%% file: HabilSiegfriedBeckus.tex
\documentclass[a4paper,11pt,reqno]{book}
\usepackage{amscd,amssymb,amsmath,latexsym,amsthm,bbm}
\usepackage{mathtools}
\usepackage{multirow}
\usepackage[ansinew]{inputenc}
\usepackage{caption}
\usepackage{enumitem}

\newcommand{\eat}[1]{}

\usepackage{pdfpages}

\usepackage{mathrsfs}

\usepackage[breaklinks=true,colorlinks=true,
linkcolor=blue,urlcolor=blue,citecolor=blue,
bookmarks=true,bookmarksopenlevel=2]{hyperref}

\usepackage{graphicx}
\DeclareMathAlphabet{\mathpzc}{OT1}{pzc}{m}{it}

\usepackage{xcolor}

\definecolor{rot}{RGB}{255,0,0}
\definecolor{gelb}{RGB}{225,225,0}
\definecolor{blau}{RGB}{0,0,255}
\definecolor{grau}{gray}{0.8}

\usepackage{tikz}
\usetikzlibrary{decorations,arrows}
\usetikzlibrary{decorations.pathmorphing}
\usepgflibrary{decorations.pathreplacing} 
\usepackage{pgf}
\usepackage{tikz-3dplot}
\usetikzlibrary{patterns}
\usepackage{tikz-cd}

\usepackage{chngcntr}
\counterwithin{figure}{chapter}

\newtheorem{theorem}{\textbf{Theorem}}[chapter]
\newtheorem{proposition}[theorem]{\textbf{Proposition}}
\newtheorem{lemma}[theorem]{\textbf{Lemma}}
\newtheorem{corollary}[theorem]{\textbf{Corollary}}
\newtheorem{example}[theorem]{\textbf{Example}}
\newtheorem{definition}[theorem]{\textbf{Definition}}

\newtheorem{remark}[theorem]{\textbf{Remark}}

\usepackage{cleveref}

\newcommand{\C}{{\mathbb C}}

\newcommand{\HH}{{\mathbb H}}

\newcommand{\N}{{\mathbb N}}
\newcommand{\Q}{{\mathbb Q}}
\newcommand{\R}{{\mathbb R}}
\newcommand{\T}{{\mathbb T}}
\newcommand{\Z}{{\mathbb Z}}

\newcommand{\Aa}{{\mathcal A}}

\newcommand{\Dd}{{\mathcal D}}
\newcommand{\Ee}{{\mathcal E}}

\newcommand{\Kk}{{\mathcal K}}
\newcommand{\Ll}{{\mathcal L}}
\newcommand{\Mm}{{\mathcal M}}

\newcommand{\Rr}{{\mathcal R}}
\newcommand{\Tt}{{\mathcal T}}

\newcommand{\Ss}{{\mathcal S}}
\newcommand{\Uu}{{\mathcal U}}
\newcommand{\Vv}{{\mathcal V}}

\newcommand{\Zz}{{\mathcal Z}}

\newcommand{\bs}{{\mathscr B}}

\newcommand{\oB}{\overline{B}}

\newcommand{\tepsilon}{{\widetilde{\varepsilon}}}

\newcommand{\inv}{{\mathscr S}}				
\newcommand{\Minv}{{M_\inv}}				
\newcommand{\Orb}{{\mathit Orb}}				
\newcommand{\Stab}{{\mathit Stab}}				
\DeclareMathOperator{\dist}{dist}               

\newcommand{\Del}{{\mathit Del}(r_0,R_0)}				
\newcommand{\DelL}{{\mathit Del}(r_0,R_0,L)}			
\newcommand{\wDel}{{\mathit Del}(r_0,R_0,\sigma)}		
\newcommand{\wPat}{{\mathcal U}(r_0,\sigma)}			
\newcommand{\wInv}{\inv_\sigma}							
\newcommand{\wInvT}{\inv_\sigma^e}						

\newcommand{\Pat}{\mathit{Pat}}							

\newcommand{\OmTabl}{\Omega_{\mathrm{table}}}			

\newcommand{\hull}{\textrm{hull}}		
\newcommand{\tran}{\textrm{tran}}		

\DeclareMathOperator{\supp}{supp}				

\newcommand{\dw}{d_{wD}}						
\newcommand{\wdelt}{\delta_{wD}}				
\newcommand{\dH}{d_H}							
\newcommand{\dJ}{\delta_H}						
\newcommand{\dAG}{d_{\Aa}}						

\newcommand{\Hamo}{{\mathit H_{\alpha,V,\theta}^{\textrm{A}}}}				
\newcommand{\HamoB}{{\mathit H_{\beta,V,\theta}^{\textrm{A}}}}				

\newcommand{\sigAMO}{\sigma_{\alpha,V}^{\textrm{A}}}				

\newcommand{\Hkoha}{{\mathit H_{\alpha,V,\theta}^{\textrm{K}}}}				

\newcommand{\AMO}{{\mathit H_{\alpha,V,0}^{\textrm{A}}}}				
\newcommand{\Koh}{{\mathit H_{\alpha,V,0}^{\textrm{K}}}}				
\newcommand{\Kohbet}{{\mathit H_{\beta,V,0}^{\textrm{K}}}}				

\newcommand{\Kohk}{{\mathit H_{\alpha_k,V,0}^{\textrm{K}}}}				

\newcommand{\Co}{\mathscr{C}}					
\newcommand{\co}{\textbf{c}}					

\newcommand{\spec}{{\mathit spec}}				

\newcommand{\Fib}{{\mathit H_{\phi,V,0}^{\textrm{K}}}}				

\newcommand{\gap}{\mathfrak{g}}				

\newcommand{\strict}{\subseteq_{\text{strict}}} 	
\newcommand{\ind}{\textrm{ind}_{\textrm{rel}}^A} 	

\newcommand{\ol}[1]{\overline{#1}}								
\global\long\def\set#1#2{\left\{  #1~:~#2\right\}  }			

\usepackage{titlesec}
\titleformat{\chapter}[hang]{\huge\bfseries}{\thechapter\hspace{7mm}}{0pt}{\huge\bfseries}

\allowdisplaybreaks

\makeatletter 
\let\orgdescriptionlabel\descriptionlabel 
\renewcommand*{\descriptionlabel}[1]{%
\let\orglabel\label   
\let\label\@gobble   
\phantomsection   
\edef\@currentlabel{#1}%
\let\label\orglabel   
\orgdescriptionlabel{#1}%
} 
\makeatother 

\makeatletter
\newcommand{\pushright}[1]{\ifmeasuring@#1\else\omit\hfill$\displaystyle#1$\fi\ignorespaces}
\newcommand{\pushleft}[1]{\ifmeasuring@#1\else\omit$\displaystyle#1$\hfill\fi\ignorespaces}
\makeatother

\usepackage{geometry}
\usepackage[
acronym,	
toc			
]
{glossaries}

\newglossary[slg]{symbolslist}{syi}{syg}{ }

\protect

\makeglossaries

\newcommand{\HabilTitel}[2]{%
  \thispagestyle{empty}

  \begin{center}

    {\Large Institut für Mathematik\par}
    \vspace{0.8cm}

    {\large Wissenschaftsdisziplin\par}	
    {\large Analysis\par}
    \vspace{1.0cm}

    \vspace*{\stretch{1}}
	{\parindent0cm \rule{\linewidth}{.7ex}}
	\begin{flushright}
		\vspace*{\stretch{1}}
		\sffamily\bfseries\Huge
		#1\\
		\vspace*{\stretch{1}}
	\end{flushright}
	\rule{\linewidth}{.7ex}
    \vspace{1.0cm}
    
    {\Large Habilitationsschrift\par}
    \vspace{0.8cm}
	
	{\Large zur Erlangung des akademischen Grades\par}
	\vspace{0.2cm}	
	{\Large\it doctor rerum naturalium habilitatus\par}
	{\Large (Dr. rer. nat. habil.)\par}
	\vspace{1.2cm}

    {\large eingereicht an der\par}	
    \vspace{0.2cm}
    {\large Mathematisch-Naturwissenschaftlichen Fakult\"at\par}
    {\large der Universit\"at Potsdam\par}
    \vspace{1.5cm}

    {\large eingereicht von\par}
    \vspace{0.2cm}
    {\Large\bfseries #2\par}
    \vspace{1.2cm}

  \end{center}
	\vfill
	
    {\large Datum des Kolloqiums: Potsdam, den 02.03.2026\par}
    {\large Datum der Probevorlesung: Potsdam, den 30.04.2026\par}
  \newpage
  \thispagestyle{empty}

  \vspace*{\stretch{1}}
  \vfill
  \noindent
  \begin{tabular}{@{}ll}
    Dekan: & Prof. Dr. Ralph Gräf \\[0.4cm]
    Gutachter: & Prof. Dr. David Damanik\\
    					& Prof. Dr. Maximilian Lein\\
    					& Prof. Dr. Hermann Schulz-Baldes
  \end{tabular}

  \cleardoublepage
}

\begin{document}

\pagenumbering{roman} 
\HabilTitel
  {The butterflies' effects}						
  {Dr.~rer.~nat. Siegfried Beckus}       			
\fontsize{12}{14}
\selectfont

\clearpage
\thispagestyle{empty}

\section*{Abstract}

This work studies spectral properties of Schrödinger operators in the context of aperiodic order, using weighted Delone sets to explore the interplay between the underlying dynamics and spectral properties. We study parameter-dependent families interpolating between periodic and aperiodic regimes, whose spectra form so-called spectral butterflies. These reflect fractal and self-similar structures of the spectra.

We review existing results, introduce additional examples, and establish new connections between works in the literature. The framework is largely dimension-independent and extends to non-Abelian groups and more general settings.

\vspace{2em}

\section*{Zusammenfassung}

Diese Arbeit studiert spektrale Eigenschaften von Schrö\-dinger\-operatoren im Kontext aperiodischer Ordnung, wobei gewichtete Delone-Mengen verwendet werden, um das Zusammenspiel zwischen der zugrundeliegenden Dynamik und den Spektraleigenschaften zu untersuchen. Wir studieren parameterabhängige Familien, die zwischen periodischen und aperiodischen dynamischen Systemen interpolieren, deren Spektren sogenannte spektrale Schmetterlinge bilden. Diese spiegeln fraktale und selbstähnliche Strukturen der Spektren wider.

Wir fassen bestehende Resultate zusammen, führen zusätzliche Beispiele ein und stellen neue Verbindungen zwischen Arbeiten in der Literatur her. Der entwickelte Rahmen ist weitgehend dimensionsunabhängig und erstreckt sich auf nicht-abelsche Gruppen und allgemeinere Situationen.

\clearpage

\cleardoublepage
\tableofcontents
\markboth{Table of Contents}{Table of Contents}
\clearpage

\markboth{Acknowledgments}{Acknowledgments}
  \include{acknowledgements}

\pagenumbering{arabic}
\setcounter{page}{1}

\chapter{Introduction}
\label{chap:Intro}

This work is devoted to the spectral analysis of Schrödinger-type operators in both discrete and continuous settings. It is based on a series of joint works~\cite{BeBeCo19,BecPinc20,BecPog20,BecEli21,BecHarPog21,BecDev22,BaBeLo24,BecHarPog25,BecTak25,BaBePoTe25,BaBeBiRaTh24,BecBelTho25}, and aims to develop a unified framework for models arising in the context of aperiodic order, in particular those motivated by the theory of quasicrystals. These models are built either from Delone sets or symbolic dynamical systems, and a central goal is to understand the spectral properties of the associated  operators via suitable approximations. This leads naturally to the study of weighted Delone sets and their associated dynamical systems.

\begin{figure}[htb]
    \centering
    \includegraphics[scale=1.15,angle=-90]{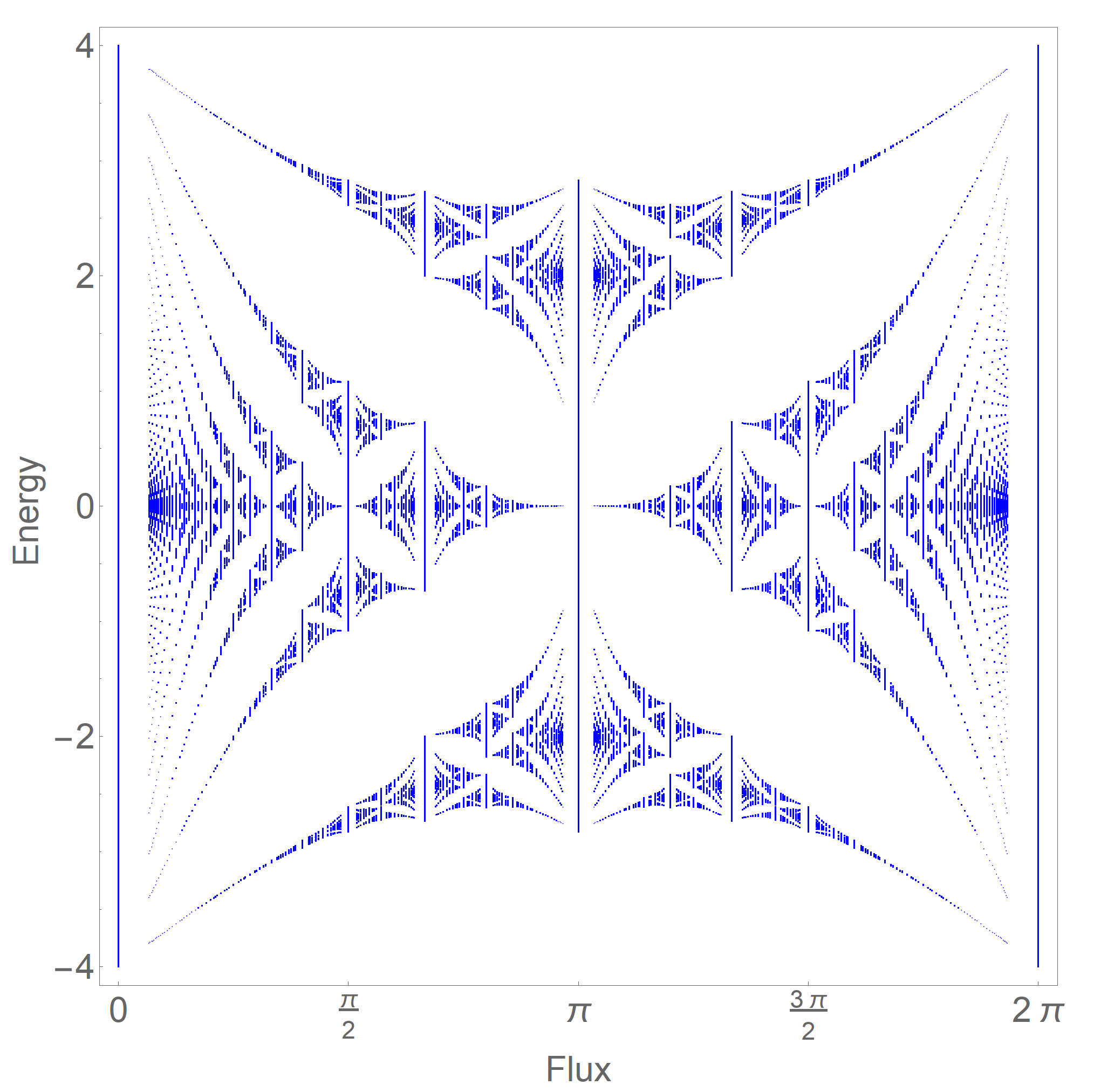}
    \captionsetup{width=0.95\linewidth}
    \caption{The Hofstadter butterfly taken from \cite{HaKaTa16}.}
    \label{fig:Hofstadter-butterfly_intro}
\end{figure}

Many of the models under consideration depend on parameters that interpolate between periodic and non-periodic dynamical systems. For periodic configurations, the spectrum of the corresponding operator can be computed numerically. Plotting the spectrum across a range of parameters reveals striking internal structures that suggest fractal or self-similar properties. The most well-known example is the Hofstadter butterfly (Figure~\ref{fig:Hofstadter-butterfly_intro}), associated with the almost Mathieu operator and closely related to the integer quantum Hall effect~\cite{ThKoNiDe82}. Numerous other spectral "butterflies" have been discovered~\cite{OK85,WilAus90,DeNittis10-thesis,BeHaJi19,BeHaJiZw21,GhaSugToh22,BarExnLip23,CedFilOng23,Nuckolls25}, each encoding complex spectral behavior.

A second prominent example is the Kohmoto butterfly \cite{KKT83,OK85}, see Figure~\ref{fig:Kohmoto-butterfly_intro}. It is connected to Sturmian dynamical systems and serves as a model for one-dimensional quasicrystals. In all these cases, a central question arises: How much spectral information about the aperiodic system is encoded in these periodic approximations? In other words, what can these "butterflies" tell us about the spectral properties of the underlying aperiodic models?

\begin{figure}[htb]
    \centering
    \includegraphics[scale=2.8]{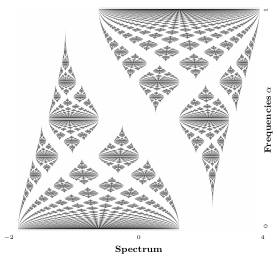}
    \captionsetup{width=0.95\linewidth}
    \caption{The Kohmoto butterfly taken from \cite{Biber22}.}
    \label{fig:Kohmoto-butterfly_intro}
\end{figure}

Several contributions addressing these questions are discussed here. We introduce the relevant concepts, establish connections to other results in the literature, and explain how the dynamical structure governs particular spectral behavior. While many of the examples considered are one-dimensional, the results themselves are often dimension-independent and extend to broader settings even to an ambient non-Abelian group.

The following three questions form the central thematic focus of this Habilitation thesis:
\begin{enumerate}[label=\textbf{(\Alph*)},labelwidth=!,align=left]
\item\label{item:StrictErg} Do there exist strictly ergodic and aperiodic dynamical systems beyond the Abelian setting?
\item\label{item:SpectrInfor} Which spectral properties are governed by the underlying structure -- such as geometry, combinatorics, or dynamics?
\item\label{item:SpectrGaps} Are all spectral gaps that are predicted by the underlying structure actually there?
\end{enumerate}
We briefly outline the central themes and methods addressed in this thesis in connection with the above questions. A more detailed discussion and outlook is provided in Chapters~\ref{chap:Dynam_LR} to~\ref{chap:discussion}.

In the study of aperiodic order, two principal classes of (weighted) Delone sets are commonly considered: cut-and-project sets and substitution systems. Some of the most prominent examples -- such as the Penrose tiling or the Fibonacci sequence -- can be described within both frameworks. In particular, the Sturmian dynamical systems mentioned earlier are one-dimensional cut-and-project sets parametrized by an irrational frequency $\alpha$ and can, under suitable conditions on the continued fraction expansion of $\alpha$, also be described via substitutions; see \cite{Fog02,BaaGri13}.

Substitution-based models are of particular interest in the spectral analysis of associated Schrödinger operators. In various one-dimensional cases, the substitution structure has been used to construct periodic approximations by iterating the substitution map to a fixed periodic configuration.

The standard literature on aperiodic order primarily focuses on dynamical systems over Abelian groups. However, recent developments have extended the construction of cut-and-project sets to certain non-Abelian groups, giving rise to so-called approximate lattices. These are studied in the context of geometric group theory; see \cite{BjHa18,BjHaPo16,BjHaPo17}. Such constructions yield point sets $D$ -- often Delone sets -- in a group $G$, where $G$ acts freely on $D$; that is, the stabilizer of $D$ is trivial. These sets are referred to as aperiodic, and the associated orbit closure (also known as the hull) defines a dynamical system.

Of particular interest are dynamical systems that are both minimal (every orbit is dense) and uniquely ergodic (there exists a unique $G$-invariant probability measure on the hull). Systems satisfying both properties are called \emph{strictly ergodic}. Specific cut-and-project models constructed in \cite{BjHaPo16} are strictly ergodic. Motivated also by spectral applications and the central question~\ref{item:StrictErg}, one aims to identify substitution systems beyond the Abelian setting that define strictly ergodic dynamical systems.

In the Abelian setting, primitive substitutions yield linearly repetitive configurations. For \( G = \R^d \), the class of linearly repetitive Delone sets is strictly ergodic; see~\cite{Dur00,LaPl03}. A Delone set is called linearly repetitive if there exists a constant \( C > 0 \) such that every patch of \( D \) supported in a translated ball of radius \( r \) occurs in every patch supported in any translated ball of radius \( Cr \).

However, when extending the analysis to general amenable and unimodular groups -- such as the lamplighter group -- balls may grow exponentially and therefore fail to form a strong F\o lner sequence. This therefore requires a metric-free criterion. To this end, the notion of tempered repetitive weighted Delone sets was introduced in the joint work~\cite{BecHarPog25}. It is shown there that the associated dynamical systems are strictly ergodic for a broad class of ambient groups. The proof draws on tools from ergodic theory, including the refined quasi-tiling techniques of Ornstein and Weiss~\cite{OrnWei87,PogSch16}.
These results generalize earlier work not only by encompassing general amenable groups and tempered repetitivity, but also by admitting weighted Delone sets that do not necessarily have finite local complexity; see Section~\ref{sec:Repet_WeightDelone}.

While linearly repetitive structures have been thoroughly studied in Abelian settings, non-trivial examples in non-Abelian groups were, until recently, not available. This gap has been addressed in the joint work \cite{BecHarPog21}, where the classical theory is extended to substitution systems on homogeneous Lie groups (Section~\ref{sec:Subst-Approximations}).

Initial attempts to generalize substitutions from the Abelian setting to the Heisenberg group -- by exploiting its representation in~$\R^3$ -- were unsuccessful. The main obstacle lies in the geometric distortion that occurs when shifting geometric objects, such as cubes or balls, in the Heisenberg group. As illustrated in Example~\ref{ex:HeisenbergExample-Subst}, such translations deform the shape in ways that are difficult to control. This challenge was overcome by introducing a systematic approach: substitutions are now defined in terms of a dilation datum and a substitution datum that respects the underlying group structure. This framework enabled the construction of a broad class of strictly ergodic and aperiodic (weighted) Delone sets beyond the Abelian setting, revealing new structural phenomena. A concrete example of a substitution system on the Heisenberg group is presented in~\cite{BecHarPog21}; see also Example~\ref{ex:HeisenbergExample-Subst}. Notably, the same substitution framework also encompasses many classical Abelian substitution systems (Section~\ref{sec:Subst_Examples}).

At this point, the connection between symbolic dynamical systems and Delone sets becomes central. Both classes of dynamical systems have been extensively studied in the literature, and the notion of weighted Delone sets provides a unifying framework for them. This is elaborated in Section~\ref{sec:Weight_Delone}, where we introduce a metric structure suitable for spectral applications discussed below. Although it is folklore that the space of (weighted) Delone sets can be equipped with a metric, we construct a specific metric under which the group action becomes Lipschitz continuous (Definition~\ref{def:LipschitzAction}). The construction is outlined in Section~\ref{sec:Weight_Delone}, and full details are deferred to Appendix~\ref{App:Weight_Delone}. This framework further allows us to characterize convergence of the underlying dynamical systems in terms of convergence of patches for general weighted Delone sets, see Corollary~\ref{cor:Charact_Converg_hulls_Patches}.

Within the author's PhD, the convergence of dynamical systems was characterized in terms of the convergence of spectra of associated operators. This was established in the joint work~\cite{BBdN18} and further extended to the setting of groupoid $C^\ast$-algebras. The proof relied on $C^\ast$-algebraic techniques and is particularly notable as it allows one to reduce spectral convergence to the convergence of the underlying dynamical structure.

This naturally leads to the central question~\ref{item:SpectrInfor}, under which we explore several key insights throughout this work.

In the joint work~\cite{BecPog20}, we proved that the convergence of dynamical systems also induces semicontinuity properties for associated invariant measures. In particular, this implies the convergence of the density of states and the autocorrelation measure, provided the limiting system is uniquely ergodic (see Section~\ref{sec:SemiContMeasure}). One of the main objectives of that work was to explore periodic approximations in higher dimensions, with a particular emphasis on cut-and-project sets. This was achieved by smoothing the characteristic function of the corresponding window. As a result, the convergence of the underlying dynamical systems also entails the convergence of relevant measure-theoretic quantities. It is worth emphasizing that the underlying geometry and combinatorics of the configurations are incorporated into the notion of dynamical convergence.

Having the spectral characterization in mind, a natural question arises: do periodic approximations exist, and if so, how fast do they converge? A review of existing spectral results shows that, for substitution systems, a common approach is to iterate the substitution rule on a fixed periodic configuration. Sufficient criteria for spectral convergence in the Hausdorff metric have been established for one-dimensional substitution systems~\cite{BBdN20} and for block substitutions on~$\Z^d$~\cite{Bec16}.

In the joint work~\cite{BaBePoTe25}, we address the general question of when the iterative application of a substitution to a configuration leads to convergence of the associated dynamical systems (see Section~\ref{sec:Subst-Approximations}). Building on the substitution framework developed in~\cite{BecHarPog21}, we establish a general characterization of convergence in terms of so-called substitution graphs. Notably, we provide a necessary and sufficient condition for the convergence based on the structure of these graphs.

A key advantage of this approach is that convergence can be verified algorithmically by analyzing these finite graphs. Moreover, we establish explicit convergence rates for the dynamical systems (Theorem~\ref{thm:Rate_Convergence_Substitutions}). Combined with recent results~\cite{BeBeCo19,BecTak25} discussed below, this also yields exponential convergence rates for the spectra (see Theorem~\ref{thm:Rate_Convergence_Substitutions_spectrum}).

Although the framework is formulated in symbolic terms, the underlying methods can be adapted to specific Delone sets. This approach is illustrated by the case of the octagonal tiling, as discussed in Section~\ref{sec:Octagonal_Penrose}.

In an ongoing project~\cite{BaBePoTe25-Spec}, the above characterization is used to demonstrate that certain natural choices of periodic approximations may fail to converge. In contrast to the one-dimensional case, new phenomena arise in higher dimensions, where such poorly chosen approximations can introduce spectral defects that alter both the essential spectrum and its Lebesgue measure.
This particularly shows that a more careful choice of periodic approximations is necessary starting from dimension two.

To study such spectral consequences, we focus on dynamically-defined operators (Chapter~\ref{chap:WorldButterflies}). Section~\ref{sec:SpectralEstimates} discusses the joint works~\cite{BeBeCo19,BecTak25} proving particular spectral estimates. More precisely, we established qualitative estimates for the Hausdorff distance between spectra in terms of the distance between the underlying dynamical systems. While spectral convergence was already known from earlier results~\cite{BBdN18} using $C^\ast$-algebraic techniques, these methods did not provide access to explicit convergence rates.

The works~\cite{BeBeCo19} and~\cite{BecTak25} approach this problem from entirely different techniques. The former focuses on symbolic dynamical systems and adapts techniques from~\cite{CoPu12,CoPu15}, including Lipschitz partitions of unity tailored to regions where patches coincide, combined with resolvent estimates. In contrast,~\cite{BecTak25} introduces a unifying method based on approximate eigenfunctions. This approach encompasses classical examples such as the almost Mathieu operator, the Kohmoto model, limit periodic potentials and Schrödinger operators assigned with skew-shifts.

Moreover, \cite{BecTak25} provides new conceptual insight into the origin of the observed square-root behavior in the spectral estimates. Specifically, it stems from the interplay between the amenability of the underlying group and the regularity of the operator coefficients, see Remark~\ref{rem:Optimization-Regularity-Amenable}. It is worth emphasizing that the qualitative behavior of these estimates is optimal, as demonstrated in the context of the Hofstadter butterfly~\cite{BeRa90} and the Kohmoto butterfly~\cite{BecBelTho25}.

These results apply to a broad class of operators and recover known estimates, such as those for the Hofstadter butterfly~\cite{AvrMouSim90} (Section~\ref{sec:Hofstadter}). Moreover, they lead to new spectral bounds for substitution systems on homogeneous Lie groups (Section~\ref{sec:Subst-Approximations}, based on the joint work~\cite{BaBePoTe25}), for the Kohmoto butterfly (Section~\ref{sec:Kohmoto}, based on the joint work~\cite{BecBelTho25}), and for the unitary almost Mathieu operator (Section~\ref{sec:Unitary-AMO}). For the latter, minor generalizations of the proof in~\cite{BecTak25} are required; these are provided in Appendix~\ref{App:SpectralEstimates} for completeness.

These spectral approximation results also allow for more refined structural insights. For instance, the analysis of spectral defects in the Kohmoto butterfly, as presented in Section~\ref{sec:Kohmoto}, is based on~\cite{BecBelTho25}. Numerical simulations in Figure~\ref{fig:Kohmoto-butterfly_intro} suggest that such defects occur at rational frequencies; see also Figure~\ref{fig:Defects-Kohmoto-butterfly} on page~\pageref{fig:Defects-Kohmoto-butterfly}. This was originally pointed out in~\cite[Rem.~3]{BIT91}. A rigorous proof of this observation is provided in~\cite{BecBelTho25}. The result relies not only on spectral convergence but also on a symbolic encoding of the periodic approximations in the Kohmoto model, discussed further below.

While the main examples discussed so far are one-dimensional, the underlying methods extend beyond this setting. In particular, they apply to more general groups such as the lamplighter group (see Section~\ref{sec:Lamplighter}), an amenable group with exponential growth. The details for proving the corresponding spectral estimates -- using the methods of \cite{BecTak25} -- are provided in Appendix~\ref{App:SpectralEstimates_Lampl}.

The previous discussion relates directly to the theme of question~\ref{item:SpectrInfor}. Moreover, we investigated criteria in the joint works~\cite{BecPinc20,BecDev22} linking the spectrum of an operator, viewed as a set, to the growth properties of its generalized eigenfunctions. These results apply to both continuous and discrete settings, including operators defined on Delone graphs; see Section~\ref{sec:EllOp_RiemManif}. A key motivation was the conjecture posed in~\cite[Conj.~9.9]{DevFraPin14}, which asks whether a generalized eigenfunction bounded by the Agmon ground state implies that the corresponding eigenvalue lies in the spectrum. This question was affirmatively answered in~\cite{BecPinc20}.

This result is closely related to so-called Shnol'-type theorems~\cite{Shnol57,Sim81}, where the growth of generalized eigenfunctions is controlled by the underlying geometry. In~\cite{BecDev22}, we provided a unified treatment of both perspectives, combining them via a Weyl sequence argument, suitable cut-off functions, and Caccioppoli-type estimates.

In~\cite{BecEli21}, we also studied the essential spectrum of Schrödinger operators on graphs. Under appropriate geometric conditions, we proved that the essential spectrum coincides with the set of generalized eigenvalues corresponding to bounded generalized eigenfunctions, using again a Shnol'-type argument; see Section~\ref{sec:DiscretOp_generEigenfct}.

We now turn to the final central question~\ref{item:SpectrGaps}, which is discussed in detail in Chapter~\ref{chap:DTMP} and brings us back to the spectral butterflies introduced at the beginning. Building on the previous developments, a natural question arises: How can the spectral approximation schemes discussed above be applied to analyze the spectral properties of the corresponding operators in the aperiodic regime? Furthermore, what insights do these methods offer into the structure of the associated butterflies? As previously indicated, the dynamical approach enables new structural results for the Kohmoto butterfly.

One is particularly interested in the aperiodic systems within the Kohmoto model namely, those operators associated with an irrational frequency. These operators are also referred to as \emph{Sturmian Hamiltonians}, since the underlying dynamical systems are Sturmian whenever the frequency is irrational.

In the joint work~\cite{BaBeLo24}, periodic approximations of irrational frequencies, together with a symbolic encoding of the spectral bands (see Section~\ref{sec:Types_Spectral_Bands}), are used to represent the spectrum as the boundary of a spectral tree (see Section~\ref{sec:SpectralTree}). This construction is inspired by a coding scheme introduced by Raymond~\cite{Raym95} for large coupling constants that control the amplitude of the potential. His approach was revisited and extended in~\cite{BaBeBiRaTh24}, incorporating recent advances from~\cite{BaBeLo24}. A key achievement of~\cite{BaBeLo24} was the extension of this symbolic encoding to all non-zero coupling constants.

Together with the spectral tree representation, these developments ultimately lead to a resolution of the \emph{dry ten Martini problem} (DTMP) for Sturmian Hamiltonians, as detailed in Section~\ref{sec:DTMP_Sturmian}. The DTMP originated in the study of the almost Mathieu operator (the Hofstadter butterfly) and was addressed in 1981 by Kac: \emph{Are all spectral gaps there?} Over the past decades, this question has been extended to other models, including the Sturmian Hamiltonians. A more comprehensive historical and mathematical account is provided in Chapter~\ref{chap:DTMP}.

Spectral gaps are the connected components of the complement of the spectrum of an operator -- visually represented as white regions on each horizontal line in the Kohmoto or Hofstadter butterfly. The set of spectral gaps is countable, and they can be indexed via the associated integrated density of states. In the case of dynamically-defined operators, the underlying dynamical system determines the set of admissible gap labels, as predicted by the Gap Labeling Theorem. With this in mind, Kac's question can be reformulated as follows: For a fixed irrational frequency, does there exist a spectral gap in the spectrum for each admissible gap label?

\begin{figure}[htb]
    \centering
    \includegraphics[scale=0.37]{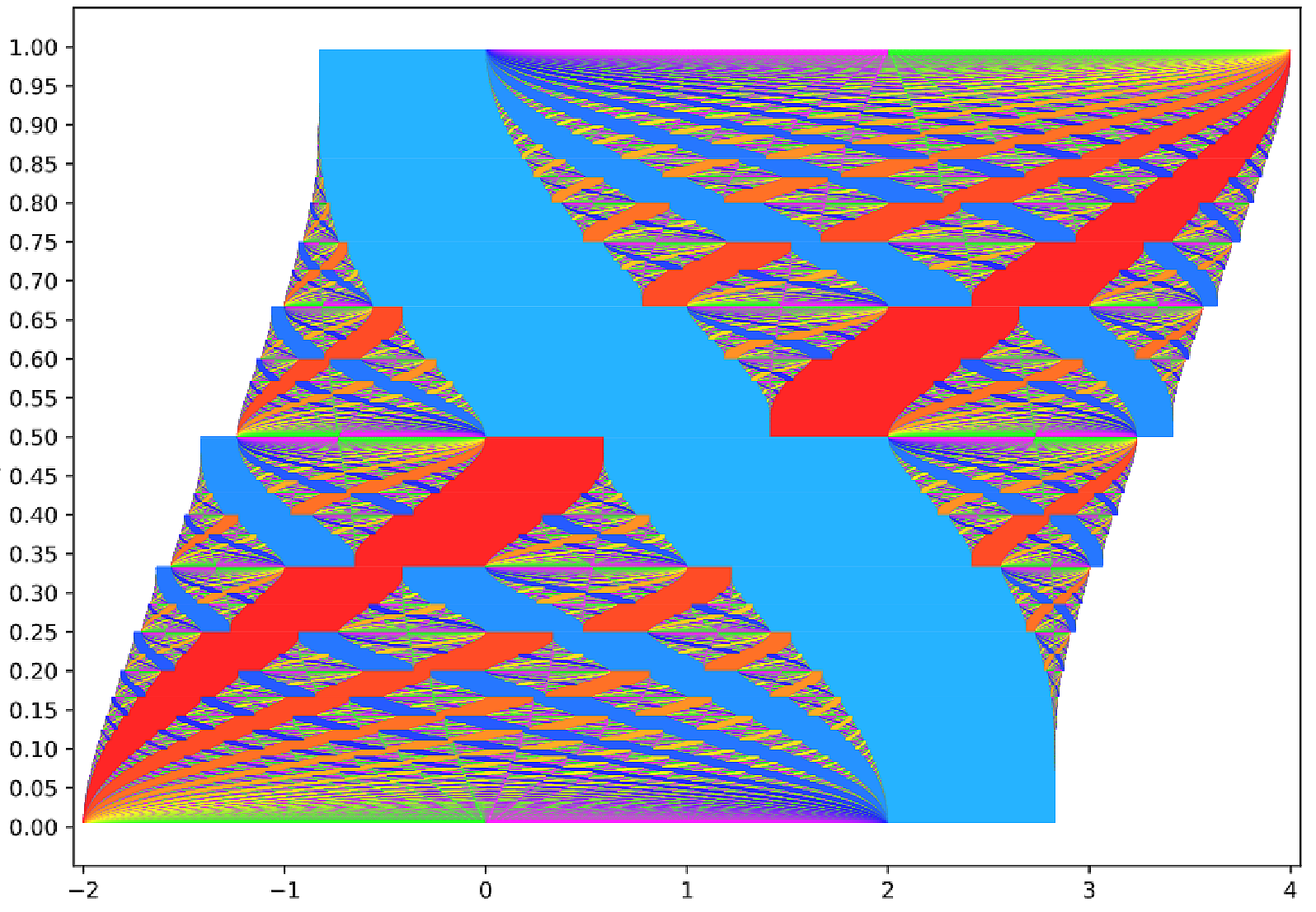}
    \captionsetup{width=0.95\linewidth}
	\caption{This figure was created in collaboration with Ram Band as part of the joint project~\cite{BanBec25}. For each rational value \( \alpha = \tfrac{p}{q} \), spectral gaps are indexed by their gap index, and each index is assigned a distinct color, see a more detailed discussion in Section~\ref{sec:DTMP_Sturmian}. We use the same color scheme as in~\cite{OsaAvr01}. 
	}
    \label{fig:ColoredKohmoto_intro}
\end{figure}

Significant progress towards resolving the original DTMP for the almost Mathieu operator has been made over the past decades~\cite{ChElYu90,Puig04,AvJit10-JEMS,LiYu15}, and most recently by Avila, You and Zhou in~\cite{AviYouZho23}. For the Sturmian Hamiltonians, the DTMP remained open for more than 30 years. A first major advance was achieved by Raymond~\cite{Raym95}, who proved that all gaps are open for sufficiently large coupling constants. For the special case where the frequency equals the inverse of the golden ratio, Damanik and Gorodetski~\cite{DamGor11} showed that all predicted gaps are open for sufficiently small couplings; this model is known as the Fibonacci Hamiltonian. Mei~\cite{Mei14} extended this result to irrational frequencies with eventually periodic continued fraction expansions. Later, Damanik, Gorodetski, and Yessen~\cite{DaGoYe16} demonstrated -- among other remarkable results -- that all spectral gaps are open in the Fibonacci case. 
All of these frequencies satisfy a Diophantine condition, and it remained unclear at that stage how to treat general irrational frequencies beyond this class.

A central difficulty arises in the small coupling regime, where the kinetic term dominates the potential, and spectral bands begin to overlap. In this setting, it was unclear whether a symbolic encoding of the spectrum was still feasible, or how to treat general irrational frequencies in the context of the DTMP. Both challenges were overcome in the joint work~\cite{BaBeLo24}, which provided a complete resolution of the DTMP for Sturmian Hamiltonians. Key innovations in that work include:
\begin{itemize}
  \item a simultaneous analysis of the space of finite continued fraction expansions, rather than focusing on a single irrational frequency;
  \item the use of eigenvalue interlacing arguments to control the positions and overlaps of spectral bands.
\end{itemize}

It is worth emphasizing the consequences of this result: no matter how small the amplitude of the potential term is, all spectral gaps open instantaneously as soon as the amplitude becomes nonzero. As a result, the spectrum becomes a Cantor set. In contrast, when the amplitude vanishes, the spectrum is an interval, and the associated spectral measure is purely absolutely continuous.

The topological indices of the almost Mathieu operator were used to color the Hofstadter butterfly in~\cite{OsaAvr01}. Adopting a similar strategy, we assigned colors to the spectral gaps of the Kohmoto butterfly according to their gap indices, as shown in Figure~\ref{fig:ColoredKohmoto_intro}. This visualization complements Figure~\ref{fig:Kohmoto-butterfly_intro}, where the spectral bands are displayed in black. Further analysis specific to the Kohmoto model is needed and is carried out in the ongoing project~\cite{BanBec25}.

Finally, let us emphasize that the new methods developed in~\cite{BaBeLo24} open the door to further investigations into the fractal nature of the Kohmoto butterfly. In a joint project with Ram Band and Yannik Thomas, we explore whether the Kohmoto butterfly exhibits self-similarity, using the spectral tree as a central tool. As noted at the beginning of this introduction, such questions lie at the heart of the study of spectral butterflies. Similar questions have also been studied in the physics literature, for instance in the context of the Hofstadter butterfly, see e.g.~\cite{SatWil20,Sat25}.

A more detailed discussion and outlook of the previously discussed material is provided in Chapters~\ref{chap:Dynam_LR} to~\ref{chap:discussion}.

The following joint works form the basis of this Habilitation thesis, which are provided in the Chapter~\ref{chap:manuscripts}. In each of them, the contributions were equally shared among the authors.
\footnote{This reflects the standard practice in mathematics, where individual contributions are typically inseparable. We refer the reader to the culture of research statement by the American Mathematical Society \cite{AMS}.}

\begin{enumerate}
\item \cite{BecPinc20}: S.~Beckus and Y.~Pinchover, \textit{Shnol-type theorem for the Agmon ground state}, J. of Spectr. Theory, \textbf{10}, no. 2, 355 -- 377 (2020), DOI: 10.4171/JST/296.
\item \cite{BecPog20}: S.~Beckus and F.~Pogorzelski, \textit{Delone dynamical systems and spectral convergence}, Ergodic Theory Dynam. Systems, \textbf{40} (2020), no. 6, 1510 -- 1544, DOI: 10.1017/etds.2018.116. 
\item \cite{BecHarPog21}: S.~Beckus, T.~Hartnick and F.~Pogorzelski, \textit{Symbolic substitution systems beyond abelian groups}, Preprint, arXiv:2109.15210 (2021)
\item \cite{BaBeLo24}:  R.~Band, S.~Beckus and R.~Loewy,  \textit{The dry ten Martini problem for Sturmian Hamiltonians}, Preprint, arXiv:2402.16703 (2024).
\item \cite{BaBePoTe25}: R.~Band, S.~Beckus, F.~Pogorzelski and L.~Tenenbaum, \textit{Spectral approximation for substitution systems}, Journal d'Analyse Mathématique (2026), DOI: 10.1007/ s11854-026-0445-0.
\item \cite{BecTak25}: S.~Beckus and A.~Takase, \textit{Spectral estimates of dynamically-defined and amen\-able operator families}, J. Spectr. Theory 15 (2025), no. 2, 563 -- 610, DOI: 10.4171/ JST/554.
\end{enumerate}

Further relevant joint works are~\cite{BeBeCo19,BecEli21,BecDev22,BecHarPog25,BaBeBiRaTh24,BecBelTho25}. We also briefly discuss their relations in the detailed discussion in Chapters~\ref{chap:Dynam_LR} to~\ref{chap:discussion}.

\chapter{Manuscripts}
\label{chap:manuscripts}

This chapter contains the following articles:
\begin{enumerate}
\item \cite{BecPinc20}: S.~Beckus and Y.~Pinchover, \textit{Shnol-type theorem for the Agmon ground state}, J. of Spectr. Theory, \textbf{10}, no. 2, 355 -- 377 (2020), DOI: 10.4171/JST/296.
\item \cite{BecPog20}: S.~Beckus and F.~Pogorzelski, \textit{Delone dynamical systems and spectral convergence}, Ergodic Theory Dynam. Systems, \textbf{40} (2020), no. 6, 1510 -- 1544, DOI: 10.1017/etds.2018.116. 
\item \cite{BecHarPog21}: S.~Beckus, T.~Hartnick and F.~Pogorzelski, \textit{Symbolic substitution systems beyond abelian groups}, Preprint, arXiv:2109.15210 (2021)
\item \cite{BaBeLo24}:  R.~Band, S.~Beckus and R.~Loewy,  \textit{The dry ten Martini problem for Sturmian Hamiltonians}, Preprint, arXiv:2402.16703 (2024).
\item \cite{BaBePoTe25}: R.~Band, S.~Beckus, F.~Pogorzelski and L.~Tenenbaum, \textit{Spectral approximation for substitution systems}, Journal d'Analyse Mathématique (2026), DOI: 10.1007/ s11854-026-0445-0.
\item \cite{BecTak25}: S.~Beckus and A.~Takase, \textit{Spectral estimates of dynamically-defined and amen\-able operator families}, J. Spectr. Theory 15 (2025), no. 2, 563 -- 610, DOI: 10.4171/ JST/554.
\end{enumerate}

\chapter{The dynamical and ergodic theoretical perspective}
\label{chap:Dynam_LR}

Throughout this work, $G$ denotes a locally compact, second countable, Hausdorff (lcsc) group. For $g,h\in G$, $gh$ denotes the group multiplication of the two elements. Further, $e\in G$ is the neutral element, i.e. $eg=ge=g$ for all $g\in G$. Then $g^{-1}$ denotes the inverse of $g\in G$, namely $gg^{-1}=g^{-1} g =e$. Following \cite[Prop.~2.A.9, Thm.~2.B.4]{CorHar16} (as well as a discussion in \cite[Sec.~2.1]{BecHarPog25}), there exists a metric $d_G:G\times G\to[0,\infty)$ on the lcsc group $G$ that is
\begin{itemize}
\item {\em continuous} (and so $d_G$ generates the topology of $G$),
\item {\em left-invariant}, i.e. $d_G(gh_1,gh_2) = d_G(h_1,h_2)$ for $g,h_1,h_2\in G$,
\item {\em proper}, i.e. \emph{closed balls} $\ol{B}(g,r) := \set{h\in G}{d_G(g,h)\leq r}$ of radius $r\geq 0$ and center $g\in G$ are compact,
\item {\em locally bounded}, i.e. every $g\in G$ admits a neighborhood with finite diameter with respect to $d_G$.
\end{itemize}
A metric  $d_G$ satisyfing the previous conditions is called {\em adapted}.
We denote by 
\[
B(g,r):= \set{h\in G}{d_G(g,h)< r}
\] 
the {\em open ball} of radius $r\geq 0$ and center $g\in G$. Furthermore, let $B(r):=B(e,r)$ and $\oB(r):=\oB(e,r)$.

The central object in this work are continuous group actions on a compact metric space -- so-called topological dynamical systems --  and the associated operators, see Chapter~\ref{chap:WorldButterflies} for details.
A {\em (topological) dynamical system} is a triple $(Z,G,\tau)$ where 
\begin{itemize}
\item $Z$ is a compact metric space with metric $d$, 
\item $G$ is an lcsc group and 
\item $\tau:G\times Z\to Z$ is a left-continuous action of the group $G$ on $Z$, namely for all $x\in Z$ and $g,h\in G$,
\[
	\tau(e,x)=x
\qquad\textrm{and}\qquad
	\tau\big(g,\tau(h,x)\big) = \tau(gh,x).
\]
\end{itemize}
Whenever the action is clear from the context, we use the short notation $gx=\tau(g,x)$ for $g\in G$ and $x\in Z$.
Like in groups $B(x,r) = \set{y\in Z}{d(x,y)<r}$ is the \emph{open ball} in $Z$ and $\oB(x,r)= \set{y\in Z}{d(x,y)\leq r}$ the \emph{closed ball}.

In general it is enough to know that the group action on $Z$ is continuous. For the spectral application, discussed in Chapter~\ref{chap:WorldButterflies}, one often requires more. For this purpose so-called Lipschitz-continuous actions are of particular interest \cite{BecTak25}.

\begin{definition}
\label{def:LipschitzAction}
Let $(Z,G,\tau)$ be a topological dynamical system and $d$ be a metric on $Z$. Then the dynamical system is called {\em Lipschitz continuous} if there exists a constant $C_\tau>0$ such that
\[
d(gx,gy) \leq \big( C_\tau d_G(e,g) + 1 \big) d(x,y),
	\qquad \text{for all }\; x,y\in Z, \, g\in G.
\]
\end{definition}

In particular, for fixed $g \in G$, the map $\tau(g,\cdot) : Z \to Z$ is Lipschitz continuous, and we require that its Lipschitz constant grows at most linearly with the distance of $g$ from the identity element $e \in G$. 
Note that $C_\tau>0$ is independent of $x,y\in Z$ and $g\in G$.
This assumption is satisfied in most examples considered in this work, see Proposition~\ref{prop:wDel_groupAction_Lipschitz}. The restriction to linear growth is mainly a practical one, as it facilitates explicit estimates of spectral distances. A relaxation of this condition -- such as sublinear growth or merely Hölder continuity of the group action -- will lead to modified, typically weaker, spectral estimates; see e.g. \cite[Thm.~2.4]{BecTak25}.

A subset $Y\subseteq Z$ is called {\em ($G$-)invariant} if $gY := \set{gx}{x\in Y} \subseteq Y$ for all $g\in G$. Particular examples are the orbits of an element $x\in Y$, namely 
\[
\Orb(x) := \set{gx}{g\in G}.
\] 
It is straightforward to check that the orbit closure $Y:=\ol{\Orb(x)}\subseteq Z$ is $G$-invariant, closed (and so compact as $Z$ is compact) and non-empty. Thus, $(Y,G,\tau)$ is again a topological dynamical system (where $\tau$ is restricted to $Y$) -- a \emph{dynamical subsystem} of $(Z,G,\tau)$.

A central object in this work is the set $\inv$ of all {\em dynamical subsystems} of $(Z,G,\tau)$ is defined by
\[
\inv := \set{Y\subseteq Z}{Y \textrm{ invariant, closed and non-empty}}.
\]
As discussed before, we have $\ol{\Orb(x)}\in\inv$ for every $x\in Z$. However, there might be more elements in $\inv$ that are not equal to the orbit closure of a single element $x\in Z$.

Clearly, $\inv\subseteq \Kk(Z) := \{K\subseteq Z \textrm{ compact}\}$ and the latter is naturally equipped with the Hausdorff metric \cite{CastaingValadier77,vMil01,Heij20} defined by
\begin{equation}
\label{eq:Hausdorff_Metric}
\delta_H:\Kk(Z)\times\Kk(Z)\to[0,\infty),\quad
	\delta_H(X,Y) := \max\{ \sup_{x\in X} d(x,Y) , \sup_{y\in Y} d(y,X) \},
\end{equation}
where $d(x,Y) := \inf_{y\in Y} d(x,y)$ denotes the distance from the point $x$ to the compact set $Y$. Since $Z$ is compact, the Hausdorff topology coincides with the Chabauty-Fell topology \cite{Cha50,Fel62}. 
For a detailed discussion of this equivalence, we refer the reader to \cite[Sec.~2.2]{Bec16} and references therein.

\begin{proposition}[\cite{BBdN18}, Prop.~4]
Let $(Z,G,\tau)$ be topological dynamical system and $\delta_H$ be the Hausdorff metric on $\Kk(Z)$ induced by the metric $d$ on $Z$. Then $(\inv,\delta_H)$ is a compact metric space.
\end{proposition}

In \cite{Bec16,BBdN18,BBdN20}, it was discovered that equipping $\inv$ with the induced Hausdorff metric, the convergence of dynamical subsystems implies the convergence of the associated spectra. In fact, it turns out that this convergence of dynamical systems is equivalent to the requirement that a suitable class of associated operators has spectrum varying continuously under continuous changes of the underlying dynamical system. This relationship is made more precise in Section~\ref{sec:SpectralEstimates} and forms the starting point of the present work.

Building on this, we show that
\begin{itemize}
  \item the Hausdorff distance of spectra is bounded by the distance between the corresponding dynamical systems, see Section~\ref{sec:SpectralEstimates} based on \cite{BeBeCo19,BecTak25}; 
  \item associated measure-theoretic quantities converge, see Section~\ref{sec:SemiContMeasure} based on \cite{BecPog20};
  \item this framework can be used to construct so-called periodic approximations (Section~\ref{sec:Subst-Approximations} based on \cite{BaBePoTe25}), which in turn allow one to analyze spectral properties of more complex systems, such as quasicrystals, see Section~\ref{sec:Kohmoto} and Chapter~\ref{chap:DTMP} based on \cite{BaBeLo24,BaBeBiRaTh24,BecBelTho25}.
\end{itemize}

A particular focus will be on symbolic dynamical systems defined as follows.

\begin{example}[Symbolic dynamical systems]
\label{ex:SymbDynSyst}
Let $G$ be a discrete and countable group and $\Aa$ be a finite set (called {\em alphabet}) equipped with the discrete topology. Then $\Aa^G := \{\omega:G\to\Aa\}$ equipped with the product topology is a compact metrizable space. A particular metric suitable for the spectral applications is defined later in Section~\ref{subsec:symb_syst_as_weighted_Delone}. The group $G$ acts via the shift on $\Aa^G$, namely, 
\[
\tau:G\times \Aa^G\to\Aa^G,\qquad \tau(g,\omega) := \omega\circ g^{-1},
\]
where $\circ$ denotes the composition of maps.
Thus, $(\Aa^G,G,\tau)$ is a topological dynamical system. Elements of $\inv$ for this dynamical system are also called {\em subshift}.
\end{example}

If, for instance, $G = \Z^d$, then symbolic dynamical systems are used to model solids in which the letters $a \in \Aa$ represent different atomic species positioned on the lattice $\Z^d$. In more general settings, the atoms are not restricted to lattice positions. Instead, solids are modeled by point sets in the ambient space (e.g., $G = \R^d$), where each point is assigned a label (below called a weight) representing its species.

These point sets are not arbitrary: two atoms cannot be arbitrarily close (the set must be uniformly discrete), and there should not be arbitrarily large gaps (the set must be relatively dense). Sets satisfying both conditions are called (weighted) Delone sets, which are introduced and studied in the following subsection.

We equip the space of weighted Delone sets with a specific metric making it a compact metric space and induces the various topologies commonly used in the literature. 
The ambient group (or any subgroup) acts Lipschitz continuously (Proposition~\ref{prop:wDel_groupAction_Lipschitz}) on this space of weighted Delone sets and so defines a suitable class of dynamical systems.
Moreover, we show that this framework also encompasses symbolic dynamical systems, thereby providing a unified perspective.

Before proceeding, we introduce some basic terminology from dynamical systems that will be used throughout this work. 
We are particularly interested in both periodic and non-periodic elements of the dynamical system $(Z, G, \tau)$. To formalize this distinction, we use the concept of the \emph{stabilizer} $\Stab(x)$ of an element $x \in Z$, defined by
\[
\Stab(x) := \set{g \in G}{g x = x}.
\]
We note that the stabilizer is always a subgroup of $G$.

\begin{definition}
\label{def:(non-)periodic}
Let $(Z,G,\tau)$ be a dynamical system. We call $x\in Z$ 
\begin{itemize}
\item \emph{non-periodic} if its stabilizer $\Stab(x)$ is trivial, i.e. $\Stab(x)=\{e\}$. 
\item \emph{periodic} if there is discrete subgroup $\Gamma \subseteq \Stab(x)$ and a relative compact $V\subseteq G$ such that $G:= \bigsqcup_{\gamma\in\Gamma} V\gamma$.
\end{itemize}
If $g\in\Stab(x)\setminus\{e\}$, then we say that $x$ has {\em period $g$}.
\end{definition}

Let us comment on the notion of periodicity. Suppose $\Gamma \subseteq \Stab(x)$ is a discrete subgroup, and let \( V \subseteq G \) be a relatively compact set such that
\[
G = \bigsqcup_{\gamma \in \Gamma} V\gamma,
\]
denoting a disjoint union. Then, for each \( g \in G \), there exist \( v \in V \) and \( \gamma \in \Gamma \) such that
\[
g x = (v \gamma) x = v x,
\]
using that \( \gamma \in \Stab(x) \). Since the action of \( G \) on \( Z \) is continuous and \( V \) is relatively compact (i.e. $\ol{V}$ is compact), it follows that
\[
\overline{\Orb(x)} = \set{g x}{g \in \overline{V}} = \Orb(x).
\]
In particular, the orbit of $X$ is itself close.
Such a condition reflects classical periodicity, as seen in lattice-based models where \( \Gamma \) plays the role of a lattice and \( x \) is invariant under translation by its elements. Note that \( V \) serves as a fundamental domain of the group \( \Gamma \) in \( G \); such a fundamental cell exists if \( \Gamma \) is discrete, closed, and a normal subgroup (see, e.g., \cite[Chap.~5]{HewittRoss1963}).

If $G=\Z$, then it is straightforward to check that if there is a $g\in\Stab(x)\setminus\{0\}\subseteq \Z$, then $x$ is periodic with 
\[
\Gamma:=\set{n\cdot \gamma}{n\in\Z}\qquad 
\textrm{and}\qquad 
V:=\big[0,|\gamma|\big)\cap\Z.
\]
Here $n\cdot \gamma$ denotes the product of these real numbers.
If $G\neq \Z$, this implication is false in general since $x$ can have a period in one direction but not the others. Note that the notion of periodicity varies in the literature if $G\neq \Z$ and our notion is also called strong periodicity.

Another central concept for this work is that of a dynamical system which cannot be decomposed (topologically) into a smaller dynamical subsystem:

\begin{definition}
\label{def:minimal}
A dynamical system $(Z,G,\tau)$ is called \emph{minimal} if every orbit in $Z$ is dense; that is, for all $x \in Z$, the orbit \(\Orb(x) \) is dense in $Z$.
\end{definition}

Note that $(Z,G,\tau)$ is minimal if and only if $\inv = \{Z\}$, namely, it does not contain any other dynamical subsystem than itself. In this work, we are mainly interested in a large, unifying dynamical system $(Z,G,\tau)$ which is not minimal, and we study approximations of minimal dynamical subsystems in the space $(\inv,\delta_H)$.

\section{Semicontinuity of the measures}
\label{sec:SemiContMeasure}

Let $(Z,G,\tau)$ be a topological dynamical system and denote by $\bs(Z)$ the Borel $\sigma$-algebra of $Z$.  
A measure $\mu:\bs(Z)\to[0,\infty)$ is called a \emph{probability measure} if $\mu(Z) = 1$. The set of all probability measures on $Z$ is denoted by $\Mm^1(Z)$.  
Furthermore, a measure $\mu \in \Mm^1(Z)$ is called \emph{$G$-invariant} if
\[
\mu(g^{-1}A) = \mu(A) \qquad \text{for all } A \in \bs(Z), \ g \in G.
\]
If the group $G$ is amenable, then there exists a $G$-invariant probability measure, see e.g. \cite[Chap.~1]{Paterson1988}. We postpone the definition of amenability to Appendix~\ref{App:SpectralEstimates} when we need it specifically.
Denote the set of all $G$-invariant probability measures by
\[
\Mm^1(Z,G) := \set{\mu \in \Mm^1(Z)}{\mu \text{ is } G\text{-invariant}}.
\]
By the Banach-Alaoglu theorem, $\Mm^1(Z,G)$ is compact and metrizable in the weak\(^*\)-topology induced from the duality with $C(Z)$, the space of continuous functions on $Z$.

A measure $\mu \in \Mm^1(Z,G)$ is called \emph{ergodic} if every $G$-invariant subset $A \subseteq Z$ satisfies $\mu(A) \in \{0,1\}$. That is, the measure space $(Z, \bs(Z), \mu)$ cannot be decomposed into smaller (up to $\mu$-null sets) $G$-invariant measurable subsystems. In this sense, ergodicity is the measure-theoretic analog of (topological) minimality.

Note that $\Mm^1(Z,G)$ is a convex set. In particular, the convex combination of two distinct ergodic measures on $(Z,G,\tau)$ is again a $G$-invariant measure, but not ergodic in general.  
Using the Krein-Milman theorem (see e.g. \cite[Thm.~3.15]{Rudin1991}), one can show that $\Mm^1(Z,G)$ is the closed convex hull of its extreme points, which are precisely the ergodic measures on $(Z,G,\tau)$.

In particular, if there exists only one $G$-invariant probability measure $\mu$ on $(Z,G,\tau)$, then this measure is automatically ergodic explaining the following terminology.

\begin{definition}
\label{def:uniquely_ergodic}
We call a dynamical system $(Z,G,\tau)$ \emph{uniquely ergodic} if there exists exactly one $G$-invariant probability measure on $Z$.
\end{definition}

The \emph{support} of a measure $\mu \in \Mm^1(Z,G)$ is defined as
\[
\supp(\mu) := \set{x \in Z}{\mu(f) > 0 \text{ for all } f \in C(Z) \text{ with } f \geq 0 \text{ and } f(x) > 0},
\]
where $\mu(f) := \int_Z f(x) \, \mathrm{d}\mu(x)$.

Let \( Y \in \inv \). Then every measure \( \mu \in \Mm^1(Y,G) \) can be extended to a measure \( \tilde{\mu} \in \Mm^1(Z,G) \) by setting
\[
\tilde{\mu}(A) := \mu(A \cap Y), \quad A \in \bs(Z),
\]
a construction also referred to as the \emph{canonical extension}. Via this embedding, \( \Mm^1(Y,G) \) can be viewed as a compact subset of \( \Mm^1(Z,G) \) with respect to the weak\(^*\)-topology \cite[Cor.~6.5]{Walters1982}.

We note that, since \( Z \) is compact, the weak\(^*\)-topology on \( \Mm^1(Z,G) \) is metrizable. Hence, the collection of compact subsets \( \Kk\big(\Mm^1(Z,G)\big) \) can be equipped with the induced Hausdorff metric or, equivalently, with the Chabauty-Fell topology.

This allows us to study the continuity properties of the map
\[
\Minv \colon \inv \to \Kk\big(\Mm^1(Z,G)\big), \qquad \Minv(Y) := \Mm^1(Y,G).
\]

In general, $\Minv$ is not continuous as can be seen by the following example.

\begin{example}
Consider the dynamical system $(\Aa^\Z, \Z, \tau)$ where $\Aa = \{a, b\}$. For a finite word 
$v = a_1 a_2 \ldots a_n \in \Aa^n$ (with $a_i \in \Aa$) of length $n \in \N$ and $m \in \N$, define the $m$-times concatenation of $v$ by
\[
v^m = \underbrace{vv \ldots v}_{m \text{ times}} \in \Aa^{mn}.
\]
Moreover, let $v^\infty \in \Aa^\Z$ be the two-sided infinite concatenation of the word $v$, whose first letter of $v$ is at the origin $0 \in \Z$. Then $v^\infty$ has a finite orbit and so is periodic.

Let $\omega,\rho \in \Aa^\Z$ be defined by
\[
\omega(n) := 
\begin{cases} 
    a, & n < 0, \\
    b, & n \geq 0,
\end{cases}
\qquad
\rho(n) := 
\begin{cases} 
    b, & n < 0, \\
    a, & n \geq 0,
\end{cases}
\qquad 
n\in\Z.
\]
Consider the associated subshift 
\[
\Omega := \overline{\Orb(\omega)} \cup \overline{\Orb(\rho)} = \Orb(\omega) \cup \Orb(\rho) \cup \{a^\infty, b^\infty\}.
\] 
Any ergodic measure $\mu$ on $\Omega$ is supported on an invariant subset of $\Omega$. The only minimal invariant subsets of $\Omega$ are $\{a^\infty\}$ and $\{b^\infty\}$ with associated ergodic measures $\delta_{a^\infty}$ and $\delta_{b^\infty}$. Here
\[
\delta_\omega: B(\Aa^\Z) \to \{0, 1\}, \quad \delta_\omega(A) = 
\begin{cases} 
    1, & \omega \in A, \\
    0, & \omega \notin A.
\end{cases}
\]
Thus, the invariant measures on $\Omega$ are given by the convex hull of these measures, namely
\[
\Mm^1(\Omega, \Z) = \set{t \cdot \delta_{a^\infty} + (1 - t) \cdot \delta_{b^\infty}}{t \in [0, 1]}.
\]
For two monotone increasing sequences $(m_k)_{k \in \N}, (n_k)_{k \in \N} \subset \N$ tending to infinity, define the periodic sequence
$\omega_k = (a^{m_k} b^{n_k})^\infty \in \Aa^\Z$ and its subshift $\Omega_k:=\Orb(\omega_k)\in\inv$ is periodic (it has exactly $n_k+m_k$ elements). Since $m_k\to\infty$ and $n_k\to\infty$, it is straightforward to show that $\lim_{k\to\infty}\Omega_k=\Omega$ in $\inv$ using \cite[Cor.~5.5]{BBdN20}, see also Proposition~\ref{prop:Conv_Symb-Dyn-Syst} below. The subshift $\Omega_k$ is uniquely ergodic (since it is periodic) and its ergodic measure is given by the uniform distribution
$$
\mu_k := \frac{1}{n_k+m_k} \sum_{\omega\in\Omega_k} \delta_\omega.
$$
Choose for instance $m_k:=2^k$ and $n_k:=k$. Since $(m_k)$ escapes to infinity much faster than $(n_k)$, the measure $\mu_k$ is mainly supported in a neighborhood of $a^\infty\in\Aa^\Z$. Thus, it is straightforward to check that $(\mu_k)_{k\in\N}$ converge to $\delta_{a^\infty}$ in the weak\(^*\)-topology of $\Mm^1(\Aa^\Z,\Z)$. Thus, 
\[
\lim_{k\to\infty} \Mm^1(\Omega_k,\Z) =\{\delta_{a^\infty}\} \subsetneq \Mm^1(\Omega,\Z)
\] 
follows proving that $\Minv$ is not continuous at $\Omega\in\inv$.
In general, we see that any accumulation point of measures is an invariant measure on the limiting dynamical system (a semi-continuity of the measures, see Theorem~\ref{thm:SemiContinuity_measures}) but as in this example the convex set of invariant measures is not converging in general. For the particular example described before, it is straightforward to check that any $\mu\in\Mm^1(\Omega,\Z)$ can be approximated by the invariant measures of $\Omega_k$ for a suitable choice of $(m_k)$ and $(n_k)$.
If for instance $m_k=n_k=k$ for $k\in\N$, then 
\[
\lim_{k\to\infty} \Mm^1(\Omega_k,\Z) =\left\{ \tfrac{1}{2}\delta_{a^\infty} + \tfrac{1}{2}\delta_{b^\infty} \right\} \subsetneq \Mm^1(\Omega,\Z).
\] 
Thus, $(\mu_k)$ is not even converging to an ergodic measure, in general. Thus, the range of $\Minv$ in $\Kk\big(\Mm^1(Z,G) \big)$ is not closed.
\end{example}

Keeping the previous example in mind, we show that $\Minv$ is semi-continuous. In order to do so, we recall some basic facts for the Hausdorff metric.

Let $(X,d_X)$ be a compact metric space, and let
\[
\Kk(X) := \set{A \subseteq X }{ A \text{ compact}}
\]
denote the collection of all non-empty compact subsets of \( X \).  
In eq.~\eqref{eq:Hausdorff_Metric}, the Hausdorff metric \( \delta_H \) on \( \Kk(X) \) is defined in terms of the underlying metric \( d_X \).  
The topology induced by the Hausdorff metric \( \delta_H \) on \( \Kk(X) \) coincides with the Chabauty-Fell topology; see, e.g., \cite[Sec.~2.2]{Bec16} and references therein.

The Chabauty-Fell topology \cite{Cha50,Fel62} is generated by the sets
\begin{align*}
\Uu(F) &:= \set{A \in \Kk(X)}{A \cap F = \emptyset}, \\
\Vv(O) &:= \set{A \in \Kk(X)}{A \cap O \neq \emptyset},
\end{align*}
where \( F \subseteq X \) is compact and \( O \subseteq X \) is open.

Following \cite[Lem.~1.11.2]{vMil01}, a sequence \( (A_n)_{n \in \N} \) in \( \Kk(X) \) converges to \( A \in \Kk(X) \) in the Hausdorff metric if and only if
\begin{equation}
\label{eq:char_conv_Hausdorff_metric}
A \subseteq \liminf_{n \to \infty} A_n \subseteq \limsup_{n \to \infty} A_n \subseteq A,
\end{equation}
where
\begin{align*}
\liminf_{n \to \infty} A_n &:= \set{x \in X}{
\begin{array}{l}
\text{for each } n \in \N, \text{ there exists } x_n \in A_n \\
\text{such that } \lim_{n \to \infty} x_n = x
\end{array}
}, \\
\limsup_{n \to \infty} A_n &:= \set{x \in X}{
\begin{array}{l}
\text{there exists a subsequence } (n_k)_{k \in \N}, \text{ and for each } k \in \N, \\
\text{there is an element } x_{n_k} \in A_{n_k} \text{ such that } \lim_{k \to \infty} x_{n_k} = x
\end{array}
}.
\end{align*}
The inclusion \( \liminf_{n \to \infty} A_n \subseteq \limsup_{n \to \infty} A_n \) always holds.  
Moreover, we have
\[
\limsup_{n \to \infty} A_n = \bigcap_{n \in \N} \overline{ \bigcup_{k \geq n} A_k },
\]
see, e.g., \cite[Eq.~7.15]{Heij20}.

Let $T$ be a metric space and $f:T\to\Kk(X)$. We call $f$ {\em upper semi-continuous} (respectively {\em lower semi-continuous}) at $t_0\in T$, if for all sequences $(t_n)_{n\in\N}\subseteq T$ converging to $t_0$, we have $ \limsup_{n\to\infty}f(t_n)\subseteq f(t_0)$ (respectively $f(t_0) \subseteq \liminf_{n\to\infty}f(t_n)$).

We use these considerations for the compact space $X=\Mm^1(Z,G)$ equipped with the weak\(^*\)-topology. We do not choose here a particular metric but note that since $Z$ is compact, the weak\(^*\)-topology on $\Mm^1(Z,G)$ is compact, second countable and Hausdorff. Thus, this space is metrizable. We are in particular interested in the compact, convex subsetes $\Mm^1(Y,G)$ of $\Mm^1(Z,G)$ for $Y\in\inv$. We obtain the following semi-continuity results for the map $\Minv$ proven in the joint work \cite[Thm.~I]{BecPog20}.

\begin{theorem}
\label{thm:SemiContinuity_measures}
Let $(Z,G,\tau)$ be a dynamical system. Then  the map 
\[
\Minv:\inv\to\Kk\big(\Mm^1(Z,G)\big), \qquad \Minv(Y):= \Mm^1(Y,G),
\]
is upper semi-continuous, namely if $\lim\limits_{n\to\infty}Y_n=Y$ in $\inv$, then $\limsup_{n\to\infty} \Mm^1(Y_n,G)\subseteq \Mm^1(Y,G)$ is a compact subset.
\end{theorem}

Since the presentation here deviates from the one in \cite[Thm.~I]{BecPog20}, let us shortly explain the proof.

\begin{proof}
First note that the set $\limsup_{n\to\infty} \Mm^1(Y_n,G)$ is clearly compact as an intersection of compact sets. It is also straightforward to see that this set coincides with the set of measures $\mu\in\Mm^1(Z,G)$ that are limit points of arbitrary sequences $\mu_k\in\Mm^1(Y_k,G), k\in\N$. It is worth emphasizing that $(\mu_k)_{k\in\N}\subseteq \Mm^1(Z,G)$ is not necessarily convergent in the weak\(^*\)-topology. However, since $\Mm^1(Z,G)$ is compact, it admits a limit point. That this limit point defines an invariant measure in $\Mm^1(Y,G)$ is proven \cite[Thm.~I]{BecPog20}. This relies on the principle that the support of the measure must be contained in the limiting set $Y$ if $\lim_{n\to\infty}Y_n=Y$.
\end{proof}

The reader is also referred to Proposition~\ref{prop:Semi-cont-Spectrum_app} and Remark~\ref{rem:UpperSemiCont-Spectrum} for analogous semi-continuity results of the spectrum.

The map $\Minv$ is even continuous at the points where $Y\in\Mm^1(Z,G)$ is uniquely ergodic, see \cite[Thm.~I]{BecPog20}:

\begin{corollary}
\label{cor:ContinuityMeasures_UniqueErgod}
Let $(Z,G,\tau)$ be a topological dynamical system. Then the map
\[
\Minv:\inv\to\Kk\big(\Mm^1(Z,G)\big), \qquad \Minv(Y):= \Mm^1(Y,G),
\]
is continuous in each $Y\in\inv$ that is uniquely ergodic. Specifically, if $Y$ is uniquely ergodic and $\lim\limits_{n\to\infty}Y_n=Y$ in $\inv$, then
\[
\lim_{n\to\infty}\Mm^1(Y_n,G) = \limsup_{n\to\infty}\Mm^1(Y_n,G) = \Mm^1(Y,G).
\]
\end{corollary}

\begin{remark}
Note that in Corollary~\ref{cor:ContinuityMeasures_UniqueErgod} the convergence $\lim_{n\to\infty}Y_n=Y$ in $\inv$ together with the unique ergodicity of $Y$ imply that the limit $\lim_{n\to\infty}\Mm^1(Y_n,G)$ exists in $\Kk\big(\Mm^1(Z,G)\big)$.
\end{remark}

\begin{proof}
In order to prove continuity of the map $\Minv$ at $Y\in\inv$, fix an arbitrary sequence $Y_n\in\inv, \ n\in\N,$ such that $\lim_{n\to\infty}Y_n=Y$. 
By Theorem~\ref{thm:SemiContinuity_measures}, we have 
\[ 
\limsup_{n\to\infty}\Mm^1(Y_n,G) \subseteq \Mm^1(Y,G).
\] 
Thus, \cite[Lem~1.11.2]{vMil01} implies that it suffices to prove 
\[
\Mm^1(Y,G) \subseteq \liminf_{n\to\infty}\Mm^1(Y_n,G), 
\]
see also eq.~\eqref{eq:char_conv_Hausdorff_metric}. Specifically, we have to show that every sequence of measures $\mu_n\in \Mm^1(Y_n,G)$ converge in the weak\(^*\)-topology to $\mu\in\Mm^1(Y,G)$ where $\mu$ is the unique invariant measure on the dynamical system $(Y,G,\tau)$. 
Assume towards contradiction that this is not the case. Then there is a subsequence $\mu_{n_k}\in \Mm^1(Y_{n_k},G)$ that do not converge to $\mu$. By compactness of $\Mm^1(Z,G)$, there is no loss of generality in assuming that the limit $\lim_{k\to\infty}\mu_{n_k}=:\nu\in\Mm^1(Z,G)$ exists. Then \cite[Lem.~2.1]{BecPog20} implies that the support $\supp(\nu)$ is contained in $Y$ since $\lim_{n\to\infty}Y_n=Y$. Hence, $\nu\in\Mm^1(Y,G)$ and since $Y$ is uniquely ergodic, we conclude $\nu=\mu$, a contradiction.
\end{proof}

In the upcoming sections, we discuss certain classes of dynamical systems admitting suitable (periodic) approximations. The results from this section then imply convergence -- or, in some cases, semi-continuity -- of the corresponding invariant measures. 

This has further implications for the convergence of two central objects: the \emph{autocorrelation measure} and the \emph{density of states} of an associated operator. We briefly outline their relevance in the literature and refer to the joint work~\cite{BecPog20} for precise definitions and formal statements concerning the continuity properties of these quantities.

\begin{itemize}
\item The \emph{autocorrelation measure} was developed to provide a rigorous mathematical framework for diffraction experiments. The discovery of quasicrystals by Dan Shechtman \cite{SBGC84}, for example, was based on a diffraction experiment and relied on the crystallographic restriction theorem; see e.g. \cite[Sec.~3.1]{BaaGri13}. 

These solids are modeled by point sets in space that represent the positions of atomic nuclei. It is assumed that no two points can be arbitrarily close to each other, and that there are no arbitrarily large holes in the material -- 	a condition formalized by the notion of \emph{Delone sets}. By distinguishing different atomic species using colors (or letters), one obtains \emph{weighted Delone sets}, which are formally introduced in Section~\ref{sec:Weight_Delone}.

It has turned out that diffractive properties of Delone sets are reflected in the associated dynamical systems; see, for instance, the non-exhaustive list of references \cite{Dwo93,Sch00,LeMoSo02,BaLe04,BaaMoo04,Gou05,LeStr09,Queff10,BaaGri13} and references therein. More recently, a non-commutative spherical diffraction theory has been developed for so-called regular model sets \cite{BjHa18,BjHaPo16,BjHaPo17}, a special class of cut-and-project sets.

Assuming the conditions of Corollary~\ref{cor:ContinuityMeasures_UniqueErgod}, the autocorrelation measure converges in the weak-$\ast$ topology whenever the underlying dynamical systems converge to a uniquely ergodic dynamical system, see~\cite[Thm.~V.A]{BecPog20} for details.

\item A similar convergence result holds for the \emph{density of states measure} (DOS) associated with a random operator defined over dynamical systems. A particular instance of such operators are the \emph{dynamically-defined operators} introduced in Section~\ref{sec:SpectralEstimates}, considered together with an invariant measure on the corresponding dynamical system.

One is typically interested in the \emph{integrated density of states} (IDS), which is the cumulative distribution function of the DOS. In quantum mechanical models, it quantifies the number of energy levels per unit volume lying below a given energy threshold, see also a discussion in Chapter~\ref{chap:WorldButterflies} and its relation to the Gap Labeling theorem.

For suitable uniquely ergodic systems, the IDS can be approximated via finite-volume analogs using an ergodic theorem; see e.g. \cite{Bel86,Kirsch89,LenSto03Alg,LenSto03Del,LenSto05,LenPeyVes07,LenMulVes08,Kirsch08,LenVes09,LenSchVes11,Pog14,PogSch16,DaFi24-book_2} and references therein. In~\cite{BecPog20}, the Pastur-Shubin trace formula \cite{Pas71,Shu79} provides the appropriate framework: it expresses the IDS as an integral over the dynamical system.

Assuming the conditions of Corollary~\ref{cor:ContinuityMeasures_UniqueErgod}, the DOS measure converges in the weak-$\ast$ topology whenever the underlying dynamical systems converge to a uniquely ergodic dynamical system, see \cite[Thm.~V.B]{BecPog20} for details.
\end{itemize}

We also addressed in this work \cite{BecPog20} if the hulls of regular model sets can be approximated by changing the underlying lattice. We obtain positive results by approximating the window function via continuous functions with compact support leading naturally to the notion of weighted Delone sets, studied in the upcoming section.
In \cite{BaBePoTe25}, approximations for symbolic dynamical systems are discussed; see a discussion in Section~\ref{sec:Subst-Approximations}. Note that all of these examples define dynamical systems that are uniquely ergodic.

We note that the integrated density of states also plays a central role in Chapter~\ref{chap:DTMP}, where the joint works~\cite{BaBeBiRaTh24,BaBeLo24} are discussed solving the dry ten Martini problem for Sturmian dynamical systems.

\section{The space of weighted Delone sets}
\label{sec:Weight_Delone}

In order to analyze the convergence of dynamical systems arising in the context of aperiodic order within a unified framework, we introduce and study weighted Delone sets in this section.

Some of the results presented here are new\footnote{According to the best knowledge of the author.} in the literature, although they have been considered as folklore. The specific definition of the metric on the space of weighted Delone sets is inspired by the construction of Solomyak~\cite{Sol98,Sol06} (in the setting of tilings with a fixed collection of prototiles) and the recent joint work~\cite{BecHarPog25}, which studies weighted Delone sets as a topological space (see Section~\ref{sec:Repet_WeightDelone}).

We consider weighted Delone sets in a locally compact, second countable, Hausdorff group \( G \), define a suitable metric on this space, and characterize the convergence of their associated hulls through the convergence of local patterns.

For the sake of readability, the technical details and proofs are deferred to Appendix~\ref{App:Weight_Delone}. We also refer the reader to~\cite{ArgGil74,Baa01,LenSto03Alg,Moo02,BaaMoo04,BaLe04,LeRi07,MueRic13,Str14,HuRi15,RiSt17,BaHuSt17,BjoHarPog22} for earlier work on weighted Delone sets.

\begin{definition}[Delone sets]
\label{def:Delone}
Let $r_0,R_0>0$. A subset $D\subseteq G$ is called 
\begin{itemize}
\item {\em ($r_0$-)uniformly discrete}, if $\sharp B(x,r_0)\cap D\leq 1$ for all $x\in G$;
\item {\em ($R_0$-)relatively dense}, if $G = \bigcup_{x\in D} \ol{B}(R_0)(x)$.
\end{itemize}
If $D$ is $r_0$-uniformly discrete and $R_0$-relatively dense, then $D$ is called an {\em $(r_0,R_0)$-Delone set}. The set of all $(r_0,R_0)$-Delone sets is denoted by $\Del$.
\end{definition}

The group $G$ acts on $\Del$ via left multiplication, namely 
\[
gD := \set{gx}{x\in D},\qquad g\in G,\, D\in\Del.
\]
In order to include also symbolic dynamics within the theory developed here, we also define weighted Delone sets.

\begin{definition}[Weight]
\label{def:weight}
Let \( \sigma \geq 1 \) and let \( P \subseteq G \) be an \( r_0 \)-uniformly discrete set.  
A map \( \alpha : P \to [\sigma^{-1}, \sigma] \) is called a \emph{($\sigma$-)weight of \( P \)}.  

A tuple \( (D, \alpha) \) is called a \emph{(\( \sigma \)-)weighted Delone set} if \( D \in \Del \) and \( \alpha \) is a \( \sigma \)-weight of \( D \).  
The set of all \( \sigma \)-weighted Delone sets is denoted by
\[
\wDel := \set{(D,\alpha)}{(D,\alpha) \text{ is a } \sigma\text{-weighted Delone set}}.
\]
\end{definition}

The group \( G \) naturally acts on \( \wDel \) via
\[
\tau : G \times \wDel \to \wDel, \quad 
	\big(g,(D,\alpha)\big) \mapsto g \cdot (D,\alpha) := (gD, \alpha \circ g^{-1}).
\]
While it is elementary to verify that this defines a group action, we still need to specify a topology on \( \wDel \) in order to determine whether this action is continuous.  
To this end, we introduce a specific metric on \( \wDel \) and compare its induced topology with those considered in special cases in the literature.  
Here the assumption $\sigma\geq 1$ (namely $\sigma^{-1}>0$ allows us to work with the associated weighted dirac combs. 
Before doing so, let us briefly discuss two main examples treated in this work.

\begin{example}[Delone dynamical systems]
\label{ex:Delone-As-WeightDelone}
If $\sigma=1$, then every weight is trivial (constant map with range $\{1\}$). Thus, we identify $\Del$ with $\wDel$ if $\sigma=1$ via the bijective map $D\mapsto (D,1)$. Note that this map is compatible with the group action on $\Del$ since $g.(D,1)=(gD,1)$ for $D\in\Del$.
\end{example}

\begin{example}[Symbolic dynamical systems]
\label{ex:SymbSyst-As-WeightDelone}
Let $\Gamma\subseteq G$ be a discrete cocompact subgroup (i.e. $G/\Gamma$ is compact) and $\Aa$ a finite set. 
Then $\Gamma\in\Del$ for a suitable choice of $r_0,R_0>0$. 
The elements of the configuration space $\Aa^\Gamma:=\{\omega:\Gamma\to\Aa\}$ can be represented as weighted Delone sets. In order to do so, we assume without loss of generality (identify letters with positive real numbers) that $\Aa\subseteq [\sigma^{-1},\sigma]$ for some $\sigma>1$. Then the map
\[
\Aa^\Gamma\to \wDel, \quad 
	\omega\mapsto (\Gamma,\omega)
\]
is injective and provides an embedding of $\Aa^\Gamma$ into $\wDel$. The group $\Gamma$ naturally acts on $\Aa^\Gamma$ via $g\omega:=\omega\circ g^{-1}$ for $g\in\Gamma$ and $\omega\in\Aa^\Gamma$, see Example~\ref{ex:SymbDynSyst} where symbolic dynamical systems are defined.
Then the previously defined embedding is compatible with the \( \Gamma \)-action.

After defining the metric structure on $\wDel$, we show in Section~\ref{subsec:symb_syst_as_weighted_Delone} that this embedding is homeomorphic (even bi-Lipschitz). Therefore symbolic dynamical systems are represented by weighted Delone sets. In the literature also compact alphabets are studied. Note that the previous construction works also for alphabets being compact subsets embedded into $[\sigma^{-1},\sigma]$.
\end{example}

Both examples are naturally endowed with a metrizable topology. Specifically, \( \Aa^\Gamma \) is equipped with the product topology, which defines a compact metrizable space. In this topology, two configurations in \( \Aa^\Gamma \) are close to each other if they coincide on a large ball around the identity element \( e \in G \).

Since each Delone set is a closed subset of \( G \), the space \( \Del \) can be equipped with the Chabauty-Fell topology \cite{Cha50,Fel62}, making \( \Del \) a compact, metrizable space; see e.g.\ \cite[Sec.~3]{BjoHarPog22}. Roughly speaking, two Delone sets are close to each other if they coincide on a large ball around the identity element \( e \in G \), up to a small displacement.  

We emphasize that the Chabauty-Fell topology on the space of Delone sets is generally not induced by the Hausdorff metric on subsets of \( G \), since Delone sets are merely closed subsets of \( G \). In particular, the Hausdorff metric is defined on \( \Kk(G) \), the space of \emph{compact} subsets of \( G \), while Delone sets are typically \emph{closed and discrete}, but not compact.

For Delone sets, various similar notions of convergence and topologies have been proposed, particularly for certain subclasses such as sets of finite local complexity or repetitive sets; see, e.g., \cite{Dwo93,Sol98,Sol06}.

By classical metrization theorems, the Chabauty-Fell topology on \( \Del \) is metrizable, as it is compact, second countable, and Hausdorff. Consequently, one can find several constructions of explicit metrics compatible with this topology, for instance in the setting of Delone sets of finite local complexity \cite{Bir18} and tiling spaces \cite{Rob96,Rob04}.

Following \cite{BaLe04,BecHarPog21,BjoHarPog22}, the weighted Delone sets can be endowed with the weak\(^*\)-topology by identifying weighted Delone sets with their associated weighted dirac comb - defining a Radon measure on the Borel $\sigma$-algebra of $G$. 
We follow a similar approach here but also defining a metric on the set of $r_0$-uniformly discrete weighted sets \label{page:wPat}
\[
\wPat := \set{(P,\alpha)}{P \textrm{ is $r_0$-uniformly discrete and } \alpha \textrm{ is a $\sigma$-weight of } P}.
\] 
Clearly $\wDel\subseteq \wPat$ holds since we just dropped the assumption of being $R_0$-relatively dense. 

For $\Pi:=(P,\alpha)\in\wPat$, the weighted dirac comb on the Borel $\sigma$-algebra $\bs(G)$ of $G$ is defined by
\begin{equation}
\label{eq:DiracComp_measure}
\mu_{\Pi}:\bs(G)\to[0,\infty], \qquad
\mu_{\Pi} := \sum_{x\in P} \alpha(x) \delta_x.
\end{equation}
Then $\mu_\Pi$ is a Radon measure (regular Borel measure on $\bs(G)$). The support of the measure $\mu_{\Pi}$ satisfies 
\[
\supp(\Pi):= \supp(\mu_\Pi)=P.
\]
For $\Pi,\Xi\in\wPat$, define the \emph{local discrepancy} (illustrated in Figure~\ref{fig:local-discrepancy})
\begin{equation}
\label{eq:delta_weightedPatches}
\wdelt(\Pi,\Xi) 
	:= \inf\set{\varepsilon>0}{ 
		\begin{array}{c}
		\big|\mu_\Pi(B(x,\varepsilon)) - \mu_\Xi(B(x,\varepsilon))\big|<\varepsilon\\[0.2cm]
		\textrm{ for all } x\in B(\tfrac{1}{\varepsilon})\cap \big( \supp(\Pi)\cup \supp(\Xi) \big)
		\end{array} },
\end{equation}
where the infimum is set to be infinity if there is no such $\varepsilon>0$.
Note that $\mu_\Pi(B(x,\varepsilon))>0$ if and only if $B(x,\varepsilon)\cap \supp(\Pi)\neq 0$. 
In particular, if $\varepsilon<r_0$ and $\mu_\Pi(B(x,\varepsilon))>0$, then $B(x,\varepsilon)\cap \supp(\Pi)$ contains exactly one point, see a sketch in Figure~\ref{fig:local-discrepancy}.

\begin{figure}[htb]
    \centering
    \includegraphics[scale=1]{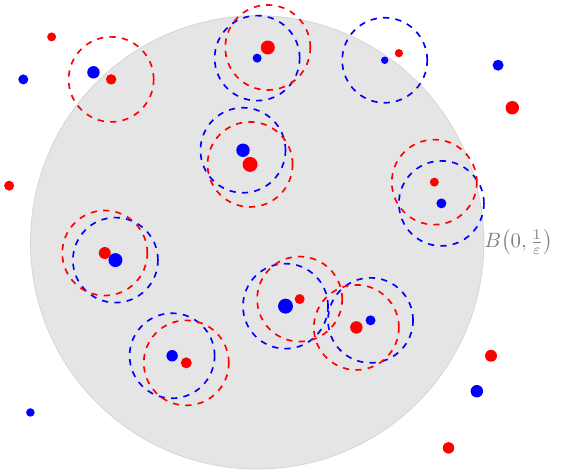}
    \captionsetup{width=0.95\linewidth}
    \caption{A sketch of the local discrepancy for two uniformly discrete weighted $\Pi$ (blue points) and $\Xi$ (red points). The dashed colored circles indicate the $\varepsilon$-balls centered at relevant points, provided the point lies within the gray ball $B\big(0,\tfrac{1}{\varepsilon}\big)$. The thickness of each point reflects the local mass contribution -- the associated weight taking values in $[\sigma^{-1},\sigma]$.}
    \label{fig:local-discrepancy}
\end{figure}

By definition, the map $\wdelt:\wPat\times\wPat\to[0,\infty]$ is symmetric and definite, which can be extended to a metric. The proof is inspired by \cite{Sol98,Sol06}.

\begin{proposition}
\label{prop:metric_weightDelone}
Let $G$ be a lcsc group, $r_0,R_0>0$ and $\sigma\geq 1$. Then $(\wPat,\dw)$ is a compact metric space where
\begin{equation}
\label{eq:Metric_wDel}
\dw:\wPat\times\wPat\to[0,\infty),\quad
	\dw(\Pi,\Xi):= \min\big\{\tfrac{r_0}{2},\tfrac{\sigma^{-1}}{2},\wdelt(\Pi,\Xi) \big\}
\end{equation}
induces the weak\(^*\)-topology on $\wPat$ viewing every element $\Pi\in\wPat$ as the Radon measure $\mu_\Pi$ defined in \eqref{eq:DiracComp_measure}.
Furthermore $\wDel\subseteq\wPat$ is a closed subset and so $(\wDel,\dw)$ is a compact metric space.
\end{proposition}

\begin{proof}
We postpone the proof to the Appendix~\ref{App:Weight_Delone} in Proposition~\ref{prop:metric_weightDelone-App}.
\end{proof}

Recall that the group $G$ naturally acts on $\wDel$ via
\[
\tau:G\times \wPat\to\wPat,  \quad 
	\big(g,(D,\alpha)\big) \mapsto g.(D,\alpha):= (gD,\alpha\circ g^{-1}).
\]
It is well-known that $(\wPat,G,\tau)$ and $(\wDel,G,\tau)$ define a topological dynamical system \cite{ArgGil74,BaaMoo04,BaLe04,MueRic13,BjoHarPog22,BecHarPog25}. For the spectral application, we need more and so we prove that this group action is actually Lipschitz continuous (see Definition~\ref{def:LipschitzAction}). 

\begin{proposition}
\label{prop:wDel_groupAction_Lipschitz}
Let $G$ be a lcsc group, $r_0,R_0>0$ and $\sigma\geq 1$.
Then $(\wPat,G,\tau)$ and $(\wDel,G,\tau)$ are a topological dynamical system. Moreover, the group action is Lipschitz continuous, namely for all $g\in G$ and $\Pi,\Xi\in\wPat$, we have
\[
\dw(g.\Pi,g.\Xi) \leq (d_G(e,g)+1) \dw(\Pi,\Xi).
\]
\end{proposition}

\begin{proof}
We postpone the proof to the Appendix~\ref{App:Weight_Delone} in Proposition~\ref{prop:wDel_groupAction_Lipschitz-App}.
\end{proof}

Following \cite{Bec16,BBdN20}, the convergence of invariant closed subsets of the dynamical system $(\wDel,G,\tau)$ can be characterized by the convergence of the associated patches. We provide a unifying framework here for weighted Delone sets.

Let \( \wInv \) denote the compact metric space of invariant, closed, and non-empty subsets of \( \wDel \), equipped with the Hausdorff metric \( \delta_H \) induced by the metric \( \dw \) on \( \wDel \).  
We refer to the elements of \( \wInv \) as \emph{(weighted) Delone dynamical systems}.

A canonical way to obtain elements of \( \wInv \) is via the hull of a weighted Delone set.  
For a weighted Delone set \( \Pi \in \wDel \), its \emph{hull} is defined as
\[
\hull(\Pi) := \overline{ \set{ g . \Pi }{ g \in G } }
= \overline{ \Orb(\Pi) } \in \wInv,
\]
where the closure is taken with respect to the metric \( \dw \).

In the literature, it is common to consider only the \emph{transversal} associated with a Delone set; see \cite{LenSto03Alg,LenSto03Del,Bel15,BBdN18,BecPog20}.  
For a weighted Delone set $\Pi \in \wDel$, its transversal is defined as \label{page:transversal}
\[
\tran(\Pi) := \set{ \Xi \in \hull(\Pi) }{ \mu_\Xi(\{e\}) > 0 }
= \overline{ \set{ g^{-1}.\Pi }{ g \in \supp(\Pi) } },
\]
where the closure is taken with respect to the metric $\dw$.

More generally, for a weighted Delone dynamical system $\Omega \in \wInv$, the associated transversal is given by
\[
\tran(\Omega) := \set{ \Xi \in \Omega }{ \mu_\Xi(\{e\}) > 0 }
= \overline{ \set{ g^{-1}.\Pi }{ \Pi \in \Omega,\; g \in \supp(\Pi) } }.
\]

Since $\mu_\Xi(\{e\}) > 0$ is equivalent to $\mu_\Xi(\{e\}) \geq \sigma^{-1} > 0$ for all $\Xi \in \wDel$, every transversal $\tran(\Omega)$ is a closed subset of $\wDel$ with respect to $\dw$, and the collection 
\[
\wInvT := \set{\tran(\Omega)}{\Omega\in\wInv}
\] 
of all such transversals can be viewed as a subset of $\Kk(\wDel)$.

It is straightforward to verify that $\wInvT$ is closed in $\Kk(\wDel)$ with respect to the induced Hausdorff metric $\delta_H$. In particular, the limit of a sequence of transversals is again a transversal. Therefore, $(\wInvT, \delta_H)$ is a compact metric space.

Note that a transversal is, in general, not invariant under the action of $G$; however, if $\Xi \in \tran(\Omega)$ and $g \in \supp(\Xi)$, then $g^{-1}.\Xi \in \tran(\Omega)$.  
To each transversal, one naturally associates a groupoid, and elements of the associated reduced groupoid $C^*$-algebra define the operators of interest in this work; see \cite{LenSto03Alg,LenSto03Del,BBdN18} and Example~\ref{ex:Graph_Delone}.

The (weighted) Delone dynamical system $\Omega \in \wInv$ can be reconstructed from its transversal $\tran(\Omega) \subseteq \Omega$ as follows.

\begin{lemma}
\label{lem:Connec_Transv-Hull}
Let $r_0,R_0>0$ and $\sigma\geq 1$. If $\Omega\in\wInv$, then 
$$
\Omega
	= \oB(R_0).\tran(\Omega) 
	:= \set{g.\Pi}{g\in \oB(R_0),\, \Pi\in\tran(\Omega)}.
$$
In particular, $\Omega=G.\tran(\Omega)$.
\end{lemma}

\begin{proof}
The inclusion $\oB(R_0).\tran(\Omega) \subseteq \Omega$ is obvious. For the converse, let $\Xi\in\Omega\subseteq \wDel$. Since $\supp(\Xi)$ is $R_0$-relatively dense, there is a $g\in\supp(\Xi)\cap\oB(R_0)$. Then $\Pi:=g^{-1}.\Xi\in\tran(\Omega)$ follows and so $\Xi=g.\Pi\in \oB(R_0).\tran(\Omega)$ is concluded.
\end{proof}

With this at hand, we conclude that the convergence of weighted Delone dynamical systems is characterized by the convergence of the associated transversals. This has been proven for Delone sets in \cite[Thm.~2.A]{BecPog20}. The extension to weighted Delone sets is straightforward and follows similar lines. We provide in the appendix a different proof for weighted Delone sets using Lemma~\ref{lem:Connec_Transv-Hull}.

\begin{proposition}[\cite{BecPog20}]
\label{prop:Homeo_hulls_transversals}
Let $R_0>r_0>0$ and $\sigma\geq 1$.
The map 
\[
\Phi:\wInvT\to\wInv,\qquad \Tt\mapsto G.\Tt=\oB(R_0).\Tt
\] 
is a homeomorphism. Moreover, $\Phi$ is bi-Lipschitz, namely
\[
\frac{1}{2}\delta_H(\Tt_1,\Tt_2) 
	\leq \delta_H\big( \Phi(\Tt_1),\Phi(\Tt_2) \big)
	\leq (R_0+1)\delta_H(\Tt_1,\Tt_2),
	\qquad \Tt_1,\Tt_2\in\wInvT.
\]
\end{proposition}

\begin{proof}
The topological statement is proven in \cite[Thm.~2.A]{BecPog20} for Delone sets without weights. An alternative proof is provided in the appendix in Proposition~\ref{prop:Homeo_hulls_transversals_Appendix}, where we also prove that $\Phi$ is bi-Lipschitz.
\end{proof}

\begin{remark}
We note that this result is false if we consider just uniformly discrete sets. In this case the hull of a $\Pi\in\wPat$ contains the empty set if $\Pi$ is not relatively dense. On the other hand, the transversal of $\Pi$ does not contain the empty set since for every element in $\Xi\in\tran(\Pi)$, we have $e\in\supp(\Xi)$.
\end{remark}

Properties of Delone sets (or tilings) are often analyzed by studying the associated patches; see, e.g., \cite{Fog02,Queff10,BaaGri13}.  
Since general Delone sets may admit infinitely many patches of a fixed size, one typically classifies them up to $\varepsilon$-similarity; see \cite{FrRi14,BecHarPog25}.

\begin{definition}[$\varepsilon$-similar patches]
\label{def:epsilon-similar}
For $\varepsilon>0$, we say two patches $\Pi,\Xi\in\wPat$ are {\em $\varepsilon$-similar ($\Pi=_\varepsilon \Xi$)}, if $\wdelt(\Pi,\Xi)<\varepsilon$. Two subsets $W,W'\subseteq \wPat$ are called {\em $\varepsilon$-similar ($W=_\varepsilon W'$)} if
\begin{itemize}
\item for all $\Lambda\in W$, there is a $\Theta\in W'$ such that $\Lambda$ and $\Theta$ are $\varepsilon$-similar, and
\item for all $\Theta\in W'$, there is a $\Lambda\in W$ such that $\Lambda$ and $\Theta$ are $\varepsilon$-similar.
\end{itemize}
\end{definition}
 
We emphasize that the relation of being $\varepsilon$-similar is not an equivalence relation. This relation is only reflexive and symmetric but not transitive. We next introduce the set of $R$-patches associated with a weighted Delone set or weighted Delone dynamical system.

\begin{definition}
\label{def:Weighted_Dictionary}
For $R>0$, and $\Pi\in\wDel$, define the {\em $R$-patches} by
$$
W_R(\Pi) : = \set{\big(g^{-1}.\Pi\big)|_{B(R)}}{g\in \supp(\Pi)}
$$
and for $\Omega\in\wInv$, the {\em $R$-patches} are defined by
$$
W_R(\Omega) : = \set{\big(g^{-1}.\Pi\big)|_{B(R)}}{\Pi\in\Omega, \, g\in \supp(\Pi)}
	= \bigcup_{\Pi\in\Omega} W_R(\Pi).
$$
\end{definition}

Note that we choose patches to be anchored at the origin (the neutral element $e\in G$).  
Specifically, for each $\Lambda \in W_R(\Pi)$, we have $\mu_\Lambda(\{e\}) \geq \sigma^{-1} > 0$.

In general, the set of $R-$patches $W_R(\Omega)$ is not closed in $(\wPat, \dw)$; see the following example.

\begin{example}
Let \( G = \R \), \( r_0 = \tfrac{1}{2} \), \( R_0 = 2 \), and \( \sigma = 2 \).  
Consider the weighted Delone set \( (\Z, \alpha) \in \wDel \), where \( \alpha : \Z \to [\sigma^{-1}, \sigma] \) is defined by
\[
\alpha(n) := 1 + \tfrac{1}{|n| + 2}.
\]
Then the set of $1$-patches is given by
\[
W_1(\Z, \alpha) = \set{ (0,\, 0 \mapsto 1 + \tfrac{1}{|n| + 2}) }{ n \in \Z }.
\]
This set is not closed in \( (\wPat, \dw) \), since its closure contains the patch \( (0,\, 0 \mapsto 1) \), which is not an element of \( W_1(\Z, \alpha) \).
A similar statement remains true for unweighted Delone sets (i.e. $\sigma=1$). An example is the Delone set
\[
D := \set{ n + \tfrac{1}{|n| + 1} }{ n \in \Z }.
\]
\end{example}

A weighted Delone set \( \Pi =(D,\alpha) \) is called of \emph{finite local complexity (FLC)} if \( W_R(\Pi) \) is finite for all \( R > 0 \).  
This condition can be characterized by the fact that the difference set \( D^{-1} D \) is locally finite and the weight function \( \alpha \) takes only finitely many values; see \cite{Lag99A,Sch00} in the unweighted case.

While not all Delone sets are of finite local complexity, one can partition any hull or transversal by neighborhoods of patches of arbitrary fixed size.

\begin{lemma}
\label{lem:Basic_Diction_Transv_Hull}
Let $r_0,R_0>0$ and $\sigma\geq 1$. 
\begin{enumerate}[label=(\alph*)]
\item For $\Pi\in\wDel$, we have $W_R(\Pi)=\set{ \Xi|_{B(R)}}{\Xi\in\tran(\Pi)}$.
\item For $\Omega\in\wInv$, we have $W_R(\Omega)=\set{ \Xi|_{B(R)}}{\Xi\in\tran(\Omega)}$.
\item For all $\Omega\in\wInv$ and each $\min\big\{ \tfrac{\sigma^{-1}}{2}, \tfrac{r_0}{2} \big\} > \varepsilon>0$, there are finitely many $\Pi_1,\ldots,\Pi_N\in \Omega$ such that 
$$
\Omega \subseteq \bigcup_{j=1}^N B(\Pi_j,\varepsilon)
	= \bigcup_{j=1}^N \set{\Xi\in\wDel}{\Xi|_{B(\tfrac{1}{\varepsilon})} =_\varepsilon \Pi_j|_{B(\tfrac{1}{\varepsilon})}}
$$
where $B(\Pi_j,\varepsilon):=\set{\Xi\in\wDel}{\dw(\Pi_j,\Xi)<\varepsilon}$ is the open ball of radius $\varepsilon$ around $\Pi_j\in\wDel$.
\end{enumerate}
\end{lemma}

\begin{proof}
Statement (a) follows immediately from the definition of the transversal. Then (a) implies (b).

\medskip

For (c), the first equality follows simply by compactness of $\Omega$ and the second one follows from the definition of $\dw$ (see eq.~\eqref{eq:Metric_wDel}) and $\varepsilon$-similar patches. Here we use that $\dw(\Xi,\Pi_j)<\varepsilon$ if and only if $\wdelt(\Xi,\Pi_j)<\varepsilon$ since $\varepsilon < \min\big\{ \tfrac{\sigma^{-1}}{2}, \tfrac{r_0}{2} \big\}$.
\end{proof}

Statement (c) is well-known and has a natural interpretation: the ``cylinder sets'' corresponding to the patches \( \Pi_j|_{B(\tfrac{1}{\varepsilon})} \) cover the hull \( \Omega \).  
The notion of cylinder sets is widely used in symbolic dynamical systems, where convergence of the associated dynamical systems is characterized by the convergence of local patches; see \cite[Thm~3.3.22]{Bec16} and \cite[Cor.~5.5]{BBdN20}.
Note however that in general, $B(\Pi_j,\varepsilon)$ may intersect $B(\Pi_i,\varepsilon)$ if $i\neq j$.

Building on this principle, the convergence of dynamical systems is used in \cite{BeBeCo19} to estimate spectral distances (see also Theorem~\ref{thm:SpecEst_DynDefOp}), and in \cite{BaBePoTe25} to characterize the existence of (periodic) approximations for substitution systems (see also Theorem~\ref{thm:Charact_Convergence_Substitutions}).

The analogue of this patchwise convergence characterization for weighted Delone sets is provided in the following.

In light of Lemma~\ref{lem:Basic_Diction_Transv_Hull}~(b) and Proposition~\ref{prop:Homeo_hulls_transversals}, we define the $R$-patches of a transversal $\Tt\in\wInvT$ by $W_R(\Tt):= W_R(\Omega)$ where $\Omega\in\wInv$ is defined by $\Omega=G.\Tt$. The Hausdorff distance of the transversals (and via Proposition~\ref{prop:Homeo_hulls_transversals} also of Delone dynamical systems) can be expressed by the infimum over positive numbers $\varepsilon$ where the $R$-patches are $\varepsilon$-similar.

\begin{proposition}
\label{prop:HausdorffDistance_transv}
Let $r_0,R_0>0$ and $\sigma\geq 1$. For $\Tt_1,\Tt_2\in\wInvT$, we have
$$
\delta_H(\Tt_1,\Tt_2) 
	= \inf\set{\varepsilon>0}{W_{1/\varepsilon}(\Tt_1) =_\varepsilon W_{1/\varepsilon}(\Tt_2)} 
		\cup
		\left\{
			\tfrac{\sigma^{-1}}{2},\, 
			\tfrac{r_0}{2}
		\right\}.
$$
\end{proposition}

\begin{proof}
Let $\Tt_1,\Tt_2\in\wInvT$ be two transversals. Set $\delta_{\text{tran}} := \delta_H(\Tt_1,\Tt_2)$,
$$
\delta:= \inf\set{\varepsilon>0}{W_{1/\varepsilon}(\Tt_1) =_\varepsilon W_{1/\varepsilon}(\Tt_2)} 
		\cup
		\left\{
			c_0
		\right\}
		\qquad\text{and}\qquad
		c_0:=\min\left\{\frac{\sigma^{-1}}{2}, \frac{r_0}{2}\right\}.
$$
We first prove $\delta_{\text{tran}} \leq \delta$. Without loss of generality we can assume $\delta<c_0$ since $\delta_{\text{tran}}\leq c_0$.
Let $\varepsilon>0$ be such that $\tepsilon:=\delta+\varepsilon<c_0$. Then $W_{1/\tepsilon}(\Tt_1) =_\tepsilon W_{1/\tepsilon}(\Tt_2)$ holds. 
Let $\Xi\in\Tt_1$ be arbitrary, then $\Xi|_{B(1/\tepsilon)}\in W_{1/\tepsilon}(\Tt_1)$. Thus, there is a $\Pi\in\Tt_2$ such that $\Pi|_{B(1/\tepsilon)} =_\tepsilon \Xi|_{B(1/\tepsilon)}$ implying $\wdelt(\Pi,\Xi)<\tepsilon<c_0$. By interchanging the role of $\Tt_1$ and $\Tt_2$, we conclude  $\delta_{\text{tran}} \leq \tepsilon = \delta+\varepsilon$. Sending $\varepsilon$ to zero implies $\delta_{\text{tran}} \leq \delta$.

\medskip

Next we prove $\delta\leq \delta_{\text{tran}}$.
Without loss if generality assume $\delta_{\text{tran}}<c_0$ since $\delta\leq c_0$.
Let $\varepsilon>0$ be such that $\tepsilon:=\delta_{\text{tran}}+\varepsilon<c_0$. 
Let $\Pi|_{B(1/\tepsilon)}\in W_{1/\tepsilon}(\Tt_1)$ be arbitrary. Then there is a $\Xi\in \Tt_2$ satisfying $\dw(\Pi,\Xi)<\tepsilon$ implying $\wdelt(\Pi|_{B(1/\tepsilon)},\Xi|_{B(1/\tepsilon)})<\tepsilon$. Hence, $\Pi|_{B(1/\tepsilon)} =_\tepsilon \Xi|_{B(1/\tepsilon)}\in W_{1/\tepsilon}(\Tt_2)$ follows. 
By interchanging the role of $\Tt_1$ and $\Tt_2$, we conclude $\delta\leq \tepsilon = \delta_{\text{tran}}+\varepsilon$. Sending $\varepsilon$ to zero implies $\delta\leq \delta_{\text{tran}}$.
\end{proof}

The previous equality is a practical tool to construct periodic approximations and estimate the rate of convergence \cite{Bec16,BBdN20,BaBePoTe25} (see also Section~\ref{sec:Subst-Approximations}) or identify possible spectral defects \cite{BecBelTho25} (Section~\ref{sec:Kohmoto}). The estimates on the rate of convergence for the dynamical systems yields then corresponding estimates for the rate of convergence of the spectrum \cite{BeBeCo19,BecTak25} (Section~\ref{sec:SpectralEstimates}).
If one is only interested in convergence of Delone dynamical systems, we can summarize the previous results in the following corollary.

\begin{corollary}
\label{cor:Charact_Converg_hulls_Patches}
Let $r_0,R_0>0$ and $\sigma\geq 1$. Then the following are equivalent for a sequence of weighted Delone dynamical systems $\Omega\in\wInv, n\in\N$, and $\Omega\in\wInv$.
\begin{enumerate}[label=(\roman*)]
\item $\lim\limits_{n\to\infty}\Omega_n=\Omega$ in $\inv$.
\item $\lim\limits_{n\to\infty}\tran(\Omega_n)=\tran(\Omega)$ in $\inv^e$.
\item For all $R>0$ and $\varepsilon>0$, there is an $n_0\in\N$ such that
$$
W_R(\Omega_n) =_\varepsilon W_R(\Omega) \qquad \textrm{for all } n\geq n_0.
$$
\end{enumerate}
\end{corollary}

\begin{proof}
\underline{(i) $\Leftrightarrow$ (ii):} This is proven in \cite[Thm.~II.A]{BecPog20} for Delone sets and extends to weighted Delone sets, see Proposition~\ref{prop:Homeo_hulls_transversals}.

\medskip

\underline{(ii) $\Leftrightarrow$ (iii):} This is an immediate consequence of Proposition~\ref{prop:HausdorffDistance_transv}.
\end{proof}

We finish the section by briefly discussing the implications for the particular case of Delone sets (without a weight) and symbolic dynamical systems of the previous results.

\subsection{Delone dynamical systems}
\label{subsec:Delone_dyn_syst}

Following Example~\ref{ex:Delone-As-WeightDelone}, we identify \( \Del \) with \( \wDel \) by setting \( \sigma = 1 \), corresponding to trivial weights.  
In this case, the topology on \( \Del = \wDel \) induced by \( \dw \) coincides with the Chabauty-Fell topology; see, e.g., \cite[Lem.~2]{BaLe04}, \cite[Sec.~2.1]{MueRic13} and \cite[Prop.~3.2]{BjoHarPog22}.  
As for weighted Delone sets, the \emph{hull} of a Delone set is defined as the orbit closure
\[
\hull(D) := \overline{ \set{ g . D }{ g \in G } }.
\]
For the sake of simplicity, we focus on Delone sets of finite local complexity (FLC), meaning that the set of associated $R$-patches
\[
W_R(D) := \set{ g^{-1} D \cap \oB(R) }{ g \in D }
\]
is finite for each \( R > 0 \). This is equivalent to the condition that the difference set \( D^{-1} D \) is locally finite, or, equivalently, that \( D^{-1} D \) is closed and discrete.
These patches describe the hull (under the FLC assumption) via 
\[
\hull(D) = \set{ P \in \Del }{ W_R(P) \subseteq W_R(D) \text{ for all } R > 0 }.
\]

Let \( L \subseteq G \) be a discrete and closed set. Then \( L \) is locally finite, and every Delone set \( D \in \Del \) with \( D^{-1} D \subseteq L \) is automatically of finite local complexity.  
The corresponding space
\[
\DelL := \set{ D \in \Del }{ D^{-1} D \subseteq L }
\]
is a closed subset of \( \Del \). This space is also endowed with the so-called local rubber topology \cite{BaLe04}, which coincides with the previously described topology on \( \Del \).

We now reformulate Corollary~\ref{cor:Charact_Converg_hulls_Patches} when restricting to \( \DelL \).

\begin{proposition}
\label{prop:Char_Conv_Hull_Delone}
Let $r_0,R_0>0$ and $L\subseteq G$ be a discrete and closed set.  
Then the following are equivalent for $D_n,D\in\DelL, n\in\N$.
\begin{enumerate}[label=(\roman*)]
\item $\lim_{n\to\infty} \hull(D_n)=\hull(D)$ in $\wInv$.
\item For all $R>0$, there exists an $n_0\in\N$ such that $W_R(D_n)=W_R(D)$ for all $n\geq n_0$.
\end{enumerate}
\end{proposition}

\begin{proof}
Since \( D^{-1} D \subseteq L \), each \( R \)-patch of \( D \in \DelL \) is a subset of \( L \cap \oB(R) \).  
This intersection is finite because \( L \) is discrete and closed.  
Consequently, there are only finitely many subsets of \( L \cap \oB(R) \), and thus only finitely many possible \( R \)-patches, uniformly for all \( D \in \DelL \).  
Hence, the statement follows immediately from Proposition~\ref{prop:HausdorffDistance_transv}.
\end{proof}

\subsection{Symbolic dynamical systems}
\label{subsec:symb_syst_as_weighted_Delone}
Let $\Gamma\subseteq G$ be a discrete cocompact subgroup (i.e. $G/\Gamma$ is compact) and $\Aa$ a finite set. If $G$ was discrete, $\Gamma=G$ would be a possible chopice. Then $\Gamma\in\Del$ for a suitable choice of $r_0,R_0>0$. 
Recall the notion of the symbolic dynamical system $\Aa^\Gamma:=\{\omega:\Gamma\to\Aa\}$ introduced in Example~\ref{ex:SymbDynSyst}. The set $\Aa^\Gamma$ is equipped with the metric
\[
\dAG( \alpha, \beta ) := \inf\set{\frac{1}{r+1}}{r\geq 0 \textrm{ such that } \alpha|_{B(r)}=\beta|_{B(r)}}, \qquad \alpha,\beta\in\Aa^\Gamma.
\]
Note that if $\alpha(e)\neq \beta(e)$, then $\dAG( \alpha, \beta )=1$. Let 
\[
\delta_\Aa:\Kk(\Aa^\Gamma)\times \Kk(\Aa^\Gamma) \to [0,1]
\] 
be the Hausdorff metric induced from $\dAG$, see eq.~\eqref{eq:Hausdorff_Metric}.
This choice of a metric is made in \cite{BeBeCo19,BecTak25,BaBePoTe25} to provide spectral estimates by the corresponding dynamical subsystems in $\Aa^\Gamma$.
We will show here that $\Aa^\Gamma$ can be represented as a closed subset of the weighted Delone sets $\wDel$ and that the Hausdorff metric $\delta_H$ inherited from $\dw$ is equivalent to the Hausdorff metric $\delta_\Aa$. 

Without loss of generality (by identifying letters with positive real numbers), we assume that \( \Aa \subseteq [\sigma^{-1}, \sigma] \) for some \( \sigma > 1 \).  
Then the configuration space \( \Aa^\Gamma \) embeds naturally into \( \wDel \) via
\[
\Aa^\Gamma \ni \omega \mapsto (\Gamma, \omega) \in \wDel.
\]
In this way, \( \Aa^\Gamma \) is regarded as a subset of \( \wDel \).

For $(\Gamma,\alpha)\in\wDel$ with $\alpha\in\Aa^\Gamma$, we have 
\[
\tran(\Gamma,\alpha) 
	= \set{(\Gamma,\omega)}{\omega \in \ol{\Orb_\Gamma(\alpha)}},
\]
where $\Orb_\Gamma(\alpha):=\set{\gamma\alpha}{\gamma\in\Gamma}$ is the orbit of $\alpha\in\Aa^\Gamma$ with respect to the \( \Gamma \)-action, see Example~\ref{ex:SymbDynSyst}. For $\alpha\in\Aa^\Gamma$, we then identify the $R$-patches $W_R(\Gamma,\alpha)$ with the set
\begin{equation}
\label{eq:R-patch_Symbolic}
W_R(\alpha) =
	\set{(\gamma\alpha)|_{B(R)}}{\gamma\in\Gamma}.
\end{equation}
The following statement shows that the metric $\dAG$ is equivalent to the metric inherited from $(\wDel,\dw)$ if we view $\Aa^\Gamma$ as a subset of $\wDel$.

\begin{proposition}
\label{prop:Conv_Symb-Dyn-Syst}
Let $r_0,R_0>0$, $\sigma>1$ and $\Aa\subseteq [\sigma^{-1},\sigma]$. Then the following assertions hold for $\Gamma\in\Del$.
\begin{enumerate}[label=(\alph*)]
\item For $\alpha,\beta\in\Aa^\Gamma$, we have
$$
c_1 \dAG(\alpha,\beta) \leq \dw\big( (\Gamma,\alpha), (\Gamma,\beta) \big) \leq 2 \dAG(\alpha,\beta) 
$$
where  $c_1 := \min\big\{ \tfrac{r_0}{2},\, \tfrac{\sigma^{-1}}{2},\, \min_{a\neq b\in\Aa}|a-b|\big\}$.
\item The symbolic dynamical system \( (\Aa^\Gamma, \Gamma,\tau) \) is Lipschitz continuous if $\Aa^\Gamma$ is equipped with the metric $\dAG$.
\item The set $\set{(\Gamma,\alpha)}{\alpha\in\Aa^\Gamma}$ is a closed (in $\dw$) and $\Gamma$-invariant subset of $\wDel$. Furthermore, the induced topology coincides with the product topology on $\Aa^\Gamma$.
\item For all $\alpha,\beta\in\Aa^\Gamma$, we have
\[
c_1\delta_\Aa \left(\ol{\Orb_\Gamma(\alpha)},\ol{\Orb_\Gamma(\beta)}\right)
	\leq \delta_H\left(\tran(\Gamma,\alpha),\tran(\Gamma,\beta)\right)
	\leq 2\delta_\Aa \left(\ol{\Orb_\Gamma(\alpha)},\ol{\Orb_\Gamma(\beta)}\right)
\]
and
$$
\delta_\Aa \left(\ol{\Orb_\Gamma(\alpha)},\ol{\Orb_\Gamma(\beta)} \right)
	= \inf\set{\frac{1}{r+1}}{r\geq 0 \textrm{ such that } W_r(\alpha)=W_r(\beta)}
$$
\end{enumerate}
\end{proposition}

\begin{proof}
(a) Let $\alpha,\beta\in\Aa^\Gamma$. We first prove the first inequality. If $\dw\big( (\Gamma,\alpha), (\Gamma,\beta) \big)\geq c_1$, then $\dAG(\alpha,\beta)\leq 1$ yields the estimate. If $\dw\big( (\Gamma,\alpha), (\Gamma,\beta) \big)< c_1$, choose an $r>0$ and an $\varepsilon>0$ such that
$$
\dw\big( (\Gamma,\alpha), (\Gamma,\beta) \big)
	< \frac{1}{r}
	<\dw\big( (\Gamma,\alpha), (\Gamma,\beta) \big)+\varepsilon 
	< c_1.
$$
Thus, $\alpha|_{B(r)}=\beta|_{B(r)}$ follows from the first inequality, the definition of $\dw$ and $c_1\leq \min_{a\neq b\in\Aa}|a-b|$. Then $\dAG(\alpha,\beta)\leq \tfrac{1}{r} < \dw\big( (\Gamma,\alpha), (\Gamma,\beta) \big)+\varepsilon$ is concluded proving the first claimed estimate.

Next, we prove 
\[
\dw\big( (\Gamma,\alpha), (\Gamma,\beta) \big) \leq 2 \dAG(\alpha,\beta).
\]
Recall that $\dw$ takes only values smaller or equal than $c_0:=\min\big\{\frac{\sigma^{-1}}{2}, \frac{r_0}{2}\big\}$. Thus, the claimed estimate holds whenever $\dAG(\alpha,\beta)\geq \tfrac{c_0}{2}$.
Next suppose $\dAG(\alpha,\beta)< \tfrac{c_0}{2}<\tfrac{1}{2}$. Let $\varepsilon>0$ and $r>1$ be such that
$$
\dAG(\alpha,\beta) < \tfrac{1}{r+1} \leq \dAG(\alpha,\beta) + \varepsilon<\tfrac{c_0}{2}.
$$
Observe that $\tfrac{1}{r}<c_0$.
Thus, the first inequality yields $\alpha|_{B(r)}=\beta|_{B(r)}$ implying
$$
\dw\big( (\Gamma,\alpha), (\Gamma,\beta) \big) 
	\leq \tfrac{1}{r}
	\leq 2\tfrac{1}{r+1}
	\leq 2\big(\dAG(\alpha,\beta) + \varepsilon\big).
$$
Sending $\varepsilon$ to zero provides the second inequality.

\medskip

(b) This follows by straightforward computations involving (a) and Proposition~\ref{prop:wDel_groupAction_Lipschitz}.

\medskip

(c) This is an immediate consequence of (a).

\medskip

(d) The estimates between the Hausdorff metrics is immediate from (a). If $\omega\in\ol{\Orb_\Gamma(\alpha)}$, then $W_r(\omega)\subseteq W_r(\alpha)$ follows for $r\geq 0$, see e.g. \cite[Cor.~3.3.24]{Bec16}.
The presentation of the Hausdorff metric $\delta_\Aa$ follows then straightforwardly from the definition $\dAG$ using eq.~\eqref{eq:R-patch_Symbolic}, similarly as in Proposition~\ref{prop:HausdorffDistance_transv}.
\end{proof}

\section{Repetitivity of weighted Delone sets}
\label{sec:Repet_WeightDelone}

We address the question of how dynamical properties of the hull of a Delone set are reflected in the structure of the underlying Delone set.  
In this section, we focus on minimality and unique ergodicity of Delone dynamical systems, following the framework developed in \cite{BecHarPog25}.

Minimality of a Delone dynamical system is characterized by repetitivity of the underlying Delone set; see, e.g., \cite{LaPl03, Yok05, FrRi14, BecHarPog25}.  
Roughly speaking, this means that every finite patch occurs in any sufficiently large compact region, uniformly over the space.  
If the required region size depends nicely on the patch size (e.g., linearly), then the corresponding dynamical system is even uniquely ergodic.  
More precisely, linearly repetitive (or linearly recurrent) configurations define uniquely ergodic systems -- a result due to Durand \cite{Dur00} in the symbolic setting for $G=\Z$ and Lagarias and Pleasants \cite{LaPl03} for Delone sets in \( \R^d \).  
In fact, these systems are strictly ergodic, i.e., both minimal and uniquely ergodic.

Various extensions of this principle have been developed in different directions; see \cite{DaLe01, Le02PW, Len03, Bes08, LenMulVes08, Pog13, BesBosLen13, FrRi14, PogSch16}.  
The present work \cite{BecHarPog25} generalizes these results to the non-Abelian setting and to weighted Delone sets.

Following an approach by Damanik and Lenz \cite{DaLe01}, Delone sets that are linearly repetitive satisfy uniform subadditive convergence theorems.  
Such ergodic theorems have a wide range of applications, including the existence of Banach densities \cite{DaHuZh19} and of the integrated density of states (IDS) \cite{LenSto05, PogSch16}.  
The IDS will play a fundamental role in the developments described in Chapter~\ref{chap:DTMP}.

Classically, the ergodic averages are taken over balls in the group.  
In the setting of general groups, these are replaced by so-called (strong) Følner sequences.  
To handle this in our recent work \cite{BecHarPog25}, we employed the Ornstein-Weiss machinery \cite{OW87} based on $\varepsilon$-quasi tilings, further developed in \cite{PogSch16}.

In the context of aperiodic order, two prominent classes of Delone sets are of central interest: cut-and-project sets and substitution systems \cite{Queff10, Fog02, BaaGri13}.  
Recently, there has been considerable interest in extending related results beyond the Abelian case.  
For cut-and-project sets, see e.g. \cite{BjHaPo16, BjHaPo17, BjHa18, BjHaPo22}; for substitution systems, see \cite{BecHarPog21, BaBePoTe25}.

We emphasize that, thanks to the results in \cite{BecHarPog21} (see also Section~\ref{sec:Subst_BeyondAbel}), there exists a rich class of linearly repetitive systems far beyond the Abelian setting.

In the work \cite{BecHarPog25}, the notion of linear repetitivity has been extended in several directions:
\begin{itemize}
\item to weighted Delone sets, thereby covering both classical Delone sets and symbolic dynamical systems as discussed in the previous section;
\item to amenable, unimodular lcsc groups \( G \): if the group \( G \) is not of exact polynomial growth, it is unclear whether balls still form strong F\o lner sequences.  
      This motivates a metric-free formulation, referred to as \emph{tempered repetitivity} of a (weighted) Delone set with respect to a given strong F\o lner sequence;
\item to weighted Delone sets without finite local complexity, leading to the notion of \emph{almost tempered repetitivity}, inspired by ideas from Frettlöh and Richard \cite{FrRi14}.
\end{itemize}

For the sake of presentation, we restrict ourselves here to the case where the group has exact polynomial growth, and refer the reader to \cite{BecHarPog25} for a more detailed discussion.

Let $G$ be an lcsc group with an adapted metric $d_G$ (see Chapter~\ref{chap:Dynam_LR} for the definition) and left-invariant Haar measure $m_G$ on the Borel $\sigma$-algebra of $G$ \cite{Fol16}. Recall that $B(r):=B(e,r)$ denotes the open ball of radius $r>0$ around $e\in G$.

\begin{definition}
\label{def:exact-polyn-growth}
The group $G$ is called of {\em exact polynomial growth} if there are constants $C>0$ and $\kappa\geq 0$ such that
$$
\lim_{t\to\infty} \frac{m_G(B(t))}{Ct^\kappa}=1.
$$
\end{definition}

Let $r_0,R_0>0$ and $\sigma\geq 1$ be fixed. In this section we need to sharpen the notion of a patch of a weighted Delone set and take into account of which set it is a restriction of. To make this clear consider first the following example.

\begin{example}
Consider Delone sets $D_1=\Z$ and $D_2:=2\Z$ in $\R$ and fix the compact set $\oB(2)\subseteq \R$. Then
$p:=D_2\cap \oB(2) = \{-2,0,2\}$ is a patch in $D_1$. This patch also occurs in $D_2$ in the sense that $p\subseteq D_2$. On the other hand, $p\neq D_1\cap \big(-g+\oB(2)\big)$ for all $g\in \R$. In light of this, we would like to say that the $\oB(2)$-patch $p$ does not occur in $D_1$, which is only possible if we remember that $p$ was a restriction of the set $\oB(2)$. 
\end{example}

\begin{definition}[$S$-patch]
For a relatively compact set $S\subseteq G$, a tuple $(\Pi|_S,S)$ is called an $S$-patch of $\Pi\in\wPat$. 
\end{definition}

Note that we do this particular distinction only in this section.
Recall the notion of $\varepsilon$-similar patches introduced on page~\pageref{def:epsilon-similar}. 
We say that an \( S \)-patch \( (\Lambda, S) \) \emph{\( \varepsilon \)-occurs} in a \( T \)-patch \( (\Theta, T) \) if there exists \( g \in G \) such that \( gS \subseteq T \) and
\[
\wdelt\big( (g . \Lambda)|_{gS},\; \Theta|_{gS} \big) < \varepsilon.
\]
Recall from Definition~\ref{def:epsilon-similar} that this condition simply means that \( g . \Lambda \) is \( \varepsilon \)-similar to \( \Theta|_{gS} \).

To capture a quantitative version of approximate pattern recurrence, we introduce the notion of almost repetitivity. 
Roughly speaking, this means that every \( R \)-patch \( \varepsilon \)-occurs within a patch of radius \( \varrho(\varepsilon, R)\geq 1 \).

\begin{definition}
\label{def:almost-repetitive}
A function \( \varrho : (0,1) \times [1, \infty) \to [1, \infty) \) is called an \emph{almost repetitivity function} for \( \Pi \in \wDel \) if for every \( \varepsilon > 0 \) and all \( R \geq 1 \), the following holds:  
For all \( g, h \in G \), the patch \( \big(\Pi|_{gB(R)}, gB(R)\big) \) \( \varepsilon \)-occurs in the patch \( \big(\Pi|_{hB(\varrho(\varepsilon, R))}, hB(\varrho(\varepsilon, R))\big) \).

Whenever such a function exists for \( \Pi \in \wDel \), we call \( \Pi \) \emph{almost repetitive}.
\end{definition}

Following \cite[Prop.~2.4]{BecHarPog25}, the dynamical system \( \hull(\Pi) \) is minimal if and only if \( \Pi \in \wDel \) is almost repetitive.  
The proof follows the lines of \cite{FrRi14}, where the result is established for (unweighted) Delone sets.

We are in particular interested in the following class of almost repetitive weighted Delone sets.

\begin{definition}
\label{def:linearRepetit}
Let $d_G$ be an adapted metric on the group $G$.
A weighted Delone set $\Pi\in\wDel$ is called {\em almost linearly repetitive with respect to $d_G$} if there is a function $c_\Pi:(0,\infty)\to (0,\infty)$ such that  $\varrho:(0,1)\times[1,\infty)\to[1,\infty)$ defined by
\[
\varrho(\varepsilon,R) = c_\Pi(\varepsilon) R,
	\qquad (\varepsilon,R)\in (0,1)\times[1,\infty),
\]
is an almost repetivity function of $\Pi\in\wDel$.

A weighted Delone set $\Pi\in\wDel$ is called {\em linearly repetitive (LR) with respect to $d_G$} if it is almost linearly repetitive with respect to $d_G$ and the constant $c_\Pi>0$ is uniformly bounded in $\varepsilon>0$.
\end{definition}

With these notions at hand, we proved the following result.

\begin{theorem}[\cite{BecHarPog25}]
\label{thm:LR-Minimal_UniqErg}
Let $G$ be a locally compact second countable Hausdorff group with exact polynomial growth with respect to an adapted metric $d_G$. If $\Pi\in\wDel$ is almost linearly repetitive, then $\hull(\Pi)$ is minimal and uniquely ergodic.
\end{theorem}

\begin{proof}
This is proven in \cite[Thm~1.2]{BecHarPog25}.
\end{proof}

A metric-free version of this result -- formulated in terms of tempered repetitivity with respect to a strong F\o lner sequence -- is provided in \cite[Thm.~1.3]{BecHarPog25}.

In the following we construct a class of linearly repetitive weighted Delone sets.

\section{Substitution systems: Going beyond the Abelian world}
\label{sec:Subst_BeyondAbel}

At first glance, it is not obvious whether linearly repetitive (LR) Delone sets exist at all beyond the Abelian setting.  
In \( \Z^d \) or \( \R^d \), so-called primitive substitution systems give rise to LR Delone sets; see \cite{Sol98a, DaLe06-subst}.  
In the joint work \cite{BecHarPog21}, our main aim was to construct LR Delone sets for homogeneous Lie groups. Therefore we developed the notion of (primitive) substitutions for homogeneous Lie groups \( G \) and lattices \( \Gamma \subseteq G \). Since substitutions inherently involve enlarging and subdividing patterns, a compatible dilation structure is required. To this end, nilpotent Lie groups equipped with a dilation structure are considered.

It turned out that this framework produces a rich class of aperiodic, linearly repetitive Delone sets.  
Combined with the results from \cite{BecHarPog25}, presented in Section~\ref{sec:Repet_WeightDelone} (Theorem~\ref{thm:LR-Minimal_UniqErg}), the associated dynamical systems are uniquely ergodic.  
Among other consequences, this leads to the continuity of certain measure-theoretic quantities -- such as the integrated density of states and the autocorrelation -- when combined with the stability results in \cite{BecPog20} (see also Section~\ref{sec:SemiContMeasure}).

A typical example of a homogeneous Lie group is the Heisenberg group \( H_3(\R) \).  
As a set, this group coincides with \( \R^3 \), which might suggest that one could transfer LR sets from \( \R^3 \) to \( H_3(\R) \).  
However, we emphasize that the notion of LR depends on shifting a ball around within the group.  
In the non-Abelian setting -- such as \( H_3(\R) \) -- balls become distorted (as subsets of \( \R^3 \)) when translated (see Figure~\ref{fig:Substitution_Supp_N=4} in Section~\ref{sec:Subst_Examples}), and thus naive extensions of LR from the Euclidean setting have failed so far.

We note that there exist various extensions of substitution systems in the Abelian setting, including, for instance:  
random substitutions \cite{RusSpi18}, substitutions over a compact alphabet \cite{MaRuWa25}, and substitutions with multiple inflation factors \cite{SmiSol21}, see also the review \cite{Man24} for more references.  
Beyond the Abelian context, substitution systems have also been explored in non-Euclidean settings such as the hyperbolic plane \cite{BedHil13} and on trees \cite{BarLep24}.

We outline here the approach of (symbolic) substitution systems developed in \cite{BecHarPog21} and provide few examples. 

We begin with one of the guiding examples: the table \emph{tiling substitution} (also called \emph{domino tiling}) \cite{Sol97,Rob99}, which defines a strictly ergodic dynamical subsystem in \( \Aa^{\Z^2} \), where \( \Aa = \{a, b, c, d\} \) is a four-letter alphabet.  
Recall from Example~\ref{ex:SymbDynSyst} that \( (\Aa^{\Z^2}, \Z^2, \tau) \) defines a symbolic dynamical system (where $\tau$ is the shift) that can be embedded into the space of weighted Delone sets in \( \R^2 \); see Proposition~\ref{prop:Conv_Symb-Dyn-Syst}.

Let
\[
\Pat := \set{ P : M \to \Aa }{ M \subseteq \Z^2 }
\]
denote the set of all finite and infinite \( \Z^2 \)-patches colored by elements of the alphabet \( \Aa \).  
Note that \( \Aa^{\Z^2} \subseteq \Pat \).

Consider the square \( K := \{-1,0\}^2 \subseteq \Z^2 \).  
Then the substitution rule for the table tiling is given by a map \( S_0 : \Aa \to \Aa^K \) defined in Figure~\ref{fig:TableTilingRule}.

\begin{figure}[htb]
    \centering
    \includegraphics[scale=3]{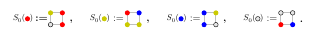}
    \captionsetup{width=0.95\linewidth}
    \caption{Substitution rule for the table tiling. For the sake of presentation, we identify the four letters  $a,b,c,d$ with the colors red, yellow, blue, and gray.}
    \label{fig:TableTilingRule}
\end{figure}

The rule \( S_0 \) naturally extends to a map \( S : \Pat \to \Pat \) by applying the substitution rule letterwise, see e.g. \cite[Chap.~5.1]{Queff10} and \cite[Chap.~4.9]{BaaGri13}.  
This process is sketched in Figure~\ref{fig:SubstitutionTableTiling-iteration}. Roughly speaking, \( S(\omega) \) is defined by scaling \( \Z^2 \) by a factor of 2 and placing at each point \( n \in \Z^2 \) the $2\times 2$ patch \( S_0(\omega(n)) \), which is supported on \( K \).  

Since \( K \) is a fundamental domain of \( 2\Z^2 \subseteq \Z^2 \), this construction is well-defined: there are no overlaps and no gaps (i.e., uncolored lattice points).  
This definition, however, is not entirely canonical, as it requires a choice of origin. This is fixed by imposing covariance with respect to the group action:
\[
S(\gamma \omega) = (2\gamma) \cdot S(\omega), \quad \text{for all } \gamma \in \Z^2.
\]

\begin{figure}[htb]
    \centering
    \includegraphics[scale=2.3]{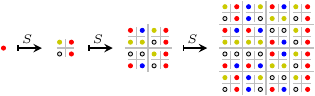}
    \captionsetup{width=0.95\linewidth}
    \caption{Iteration of the table tiling substitution on a single letter. Figure taken from \cite{BaBePoTe25}.}
    \label{fig:SubstitutionTableTiling-iteration}
\end{figure}
By a compactness argument, one concludes that the iteration on a single letter (as sketched in Figure~\ref{fig:SubstitutionTableTiling-iteration}) has a limit point \( \omega_0 \in \Aa^{\Z^2} \) satisfying \( S^2(\omega_0) = \omega_0 \), and composed only of legal patches (defined below). Such a configuration $\omega_0$ is also called fixed point of the substitution.
Since the substitution is primitive (see below), it follows that the orbit closure \( \OmTabl := \overline{\Orb(\omega_0)} \) is independent of the chosen limit point and defines a strictly ergodic dynamical system \( (\OmTabl, \Z^2, \tau) \).

This construction naturally extends to \( \Z^d \), but at first glance becomes more intricate when considering non-Abelian groups,  
as sets deform their geometry under translation. This phenomenon can already be observed in the discrete Heisenberg group \( H_3(\Z) \), as illustrated in Example~\ref{ex:HeisenbergExample-Subst} and Figure~\ref{fig:Substitution_Supp_N=4} in Section~\ref{sec:Subst_Examples}.

To overcome this difficulty, we formalized the concept in \cite{BecHarPog21} to the setting of homogeneous Lie groups \( G \).  
A substitution system in this context is described using certain geometric data, referred to as a \emph{dilation datum} \( \Dd = (G, d_G, (D_\lambda)_{\lambda > 0}, \Gamma, V) \), consisting of:
\begin{enumerate}[label=(D\arabic*)]
    \item a 1-connected, locally compact, second-countable, Hausdorff group \( G \), equipped with an adapted metric \( d_G \);
    \item a one-parameter group of automorphisms \( (D_\lambda)_{\lambda > 0} \) of \( G \), called the \emph{underlying dilation family}, such that
    \[
    d_G \big(D_\lambda(g), D_\lambda(h)\big) = \lambda \cdot d_G(g, h) \quad \text{for all } \lambda > 0 \text{ and } g, h \in G;
    \]
    \item a uniform lattice \( \Gamma < G \) (i.e., discrete and co-compact) satisfying \( D_\lambda(\Gamma) \subseteq \Gamma \) for some \( \lambda > 1 \), and a Borel measurable, relatively compact, (left) fundamental domain \( V \subseteq G \) for \( \Gamma \) (i.e. $G=\Gamma V =\bigsqcup_{\gamma\in\Gamma}\gamma V$) such that the identity element \( e \in G \) lies in the interior of \( V \).
\end{enumerate}

Alongside the dilation datum, we have combinatorial data, referred to as a \emph{substitution datum} \( \Ss = (\Aa, \lambda_0, S_0) \) over \( \Dd \), consisting of:
\begin{enumerate}[label=(S\arabic*)]
    \item a finite set \( \Aa \), referred to as the \emph{alphabet};
    \item a stretch factor \( \lambda_0 > 1 \), associated with the dilation datum \( \Dd \) (see below for details);
    \item a substitution rule \( S_0 : \Aa \to \Aa^{D_{\lambda_0}(V) \cap \Gamma} \).
\end{enumerate}

In order to define the stretch factor, we first recursively introduce the following family of sets.  
Let \( V \) be a (left) fundamental domain of a lattice \( \Gamma \subseteq G \), let \( \lambda > 1 \), \( n \in \N \), and let \( \emptyset \neq M \subseteq \Gamma \).  
We define:
\begin{equation}
\label{eq:V(n,M)-def}
V_\lambda(0, M) := M \cdot V
\qquad \text{and} \qquad
V_\lambda(n, M) := D_\lambda \big( (V_\lambda(n{-}1, M) \cap \Gamma) \cdot V \big),
\end{equation}
where \( M \cdot V := \{ x v \mid x \in M,\, v \in V \} \) denotes the product of the two subsets in the group \( G \).  
For simplicity, we write \( V_\lambda(n) := V_\lambda(n, \{e\}) \) for \( n \in \N \).

It turns out that \( V_{\lambda_0}(n, M) \) coincides with the support of the \( n \)-fold substituted patch \( S^n(\alpha) \), where \( \alpha : M \to \Aa \) and \( S \) is the substitution rule derived from the geometric and combinatorial data (see Proposition~\ref{prop:SubstRule-Properties}).  

As in the classical substitution setting, one requires that the size of these supports \( V_{\lambda_0}(n) \) grows with \( n \).  
This is expressed by demanding that, for sufficiently large \( n \), the sets \( V_{\lambda_0}(n) \) contain balls of radius proportional to \( \lambda_0^n \).  
The following definition provides a sufficient condition to ensure this growth behavior.

\begin{definition}[stretch factor]
\label{def:stretch-factor}
Let \( \Dd = (G, d_G, (D_\lambda)_{\lambda > 0}, \Gamma, V) \) be a dilation datum.  
We say that \( \lambda_0 > 1 \) is \emph{sufficiently large relative to \( V \)} if there exist a constant \( C_- > 0 \), an integer \( s \in \N_0 := \N \cup \{0\} \), and a \( z \in \Gamma \) such that for all \( n \in \N \),
\[
D_{\lambda_0}^n\big( B(z, C_-) \big) \subseteq V_{\lambda_0}(s+n).
\]
We then call \( \lambda_0 > 1 \) a \emph{stretch factor} (associated with \( \Dd \)) if \( D_{\lambda_0}(\Gamma) \subseteq \Gamma \), and  \( \lambda_0 \) is sufficiently large relative to \( V \).
\end{definition}

To simplify notation, we will henceforth write \( V(n) := V_{\lambda_0}(n) \) and \( V(n, M) := V_{\lambda_0}(n, M) \) for \( n \in \N_0 \), whenever \( \lambda_0 \) is a fixed stretch factor.

\begin{remark}
\label{rem:Histor_suff-large}
\begin{itemize}
\item[(a)] The condition that \( \lambda_0 > 1 \) is sufficiently large relative to \( V \) states that the supports \( V(n) \) contain metric balls of radius \( C_- \lambda_0^n \) using $D_{\lambda_0}$ is a dilation implying \(D_{\lambda_0}^n\big( B(z, C_-) = B\big( D_{\lambda_0}^n(z), C_-\lambda_0^n\big)\).  
This property is fundamental for proving the existence of fixed points associated with the substitution and for establishing their linear repetitivity, which in turn implies strict ergodicity; see \cite{BecHarPog21,BecHarPog25}.

\item[(b)] Historically, this notion of ``sufficiently large'' differs from the one used in the ArXiv version of \cite{BecHarPog21} (September 2021). The main focus was to construct substitutions in the non-Abelian setting and find a class of examples. We later observed that one could also include substitutions with moving centers like block substitutions substitutions with scaling factor \( \lambda_0 = 2 \), such as the table tiling.  
The revised definition presented here is weaker and accommodates all block substitutions in this framework; see \cite[Prop.~3.1]{BaBePoTe25} briefly discussed in Section~\ref{sec:Subst_Examples}.  
An updated version of \cite{BecHarPog21} will reflect this adjustment in the definition of ``sufficiently large'' accordingly. For convenience, the adjustments and proofs are provided in the Appendix~\ref{App:Substituion_Systems}.
\end{itemize}
\end{remark}

We prove in \cite{BecHarPog21} that under this condition the sets \( V(n) \) grow comparably to metric balls in $G$, see also Theorem~\ref{thm:SupportGrowth} in the Appendix~\ref{App:Substituion_Systems}. Also we provide sufficient criteria in Lemma~\ref{lem:suff-large}, which we apply in the following for some examples.

To illustrate how classical substitution systems embed into the present framework, we revisit the table tiling substitution on \( \Z^2 \) and verify that it satisfies the structural assumptions introduced above. Already in the table tiling we observe that the adjustment in the notion of sufficiently large is necessary to include this example. Specifically, $V(n)$ only contains a shifted ball with radius comparable to $\lambda_0^n$ but without the shift this fails.
Therefore, we now reinterpret the example of the table tiling in terms of the geometric and combinatorial substitution data defined in \cite{BecHarPog21}.

\begin{example}
\label{ex:Geom_Comb_data-TableTiling}
The geometric data  \( \Dd = (G, d_G, (D_\lambda)_{\lambda > 0}, \Gamma, V) \) for the table tiling example are given by:
\begin{align*}
G &= \R^2,                          & \Gamma &= \Z^2, \\
d(x, y) &= \max_{1 \leq j \leq 2} |x_j - y_j|, 
         & V &= \left[-\tfrac{1}{2}, \tfrac{1}{2}\right)^2, \\
D_\lambda(x_1, x_2) &= (\lambda x_1, \lambda x_2).
\end{align*}

In order to prove that \( \lambda_0 := 2 \) is sufficiently large relative to \( V \), we apply the sufficient criterion provided in the Appendix~\ref{App:Substituion_Systems} in Lemma~\ref{lem:suff-large}.  
Set \( r_- := \tfrac{1}{2} \) and \( r_+ := \tfrac{3}{4} \), so that
\[
B(e, r_-) \subseteq V \subseteq \overline{V} \subseteq B(e, r_+).
\]
Let \( r_0 := \tfrac{7}{4} \) and \( z_0 := (-2, -2) \in \Gamma \). Then
\[
B(z_0, r_0) = \left(-3\tfrac{3}{4}, -\tfrac{1}{4}\right)^2 \cap \Gamma = \{-3, -2, -1\}^2 \subseteq \{-3, -2, -1, 0\}^2 = V(2).
\]
Note that one can prove inductively that
\[
V(n) = \Z^2 \cap \left[-2^n + 1, 1\right)^2
\]
Hence, \( r_0 > \tfrac{r_+ \lambda_0}{\lambda_0 - 1} \) and \( B(z_0, r_0) \subseteq V(s) \) with \( s_0 = 2 \).  
According to Lemma~\ref{lem:suff-large}~(b), this implies that \( \lambda_0 = 2 \) is sufficiently large relative to \( V \), with parameters
\[
C_- = (r_0 - r_+) - \tfrac{r_0}{\lambda_0} = \tfrac{1}{8}, \qquad
s = 2, \qquad
z = (-2, -2).
\]

Observe that \( D_{\lambda_0}(V) \cap \Gamma = \{-1, 0\}^2 \).  
Now consider the finite alphabet \( \Aa = \{a, b, c, d\} \), and let \( S_0 : \Aa \to \Aa^{D_{\lambda_0}(V) \cap \Gamma} \) be the table tiling substitution as defined in Figure~\ref{fig:TableTilingRule}.  
Then \( \Ss = (\Aa, \lambda_0, S_0) \) defines a valid substitution datum.
\end{example}

This example illustrates how the classical table tiling substitution naturally fits into the general substitution framework developed in \cite{BecHarPog21}.  
More generally, this framework encompasses all block substitutions on \( \Z^d \); see Section~\ref{sec:Subst_Examples} for their definition.
The embedding of block substitutions into this framework is discussed in detail in \cite[Prop.~3.1]{BaBePoTe25}.

We now continue with the description of the substitution rule, the associated subshift, and its basic properties.  
Let \( \Aa \) be a finite alphabet and \( \Gamma \) a discrete group. Consider the set of all (possibly infinite) patches:
\[
\Pat := \left\{ P : M \to \Aa \,\middle|\, M \subseteq \Gamma \right\}.
\]
Note that \( \Aa^\Gamma \subseteq \Pat \), where \( \Aa^\Gamma \) is equipped with the product topology.  
Following \cite{BecHarPog21}, we obtain the following result.

\begin{proposition}
\label{prop:SubstRule-Properties}
Let \( \Ss = (\Aa, \lambda_0, S_0) \) be a substitution datum over the dilation datum \( \Dd = (G, d_G, (D_\lambda)_{\lambda > 0}, \Gamma, V) \).  
Then there exists a unique map \( S : \Pat \to \Pat \) satisfying the following conditions:
\begin{itemize}
\item \emph{Equivariance:}  
\( S^n(\gamma P) = D_{\lambda_0}^n(\gamma)\, S^n(P) \) for all \( P \in \Pat \), \( \gamma \in \Gamma \);
\item \emph{Restriction:}  
\( \big(S^n(P)\big)\big|_{V(n,M) \cap \Gamma} = S^n\big(P|_M\big) \) for all \( P \in \Pat \), \( n \in \N_0 \), and non-empty \( M \subseteq \Gamma \).
\end{itemize}

Moreover, the following properties hold:
\begin{enumerate}[label=(\alph*)]
\item The restriction \( S : \Aa^\Gamma \to \Aa^\Gamma \subseteq \Pat \) is continuous.
\item If \( P \in \Pat \) is supported on \( M \subseteq \Gamma \), then for each \( n \in \N_0 \), the patch \( S^n(P) \) is supported on \( V(n, M) \cap \Gamma \).
\item Let \( P : M \to \Aa \in \Pat \). Then for each \( \gamma \in V(1, M) \), there exists a unique \( \eta \in M \) such that
\[
\gamma \in D_{\lambda_0}(\eta) \big(D_{\lambda_0}(V) \cap \Gamma\big)
\quad \text{and} \quad
S(P)(\gamma) = S_0\big(P(\eta)\big)\big(D_{\lambda_0}(\eta)^{-1} \gamma\big).
\]
\end{enumerate}
\end{proposition}

\begin{proof}
This is proven in \cite[Prop.~2.7, Prop.~5.6, Prop.~5.12, Lem.~5.16]{BecHarPog21}.
\end{proof}

\begin{definition}[Substitution map]
\label{def:substitution-map}
Let \( \Ss = (\Aa, \lambda_0, S_0) \) be a substitution datum over the dilation datum \( \Dd = (G, d_G, (D_\lambda)_{\lambda > 0}, \Gamma, V) \).  
The unique map \( S : \Pat \to \Pat \) defined in Proposition~\ref{prop:SubstRule-Properties}, which satisfies the equivariance and restriction conditions, is called the \emph{substitution map} associated with \( \Dd \) and \( \Ss \).
\end{definition}

Proposition~\ref{prop:SubstRule-Properties}~(b) justifies that the recursively defined sets \( V(n, M) \) from eq.~\eqref{eq:V(n,M)-def} indeed describe the support of the \( n \)-fold substituted patch \( S^n(P) \), where \( P : M \to \Aa \).  
Moreover, Proposition~\ref{prop:SubstRule-Properties}~(c) provides an explicit expression for how the substitution map acts pointwise.

We now proceed to define the associated dynamical system \( \Omega(S) \subseteq \Aa^\Gamma \) for a substitution map \( S \) associated with \( \Dd \) and \( \Ss \).  
There are various approaches in the literature to construct such systems, either via fixed points or via the set of legal patches; see for instance \cite{Queff10,BaaGri13}.

In Section~\ref{sec:Repet_WeightDelone}, we already defined when a patch $\varepsilon$-occurs in another one. Since we study here symbolic dynamical systems over a finite alphabet, we can drop the $\varepsilon$-dependence.

\begin{definition}[Subpatch relation]
\label{def:Subpatch-relation}
A patch \( P : M \to \Aa \in \Pat \) is a \emph{subpatch of} (or \emph{occurs in patch}) \( Q : F \to \Aa \in \Pat \), if there exists \( \gamma \in \Gamma \) such that \( \gamma M \subseteq F \) and 
\[
Q(\gamma \eta) = P(\eta) \quad \text{for all } \eta \in M.
\]
In this case, we write \( P \prec Q \).  
\end{definition}

An example is illustrated in Figure~\ref{fig:Subpatches_legal}.  
Note that we do not require either patch to have finite support. Furthermore, the subpatch relation $\prec$ is an order relation.

\begin{figure}[htb]
    \centering
    \includegraphics[scale=3.6]{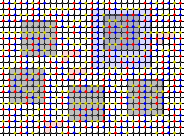}
    \captionsetup{width=0.95\linewidth}
    \caption{A larger patch is illustrated and several supatches are marked by coloring the background in gray, blue and red.}
    \label{fig:Subpatches_legal}
\end{figure}

\begin{definition}[Legal patch]
\label{def:legal_patch}
Let \( S \) be a substitution map associated with \( \Dd \) and \( \Ss \).  
A patch \( P : M \to \Aa \in \Pat \), with \( M \subseteq \Gamma \) finite, is called \emph{\( S \)-legal} if there exist \( n \in \N \) and \( a \in \Aa \) such that \( P \prec S^n(a) \). The set of all legal patches associated with the substitution map $S$ is denoted by $W(S)\subseteq \Pat$.
\end{definition}

Recall from Definition~\ref{def:Weighted_Dictionary} that \( R \)-patches were defined for weighted Delone sets and their associated dynamical systems.  
Symbolic dynamical systems can be identified with weighted Delone sets; see Proposition~\ref{prop:Conv_Symb-Dyn-Syst}.  

In analogy, for \( \omega \in \Aa^\Gamma \) and \( R > 0 \), we define the set of \( R \)-patches by
\[
W(\omega)_R := \left\{ (\gamma \omega)|_{B(e,R)} \;\middle|\; \gamma \in \Gamma \right\},
\qquad 
W(\omega) := \bigcup_{R > 0} W(\omega)_R \subseteq \Pat.
\]

Using this notation, we define the associated substitution subshift as
\[
\Omega(S) := \left\{ \omega \in \Aa^\Gamma \;\middle|\; W(\omega) \subseteq W(S) \right\} \subseteq \Aa^\Gamma.
\]
The following results (Proposition~\ref{prop:Basic_SubstitutionSubshift} and Theorem~\ref{thm:Primitive-LR}) are classical in the Euclidean setting; see e.g.~\cite{Sol97,Sol98a,AnPu98,Dur00,Fog02,DaLe06-subst,BaaGri13} and references therein.

\begin{proposition}[\cite{BecHarPog21}]
\label{prop:Basic_SubstitutionSubshift}
Let \( S \) be a substitution map associated with a dilation datum \( \Dd = (G, d_G, (D_\lambda)_{\lambda > 0}, \Gamma, V) \) and a substitution datum \( \Ss = (\Aa, \lambda_0, S_0) \).  
Then the following holds:
\begin{enumerate}[label=(\alph*)]
\item \( \Omega(S) \) is non-empty, compact, and invariant under both \( S \) and the natural \( \Gamma \)-action. In particular, \( \Omega(S) \in \inv \).
\item There exist \( k \in \N \) and \( \omega \in \Omega(S) \) such that \( S^k(\omega) = \omega \). Any such \( \omega \) is called a \emph{fixed point}.
\end{enumerate}
\end{proposition}

\begin{proof}
This is proven in \cite[Thm.~1.2]{BecHarPog21}.
\end{proof}

We now address the question under which conditions the subshift \( \Omega(S) \) is minimal and uniquely ergodic. The key tool is the notion of linear repetitivity.

\begin{definition}[Primitive substitution]
\label{def:primitive_substitution}
Let \( S \) be a substitution map associated with \( \Dd \) and \( \Ss = (\Aa, \lambda_0, S_0) \).  
We call \( S \) \emph{primitive} if there exists \( L \in \N \) such that for all \( a, b \in \Aa \), one has \( a \prec S^L(b) \).
\end{definition}

It is straightforward to check that the table tiling substitution is primitive with \( L = 2 \): each letter of the alphabet occurs in the second substitution iterate \( S^2(a) \) for any \( a \in \Aa \), as illustrated in Figure~\ref{fig:Table_2nd_Iteration}.

\begin{figure}[htb]
    \centering
    \includegraphics[scale=0.65]{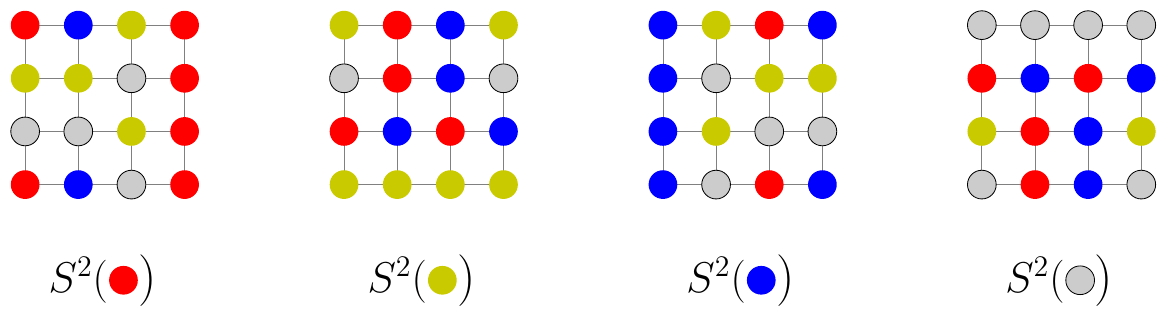}
    \captionsetup{width=0.95\linewidth}
    \caption{The second iterate \( S^2(a) \) for each $a\in \Aa$ of the table tiling contains all letters of the alphabet.}
    \label{fig:Table_2nd_Iteration}
\end{figure}

\begin{theorem}[\cite{BecHarPog21}]
\label{thm:Primitive-LR}
Let \( S \) be a substitution map associated with \( \Dd \) and \( \Ss = (\Aa, \lambda_0, S_0) \).  
If \( S \) is primitive, then the associated substitution subshift \( \Omega(S) \) is strictly ergodic, i.e., minimal and uniquely ergodic.  
Moreover, every element in \( \Omega(S) \) is linearly repetitive.
\end{theorem}

\begin{proof}
The result is proven in \cite[Thm.~1.2]{BecHarPog21}.  
For context, we note that primitivity ensures minimality and linear repetitivity of all configurations in \( \Omega(S) \); see \cite[Thm.~7.4]{BecHarPog21}.  
Moreover, the lattice \( \Gamma \) has exact polynomial growth (see Section~\ref{sec:Repet_WeightDelone} for a definition) by \cite[Sec.~3.4]{BecHarPog21}. Thus, $\Gamma$ is amenable by \cite[Prop.~4.4]{BecHarPog25}.  
Combining these properties with the results of \cite{BecHarPog25}, in particular Theorem~\ref{thm:LR-Minimal_UniqErg}, yields unique ergodicity of \( \Omega(S) \).
\end{proof}

To further analyze the dynamical structure, we investigated under which conditions elements in the subshift are non-periodic.  
We call a subshift \( \Omega \subseteq \Aa^\Gamma \) \emph{weakly aperiodic} if it contains at least one non-periodic element; see Definition~\ref{def:(non-)periodic}. If all elements are non-periodic, we call \( \Omega \) \emph{strongly aperiodic}.  
In the Abelian case, weak aperiodicity and minimality already imply strong aperiodicity. 
This implication does not directly carry over in the non-Abelian setting and requires separate analysis; see a discussion in \cite[Sec.~8.2]{BecHarPog21}.

One classical strategy to establish aperiodicity in symbolic substitutions in $\Z^d$ involves the existence of proximal pairs \cite{BaaGri13,BarOl14}.  
In geometric settings, injectivity of the substitution often suffices to construct aperiodic configurations \cite{Sol97,Sol98a,AnPu98}.  
Inspired by the latter, we introduce a notion of non-periodicity for general substitution systems:

\begin{definition}[Non-periodic substitution]
\label{def:nonperiodic-subst}
Let \( S \) be a substitution map associated with \( \Dd \) and \( \Ss = (\Aa, \lambda_0, S_0) \).  
Then \( S \) is called \emph{non-periodic} if \( S_0 \) is injective and
\[
\left( \gamma^{-1} S(a) \right)\big|_{\gamma^{-1} D(V) \cap D(V)}
\neq S(b)\big|_{\gamma^{-1} D(V) \cap D(V)}
\]
for all \( \gamma \in D(V) \cap \Gamma \setminus \{e\} \) and \( a, b \in \Aa \).
\end{definition}

This notion of non-periodic substitution has not been introduced in the Abelian setting, as alternative methods are typically used to establish aperiodicity.

\begin{theorem}[\cite{BecHarPog21}]
\label{thm:NonPeriodicSubshift}
Let \( S \) be a non-periodic substitution. Then the associated subshift \( \Omega(S) \) is weakly aperiodic.  
In particular, every fixed point \( \omega \in \Omega(S) \) is non-periodic.
\end{theorem}

We emphasize that the non-periodicity of the substitution rule \( S \) guarantees that all of its fixed points are non-periodic.

Whether non-periodicity of \( S \) also implies strong aperiodicity of \( \Omega(S) \) remains an open and subtle question.  
Partial results have been obtained for 2-step nilpotent Lie groups; see \cite[Thm.~1.4]{BecHarPog21}.  
This includes, in particular, the discrete Heisenberg group, discussed in Example~\ref{ex:HeisenbergExample-Subst} below.

In contrast to the Abelian case -- where minimality combined with weak aperiodicity implies strong aperiodicity -- there is currently no general argument that establishes this implication in the non-Abelian setting.  
We suspect, however, that the partial result for 2-step nilpotent Lie groups may extend more broadly -- possibly even to all homogeneous Lie groups -- by an inductive argument along the lines of \cite{BecHarPog21}.

\section{A zoo of examples}
\label{sec:Subst_Examples}

Building on the substitution framework developed in \cite{BecHarPog21} (as presented in the previous section), it is natural to ask under which conditions such substitutions exist and whether classical examples fall within this setting. We briefly address these questions here and refer the reader to the more detailed discussion in \cite{BecHarPog21}. In addition, we present a selection of concrete examples.

In order to construct such substitutions explicitly, the original work \cite{BecHarPog21} introduces a structural splitting condition with respect to an Abelian component of the group. This leads to the notion of \emph{RAHOGRASP} groups (rationally homogeneous Lie groups with rational spectrum) and the concept of a good substitution rule.

\begin{theorem}[\cite{BecHarPog21}]
\label{thm:Good_Subst_Rule}
For every RAHOGRASP group \( G \) of dimension at least two and every finite alphabet \( \Aa \) with \( \#\Aa \geq 2 \), there exists a dilation datum \( \Dd \) over \( G \) and a primitive and non-periodic substitution datum \( \Ss \) over \( \Dd \) with alphabet \( \Aa \).
\end{theorem}

The proof given in \cite[Thm.~1.1]{BecHarPog21} is constructive and uses the notion of a \emph{good substitution rule} as defined in \cite[Sec.~6.2]{BecHarPog21}. This construction still leaves considerable flexibility in the choice of the actual substitution, but defines a primitive and non-periodic substitution. In Example~\ref{ex:HeisenbergExample-Subst}, we provide an example of a good substitution rule for the Heisenberg group.

As a direct consequence of Theorem~\ref{thm:Good_Subst_Rule}, we conclude in \cite[Cor.~1.5]{BecHarPog21} that every RAHOGRASP group $G$ admits a linearly repetitive and non-periodic Delone set.

Let us return to the symbolic dynamical systems. A natural question is whether this setting includes classically studied examples such as block substitutions. A \emph{block substitution} on \( \Aa^{\Z^d} \) with finite alphabet \( \Aa \) is defined via a vector \( m = (m_1, \dots, m_d) \in \N^d \) with \( m_j > 1 \) for all \( 1 \leq j \leq d \), and a substitution rule \( S_0 : \Aa \to \Aa^{K(m)} \), where
\[
K(m) := \Z^d \cap \prod_{j=1}^d \left[-\tfrac{m_j}{2}, \tfrac{m_j}{2}\right)
\]
is the support of the substitution rule. As in the case of the table tiling discussed previously, the rule extends uniquely to a substitution map \( S : \Pat \to \Pat \); see \cite{Sol97,Sol98a,Fra05,BaaGri13}.

It turns out that such substitutions can be embedded into the general substitution framework by defining a suitable dilation datum and substitution datum that generate the same substitution map, similarly as we did in Example~\ref{ex:Geom_Comb_data-TableTiling}. For a formal statement we refer to \cite[Prop.~3.1]{BaBePoTe25}.

We emphasize that the entries \( m_j \) in a block substitution may differ, allowing for non-uniform scaling in different coordinate directions. For further illustrative examples, we refer to \cite{Sol97,Sol98a,Fra05,BaaGri13} and the forthcoming PhD thesis \cite{Ten25-Phd} by Lior Tenenbaum.

To demonstrate the flexibility of the framework, we two examples showing that even certain non-block substitutions can be incorporated into the setting developed in \cite{BecHarPog21}.  
We consider here an example of a digit substitution originally discussed in \cite[Exam.~2.3]{FraMan22}.

\begin{example}
\label{ex:NonSquare-Substitution-Z2}
Consider the dilation datum \( \Dd := (\R^2, d_\infty, (D_\lambda)_{\lambda > 0}, \Z^2, V) \),
where:
\begin{itemize}
    \item \( d_\infty(x,y) := \max_{j=1,2} |x_j - y_j| \);
    \item \( D_\lambda(x,y) := (\lambda x, \lambda y) \);
    \item \( V := -\left(\tfrac{1}{4},\tfrac{1}{4}\right) + \left(\left([0,1)^2 \setminus \left[\tfrac{1}{2},1\right)^2 \right) \cup \left[-\tfrac{1}{2},0\right)^2\right) \).
\end{itemize}

\begin{figure}[htb]
    \centering
    \includegraphics[scale=3.4]{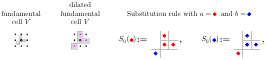}
    \captionsetup{width=0.95\linewidth}
    \caption{The gray shaded area represent first the fundamental domain $V$ and in the second picture $V(1)=D_2(V)$, where the origin is marked by a black circle. The associated lattice points $V\cap\Z^2$ and $V(1)\cap\Z^2$ are colored pink. The substitution rule \( S_0 \) of this digit substitution, originally defined in \cite{FraMan22}, is given on the right.}
    \label{fig:NonSquare-Substitution-Z2-Lam=2}
\end{figure}

In order to prove that \( \lambda_0 = 2 \) is a stretch factor for \( \Dd \), observe that $\ol{V}\subseteq B(e,r_+)$ for $r_+=1$. Then $r_0:=2.1>\tfrac{r_+\lambda_0}{\lambda_0-1}$ and for $z_0=(-6,-6)$ and $s_0=4$ we conclude $B(z_0,r_0)\subseteq V(s_0)$, see Figure~\ref{fig:NonSquare-Substitution-Z2-Lam=2-V(4)}. Thus, Lemma~\ref{lem:suff-large}~(b) implies that $\lambda_0=2$ is sufficiently large relatively to $V$.

\begin{figure}[htbp]
    \centering
    \includegraphics[scale=1]{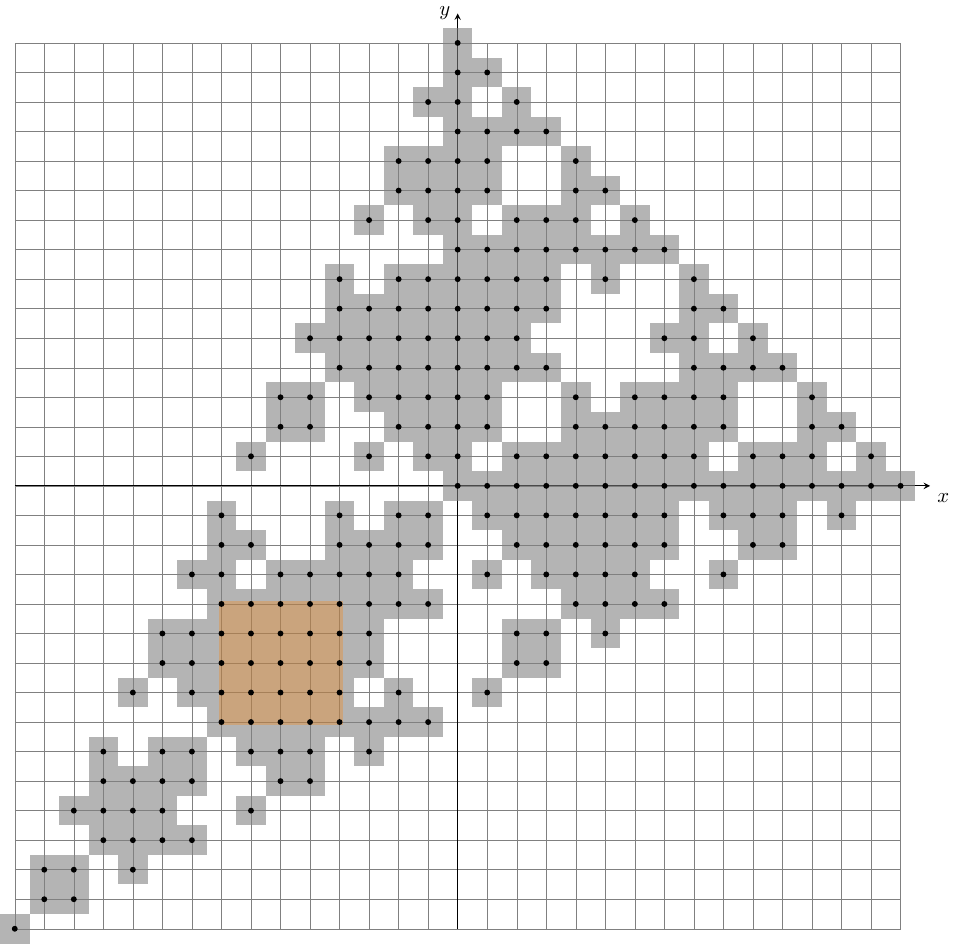}
    \captionsetup{width=0.95\linewidth}
    \caption{The gray shaded region indicates the set $V(4)$ for $\lambda_0 = 2$ in Example~\ref{ex:NonSquare-Substitution-Z2}. All lattice points in $\Z^2 \cap V(4)$ are marked by black bullets. The orange square represents the ball $B(z_0, r_0)$ in the metric $d_\infty$, which is entirely contained within $V(4)$.}
    \label{fig:NonSquare-Substitution-Z2-Lam=2-V(4)}
\end{figure}

The corresponding support after one substitution step satisfies (see also Figure~\ref{fig:NonSquare-Substitution-Z2-Lam=2} for an illistration):
\[
V(1)\cap\Z^2 = \{ (-1,-1), (0,0), (1,0), (0,1) \}.
\]

Let \( \Aa = \{a,b\} \) be the alphabet, and define the substitution rule \( S_0 : \Aa \to \Aa^{V(1)\cap\Z^2} \) as illustrated in Figure~\ref{fig:NonSquare-Substitution-Z2-Lam=2}. This rule is not of block type, as its support is not a rectangular region. Following \cite[Exam.~2.3]{FraMan22} it is a digit substitution.
Denote by $S$ the associated substitution map with $\Dd$ and $\Ss=(\Aa,2,S_0)$. Since $S$ is primitive the associated subshift $\Omega(S)\subseteq \Aa^{\Z^2}$ is strictly ergodic and every element is linearly repetitive.
\end{example}

Next, we present another example of a non-block substitution on $\Z^2$ that fits into the framework developed in \cite{BecHarPog21}.

\begin{figure}[htbp]
    \centering
    \includegraphics[scale=3.4]{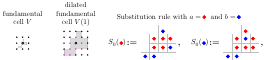}
    \captionsetup{width=0.95\linewidth}
    \caption{The gray shaded region illustrates the fundamental domain \( V \) (left) and its image under the dilation \( D_3 \), i.e., \( V(1) = D_3(V) \) (right). The origin is marked with a black circle, and the corresponding lattice points in \( V \cap \Z^2 \) and \( V(1) \cap \Z^2 \) are shown in pink. The substitution rule \( S_0 \) is defined over the support \( V(1) \cap \Z^2 \).}
    \label{fig:NonSquare-Substitution-Z2-Lam=3}
\end{figure}

\begin{example}
\label{ex:NonSquare-Substitution-Z2_lambda=3}
Consider the dilation datum \( \Dd := (\R^2, d_\infty, (D_\lambda)_{\lambda > 0}, \Z^2, V) \), where:
\begin{itemize}
    \item \( d_\infty(x,y) := \max_{j=1,2} |x_j - y_j| \);
    \item \( D_\lambda(x,y) := (\lambda x, \lambda y) \);
    \item \( 
    	V := \left[ -\tfrac{2}{3}, 0 \right) \times \left[ -\tfrac{2}{3}, -\tfrac{1}{3} \right)  
    		\cup \left[ -\tfrac{1}{3}, \tfrac{2}{3} \right) \times \left[ -\tfrac{1}{3}, \tfrac{1}{3} \right)
    		\cup \left[ 0, \tfrac{1}{3} \right) \times \left[ \tfrac{1}{3}, \tfrac{2}{3} \right)
    	\).
\end{itemize}

For \( \lambda_0 = 3 \), a direct computation yields
\[
V(1) \cap \Z^2 = \{ (-2,-2), (-1,-2), (-1,-1), (-1,0), (0,-1), (0,0), (1,-1), (1,0), (0,1) \},
\]
as visualized in Figure~\ref{fig:NonSquare-Substitution-Z2-Lam=3}.

\begin{figure}[htb]
    \centering
    \includegraphics[scale=1.22]{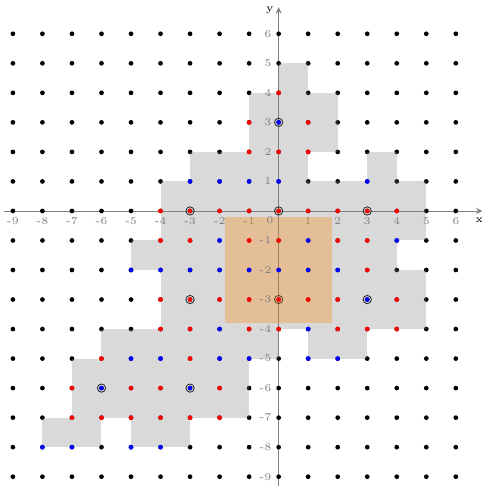}
    \captionsetup{width=0.95\linewidth}
    \caption{The second iteration \( S^2(\textcolor{rot}{\bullet}) = S^2(a) \), supported on \( V(2) \cap \Z^2 \), is shown. The gray shaded region represents \( V(2) \). Elements of \( D_{\lambda_0}(V(1) \cap \Z^2) \) are indicated by black circles. The orange square represents the ball $B\big( (0,-2), 1.8 \big)$ in the metric $d_\infty$, which is entirely contained within $V(4)$.}
    \label{fig:NonSquare-Substitution-Z2-Lam=3-support2nd}
\end{figure}
To show that \( \lambda_0 = 3 \) is a stretch factor, we apply the sufficient condition given in Lemma~\ref{lem:suff-large}~(b). Observe that 
\(\overline{V} \subseteq B(e, r_+) \)
for \( r_+ = 1 \), where balls are taken with respect to the \( d_\infty \)-metric. Furthermore, \( r_0 = 1.8 > \tfrac{r_+ \lambda_0}{\lambda_0 - 1} = 1.5 \), and one checks directly that
\[
B\big( (0,-2), r_0 \big) \subseteq V(2),
\]
as seen in Figure~\ref{fig:NonSquare-Substitution-Z2-Lam=3-support2nd}, where \( V(2) = D_{\lambda_0}( (V(1) \cap \Z^2) + V ) \) is shown.  
Thus, Lemma~\ref{lem:suff-large}~(b) implies that \( \lambda_0 = 3 \) is sufficiently large relative to \( V \), with constants \( C_- = 0.4 \), \( s = 2 \), and \( z = (0,-2) \).

Let \( \Aa = \{a, b\} \) and define the substitution rule \( S_0 : \Aa \to \Aa^{V(1)\cap\Z^2} \) as illustrated in Figure~\ref{fig:NonSquare-Substitution-Z2-Lam=3}.  
Note that this is not a block substitution, as its support is not rectangular.

Denote by $S$ the associated substitution map with $\Dd$ and $\Ss=(\Aa,3,S_0)$. Since $S$ is primitive the associated subshift $\Omega(S)\subseteq \Aa^{\Z^2}$ is strictly ergodic and every element is linearly repetitive. 
\end{example}

We conclude this section with an example in the Heisenberg group \( H_3(\R) \), a non-Abelian Lie group. This example was already provided in \cite{BecHarPog21}.
As a topological space, we identify \( H_3(\R) \) with \( \R^3 \), equipped with the group operation
\[
(x,y,z)(a,b,c) := \big(x+a,\, y+b,\, z+c + \tfrac{1}{2}(xb - ya)\big).
\]
Consider the tuple \( \Dd_\HH := (H_3(\R), d_\HH, (D^\HH_\lambda)_{\lambda > 0}, \Gamma_\HH, V) \), where:
\begin{itemize}
    \item \( d_\HH \) is the unique left-invariant metric on \( H_3(\R) \) satisfying
    \[
    d_\HH \big((x,y,z), (0,0,0)\big) = \sqrt[4]{(x^2 + y^2)^2 + z^2};
    \]
    \item \( D^\HH_\lambda(x,y,z) := (\lambda x, \lambda y, \lambda^2 z) \) defines a one-parameter dilation family;
    \item \( \Gamma_\HH := \{ (x,y,z) \in H_3(\R) \mid x,y,z \in 2\Z \} \) is a uniform lattice;
    \item \( V := [-1,1)^3 \subseteq H_3(\R) \) is a fundamental domain for \( \Gamma_\HH \).
\end{itemize}

This data defines a dilation datum \( \Dd_\HH \). Setting \( r_- = 1 \) and \( r_+ = 1.5 \), we observe
\[
B(e, r_-) \subseteq V \subseteq \overline{V} \subseteq B(e, r_+),
\]
and since \( D_{\lambda_0}^\HH(\Gamma_\HH) \subseteq \Gamma_\HH \) for \( \lambda_0 = 3 > 1 + \tfrac{r_+}{r_-} = 2.5 \), we conclude that \( \lambda_0 = 3 \) is a stretch factor for \( \Dd_\HH \) using Lemma~\ref{lem:suff-large}~(a). 

\begin{example}
\label{ex:HeisenbergExample-Subst}
Let \( \Dd_\HH \) and \( \lambda_0 = 3 \) be as above. Then
\[
V(1) \cap \Gamma_\HH 
= \big([-3,3)^2 \times [-9,9)\big) \cap \Gamma_\HH 
= \{ -2, 0, 2 \}^2 \times \{ -8, -6, \dots, 6, 8 \}.
\]
Define the substitution datum \( \Ss_\HH := (\Aa, \lambda_0, S_0) \) with \( \Aa = \{a,b\} \) and substitution rule \( S_0 \) shown in Figure~\ref{fig:SubstRule-Heisen}. 

\begin{figure}[htb]
    \centering
    \includegraphics[scale=1]{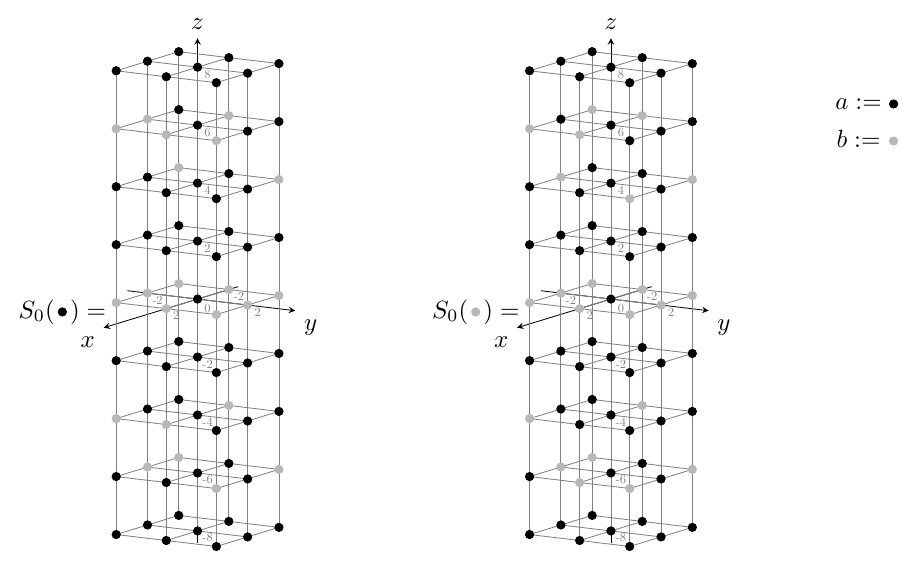}
    \captionsetup{width=0.95\linewidth}
    \caption{Substitution rule \( S_0 \) for the Heisenberg example. Figure taken from \cite{BecHarPog21}.}
    \label{fig:SubstRule-Heisen}
\end{figure}

As shown in \cite[Thm.~1.4, Exam.~6.7, Prop.~6.6]{BecHarPog21}, the associated substitution map \( S \) is primitive and non-periodic. 
Hence, the resulting subshift \( \Omega(S) \) is strictly ergodic, strongly aperiodic, and every configuration in \( \Omega(S) \) is linearly repetitive.
\end{example}

This example highlights the geometric challenges of substitutions in non-Abelian settings compared to the Abelian case.  
To illustrate this, consider the dilation datum \( \Dd_\R := (\R^3, d_\R, (D^\R_\lambda)_{\lambda > 0}, \Gamma_\R, V) \), where:
\begin{itemize}
    \item \( d_\R \) is the standard Euclidean metric;
    \item \( D^\R_\lambda(x,y,z) := (\lambda x, \lambda y, \lambda^2 z) \);
    \item \( \Gamma_\R := \{ (x,y,z) \in \R^3 \mid x,y,z \in 2\Z \} \);
    \item \( V := [-1,1)^3 \).
\end{itemize}

\begin{figure}[htb]
    \centering
    \includegraphics[scale=0.45]{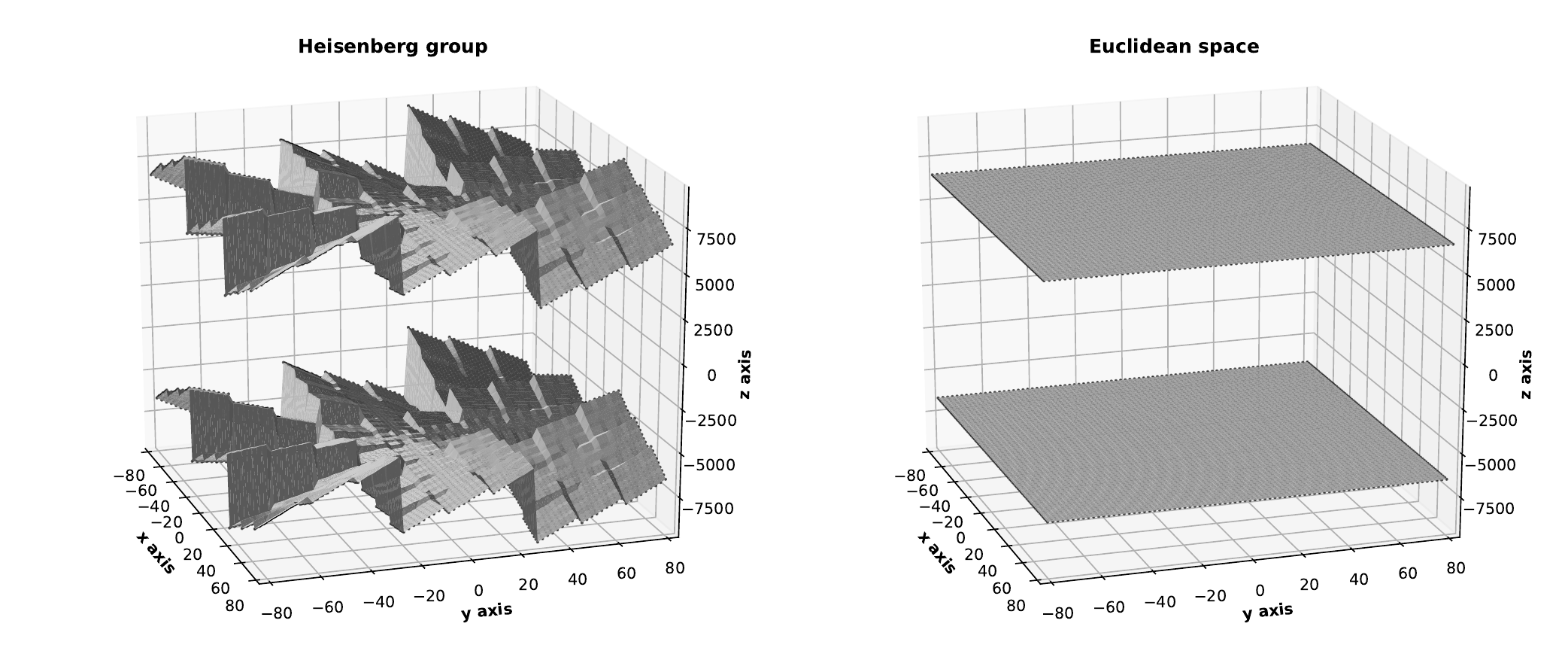}
    \captionsetup{width=0.95\linewidth}
    \caption{Support \( V(4) \) for \( \Dd_\HH \) (left) and \( \Dd_\R \) (right) with \( \lambda_0 = 3 \).  
	The maximal and minimal values of the third coordinate are plotted as functions of the \( (x,y) \)-position. 
	Figure taken from \cite{BecHarPog21}.}
    \label{fig:Substitution_Supp_N=4}
\end{figure}

The key difference between \( \Dd_\HH \) and \( \Dd_\R \) lies in the underlying group structure.  
Since \( \Ss_\HH \) also defines a substitution rule over \( \Dd_\R \), we may consider the associated data \( \Ss_\R := \Ss_\HH \) over \( \Dd_\R \).  
Figure~\ref{fig:Substitution_Supp_N=4} compares the supports \( V(4) \) for both settings.
The fluctuations in the third coordinate -- resulting from the group structure in \( H_3(\R) \) -- are clearly visible in the non-Abelian setting (with $\Dd_\HH$ and $\Ss_\HH$), but absent in the Euclidean case (with $\Dd_\R$ and $\Ss_\R$).

\section{Approximation for symbolic substitutions}
\label{sec:Subst-Approximations}

In Section~\ref{sec:Weight_Delone}, we discussed conditions under which the hulls of weighted Delone dynamical systems converge. As a direct application, we showed for symbolic systems (Proposition~\ref{prop:Conv_Symb-Dyn-Syst}) that two subshifts are close to each other if they share the same patches up to a fixed size.

As previously mentioned, convergence of dynamical systems implies convergence of certain measure theoretic quantities (Section~\ref{sec:SemiContMeasure}) and we see further spectral applications in Section~\ref{sec:SpectralEstimates}. Motivated by various approaches in one dimension, and based on preliminary ideas in \cite{Bec16}, we raised the following question in \cite{BaBePoTe25}:

Given a substitution map \( S \) associated with a dilation datum \( \Dd \) and a substitution datum \( \Ss \), for which \( \omega_0 \in \Aa^\Gamma \) does the following convergence hold
\begin{equation}
\label{eq:Conv_Subst_Subshift}
\overline{\Orb\big(S^n(\omega_0)\big)} \xrightarrow[n\to\infty]{} \Omega(S)\ ?
\end{equation}
Here, convergence of subshifts is defined via the Hausdorff metric \( \delta_\Aa \) (Section~\ref{subsec:symb_syst_as_weighted_Delone}) on the space
\[
\inv := \set{ \Omega \subseteq \Aa^\Gamma }{ \Omega \text{ closed, non-empty, } \Gamma\text{-invariant} },
\]
where the group \( \Gamma \) acts on \( \Aa^\Gamma \) by shifting (see Example~\ref{ex:SymbDynSyst}). We focus only on primitive substitutions while similar results can be obtained under the additional assumption that all letters in $\Aa$ occur in $\omega_0$.

To illustrate the problem, consider the following substitution, originally introduced in \cite{Ten24}.

\begin{example}
\label{ex:Subst-1dim-nonConv}
Consider the dilation datum 
\[
\Dd = \big(\R, |\cdot|, (D_\lambda)_{\lambda > 0}, \Z, [-\tfrac{1}{3}, \tfrac{2}{3}) \big)
\]
with dilation \( D_\lambda : \R \to \R,\, D_\lambda(x) := \lambda x \). Let \( \Ss = (\Aa, 3, S_0) \) be the substitution datum over \( \Dd \) with alphabet \( \Aa = \{a, b, c\} \) and substitution rule \( S_0 : \Aa \to \Aa^3 \) defined by
\[
S_0:\quad 
	a \mapsto aab, \qquad 
	b \mapsto caa, \qquad
	c \mapsto bac.
\]

We follow the standard convention of identifying patches $\Aa^n:=\{u:\{1,\ldots,n\}\to\Aa\}$ with finite words, see e.g.\ \cite{Queff10,BaaGri13}. The substitution is primitive and therefore the associated subshift \( \Omega(S) \) is strictly ergodic. Using standard arguments (e.g.\ proximal pairs), one also shows that \( \Omega(S) \) is strongly aperiodic. Here, \( S \) denotes the substitution map associated to \( \Ss \).

Now consider the constant configuration \( \omega_c \in \Aa^\Z \) given by \( \omega_c(n) = c \) for all \( n \in \Z \). Applying the substitution map \( S \) twice yields:
\[
\omega_c = \ldots c\ c\ \boxed{c\,|\,c}\ c\ c \ldots 
	\xmapsto{S} \ldots a\ c\ b\ a\ \boxed{c\,|\,b}\ a\ c\ b\ a \ldots 
	\xmapsto{S} \ldots a\ b\ b\ a\ \boxed{c\,|\,c}\ a\ a\ a\ a \ldots \,.
\]
Here the vertical bar \( | \) indicates the origin and we interpret \( \omega_c \) as a two-sided infinite concatenation of letters.

This computation inductively shows that the word \( cc \in \Aa^2 \) occurs in \( S^{2n}(\omega_c) \) for all \( n \in \N \), and likewise \( cb \in \Aa^2 \) occurs in \( S^{2n+1}(\omega_c) \) for all \( n \in \N \), see the boxed subwords before.

On the other hand, one can verify that the words \( cb \) and \( cc \) are not legal for this substitution, see \cite{Ten24}. Using Proposition~\ref{prop:Conv_Symb-Dyn-Syst}, this implies that the subshifts 
\[
\Omega_n := \Orb(S^{2n}(\omega_c)) \quad \text{and} \quad \tilde{\Omega}_n := \Orb(S^{2n+1}(\omega_c))
\]
do not converge to the substitution subshift \( \Omega(S) \), since both contain non-legal words \( cb \) or \( cc \), i.e. they do not occur in \( W(S) \) defined in Definition~\ref{def:legal_patch}. Note that we omitted the closure for these orbits as these configurations are periodic and so their orbit is already closed.

On the other hand, convergence does hold when starting from the constant configuration \( \omega_a \in \Aa^\Z \), defined by \( \omega_a(n) = a \) for all \( n \in \Z \).  
Since \( a \mapsto aab \), the word $aa$ is legal in $\omega_a$. Thus, one can show (see e.g. \cite{BBdN20,Ten24}) that
\[
\lim_{n \to \infty} \Orb\big(S^n(\omega_a)\big) = \Omega(S).
\]
\end{example}

\begin{figure}[htb]
    \centering
    \includegraphics[scale=1.7]{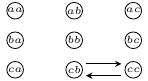}
    \captionsetup{width=0.95\linewidth}
    \caption{The substitution graph $G_S(T;N_T)$ with $T=\{0,1\}$ and $N_T=1$ discussed in Example~\ref{ex:Subst-1dim-nonConv}, see \cite{Ten24}.}
    \label{fig:SubstGraph_1dim}
\end{figure}

This simple example shows that convergence as requested in eq.~\eqref{eq:Conv_Subst_Subshift} fails in general. Already in \cite{Bec16} it was observed that if all the subpatches of $\omega$ of size two (in $\Z^d$ it would be squares $\{0,1\}^d$ of size two) are legal, then the convergence in eq.~\eqref{eq:Conv_Subst_Subshift} holds. Based on this, a characterization of this convergence was provided in \cite{Ten24} for one-dimensional substitution systems. We provided a corresponding characterization (Theorem~\ref{thm:Charact_Convergence_Substitutions}) for the general setting of substitution systems defined in \cite{BecHarPog21} (discussed in the previous Section~\ref{sec:Subst_BeyondAbel}). Moreover, we were able to estimate the rate of convergence of these subshifts (Theorem~\ref{thm:Rate_Convergence_Substitutions}) in the metric $d_\Aa$. In combination \cite{BeBeCo19,BecTak25}, these estimates lead to exponential convergence rates of the spectra discussed in Section~\ref{sec:SpectralEstimates}, see Theorem~\ref{thm:Rate_Convergence_Substitutions}. 

The formal analog of the patches of a sufficiently large size is overcome by introducing a so-called testing tuple $(T,N_T)$ with $T\subseteq \Gamma$ a finite subset and $N_T\in\N$.
We omit their formal definition here and refer the reader to \cite[Def.~2.11]{BaBePoTe25}. However, note that $(T,1)$ with $T=\{0,1\}^d$ is a testing tuple for block substitutions in $\Z^d$, see \cite[Prop.~3.2]{BaBePoTe25}. Moreover, there is a testing tuple for any dilation datum and an associated stretch factor $\lambda_0$, see \cite[Prop~5.3.]{BaBePoTe25}. 

For practical purposes one aims to minimize the size of the testing tuple \cite[Sec.~6]{BaBePoTe25}, which gets more involved in the non-Abelian setting. Therefore we also provided an algorithm to reduce the size.

We continue with the presentation of the main results from the joint work \cite{BaBePoTe25}.  
For the reader's convenience, we present the general results alongside with the guiding example of the table tiling substitution \( S \) on \( \Aa^{\Z^2} \), where \( \Aa = \{a,b,c,d\} \), introduced in Example~\ref{ex:Geom_Comb_data-TableTiling} in Section~\ref{sec:Subst_BeyondAbel}.

Recall that the set of subpatches occurring in a configuration \( \omega_0 \in \Aa^\Gamma \) is given by
\[
W(\omega_0) := \left\{ (\gamma \omega_0)|_F \,\middle|\, \gamma \in \Gamma,\; F \subseteq \Gamma \text{ finite} \right\}.
\]
For \( r > 0 \), define the set of $r$-patches in $\omega_0$ by \( W(\omega_0)_r := W(\omega_0) \cap \Aa^{B(e,r)} \), which is a finite set.

Using Proposition~\ref{prop:Conv_Symb-Dyn-Syst}, convergence $\lim_{n\to\infty}\ol{\Orb\big(S^n(\omega_0)\big)} =\Omega(S)$ in~\eqref{eq:Conv_Subst_Subshift} holds if, for all \( r > 0 \), there exists \( n_r \in \N \) such that
\[
W(S^n(\omega_0))_r = W(S)_r, \qquad \text{for all } n \geq n_r.
\]
We proved that such an \( n_r \) exists whenever \( W(S^{n_0}(\omega_0)) \cap \Aa^T \subseteq W(S) \cap \Aa^T \) for some \( n_0 \in \N \), where \( T \subseteq \Gamma \) is the finite set from the testing tuple. 

In light of this, it suffices to monitor the patches in \( \Aa^T \) that occur in \( \omega_0 \).

A central idea is to associate a finite graph to a substitution and to characterize convergence as in equation~\eqref{eq:Conv_Subst_Subshift} via graph-theoretic properties.  
To that end, recall the subpatch relation \( \prec \) (Definition~\ref{def:Subpatch-relation}) and the notion of legal patches \( W(S) \) associated with a substitution (Definition~\ref{def:legal_patch}).  

A tuple \( G = (\Vv, \Ee) \) is called a \emph{directed graph} if \( \Vv \) is a countable set and \( \Ee \subseteq \Vv \times \Vv \).  
We emphasize that elements of \( \Ee \) are ordered pairs, so in general \( (P, Q) \neq (Q, P) \).  
The set \( \Vv \) is referred to as the \emph{vertex set}, and \( \Ee \) as the \emph{edge set}.  
An element \( (P, Q) \in \Ee \) is interpreted as a directed edge from \( P \) to \( Q \). A \emph{(directed) path} in a directed graph is a chain of vertices $(P_0,\ldots,P_n)$ such that $(P_i,P_{i+1})\in\Ee$ for all $i\in\{0,\ldots,n-1\}$.

\begin{definition}[Substitution graph]
\label{def:Subst_graph}
Let \( S \) be a substitution map associated with a dilation datum \( \Dd = (G, d_G, (D_\lambda)_{\lambda > 0}, \Gamma, V) \) and a substitution datum \( \Ss = (\Aa, \lambda_0, S_0) \). For a testing tuple $(T,N_T)$, define the \emph{substitution graph} $G_S(T;N_T)$ by the directed graph with vertex set $\Vv_S=\Aa^T$ and edge set $\Ee_S$ defined by
\[
(P,Q)\in\Ee_S \qquad :\Longleftrightarrow \qquad Q\prec S^{N_T}(P) \quad\text{and}\quad P,Q\not\in W(S).
\]
\end{definition}

First note that we only draw edges between patches in $\Aa^T$ that are not legal. The graph keeps track if non-legal patches reappear under the iteration of the substitution or if they rule out. In Figure~\ref{fig:SubstGraph_1dim}, a substitution graph is sketched for the substitution defined in Example~\ref{ex:Subst-1dim-nonConv}.

\begin{figure}[htb]
    \centering
    \includegraphics[scale=3.16]{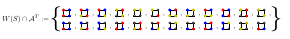}
    \captionsetup{width=0.95\linewidth}
    \caption{The legal patches of the table tiling substitution of support $T=\{0,1\}^2$. Figure taken from \cite{BaBePoTe25}.}
    \label{fig:TableTiling_2x2_patterns}
\end{figure}

Let us return to the table tiling substitution (Example~\ref{ex:Geom_Comb_data-TableTiling}). 

\begin{example}
\label{ex:TableTiling_Patches_Cycles_Graph}
The tuple $(T,1)$ with $T=\{0,1\}^2$ is a testing tuple and $T$ contains four points. Since we have four letters in the alphabet $\Aa$, the vertex set $\Vv=\Aa^T$ contains $4^4=256$ elements. Therefore it would be too much to illustrate the full graph but we illustrate some edges in the graph. By definition, only non-legal patches are relevant for the edge set. Therefore note that the table tiling has exactly $24$ legal patches in $\Aa^T$ (see e.g. \cite[Rem.~4.17]{BaaGri13}), which are illustrated in Figure~\ref{fig:TableTiling_2x2_patterns}.

\begin{figure}[htb]
    \centering
    \includegraphics[scale=2.5]{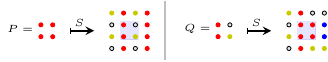}
    \captionsetup{width=0.95\linewidth}
    \caption{The non-legal patches $P$ and $Q$ for the table tiling substitution and there iterative $S(P)$ and $S(Q)$ satisfying $P\prec S(Q)$ and $Q\prec S(P)$.}
    \label{fig:TableTiling_subst_non-legal_patch}
\end{figure}

In Figure~\ref{fig:TableTiling_subst_non-legal_patch}, we observe that the non-legal patches $P, Q \notin W(S)$ satisfy $P \prec S(Q)$ and $Q \prec S(P)$. Therefore, $(P, Q), (Q, P) \in \Ee_S$, which defines a closed path $(P,Q,P)$ in the substitution graph $G_S(T;1)$. 

If either $P$ or $Q$ occurs in a configuration $\omega_0 \in \Aa^\Gamma$, then any iterative $S^n(\omega_0)$ contains $P$ or $Q$, which are not legal patches. Thus, the subshifts $\overline{\Orb(S^n(\omega_0))}, n\in\N,$ do not converge to the table tiling subshift $\OmTabl = \Omega(S)$. This is one idea to prove Theorem~\ref{thm:Charact_Convergence_Substitutions} using that $(P,Q,P)$ is a closed path in $G_S(T,1)$.
\end{example}

We include here a characterization of the convergence of the spectra while the class of operators is only introduced below in Section~\ref{sec:SpectralEstimates}. So-called dynamically-defined operators \( H \) are associated with the unifying symbolic dynamical system \( (\Aa^\Gamma, \Gamma, \tau) \). Here, \( H = (H_\omega)_{\omega \in \Aa^\Gamma} \) is a family of bounded linear operators on a Hilbert space that is strongly continuous in \( \omega \) and covariant under the action of \( \Gamma \); see a detailed discussion in Section~\ref{sec:SpectralEstimates}.

Since each \( H_\omega \) is bounded, its spectrum is a compact subset of \( \C \), and spectral convergence is understood in terms of the Hausdorff metric $d_H$ induced by the Euclidean metric on \( \C \). Furthermore, a bounded linear operator is called \emph{normal} if it commutes with its adjoint.

\begin{theorem}[\cite{BaBePoTe25}]
\label{thm:Charact_Convergence_Substitutions}
Let \( S \) be a substitution map associated with a dilation datum \( \Dd = (G, d_G, (D_\lambda)_{\lambda > 0}, \Gamma, V) \) and a substitution datum \( \Ss = (\Aa, \lambda_0, S_0) \). Assume that \( S \) is primitive and that \( (T, N_T) \) is a testing tuple. Then the following statements are equivalent:
\begin{enumerate}[label=(\roman*)]
\item For every dynamically-defined operator \( H \) that is normal, we have
\[
\lim_{n \to \infty} \sigma(H_{S^n(\omega_0)}) = \sigma(H_\omega), \qquad \text{for all } \omega \in \Omega(S).
\]
\item The iterative subshifts \( \overline{\Orb(S^n(\omega_0))} \), \( n \in \N \), converge to \( \Omega(S) \) in the metric space \( (\inv, \delta_\Aa) \).
\item Each directed path in the substitution graph \( G_S(T; N_T) \) that starts in a vertex from \( W(\omega_0) \cap \Aa^T \) does not contain a closed cycle.
\end{enumerate}
\end{theorem}

\begin{proof}
We first observe that \( \sigma(H_\omega) = \sigma(H_\rho) \) for all \( \omega, \rho \in \Omega(S) \), which follows from standard arguments using the minimality of \( \Omega(S) \), as ensured by the primitivity of \( S \); see e.g.\ Corollary~\ref{cor:SemiCont_Spectrum}. The equivalence between (i) and (ii) was established in \cite{BBdN18}, while the equivalence between (ii) and (iii) is proven in \cite[Thm.~2.14]{BaBePoTe25}.
\end{proof}

Theorem~\ref{thm:Charact_Convergence_Substitutions} provides a graph-theoretic criterion to determine the convergence described in eq.~\eqref{eq:Conv_Subst_Subshift} respectively of the spectra. This criterion enables the explicit construction of periodic approximations for the associated operators; see \cite[Cor.~2.16]{BaBePoTe25}. A sufficient criteria for the convergence of these subshifts for block substitutions in $\Z^d$ was existing before \cite{Bec16}. The new characterization to these graphs for a much larger class of substitutions opens also the door for further analysis such as the study of spectral defects and dynamical impurities \cite{BaBePoTe25-Spec}; see a similar discussion in Section~\ref{sec:Kohmoto}.

Another consequence of Theorem~\ref{thm:Charact_Convergence_Substitutions} is that there exists a primitive substitution in \( \Aa^{\Z^2} \) whose associated subshift is not periodically approximable; see  \cite[Cor.~2.17]{BaBePoTe25}. This stands in contrast to the one-dimensional setting, where all primitive substitutions yield periodically approximable subshifts; see \cite{BBdN20}.

Recall from Definition~\ref{def:(non-)periodic} the notion of periodic elements in a dynamical system. A subshift \( \Omega \in \inv \) is called \emph{periodically approximable} if there exists a sequence of periodic configurations \( (\omega_n)_{n\in\N} \) such that their orbits \( \Orb(\omega_n) \) converge to \( \Omega \) in \( (\inv, \delta_\Aa) \).

In \cite[Prop.~2.15]{BaBePoTe25}, it is shown that if \( \omega_0 \in \Aa^\Gamma \) is periodic, then \( S(\omega_0) \) is also periodic. In combination with Theorem~\ref{thm:Charact_Convergence_Substitutions}, this implies that \( \Omega(S) \) is periodically approximable whenever there exists a periodic configuration \( \omega_0 \in \Aa^\Gamma \) satisfying one of the equivalent conditions of the theorem.

To illustrate this, we return to the table tiling substitution. The following result was originally established in \cite[Prop.~7.7.1]{Bec16} using a sufficient condition for convergence, and is restated in \cite[Prop.~1.3]{BaBePoTe25}. However, Theorem~\ref{thm:Charact_Convergence_Substitutions} additional gives that single letter approximations fail to converge for the table tiling substitution:

\begin{proposition}
\label{prop:PeriodAppr_Table}
Let \( S \) be the table tiling substitution. Then the following assertions hold.
\begin{enumerate}[label=(\alph*)]
\item Let \( \omega_{rb} \in \Aa^{\Z^2} \) be the configuration defined by
\[
\omega_{rb}(\gamma) 
	:= \begin{cases}
		\textcolor{rot}{\bullet}, & \text{if } \gamma \in (2\Z)^2 \cup \big( (1,1) + (2\Z)^2 \big),\\
		\textcolor{blue}{\bullet}, & \text{otherwise},
	\end{cases}
\]
then
\[
\lim_{n\to\infty} \Orb(S^n(\omega_{rb})) = \OmTabl.
\]
\item Let $a\in\Aa$ and $\omega_a\in\Aa^{\Z^2}$ be defined by $\omega_a(n)=a$ for all $n\in\Z^2$. Then the subshifts $\big(\Orb(S^n(\omega_{a}))\big)_{n\in\N}$ do not converge to $\OmTabl$.
\end{enumerate}
\end{proposition}

\begin{proof}
The assertion (a) follows from the inclusion \( W(\omega_{rb}) \cap \Aa^T \subseteq W(S) \), which can be verified directly by inspecting the configuration \( \omega_{rb} \); see Figure~\ref{fig:TableTiling_PeriodicAppr_color} and Figure~\ref{fig:Subpatches_legal}.

Statement (b) follows from similar computations like in Figure~\ref{fig:TableTiling_subst_non-legal_patch} stating that a patch supported on $T=\{0,1\}^2$ with a constant letter $a\in\Aa$ leads to a cycle is $G_S(T,1)$.
\end{proof}

\begin{figure}[htb]
    \centering
    \includegraphics[scale=3.3]{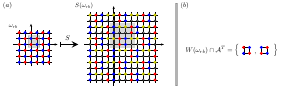}
    \captionsetup{width=0.95\linewidth}
    \caption{In panel~(a), the configuration \( \omega_{rb} \) and its image under \( S \). The subpatches \( W(\omega_{rb}) \cap \Aa^T \) are provided in panel~(b).}
    \label{fig:TableTiling_PeriodicAppr_color}
\end{figure}

We emphasize that \( \omega_{rb} \) is periodic, and thus the above result shows that the table tiling subshift \( \OmTabl \) is periodically approximable.

In a forthcoming preprint we analyze the structure of substitution graphs for block substitutions and their spectral consequences in more detail; see \cite{BaBePoTe25-Spec}. We observed that already in dimension two, new phenomena arise. These were initially observed through numerical investigations of the poor approximation behavior in the table tiling, notably when starting from constant configurations such as \( \omega_a \) from Example~\ref{ex:TableTiling_Patches_Cycles_Graph} and Proposition~\ref{prop:PeriodAppr_Table}~(b).

Once a suitable approximation is established, the natural next step is to quantify the convergence rate. This question will also be central in the upcoming  Section~\ref{sec:SpectralEstimates} discussing the results obtained in \cite{BeBeCo19,BecTak25}. Under the assumption that the subshift associated with a configuration under iteration of the substitution converges, we prove that the convergence in \( (\inv, \delta_\Aa) \) is exponentially fast:

\begin{theorem}[\cite{BaBePoTe25}]
\label{thm:Rate_Convergence_Substitutions}
Let \( S \) be a substitution map associated with a dilation datum \( \Dd = (G, d_G, (D_\lambda)_{\lambda > 0}, \Gamma, V) \) and a substitution datum \( \Ss = (\Aa, \lambda_0, S_0) \). If \( S \) is primitive, then there exist constants \( C > 0 \) and \( M_1 \in \N \) such that for every \( \omega_0 \in \Aa^\Gamma \) satisfying one of the equivalent conditions in Theorem~\ref{thm:Charact_Convergence_Substitutions}, the following estimate holds:
\[
\delta_\Aa\left( \overline{\Orb(S^n(\omega_0))} , \Omega(S) \right) \leq \frac{C}{\lambda_0^n}, \qquad \text{for all } n \geq M_1.
\]
\end{theorem}

\begin{proof}
This is proven in \cite[Thm.~2.18]{BaBePoTe25}.
\end{proof}

The constants involved in the estimate are made more explicit in \cite[Prop.~5.6, Cor.~5.7]{BaBePoTe25}, where the linear repetitivity constant plays a key role.

\section[The Ammann-Beenker tiling]{Periodic approximations: The octagonal or Ammann-Beenker tiling}
\label{sec:Octagonal_Penrose}

We conclude this chapter with a brief discussion of further applications to the particular geometric example of the octagonal tiling, also known as the Ammann--Beenker tiling. 

To every Delone set one can associate a tiling, for instance via the Voronoi construction. Conversely, one can associate to a tiling a point set. Under suitable assumptions, the dynamical properties are preserved; see e.g. \cite{GruShe89,Sen95,BaaGri13} for details.

We refer the reader to \cite{DMO89} for a comprehensive discussion of the octagonal tiling, and also to \cite[Sec.~6.1]{BaaGri13}. The tiling consists of two prototiles (up to rotation): a square and a rhombus with opening angle \( \tfrac{\pi}{4} \). In some representations, the square is replaced by a pair of triangles.

If we associate to each vertex (i.e., the point where two edges meet) a point, we obtain a Delone set of finite local complexity. The space of tilings then inherits a natural translation action of \( \R^2 \), and the resulting dynamical system can be equipped with a topology compatible with that of the associated Delone dynamical system.

\begin{figure}[htbp]
    \centering
    \includegraphics[scale=0.3]{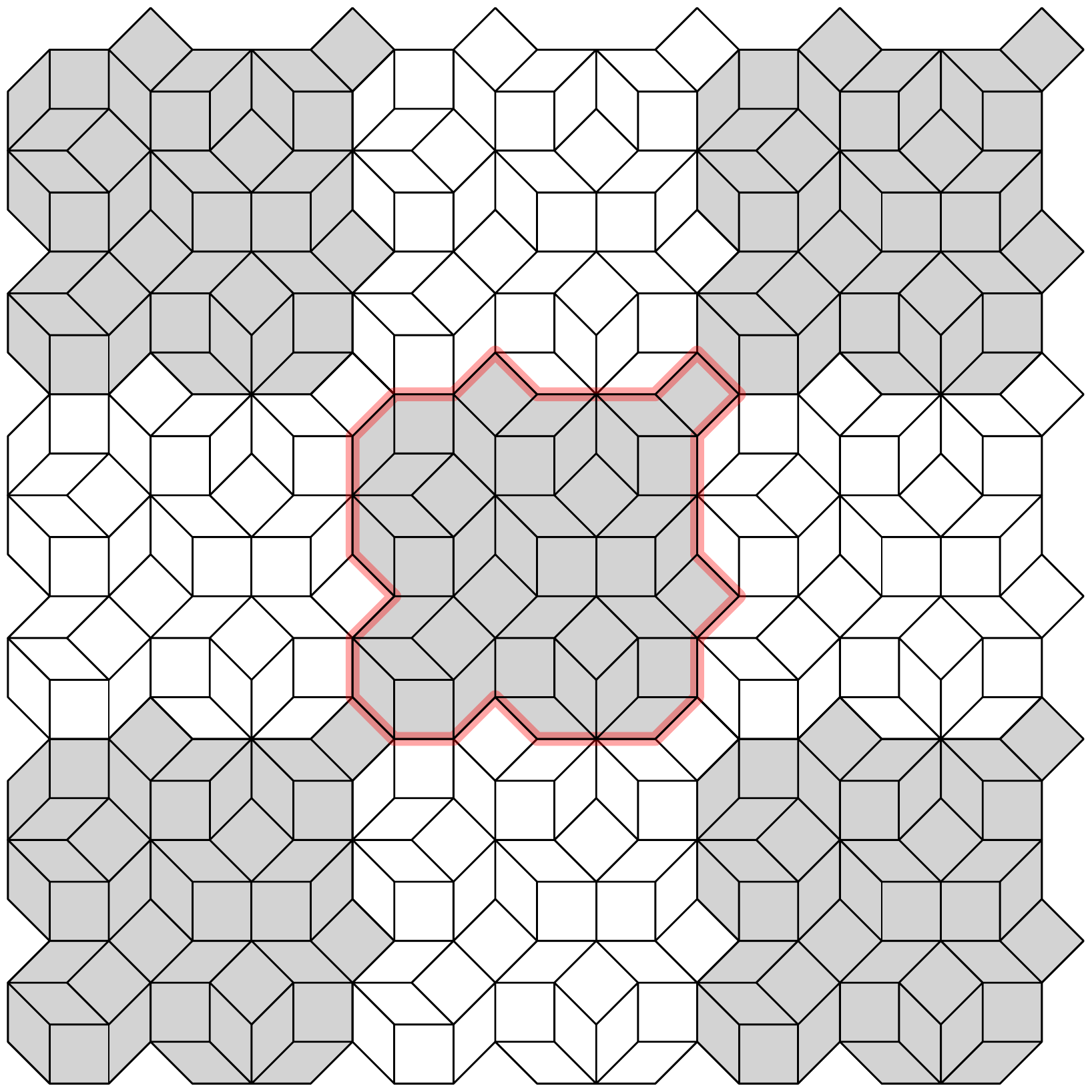}
    \captionsetup{width=0.95\linewidth}
    \caption{A sketch of a periodic tiling \( T_0 \) defined in \cite{DMO89}. For convenience, some of the fundamental cells are shaded in gray to distinguish them.}
    \label{fig:Octagonal_periodic}
\end{figure}

In \cite{DMO89}, the authors constructed a periodic tiling \( T_0 \) using the aforementioned two tiles, as sketched in Figure~\ref{fig:Octagonal_periodic}. For clarity, the boundary of a fundamental cell is marked with a thick red line. The $1$-corona of a tile \( V \) refers to the collection of all tiles intersecting \( V \), i.e., those within one adjacency step.

It can be verified that all $1$-corona in the tiling \( T_0 \) are legal patches in the octagonal tiling, meaning they appear in the infinite aperiodic tiling. These patches play the role of the set \( \Aa^T \) where $T$ was one of the parameters in the testing tuple discussed in Section~\ref{sec:Subst-Approximations}.

The octagonal tiling is defined via a substitution rule. Let \( S \) denote the associated substitution map. Adapting the convergence criteria from \cite[Sec.~4.1]{BaBePoTe25}, one concludes that the orbits of the periodic tilings \( S^n(T_0) \) converge to the hull of octagonal tiling. Hence, the octagonal tiling is periodically approximable.

\begin{remark}
In contrast to the Ammann--Beenker tiling, a Penrose tiling does not admit
periodic approximations using the same tiles.
Consider for instance the Penrose tiling described in
\cite[Sec.~10.3]{GruShe89} and assume there exists a sequence of periodic tilings
converging to it.
Then Proposition~\ref{prop:Char_Conv_Hull_Delone} implies the existence of a
periodic tiling having the same $1$-coronas as the Penrose tiling.
However, by \cite[Sec.~10.3.1]{GruShe89}, no such periodic tiling exists for this
set of tiles.
Consequently, periodic approximations of the Penrose tiling can only be obtained
by deforming the underlying tiles.
\end{remark}

\chapter{The world of spectral butterflies}
\label{chap:WorldButterflies}

In 1976, Douglas Hofstadter~\cite{Hofst76} numerically computed the spectrum of the almost Mathieu operator (see Section~\ref{sec:Hofstadter}), a family of self-adjoint operators on $\ell^2(\Z)$ parametrized by a coupling constant $V\in\R$ and a rotation number $\alpha\in[0,1]$. 
Figure~\ref{fig:Hofstadter-butterfly} shows the spectrum (based on numerical computations from~\cite{HaKaTa16}), which is known as the \emph{Hofstadter butterfly}. The spectrum is plotted only for rational values of $\alpha$, since in these cases the operator is periodic and the spectrum can be computed numerically via Floquet-Bloch theory.

\begin{figure}[htb]
    \centering
    \includegraphics[scale=1.1,angle=-90]{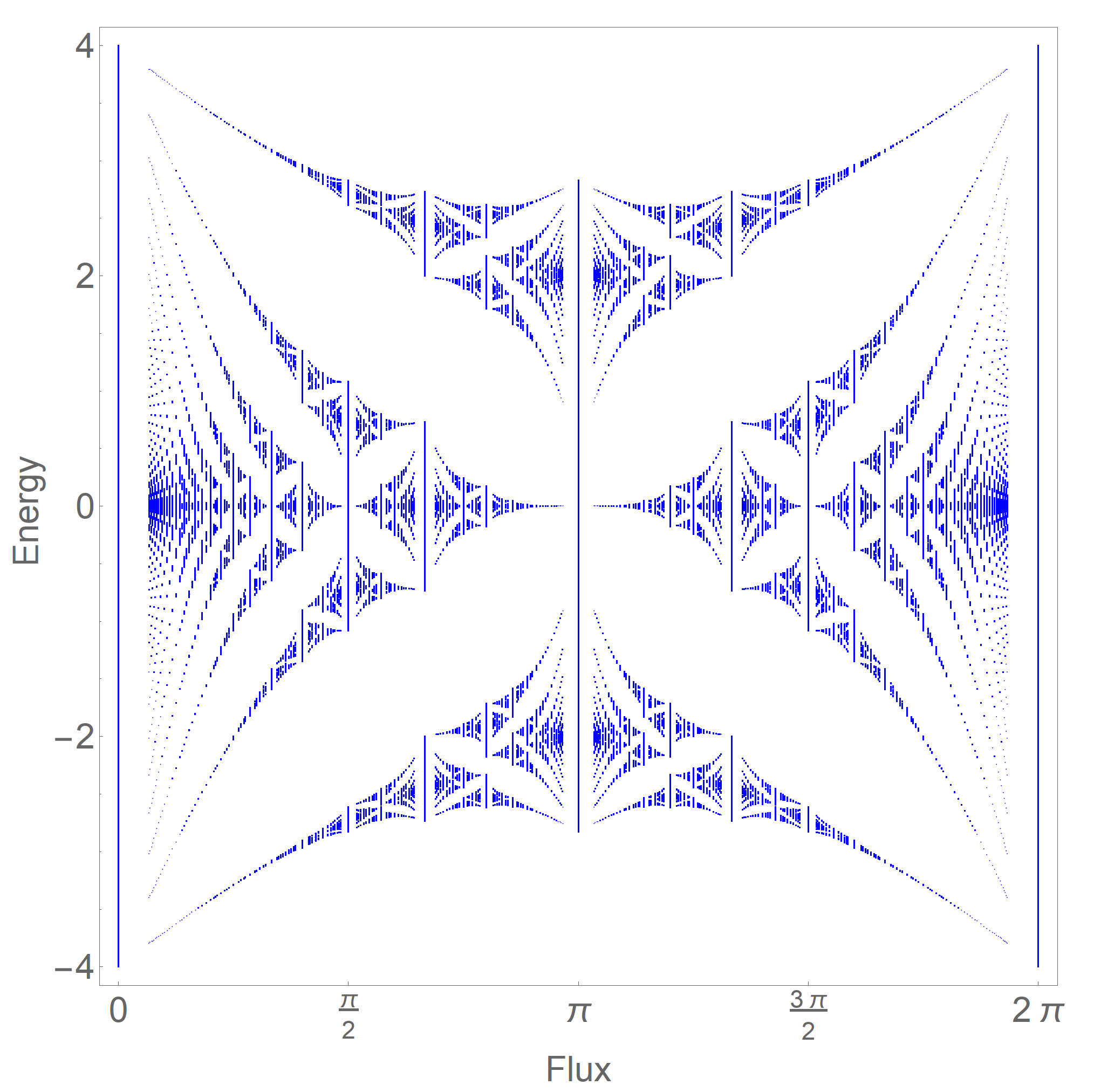}
    \captionsetup{width=0.95\linewidth}
    \caption{The Hofstadter butterfly taken from \cite{HaKaTa16}.}
    \label{fig:Hofstadter-butterfly}
\end{figure}

These numerical computations already reveal self-similar and fractal structures\footnote{Here, we only speak about numerical observations. To the best knowledge of the author, there is currently no complete mathematically rigorous proof establishing the self-similarity of the Hofstadter butterfly. Nevertheless, this model exhibits many other intriguing spectral features, some of which are discussed in Section~\ref{sec:Hofstadter} and Chapter~\ref{chap:DTMP}.} in the spectrum, as observed earlier by Azbel~\cite{Azb64}. The reader is also referred to recent elaborations \cite{SatWil20,Sat25}. Since then, the Hofstadter butterfly has attracted considerable attention (see a more detailed discussion in Section~\ref{sec:Hofstadter} and Chapter~\ref{chap:DTMP}). Similar spectral features and butterflies have been found in other models~\cite{OK85,WilAus90,DeNittis10-thesis,BeHaJi19,BeHaJiZw21,GhaSugToh22,BarExnLip23,CedFilOng23,Nuckolls25}. Most recently, Hofstadter's butterfly was experimentally observed in twisted bilayer graphene~\cite{Nuckolls25}.

Another model is the so-called \emph{Kohmoto model}, which is studied in further detail in this work (see Section~\ref{sec:Kohmoto} and Chapter~\ref{chap:DTMP}). It was introduced in~\cite{KKT83} and today it serves as a toy model for one-dimensional quasicrystals, such as those discovered by Dan Shechtman~\cite{SBGC84}. Like the almost Mathieu operator, the Kohmoto operators are parametrized by a coupling constant $V \in \R$ and a rotation number $\alpha \in [0,1]$, and they are periodic when $\alpha$ is rational. The associated \emph{Kohmoto butterfly} (based on numerical computations from~\cite{Biber22}) is shown in Figure~\ref{fig:Kohmoto-butterfly}. Numerically, the Kohmoto model exhibits self-similar and fractal features\footnote{Here, we only speak about numerical observations. To the best knowledge of the author, there is currently no complete mathematically rigorous proof establishing the self-similarity of the Kohmoto butterfly.} like the Hofstadter butterfly.

\begin{figure}[htb]
    \centering
    \includegraphics[scale=2.8]{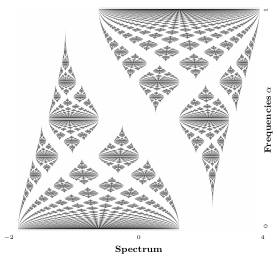}
    \captionsetup{width=0.95\linewidth}
    \caption{The Kohmoto butterfly taken from \cite{Biber22}.}
    \label{fig:Kohmoto-butterfly}
\end{figure}

In both models -- the almost Mathieu operator and the Kohmoto model -- the spectral properties of the operators with irrational $\alpha$ are of central interest. Since these operators are no longer periodic, rational approximations of $\alpha$ are employed to study their spectral nature; see the detailed discussion in~\cite{DaFi24-book_2} and in Chapter~\ref{chap:DTMP}. For these models, the convergence (and in some cases convergence rates) of the spectra under rational approximation has been rigorously established~\cite{Ell86,ChElYu90,AvrMouSim90,BIST89,BIT91,Bell94,BecBelTho25}.

While a variety of methods have been used -- including continuous fields of $C^*$-algebras and detailed analyses of approximate eigenfunctions -- we highlight a unifying strategy based on the convergence of the underlying dynamical systems. This approach was initiated in earlier joint works~\cite{BB16,BBdN18}, which used $C^*$-algebraic techniques to prove spectral convergence. In this chapter, we present recent results from the joint works~\cite{BeBeCo19,BecTak25} based on different analytic methods yielding explicit quantitative estimates. These results not only unify spectral estimates for the different models discussed so far but also show that a common mechanism governs the spectral behavior under periodic approximation. Moreover, the spectral estimates obtained are qualitatively optimal with respect to the convergence behavior \cite{BeRa90,BecBelTho25}.

In the case of the almost Mathieu operator, a square-root behavior of spectral gaps that close was previously observed \cite{BeRa90} (see also a discussion in~\cite{BB16,BecTak25}). In the joint work~\cite{BecTak25}, this phenomenon naturally emerges from the interaction between the amenability of the underlying group (in this case $\Z$) and the regularity of the operator coefficients, see also a discussion in Remark~\ref{rem:Optimization-Regularity-Amenable} and Remark~\ref{rem:Choice-cut_Interplay-amenable}.

The underlying dynamical system for the almost Mathieu operator is a torus rotation. For the Kohmoto model, the corresponding dynamical system is symbolic and, for irrational $\alpha$, these are called \emph{Sturmian dynamical system}. Hence, the operators in the Kohmoto model with irrational $\alpha$ are also referred to as \emph{Sturmian Hamiltonians}.

We demonstrate in this section that the results of~\cite{BeBeCo19,BecTak25} extend far beyond the classical models, including:
\begin{itemize}
\item higher-dimensional models (see Section~\ref{sec:Operator_Subst}), with a particular emphasis on substitution systems, building on results from the previous chapter,
\item other models studied in the literature (e.g. the unitary almsot Mathieu operator, see Section~\ref{sec:Unitary-AMO}),
\item operators defined over non-Abelian homogeneous Lie groups (see Section~\ref{sec:Operator_Subst}),
\item operators over general amenable groups possibly with exponential growth, such as the Lamplighter group \cite{KaiVer83,Lin01,Butk01}, see Section~\ref{sec:Lamplighter}.
\end{itemize}

This dynamical perspective on spectral convergence also sheds light on the structure of the Kohmoto butterfly. In the joint work~\cite{BecBelTho25}, spectral defects at rational~$\alpha$ are analyzed and previous numerical findings are confirmed rigorously. We discuss this in more detail in Section~\ref{sec:Kohmoto}.

Periodic approximations are used to analyze the spectral nature of the operators; see a detailed discussion for the Kohmoto model in Chapter~\ref{chap:DTMP}. Besides that and numerical computations, such approximations can also be used to detect spectral gaps \cite{BeSi82,ChElYu90,HegMosTeu22}. In \cite[Thm.~5.1]{ChElYu90}, the authors used spectral estimates to prove that all possible gaps are there -- thus solving the dry ten Martini problem for particular models.

In Section~\ref{sec:SpectralEstimates}, we discuss the general results on the spectral estimates obtained in \cite{BeBeCo19,BecTak25}. In the subsequent section, we apply these results to substitution systems (Section~\ref{sec:Operator_Subst}), the almost Mathieu operator (Section~\ref{sec:Hofstadter}), the unitary almost Mathieu operator (Section~\ref{sec:Unitary-AMO}), the Kohmoto model (Section~\ref{sec:Kohmoto}) and the Lamplighter group (Section~\ref{sec:Lamplighter}).

\section{Spectral estimates through the underlying dynamics}
\label{sec:SpectralEstimates}

This section is devoted to the results of \cite{BeBeCo19,BecTak25}. In both works, spectral estimates in terms of the distance of the underlying dynamics were established.

The work \cite{BeBeCo19} focuses on symbolic dynamical systems and Schrödinger-type operators acting on the Hilbert space $\ell^2(G)\otimes\C^N$, where $G$ is a lattice in $\R^d$ and $N\in\N$. In this setting we prove spectral estimates. The approach is based on localizing the resolvent using Lipschitz-partition of unity and the convergence criterion involving local patterns, see Proposition~\ref{prop:Conv_Symb-Dyn-Syst}.

In \cite{BecTak25}, this result is generalized using approximate eigenfunctions to general topological dynamical systems $(\Zz,G,\tau)$ where $G$ is an amenable, unimodular lcsc group. This setting includes the result of \cite{BeBeCo19} in the case $N=1$ and extends to other models such as the almost Mathieu operator. As a consequence, it provides a unifying principle explaining the square root behavior observed in several previously studied models.
We provide here a slight generalization of this result by allowing $N$ to take values different from $1$. Specifically, we consider dynamically-defined operators on the Hilbert space $H_N:= L^2(G) \otimes\C^N$ for some $N\in\N$. The proof extends directly following the lines of \cite{BecTak25}, and we provide the details for completeness in the appendix, see Appendix~\ref{App:SpectralEstimates}. This slight generalization allows us to apply our model also to the unitary or Almost-Mathieu operator discussed in Section~\ref{sec:Unitary-AMO}.

Recall the notion of a topological dynamical system \( (Z, G, \tau) \) introduced in Chapter~\ref{chap:Dynam_LR}. Since \( G \) is a locally compact Hausdorff group, there exists a unique left-invariant Haar measure \( m_G \) on the Borel \( \sigma \)-algebra of \( G \); see, e.g., \cite{Fol16}. Let \( L^2(G) := L^2(G, m_G) \) be the Hilbert space of square integrable functions on $G$. For simplicity, we write \( dh \) when integrating with respect to this Haar measure $m_G$. Throughout this section, we additionally assume that \( G \) is unimodular, i.e. $m_G$ is left and right $G$-invariant. For $N\in\N$,  consider the Hilbert space 
\[
H_N:=L^2(G)\otimes \C^N
\] 
that we also identify with the Hilbert space $\bigoplus_{n=1}^N \ell^2(G)$, see e.g. \cite[Rem.~2.6.8]{KadRin97}.

Let $\C^{N\times N}$ be the set of $N\times N$-matrices $A = (a_{m,n})_{m,n=1}^N$ with values in $\C$ (i.e. $a_{m,n}\in\C$) equipped with the norm
$$
\|A\|_M := \max_{1\leq m,n\leq N} |a_{m,n}|.
$$

\begin{definition}[Dynamically-defined operators]	
\label{def:DynDefOp}
Let \( (Z, G, \tau) \) be a dynamical system. A measurable function \( k: G \times Z \to \C^{N\times N} \) is called a {\em (bounded) kernel} if for each \( g \in G \), the map \( k(g,\cdot): Z \to \C^{N\times N} \) is continuous (in $\|\cdot\|_M$) and
\[
\sup_{x \in Z} \int_G \sup_{g\in G}\left\| k\big( h, g^{-1}x \big) \right\|_M \, dh < \infty.
\]
Then the operator family \( A_x \in \Ll\big(H_N\big) \) for \( x \in Z \) defined by
\[
(A_x \psi)_n(g) := \int_G  \left(k\big( g^{-1}h, g^{-1}x \big) \psi(h)\right)_n\, dh 
		= \int_G  \left(k\big( h, g^{-1}x \big) \psi(gh)\right)_n\, dh,\qquad 1\leq n \leq N,
\]
is called a \emph{dynamically-defined operator} associated with the kernel $k$.
\end{definition}

Note that $k\big( g^{-1}h, g^{-1}x \big)\in \C^{N\times N}$ is a matrix acting by matrix multiplication on the vector $\psi(h)\in\C^N$. By definition of a dynamically-defined operator it acts in the $n$th position like an integral operator.

Each operator \( A_x \) is bounded and linear, and the map \( Z \ni x \mapsto A_x \in \Ll(H_N) \) is strongly continuous. Moreover, the operator family satisfies the covariance relation
\[
A_x U_g = U_g A_{g^{-1}x}, \qquad x \in Z,\ g \in G,
\]
where \( U_g \in \Ll(L^2(G)) \) is the translation operator defined by \( (U_g \psi)(h) := \psi(g^{-1}h) \). This expresses that the operator family is covariant with respect to the group action.

Let us start with a particular example for $N=2$.

\begin{example}
\label{ex:Unitary_AMO-first}
Let \( \T \) be the torus, identified with the interval \( [0,1) \) and equipped with the induced Euclidean metric \( |\cdot| \). The compact metric space \( Z := [0,1] \times \T \) equipped with the metric
\[
d_Z\left( \begin{pmatrix} \alpha \\ \theta \end{pmatrix}, \begin{pmatrix} \beta \\ \rho \end{pmatrix} \right)
	:= \max\left\{ |\alpha - \beta|, |\theta - \rho| \right\}.
\]
For a real number \( x\in\R \), we define its \emph{fractional part} by
\[
\{x\} := x - \lfloor x \rfloor,
\]
where \( \lfloor x \rfloor := \sup \set{n \in \Z}{n \leq x} \) is the \emph{integer part} of \( x \).
The discrete group \( G = \Z \) with action
\[
\tau : \Z \times Z \to Z, \qquad
\tau\left(n, \begin{pmatrix} \alpha \\ \theta \end{pmatrix} \right)
	:= \begin{pmatrix} \alpha \\ \{-n \cdot \alpha + \theta\} \end{pmatrix}.
\]
Define the kernel $k:\Z\times Z\to\C^{2\times 2}$ by
\[
k(n,\tbinom{\alpha}{\theta}) :=
	\begin{pmatrix}
		\cos\big(2\pi (\theta-\alpha)\big) \delta_{-1}(n) 	& -\sin\big(2\pi(\theta-\alpha)\big)\delta_{-1}(n) \\
	\sin\big(2\pi (\theta+\alpha)\big) \delta_{1}(n) 		& \cos\big(2\pi(\theta+\alpha)\big) \delta_{1}(n)
	\end{pmatrix}.
\]
In fact $k$ defines a bounded kernel using that $k(n,x)$ vanished whenever $n\not\in\{-1,1\}$ for all $x\in Z$. 
This kernel defines a dynamically-defined operator family $(A_x)_{x\in Z}$ on the Hilbert space $\ell^2(\Z)\otimes\C^2$. 
For $x=\tbinom{\alpha}{\theta}$, $\psi=\tbinom{\psi^+(n)}{\psi^-(n)}_{n\in\Z}\in\ell^2(\Z)\otimes\C^2$ and $n\in\Z$, we compute
\begin{align*}
(A_x\psi)(n) 
	= &\sum_{m\in\Z} k\left(-n+m,\tau_{-n}\tbinom{\alpha}{\theta}\right) \binom{\psi^+(m)}{\psi^-(m)}\\
	= &\sum_{m\in\Z} k\left(-n+m,\tbinom{\alpha}{n\alpha + \theta}\right) \binom{\psi^+(m)}{\psi^-(m)}\\
	= 
	&\begin{pmatrix}
		\cos\big(2\pi\big((n-1)\alpha + \theta\big)\big) 	&  	-\sin\big(2\pi\big((n-1)\alpha + \theta\big)\big)\\
		0													& 	0
	\end{pmatrix}
	\binom{\psi^+(n-1)}{\psi^-(n-1)}\\[0.2cm]
	& + \begin{pmatrix}
		0												 	&  	0\\
		\sin\big(2\pi\big((n+1)\alpha + \theta\big)\big)	& 	\cos\big(2\pi\big((n+1)\alpha + \theta\big)\big)
	\end{pmatrix}
	\binom{\psi^+(n+1)}{\psi^-(n+1)}\\[0.2cm]
	=&\begin{pmatrix}
		\cos\big(2\pi\big((n-1)\alpha + \theta\big)\big)\psi^+(n-1) 	&  	-\sin\big(2\pi\big((n-1)\alpha + \theta\big)\big)\psi^-(n-1)\\
		\sin\big(2\pi\big((n+1)\alpha + \theta\big)\big)\psi^+(n+1)		& 	\cos\big(2\pi\big((n+1)\alpha + \theta\big)\big)\psi^-(n+1)
	\end{pmatrix}.
\end{align*}
This is actually the unitary almost Mathieu operator (see \cite[Lem.~4.1]{CedFilOng23} for the coupling constants $\lambda_1=\lambda_2=1$. The reader is referred to Section~\ref{sec:Unitary-AMO} for further elaborations.
\end{example}

We also mentioned various one-dimensional Schr\"odinger operators before such as the almost Mathieu operator and the Kohmoto model. All these operators are dynamically-defined:

\begin{example}[One-dimensional Schrödinger operators]
\label{ex:DynDefOp_integers}
Let \( (Z, \Z, \tau) \) be a dynamical system and $N=1$. Note that $\ell^2(\Z)\otimes\C$ is identified with $\ell^2(\Z)$. Then the Haar measure on \( \Z \) is the counting measure. For a continuous potential function \( v: Z \to \R \), define the kernel
\[
k_v: \Z \times Z \to \R, \quad 
k_v(m,x) :=
\begin{cases}
v(x), & \text{if } m = 0, \\
1, & \text{if } m \in \{-1, +1\}, \\
0, & \text{otherwise}.
\end{cases}
\]
By continuity of $v$, we have that $k_v(m,\cdot):Z\to Z$ is continuous. Since the left-invariant (normalized) Haar measure on $\Z$ is given by the counting measure, observe
\[
 \sup_{x \in Z} \sum_{h\in\Z}  \sup_{n \in \Z} \left| k\big( h, \tau(-n,x) \big) \right|
	\leq  2 + \sup_{x \in Z} |v(x)| 
	< \infty.
\]
The latter supremum is finite, since $v:Z\to\R$ is continuous and $Z$ is a compact space. Thus, $k_v$ is a bounded kernel.
A short computation yields that the associated dynamically-defined operator family \( A_x \in \Ll\big(\ell^2(\Z)\big) \) with $k_v$ satisfies
\[
(A_x \psi)(n) = \psi(n-1) + \psi(n+1) + v\big( \tau(n,x) \big) \psi(n),
\qquad\psi\in\ell^2(\Z), \, n\in\Z,
\]
which is a discrete one-dimensional Schrödinger operator with potential determined by the orbit of \( x \in Z \); see examples in Section~\ref{sec:Hofstadter} and Section~\ref{sec:Kohmoto}.
\end{example}

Already this example includes several models studied in the literature, such as discrete quasiperiodic Schrödinger operators (e.g., the almost Mathieu operator), Sturmian Hamiltonians, skew-shift Schrödinger operators, and limit-periodic Schrödinger operators. All of these examples are discussed in the joint work \cite{BecTak25}, and further examples will be provided below.
The reader is referred to Section~\ref{sec:Unitary-AMO} where the unitary almost Mathieu operator is considered -- an example of a dynamically-defined operator with $N=2$.

In \cite{BeBeCo19}, we proved Lipschitz continuity of the spectral map in the following setting:
\begin{itemize}
\item dynamical systems \( (\Aa^G, G, \tau) \) with \( \Aa \) a compact metric space and \( G = A\Z^d \), where \( A \) is an invertible \( d \times d \)-matrix;
\item strongly pattern-equivariant Hamiltonians (i.e., the kernel coefficients are locally constant).
\end{itemize}
This result relies crucially on the fact that two dynamical systems are close in \( \delta_\Aa \) if they agree on patches of large enough size; see Proposition~\ref{prop:Conv_Symb-Dyn-Syst}. The proof adapts techniques from \cite{CoPu12,CoPu15}, using a Lipschitz partition of unity (adapted to the size where the patches coincide) and resolvent estimates.

This covers, for instance, the Kohmoto model but not operators such as the almost Mathieu operator, which is defined by an irrational rotation on the torus. For that model, spectral estimates in terms of the rotation number were already known \cite{ChElYu90,AvrMouSim90,Bell94}.

These observations served as the starting point for \cite{BecTak25}, which aims to provide a unified proof that also explains the differing behavior observed across models. Specifically, why a square-root convergence rate appears in some cases but not in others, see Remark~\ref{rem:Choice-cut_Interplay-amenable}.

We return to general dynamically-defined operator families $(A_x)_{x\in Z}$. The spectrum of each operator \( A_x \) is denoted by \( \sigma(A_x) \), which is a compact subset of \( \C \). For a dynamical subsystem \( Y \in \inv \) of \( (Z, G, \tau) \), the spectrum of the operator family \( A_Y := (A_y)_{y \in Y} \) is defined as
\[
\sigma(A_Y) := \overline{\bigcup_{y \in Y} \sigma(A_y)}.
\]
Note that the restriction of the kernel \( k \) to \( Y \) corresponds to an element of the reduced \( C^* \)-algebra of the dynamical system \( (Y, G, \tau) \). In this context, \( \sigma(A_Y) \) coincides with the spectrum of this element in the \( C^* \)-algebra; see e.g. \cite{Bec16,BBdN18}.

If each \( A_x \) is normal (i.e., \( A_x \) commutes with its adjoint), then standard arguments (see e.g. Corollary~\ref{cor:SemiCont_Spectrum}) imply
\[
\sigma(A_x) = \sigma\big( A_{\overline{\Orb(x)}} \big), \qquad y \in \overline{\Orb(x)}\in\inv.
\]
We note that this equality of the spectra is a well-known fact for strongly continuous and covariant operator families; see e.g.~ \cite{CyconFroeseKirschSimon87,BIST89,Jit95,Len99,LenSto03Alg}.

It was shown in \cite{BBdN18}, using \( C^* \)-algebraic techniques, that the spectral map
\[
\Sigma_k: \inv \to \Kk(\C), \qquad Y \mapsto \sigma(A_Y),
\]
is continuous, where \( \Kk(\C) := \{ K \subseteq \C \text{ compact} \} \) is equipped with the Hausdorff metric \( d_H \); see also Chapter~\ref{chap:Dynam_LR} for the definition and the background. The regularity of this spectral map --- in particular, whether it is Lipschitz or Hölder continuous --- was further investigated in \cite{BeBeCo19,BecTak25} under suitable assumptions on \( G \) and the kernel \( k \).

To study the spectral map, we require additional assumptions on the kernel. On the one hand, the off-diagonals of the kernel (i.e., $k(g,x)$ with $g \neq e$) need to decay uniformly in \( x \in Z \) as \( g \) tends to infinity (i.e., leaves every compact set). On the other hand, the functions \( k(g,\cdot): Z \to \C \) must satisfy further regularity conditions.

In this context -- similar to \cite{BeBeCo19,BecTak25} -- we focus on kernels that satisfy a linear decay condition and are either Lipschitz continuous or locally constant in their second argument.

\begin{definition}
\label{def:Regularity_kernel}
Let \( (Z, G, \tau) \) be a dynamical system where $G$ is an unimodular lcsc group with adapted metric $d_G$. Consider a (bounded) kernel \( k: G \times Z \to \C^{N\times N} \) for $N\in\N$.
\begin{enumerate}[label=(\alph*)]
\item The kernel \( k \) is \emph{linearly decaying} if there exists a constant \( c_s > 0 \) such that
\[
\sup_{x \in Z} \int_G  \sup_{g\in G} \big\| k\big(h,g^{-1}x\big) \big\|_M \ d_G(e,h) \, dh \leq c_s.
\]

\item The kernel \( k \) \emph{Lipschitz continuous} if there exists \( c_k \in L^1_+(G) \) such that for \( m_G \)-a.e.\ \( h \in G \),
\[
\|k(h,x) - k(h,y)\|_M \leq c_k(h) \cdot d(x,y), \qquad x,y \in Z.
\]

\item The kernel \( k \) is \emph{locally constant} if there exists \( c_k \in L^\infty_+(G) \) such that for \( m_G \)-a.e.\ \( h \in G \),
\[
d(x,y) < \frac{1}{c_k(h)} \quad \Rightarrow \quad k(h,x) = k(h,y).
\]
\end{enumerate}
\end{definition}

In~\cite{BeBeCo19}, we also studied weaker decay conditions on the off-diagonal elements (than in Definition~\ref{def:Regularity_kernel}~(a)), measured in terms of the Schur \( \beta \)-norm.
Also Lipschitz continuity in (b) can be replaced by H\"older continuity if needed. All such changes will influence the particular regularity of the spectral map, see e.g. \cite[Thm.~1.2]{BeBeCo19}. Another instance is discussed in \cite[Thm.~2.4]{BecTak25} where the dynamical system is not acting Lipschitz continuously affecting the regularity of the spectral map. 
The required modifications are straightforward and are illustrated for various cases both here and in \cite{BeBeCo19,BecTak25}.

Let us return to our Example~\ref{ex:DynDefOp_integers} over the integers \( \Z \).

\begin{example}
\label{ex:DynDefOp_integers_Lipsch_LocConst}
Let \( (Z, \Z, \tau) \) be a dynamical system, $N=1$ and \( v: Z \to \R \) a continuous function. Then the kernel \( k_v: \Z \times Z \to \R \), as defined in Example~\ref{ex:DynDefOp_integers}, satisfies a linear decay condition since \( k_v(n,x) = 0 \) for all \( n \notin \{-1,0,1\} \).

Suppose \( v: Z \to \R \) is Lipschitz continuous, i.e., there exists a constant \( c_v > 0 \) such that \( |v(x) - v(y)| \leq c_v \cdot d(x,y) \). Then the kernel $k_v$ is Lipschitz continuous with
\[
c_{k_v}(h) = 
	\begin{cases}
		c_v, & \text{if } h = 0, \\
		0,   & \text{otherwise}.
	\end{cases}
\]
An example of such a kernel is provided by the almost Mathieu operator discussed in Section~\ref{sec:Hofstadter}.

Now suppose \( v: Z \to \R \) is locally constant, i.e., there exists a constant \( c_v > 0 \) such that \( d(x,y) < \frac{1}{c_v} \) implies \( v(x) = v(y) \). Then the kernel $k_v$ is locally constant with
\[
c_{k_v}(h) = 
	\begin{cases}
		c_v, & \text{if } h = 0, \\
		1,   & \text{otherwise}.
	\end{cases}
\]
An example of such a kernel is provided by the Kohmoto model or the Sturmian Hamiltonian discussed in Section~\ref{sec:Kohmoto}.
\end{example}

In \cite{BecTak25}, we raised the question of whether the various spectral regularity results known in the literature follow from a common underlying principle. This turned out to be the case. We also provided a conceptual explanation for the well-known root-Hölder behavior observed for the spectrum of the almost Mathieu operator \cite{BeRa90}; see also the discussion in \cite{BB16}. This behavior can be traced back to the interaction between the amenability of the acting group and the regularity of the kernel \( k \); see \cite{BecTak25} for details as well as Appendix~\ref{App:SpectralEstimates}.

\begin{definition}
\label{def:strict-polyn-growth}
The group $G$ has \emph{strict polynomial growth} if there exists constants \( 0 < c_0 \leq c_1 \) and \( \kappa >0 \)
such that
\[
c_0 r^\kappa \leq m_G(B(r)) \leq c_1 r^\kappa,\qquad \text{ for each } r\geq 1.
\]
\end{definition}

We now state the corresponding regularity result for the spectral map 
\[ 
\Sigma_k :\inv\to\Kk(\C),\qquad Y\mapsto \sigma(A_Y),
\] 
where $(\inv,\delta_H)$ and $(\Kk(\C),d_H)$ are the corresponding Hausdorff metrics, introduced in Chapter~\ref{chap:Dynam_LR}.
The regularity of $\Sigma_k$ is proven in the realm that the acting group \( G \) has strictly polynomial growth.\footnote{This is a weaker assumption than the exact polynomial growth considered earlier in Section~\ref{sec:Repet_WeightDelone}.} Note that every strictly polynomially growing group $G$ is amenable, see also a discussion in Appendix~\ref{App:SpectralEstimates}. In Section~\ref{sec:Lamplighter}, we also provide spectral estimates for the Lamplighter group that is exponentially growing but still amenable using cut-off functions supported on an associated F\o lner sequence.
Recall also the notion of Lipschitz continuous dynamical system defined in Definition~\ref{def:LipschitzAction}.

\begin{theorem}
\label{thm:SpecEst_DynDefOp}
Let \( (Z, G, \tau) \) be a Lipschitz continuous dynamical system such that \( G \) is unimodular and strictly polynomially growing. For $N\in\N$, let \( k : G \times Z \to \C^{N\times N} \) be a linearly decaying and bounded kernel. Assume \( A_x \) is a normal operator for each $x\in Z$ where \( (A_x)_{x\in Z} \) is the dynamically-defined operator family associated with \( k \). 
\begin{enumerate}[label=(\alph*)]
\item If \( k \) is Lipschitz continuous, then the spectral map \( \Sigma_k \) is root-Hölder continuous. More precisely, there exists a constant \( C := C(G, k) > 0 \) such that
\[
d_H\left( \sigma(A_X), \sigma(A_Y) \right) \leq C \cdot \delta_H(X,Y)^{1/2}, \qquad \text{for all } X,Y \in \inv.
\]

\item If \( k \) is locally constant, then the spectral map \( \Sigma_k \) is Lipschitz continuous. That is, there exists a constant \( C := C(G,k) > 0 \) such that
\[
d_H\left( \sigma(A_X), \sigma(A_Y) \right) \leq C \cdot \delta_H(X,Y), \qquad \text{for all } X,Y \in \inv.
\]
\end{enumerate}
\end{theorem}

\begin{proof}
This is proven in \cite[Thm.~1.3]{BecTak25} for $N=1$. Slight modifications of \cite[Thm.~1.3]{BecTak25} allow to prove the statement for all $N\in\N$, see Appendix~\ref{App:SpectralEstimates}. Statement~(b) recovers and generalizes the results previously established in \cite{BeBeCo19}. 
\end{proof}

While the theorem is formulated for a broad class of groups and kernels, the obtained regularity bounds are in fact sharp:
\begin{itemize}
\item The root-Hölder estimate in Theorem~\ref{thm:SpecEst_DynDefOp}~(a) is known to be optimal for the almost Mathieu operator; see \cite{BeRa90} and a discussion in \cite{BecTak25}.
\item The Lipschitz estimate in Theorem~\ref{thm:SpecEst_DynDefOp}~(b) is proven to be optimal for the Kohmoto model; see \cite{Tho23_MSc,BecBelTho25} as well as the discussion around eq.~\eqref{eq:Optimality_Kohmoto} in Section~\ref{sec:Kohmoto} and draw further conclusions from them.
\end{itemize}
In both cases the optimality of the spectral estimates is obtained at points where spectral gaps close in the limit, see also a discussion in \cite{BB16}.

In the following sections, we discuss various applications of this result. Note that all presented results deal with discrete groups $G$ while Theorem~\ref{thm:SpecEst_DynDefOp} applies also for non-discrete groups. 
For non-discrete groups, $\exp(-H_x)$ -- where $H_x$ is a Schr\"odinger type operator -- can be represented as a dynamically-defined operators, see a short discussion in Chapter~\ref{chap:discussion}. The main difficulty is to prove the regularity of the kernel with respect to the underlying dynamical system. Furthermore, first elaborations in a joint project \cite{BeJoUl-Frame} show potential applications to frame operators where $G$ is non-discrete as well.

\section{Operators associated with substitution systems}
\label{sec:Operator_Subst}

We can combine Theorem~\ref{thm:Rate_Convergence_Substitutions} with Theorem~\ref{thm:SpecEst_DynDefOp} to obtain explicit rates of convergence for substitution systems. A version of this result can be found in~\cite[Thm.~2.18]{BaBePoTe25}.

\begin{theorem}[\cite{BaBePoTe25}]
\label{thm:Rate_Convergence_Substitutions_spectrum}
Let \( S \) be a substitution map associated with a dilation datum \( \Dd = (G, d_G, (D_\lambda)_{\lambda > 0}, \Gamma, V) \) and a substitution datum \( \Ss = (\Aa,\lambda_0,S_0) \). 
Let $k:\Gamma\times \Aa^\Gamma \to\C$ be a linearly decaying, locally constant and bounded kernel. Assume that $A_\omega$ is normal for each $\omega\in \Aa^\Gamma$ where \( (A_\omega)_{\omega\in\Aa^\Gamma} \) is the associated dynamically-defined operator family with $k$.

If \( S \) is primitive, then there exist a constant \( C > 0 \) such that for every \( \omega_0 \in \Aa^\Gamma \) satisfying one of the equivalent conditions in Theorem~\ref{thm:Charact_Convergence_Substitutions}, the following estimate holds:
\[
d_H\left( \sigma(A_{S^n(\omega_0)}),\sigma(A_\omega) \right)  \leq \frac{C}{\lambda_0^n}, \qquad \text{for all } n \in\N,\, \omega\in\Omega(S).
\]
\end{theorem}

\begin{proof}
Following \cite{BecHarPog21}, the lattice $\Gamma$ has exact polynomial growth (Definition~\ref{def:exact-polyn-growth}) and therefore it has strict polynomial growth (Definition~\ref{def:strict-polyn-growth}).
Let $\Omega_n:=\overline{\Orb(S^n(\omega_0))}$ for $n\in\N$. Since $S$ is primitive, Theorem~\ref{thm:Primitive-LR} implies that $\Omega(S)$ is minimal, i.e. $\Omega(S)=\overline{\Orb(\omega)}$ for all $\omega\in\Omega(S)$.
Since the operator family $A_\omega$ is normal for each $\omega\in \Aa^\Gamma$, we conclude (see e.g. Corollary~\ref{cor:SemiCont_Spectrum})
\[
\sigma(A_{S^n(\omega_0)}) = \sigma(A_{\Omega_n})
	\qquad\text{and}\qquad
\sigma(A_{S^n(\omega)}) = \sigma(A_{\Omega(S)}), \quad \omega\in\Omega(S).
\]
Let $\Aa^\Gamma$ be equipped with the metric $\dAG$ defined in Section~\ref{subsec:symb_syst_as_weighted_Delone}, and let $\delta_\Aa$ denote the Hausdorff metric on $\inv$ induced by $\dAG$.

According to Proposition~\ref{prop:Conv_Symb-Dyn-Syst}~(b), the symbolic dynamical system $(\Aa^\Gamma,\Gamma,\tau)$ is Lipschitz continuous.
Thus, Theorem~\ref{thm:SpecEst_DynDefOp} implies that there is a constant $C_1>0$ such that 
\[
d_H\left( \sigma(A_{S^n(\omega_0)}),\sigma(A_\omega) \right)  \leq C_1 \delta_\Aa\big( \Omega_n,\Omega(S) \big).
\]
Since \( \omega_0 \in \Aa^\Gamma \) satisfies one of the equivalent conditions in Theorem~\ref{thm:Charact_Convergence_Substitutions}, Theorem~\ref{thm:Rate_Convergence_Substitutions} leads to the claimed estimate.
\end{proof}

In view of the discussion on the Octagonal tiling (Section~\ref{sec:Octagonal_Penrose}), we note that the proof could be modified to Delone dynamical systems and associated kernels.

\section{The Hofstadter butterfly}
\label{sec:Hofstadter}

The almost Mathieu operator \( \Hamo : \ell^2(\Z) \to \ell^2(\Z) \) is defined by
\begin{equation}
\label{eq:AMO}
(\Hamo\psi)(n):= \psi(n+1)+\psi(n-1) + 2V\cos\big(2\pi (\alpha\cdot n +\theta)\big)\psi(n)
\end{equation}
for $\psi\in\ell^2(\Z)$ and $n\in\Z$, where
\begin{itemize}
\item \( \alpha\in[0,1] \) is the \emph{rotation number},
\item \( V\in\R \) is the \emph{coupling constant}, and
\item \( \theta\in\T \) is the phase.
\end{itemize}

Here \( \T \) denotes the torus, identified with the interval \( [0,1) \) and equipped with the induced Euclidean metric \( |\cdot| \). For a real number \( x\in\R \), we define its \emph{fractional part} by
\begin{equation}
\label{eq:fractional_part}
\{x\} := x - \lfloor x \rfloor \in\T,
\end{equation}
where \( \lfloor x \rfloor := \sup \set{n \in \Z}{n \leq x} \) is the \emph{integer part} of \( x \).

The almost Mathieu operator is a dynamically-defined operator associated to the following dynamical system:
\begin{itemize}
\item The compact metric space \( Z := [0,1] \times \T \) equipped with the metric
\[
d_Z\left( \begin{pmatrix} \alpha \\ \theta \end{pmatrix}, \begin{pmatrix} \beta \\ \rho \end{pmatrix} \right)
	:= \max\left\{ |\alpha - \beta|, |\theta - \rho| \right\}.
\]
\item The discrete group \( G = \Z \) with action
\[
\tau : \Z \times Z \to Z, \qquad
\tau\left(n, \begin{pmatrix} \alpha \\ \theta \end{pmatrix} \right)
	:= \begin{pmatrix} \alpha \\ \{-n \cdot \alpha + \theta\} \end{pmatrix}.
\]
\end{itemize}
For given \( \alpha \in [0,1] \), one is interested in the spectrum
\[
\sigAMO := \overline{\bigcup\limits_{\theta\in\T}\sigma(\Hamo)}
\]
which is a compact subset of \( \R \). If \( \alpha \) is rational, the operators are periodic and so the spectrum can be computed numerically. Plotting these spectra for different \( \alpha \) produces the so-called \emph{Hofstadter butterfly}~\cite{Hofst76}, sketched in Figure~\ref{fig:Hofstadter-butterfly}. One point of view is that while only the spectrum of rational \( \alpha \) is plotted in the Hofstadter butterfly, it should also reflect the structure for the irrational case. Therefore, the continuity of the spectrum at irrational \( \alpha \) is crucial.

Note that if \( \alpha \in [0,1] \) is irrational, then the spectrum is independent of \( \theta \)~\cite{AvrSim83}, namely
\begin{equation}
\label{eq:SpectrConst_almMath}
\sigAMO = \sigma\left( \Hamo \right),\qquad \theta \in \T,\ \alpha \in [0,1] \setminus \Q.
\end{equation}
This also follows from the general fact that the irrational rotation is minimal and the spectra of these self-adjoint operators do not change over a minimal dynamical system, see e.g.~\cite{LenSto03Alg,Bec16} as well as Corollary~\ref{cor:SemiCont_Spectrum}.

This figure has drawn a lot of attention from both the physical and mathematical communities. 
The Hofstadter butterfly reflects a fractal and self-similar structure that has raised significant interest. Moreover, the nature of the spectral measure depending on the coupling constant is challenging and has been intensively studied in the past. We refer the reader to~\cite{Jit95,MarJit_etds17,Jitomirskaya2019,DaFi24-book_2} and to the more detailed discussion in Chapter~\ref{chap:DTMP} on the dry ten Martini problem.

As discussed before in Section~\ref{sec:SpectralEstimates}, the change of the spectrum in \( \alpha \), measured in the Hausdorff metric, has been studied in~\cite{Ell86,ChElYu90,AvrMouSim90,Bell94}. Using Theorem~\ref{thm:SpecEst_DynDefOp}, we recover these spectral estimates; see Proposition~\ref{prop:AMO_SpecEst}~(a).

\begin{proposition}
\label{prop:AMO_SpecEst}
Let $\alpha,\beta\in[0,1]$, $V\in\R$ and $\theta\in\T$. Then the following assertion hold for the almost Mathieu operator $\Hamo$.
\begin{enumerate}[label=(\alph*)]
\item There exists a constant $C>0$ such that
\[
\dH\left( \sigAMO,\sigma_{\beta,V}^{\textrm{A}} \right) \leq C |\alpha-\beta|^{\tfrac{1}{2}},
\qquad \alpha,\beta\in [0,1].
\]
\item There exists a constant $C>0$ such that
\[
\dH\left( \sigma(\Hamo),\sigma(H_{\alpha,V,\rho}^{\textrm{A}}) \right) \leq C |\theta-\rho|^{\tfrac{1}{2}},
\qquad \theta,\rho\in\T.
\]
\item Let $\theta\in\T$ and suppose $\beta=\tfrac{p}{q}$ where $p,q$ are coprime and $|\alpha-\beta|<\frac{1}{q^2}$. Then
\[
\dH\left( \sigma(\Hamo),\sigma(\HamoB) \right) \leq C \tfrac{1}{\sqrt{q}}.
\]
\end{enumerate}
\end{proposition}

\begin{proof}
Observe that the action of $\Z$ on $Z := [0,1]\times\T$ is Lipschitz continuous, since
\[
d_Z\left( \tau_n\tbinom{\alpha}{\omega}, \tau_n \tbinom{\beta}{\rho}\right) 
	\leq \max\left\{ |\alpha-\beta|, |n|\cdot |\alpha-\beta| + |\omega-\rho| \right\}
	\leq (|n|+1)\cdot d_Z\left( \tbinom{\alpha}{\omega}, \tbinom{\beta}{\rho} \right).
\]
Furthermore, $\Z$ is strictly polynomially growing.
Define the linearly decaying, Lipschitz continuous and bounded kernel
\[
k_V:\Z\times Z\to\R, \quad 
	k_V\left(m,\tbinom{\alpha}{\theta}\right)
		:= 
			\begin{cases}
				1, \qquad &\textrm{if } m\in\{-1,1\},\\
				2V\cdot \cos(2\pi\theta), \qquad &\textrm{if } m=0,\\
				0, \qquad &\textrm{otherwise}.
			\end{cases}
\]
The associated dynamically-defined operator $(A_x)_{x\in Z}$ (see Definition~\ref{def:DynDefOp}) is given by
\[
(A_x\psi)(n) = \sum_{m\in\Z} k_V\left(-n+m,\tau_{-n}\tbinom{\alpha}{\theta}\right) \psi(m)
	= (\Hamo\psi)(n)
\]
for $x=\tbinom{\alpha}{\theta}\in Z$, where we used that the Haar measure on $\Z$ is the counting measure.

Let us now consider the orbit closures $\Omega_{\alpha,\theta}:=\overline{\Orb\big(\tbinom{\alpha}{\theta}\big)}$ for $\tbinom{\alpha}{\theta}\in Z$. 

If $\alpha = \tfrac{p}{q}$ is rational with $p,q$ coprime, then $\Omega_{\alpha,\theta}$ contains exactly $q$ points given by
\begin{equation}
\label{eq:Om_alp,thet-AMO}
\Omega_{\alpha,\theta} = \set{ \binom{\alpha}{\big\{\theta + \tfrac{j}{q}\big\}}}{ 0\leq j \leq q-1}.
\end{equation}
If $\alpha$ is irrational, we observe $\Omega_{\alpha,\theta} = \{\alpha\}\times\T$.

\medskip

(a) Consider the dynamical subsystems $\{\alpha\}\times\T$ and $\{\beta\}\times\T$ of $(Z,\Z,\tau)$. Then
\[
\dJ(\{\alpha\}\times\T,\{\beta\}\times\T) = |\alpha-\beta|,
\]
and so Theorem~\ref{thm:SpecEst_DynDefOp}~(a) yields the claimed estimate.

\medskip

(b) Consider the dynamical subsystems $\Omega_{\alpha,\theta}$ and $\Omega_{\alpha,\rho}$. If $\alpha$ is irrational, then $\Omega_{\alpha,\theta}=\Omega_{\alpha,\rho}$ and hence $\sigma(\Hamo)=\sigma(H_{\alpha,V,\rho}^{\textrm{A}})$. If $\alpha$ is rational, then $\Omega_{\alpha,\theta}$ and $\Omega_{\alpha,\rho}$ are finite subsets of $\{\alpha\}\times\T$, and using~\eqref{eq:Om_alp,thet-AMO} we conclude
\[
\dJ(\Omega_{\alpha,\theta},\Omega_{\alpha,\rho}) \leq |\theta-\rho|.
\]
Thus, Theorem~\ref{thm:SpecEst_DynDefOp}~(a) yields the claimed estimate.

\medskip

(c) Consider the dynamical subsystems $\Omega_{\alpha,\theta}$ and $\Omega_{\beta,\theta}$.
Since $\beta = \tfrac{p}{q}$ with $p,q$ coprime, $\Omega_{\beta,\theta}\subseteq \{\beta\}\times\T$ contains $q$ equidistributed points in the second component. Since $|\alpha-\beta| < \tfrac{1}{q^2}$, $\alpha$ is either irrational or a rational number $\alpha = \tfrac{p'}{q'}$ with $p',q'$ coprime and $q' > q$. This follows from the fact that two Farey numbers cannot be too close; see~\cite[Lem.~2.4]{BecBelTho25}. Since $q' > q$, a similar argument as for $\Omega_{\beta,\theta}$ gives
\[
\dJ(\Omega_{\beta,\theta},\Omega_{\alpha,\theta}) \leq \tfrac{1}{2q}.
\]
Thus, Theorem~\ref{thm:SpecEst_DynDefOp}~(a) yields the claimed estimate.
\end{proof}

Let us briefly comment on the previous result. The estimate in (a) was already proven in \cite{AvrMouSim90,Bell94}. Statement~(c) shows that, when approximating the spectrum of an irrational $\alpha \in [0,1]$ by a sequence of rationals $\tfrac{p_k}{q_k} \in [0,1]$ with $p_k, q_k$ coprime and $|\alpha - \tfrac{p_k}{q_k}| < \tfrac{1}{q_k^2}$, the spectrum still converges for fixed $\theta \in \T$. In fact, we obtain $\tfrac{1}{4}$-Hölder continuity in this setting.

Statement~(b) can actually be improved, since the almost Mathieu operator depends norm-continuously on $\theta$:
\[
\left\| \Hamo - H_{\alpha,V,\rho}^{\textrm{A}} \right\| \leq 4\pi |V| \cdot |\theta - \rho|, \qquad \theta,\rho\in\T.
\]
This shows that the square root behavior is not present in this case. As discussed earlier, the appearance of square root behavior in general reflects the possibility of spectral gap closures, which typically reduces regularity; see a detailed elaboration in \cite{BB16}.

At the same time, it was proven in \cite{BeRa90} (see also the discussion in \cite{BecTak25}) that the square root behavior in (a) is optimal. Thus, we cannot expect an improvement of the general unifying result Theorem~\ref{thm:SpecEst_DynDefOp} in this direction.

On the other hand, it may be possible to improve the spectral estimates if the spectral gaps do not close in the limit; see a discussion in \cite{BB16}. In the case of the rotation algebra (which includes the almost Mathieu operator), such improvement has been proven in \cite{Bell94}. Current investigations aim at adapting the methods from \cite{BecTak25} to derive Lipschitz continuity under the assumption that spectral gaps remain open, see Chapter~\ref{chap:discussion}.

\section{The unitary almost Mathieu operator}
\label{sec:Unitary-AMO}

In this section, we show that the spectral estimates obtained in Theorem~\ref{thm:SpecEst_DynDefOp} apply to the unitary almost Mathieu operator. To the best of our knowledge, these estimates are new in the literature.

The unitary almost Mathieu operator models a quantum walk in a homogeneous external magnetic field. We refer the reader to \cite{FilOngZha17,CeFiGeWe20,CedFilOng23,CedFil24,CedLi25} for further background and references.
Note that also the unitary almost Mathieu operator gives rise to a butterfly and this model shares similar spectral phenomena like the almost Mathieu operator, see \cite{CeFiGeWe20,CedFilOng23} for an illustration.

As the name suggests, this operator is closely related to the almost Mathieu operator and is defined over the same dynamical system considered in Section~\ref{sec:Hofstadter}. For completeness, we briefly recall the setting.

The compact metric space \( Z := [0,1] \times \T \) is equipped with the metric
\[
d_Z\left( \begin{pmatrix} \alpha \\ \theta \end{pmatrix}, \begin{pmatrix} \beta \\ \rho \end{pmatrix} \right)
:= \max\left\{ |\alpha - \beta|,\; |\theta - \rho| \right\}.
\]
The discrete group \( \Z \) acts on \( Z \) via
\[
\tau : \Z \times Z \to Z, \qquad
\tau\left(n, \begin{pmatrix} \alpha \\ \theta \end{pmatrix} \right)
:= \begin{pmatrix} \alpha \\ \{-n \cdot \alpha + \theta\} \end{pmatrix},
\]
where \( \{x\} \) denotes the fractional part of the real number \( x \), see eq.~\eqref{eq:fractional_part}.

In Example~\ref{ex:Unitary_AMO-first}, we studied a particular kernel which belongs to the class of unitary almost Mathieu operators with coupling constants $\lambda_1 = \lambda_2 = 1$. It is straightforward to verify that the kernel in Example~\ref{ex:Unitary_AMO-first} is Lipschitz continuous and exhibits linear decay.

Since the dynamical system \( (Z, \Z, \tau) \) is Lipschitz continuous (see the proof of Proposition~\ref{prop:AMO_SpecEst}), Theorem~\ref{thm:SpecEst_DynDefOp}\,(a) implies that the associated spectral map is root-Hölder continuous. We discuss this here for general coupling constants $\lambda_1$ and $\lambda_2$.

For $\lambda \in [0,1]$, we define the shorthand $\lambda' := \sqrt{1 - \lambda^2}$. With this, we define the kernel
\[
k_{\lambda_1, \lambda_2} : \Z \times Z \to \C^{2 \times 2}
\]
for $\lambda_1, \lambda_2 \in [0,1]$ by
\[
k_{\lambda_1,\lambda_2}(n,\tbinom{\alpha}{\theta}) :=
\left(
\begin{array}{cc}
\begin{array}{c}
	\scalebox{0.9}{$ \lambda_1\big(\lambda_2 \cos\big(2\pi (\theta-\alpha)\big) + i \lambda_2'\big) \delta_{-1}(n) $}\\
	\scalebox{0.9}{$\textstyle - \lambda_1' \lambda_2\sin(2\pi\theta) \delta_0(n) $}
\end{array}
&
\begin{array}{c}
	\scalebox{0.9}{$ -\lambda_1\lambda_2 \sin\big(2\pi(\theta-\alpha)\big) \delta_{-1}(n) $} \\
	\scalebox{0.9}{$ - \lambda_1'\big( \lambda_2\cos(2\pi\theta) - i\lambda_2' \big) \delta_0(n) $}
\end{array}
\\[0.6cm]
\begin{array}{c}
	\scalebox{0.9}{$ \lambda_1 \lambda_2 \sin\big(2\pi (\theta+\alpha)\big) \delta_{1}(n) $} \\
	\scalebox{0.9}{$ + \lambda_1'\big( \lambda_2\cos(2\pi\theta) +  i\lambda_2'\big) \delta_0(n) $}
\end{array}
&
\begin{array}{c}
	\scalebox{0.9}{$ \lambda_1\big(\lambda_2 \cos\big(2\pi(\theta+\alpha)\big) - i \lambda_2' \big) \delta_{1}(n) $} \\
	\scalebox{0.9}{$ - \lambda_1' \lambda_2\sin(2\pi\theta) \delta_0(n) $}
\end{array}
\end{array}
\right)
\]
Clearly, this defines a bounded kernel, since all involved functions are continuous in $\theta$ and $\alpha$, and the kernel vanishes for all \( n \in \Z \setminus \{-1, 0, 1\} \). We also note that if \( \lambda_1 = \lambda_2 = 1 \), then we recover the kernel from Example~\ref{ex:Unitary_AMO-first} since $\lambda_1'=\lambda_2'=0$ in this case.

Next, we determine the associated dynamically-defined operator family \( (A_x)_{x \in Z} \), analogous to the computations made in Example~\ref{ex:Unitary_AMO-first}. Specifically, for
\[
x = \begin{pmatrix} \alpha \\ \theta \end{pmatrix},
\qquad
\psi = \big( \psi(n) \big)_{n \in \Z}
\in \ell^2(\Z) \otimes \C^2,
\quad \text{with } \psi(n) = \begin{pmatrix} \psi^+(n) \\ \psi^-(n) \end{pmatrix},
\]
and for each \( n \in \Z \), we obtain:
\begin{align*}
&(A_x\psi)(n) \\
	= &\sum_{m\in\Z} k\left(-n+m,\tau_{-n}\tbinom{\alpha}{\theta}\right) \binom{\psi^+(m)}{\psi^-(m)}\\
	= &\sum_{m\in\Z} k\left(-n+m,\tbinom{\alpha}{n\alpha + \theta}\right) \binom{\psi^+(m)}{\psi^-(m)}\\
	=&\left(
\begin{array}{c}
	\scalebox{0.9}{$ \lambda_1 \Big( \big(\lambda_2 \cos\big(2\pi\big((n-1)\alpha + \theta\big)\big) + i \lambda_2'\big) \psi^+(n-1) 
			- \lambda_2 \sin\big(2\pi\big((n-1)\alpha + \theta\big)\big) \psi^-(n-1)\Big) $}\\
	\scalebox{0.9}{$ -\lambda_1' \Big( \lambda_2 \sin\big(2\pi(n\alpha + \theta)\big) \psi^+(n) 
			+ \big(\lambda_2 \cos\big(2\pi(n\alpha + \theta)\big) - i \lambda_2' \big) \psi^-(n)\Big) $}\\
\\
	\scalebox{0.9}{$ \lambda_1 \Big( \lambda_2 \sin\big(2\pi\big((n+1)\alpha + \theta\big)\big) \psi^+(n+1) 
			+ \big(\lambda_2 \cos\big(2\pi\big((n+1)\alpha + \theta\big)\big) - i \lambda_2' \big) \psi^-(n+1)\Big) $}\\
	\scalebox{0.9}{$ +\lambda_1' \Big( \big(\lambda_2 \cos\big(2\pi(n\alpha + \theta)\big) + i \lambda_2'\big) \psi^+(n) 
			- \lambda_2 \sin\big(2\pi(n\alpha + \theta)\big) \psi^-(n)\Big) $}
\end{array}
\right).
\end{align*}

This shows that the dynamically-defined operator family indeed corresponds to the unitary almost Mathieu operator. Specifically, we have
\[
A_{x} = W_{\lambda_1, \lambda_2, \alpha, \theta} \qquad \text{ if } x=\tbinom{\alpha}{\theta},
\]
for the kernel \( k_{\lambda_1, \lambda_2} \) using the notation from \cite{CedFilOng23}. 
Thus, the operator $A_x$ is a unitary operator on $\ell^2(\Z)\otimes\C^2$.

The main object of interest is the spectrum of this operator family.
As in the case of the almost Mathieu operator, we define
\[
\sigma_{\lambda_1, \lambda_2, \alpha}
:= \ol{\bigcup_{\theta \in \T} \sigma\left( W_{\lambda_1, \lambda_2, \alpha, \theta} \right)},
\qquad \theta \in \T,\quad \alpha \in [0,1] \setminus \Q.
\]
Since $W_{\lambda_1, \lambda_2, \alpha, \theta}$ is a unitary (bounded) operator, the spectrum $\sigma_{\lambda_1, \lambda_2, \alpha}$ is a compact subset of $\C$ respectively the unit circle in $\C$.
If $\alpha\in[0,1]$ is irrational, the associated dynamical system is minimal and so
\[
\sigma_{\lambda_1, \lambda_2, \alpha} = \sigma\left( W_{\lambda_1, \lambda_2, \alpha, \theta} \right), \qquad \text{ for all } \theta\in\T,
\]
follows again by standard arguments, see e.g. Corollary~\ref{cor:SemiCont_Spectrum}.

In analogy with Proposition~\ref{prop:AMO_SpecEst}, we obtain the following result.

\begin{proposition}
\label{prop:unitaryAMO_SpecEst}
Let $\alpha,\beta\in[0,1]$, $\lambda_1,\lambda_2\in[0,1]$ and $\theta\in\T$. Then the following assertion hold for the unitary almost Mathieu operator $W_{\lambda_1,\lambda_2,\alpha,\theta}$.
\begin{enumerate}[label=(\alph*)]
\item There exists a constant $C>0$ such that
\[
\dH\left( \sigma_{\lambda_1,\lambda_2,\alpha}\ , \ \sigma_{\lambda_1,\lambda_2,\beta} \right) \leq C |\alpha-\beta|^{\tfrac{1}{2}},
\qquad \alpha,\beta\in [0,1].
\]
\item There exists a constant $C>0$ such that
\[
\dH\left( \sigma\left( W_{\lambda_1,\lambda_2,\alpha,\theta} \right), \ \sigma\left( W_{\lambda_1,\lambda_2,\alpha,\rho} \right) \right) \leq C |\theta-\rho|^{\tfrac{1}{2}},
\qquad \theta,\rho\in\T.
\]
\item Let $\theta\in\T$ and suppose $\beta=\tfrac{p}{q}$ where $p,q$ are coprime and $|\alpha-\beta|<\frac{1}{q^2}$. Then
\[
\dH\left( \sigma\left( W_{\lambda_1,\lambda_2,\alpha,\theta} \right) ,\sigma\left( W_{\lambda_1,\lambda_2,\beta,\theta} \right) \right) \leq C \tfrac{1}{\sqrt{q}}.
\]
\end{enumerate}
\end{proposition}

\begin{proof}
First note that the dynamical system \( (Z,\Z,\tau) \) is Lipschitz continuous (see the proof of Proposition~\ref{prop:AMO_SpecEst}). Also the group $\Z$ is strictly polynomially growing and unimodular. We already showed that the kernel $k_{\lambda_1,\lambda_2}:\Z\times Z\to\C^{2\times 2}$ for $\lambda_1,\lambda_2\in[0,1]$ is a bounded kernel and that it gives rise to the unitary almost Mathieu operator. Thus, $A_x$ is unitary for each $x\in Z$ and so it is in particular a normal operator. 

The kernel $k_{\lambda_1,\lambda_2}(n,x)$ vanishes whenever $n\in\Z\setminus\{-1,0,1\}$ and so the kernel is linearly decaying. Moreover, observe that the kernel is Lipschitz continuous since each coefficients of the matrix $k_{\lambda_1,\lambda_2}(n,\tbinom{\alpha}{\theta})$ depends Lipschitz continuously on the parameter $\alpha$ and $\theta$. Thus, Theorem~\ref{thm:SpecEst_DynDefOp}~(a) asserts that there is a constant $C>0$ such that for all $X,Y\in\inv$ (non-empty, closed and invariant subsets of $\Z$) the spectral estimate
\[
d_H\left( \sigma(A_X), \sigma(A_Y) \right) \leq C \cdot \delta_H(X,Y)^{1/2}, \qquad \text{for all } X,Y \in \inv.
\]
holds. Now statement (a), (b) and (c) follow by choosing the appropriate dynamical systems. These are exactly the same as in the proof of Proposition~\ref{prop:AMO_SpecEst}.
\end{proof}

Like for the almost Mathieu operator the estimate in (b) can be improved.

\section{The Kohmoto butterfly}
\label{sec:Kohmoto}

Let $\Aa$ be a finite set and consider the symbolic dynamical system $(\Aa^\Z,\Z,\tau)$; see Example~\ref{ex:SymbDynSyst}. Note that $\Z$ is strictly polynomially growing and recall that this dynamical system is Lipschitz continuous, see Proposition~\ref{prop:wDel_groupAction_Lipschitz-App} and Proposition~\ref{prop:Conv_Symb-Dyn-Syst}.

For $V\in\R$ and $\omega\in\Aa^\Z$, define the self-adjoint discrete Schrödinger operators $H_{\omega,V} : \ell^2(\mathbb{Z}) \to \ell^2(\mathbb{Z})$ by
\begin{equation} \label{eq:Schroedinger_Z}
(H_{\omega,V} \psi)(n) = \psi(n-1) + \psi(n+1) + V \cdot \omega(n) \cdot \psi(n),
\quad \psi \in \ell^2(\mathbb{Z}), \; n \in \mathbb{Z}.
\end{equation}
Here $V \in \mathbb{R}$ is called the coupling constant -- the amplitude of the potential term. 
A short computation yields that $H_{\omega,V}$ is a dynamically-defined operator for the bounded kernel
\[
k_V:\Z\times \Aa^\Z\to\R, \quad 
	k_V\left(m,\omega\right)
		:= 
			\begin{cases}
				1, \qquad &\textrm{if } m\in\{-1,1\},\\
				V\cdot \omega(0), \qquad &\textrm{if } m=0,\\
				0, \qquad &\textrm{otherwise}.
			\end{cases}
\]
It is straightforward to check that this kernel is linearly decaying and locally constant; see also \cite[Sec.~2.2]{BecTak25} and Example~\ref{ex:DynDefOp_integers_Lipsch_LocConst}. Thus, Theorem~\ref{thm:SpecEst_DynDefOp}~(b) applies to the operator family $(H_{\omega,V})_{\omega\in\Aa^\Z}$ with $V\in\R$.

In this section, we focus on a particular class of operators defined over the alphabet $\Aa := \{0,1\}$, sharing the same parameter space as the almost Mathieu operator discussed in Section~\ref{sec:Hofstadter}. 

Recall that $\T$ denotes the one-dimensional torus, identified with the interval $[0,1)$, and that the fractional part of a real number $x \in \R$ is given by $\{x\} := x - \lfloor x \rfloor$; see eq.~\eqref{eq:fractional_part}.

For a rotation parameter $\alpha \in [0,1]$ and a phase $\theta \in \T$, we define the configuration
\[
\omega_{\alpha,\theta} \in \Aa^\Z, \qquad \omega_{\alpha,\theta}(n) := \chi_{[1-\alpha,1)}\big( \{n\alpha + \theta\} \big) \in\Aa.
\]
Here $\chi_{[1-\alpha,1)}$ denotes the characteristic function of the interval $[1-\alpha,1)\subseteq \T$.
For simplicity, we abbreviate $\omega_\alpha := \omega_{\alpha,0}$. The associated dynamical systems are given by the orbit closures
\begin{equation}
\label{eq:Omega_alpha}
\Omega_\alpha := \overline{\Orb(\omega_\alpha)} \in \inv,
\end{equation}
in the symbolic dynamical system $(\Aa^\Z, \Z, \tau)$.

It can be shown (see e.g. \cite[Sec.~6]{Fog02} or \cite[App.~5]{DaLe04}) that $\omega_{\alpha,\theta} \in \Omega_\alpha$ for all $\theta \in \T$ and that $\Omega_\alpha$ is a minimal dynamical system in the sense of Definition~\ref{def:minimal} for every $\alpha \in [0,1]$. In the case where $\alpha = \frac{p}{q} \in [0,1]$ is rational with coprime integers $p$ and $q$, the sequence $\omega_{\alpha,\theta}$ is periodic under the shift $\tau$. More precisely, one has $\tau(q,\omega_{\alpha,\theta}) = \omega_{\alpha,\theta}$ for all $\theta \in \T$, and the orbit closure $\Omega_\alpha$ consists of exactly $q$ distinct elements.

If $\alpha \in [0,1]$ is irrational, then $\omega_\alpha$ is a non-periodic sequence and the associated dynamical system $\Omega_\alpha$ is referred to as a \emph{Sturmian dynamical system}. For a detailed exposition and further references, we refer the reader to \cite{Fog02,Queff10,BaaGri13,DaFi24-book_2}.

The so-called \emph{Kohmoto model} \cite{KKT83,OPRSS83,KO84,OK85} is the family of discrete Schrödinger operators $\Hkoha : \ell^2(\mathbb{Z}) \to \ell^2(\mathbb{Z})$ defined by
\[
\Hkoha := H_{\omega_{\alpha,\theta},V}.
\]
The operator $\Hkoha$ is also referred to as the \emph{Sturmian Hamiltonian} if $\alpha$ is irrational. This is a prototypical toy model for one-dimensional quasicrystals

Due to the minimality of the associated dynamical systems, standard arguments (see, e.g.,~\cite{LenSto03Alg,Bec16} and Corollary~\ref{cor:SemiCont_Spectrum}) yield
\begin{equation}
\label{eq:SpectrConst_Kohm}
\sigma(\Hkoha) = \sigma(\Koh), \qquad \theta \in \T, \; \alpha \in [0,1].
\end{equation}
Hence, we will omit the $\theta$-dependence in the sequel and focus on the operator 
\[
H_{\alpha,V} := \Koh
\]
and the associated configuration $\omega_\alpha \in \Aa^\Z$ for $\alpha \in [0,1]$. We emphasize that, in contrast to the almost Mathieu operator, where the spectrum is $\theta$-independent only for irrational $\alpha$, the Kohmoto model exhibits this property for all $\alpha \in [0,1]$.
This follows from the fact that $\Omega_\alpha$ is minimal for all rational $\alpha \in [0,1]$.

Let now $\alpha = \frac{p}{q} \in [0,1]$ be rational with coprime integers $p$ and $q$. Since $\omega_\alpha$ is $q$-periodic, the spectrum $\sigma(H_{\alpha,V})$ consists of a disjoint union of $q$ intervals for $V \neq 0$, known as \emph{spectral bands}; see, e.g.,~\cite{Tes00,Simon11,DaFi24-book_2}, and the detailed discussion in~\cite[Prop.~3.5]{BaBeBiRaTh24}. 

In general, one can only guarantee the presence of at most $q$ spectral bands for such periodic models. For instance, the almost Mathieu operator with $\alpha=\tfrac{1}{2}$ has only one spectral band, see Figure~\ref{fig:Hofstadter-butterfly}. However, in the case of the Kohmoto model, one obtains exactly $q$ spectral bands whenever $V \neq 0$. If $V = 0$, the potential vanishes and the spectrum equals to $[-2,2]$, independently of $\alpha \in [0,1]$.

For varying rational values of $\alpha$, the corresponding spectra are depicted in Figure~\ref{fig:Kohmoto-butterfly}, illustrating the structure commonly referred to as the \emph{Kohmoto butterfly}. As in the case of the almost Mathieu operator, this butterfly structure reflects features of the spectrum in the irrational regime. Therefore, the continuity of the spectrum at irrational $\alpha$ is of central importance -- a topic discussed in the following, along with further implications; see also Chapter~\ref{chap:DTMP}.

The mathematical foundation for the study of the Kohmoto model was laid in \cite{Cas86,Sut89,BIST89,BIT91,Raym95}. The spectrum of the Sturmian Hamiltonians ($H_{\alpha,V}$ with irrational $\alpha$) forms a Cantor set of Lebesgue measure zero \cite{Sut89,BIST89}. Building on earlier ideas by Casdagli \cite{Cas86}, Raymond \cite{Raym95,Raym95-thesis} established a coding of the spectral bands for rational $\alpha$ in the regime $V > 4$, see a more detailed discussion in Section~\ref{sec:Types_Spectral_Bands}. A comprehensive review of Raymond's results was recently provided in the joint work \cite{BaBeBiRaTh24}.

This hierarchical structure was further extended in \cite{BaBeLo24} to arbitrary couplings $V \neq 0$, which in turn allowed solving the associated dry ten Martini problem for Sturmian Hamiltonians; see Chapter~\ref{chap:DTMP} for details. 

Given that the spectrum in the irrational case is a Cantor set of zero Lebesgue measure, various fractal characteristics have been investigated in the literature; see, e.g.,~\cite{Raym95,KiKiLa03,DaEmGoTc08,LiQuWe14,DaGoYe16}. For a more comprehensive discussion, we refer to \cite{Dam07_survey,DaEmGo15-survey,Dam17-Survey,DaFi24-book_2} and references therein.

In this section, we follow the joint work \cite{BecBelTho25}, where both the regularity of the spectral map
\[
\Sigma_V : [0,1] \to \Kk(\R), \qquad \alpha \mapsto \sigma(H_{\alpha,V}),
\]
and the emergence of spectral defects in the Kohmoto butterfly were analyzed.

Using $C^\ast$-algebraic methods, the authors of \cite{BIT91} proved that the spectral map $\Sigma_V$ is continuous at all irrational values of $\alpha$, while being discontinuous at all rational $\alpha$, provided that $[0,1]$ is equipped with the Euclidean topology. This discontinuity is already visible in numerical simulations -- as illustrated in Figure~\ref{fig:Defects-Kohmoto-butterfly}.

Consider for example $r = \tfrac{2}{3}$, for which the spectrum $\sigma(H_{r,V})$ consists of three disjoint intervals. When approaching $r$ from left or right, the corresponding limiting spectra $\sigma(H_{r_+,V})$ and $\sigma(H_{r_-,V})$ exist by Theorem~\ref{thm:SpectralEstim_Sturm}, and contain isolated eigenvalues located in the spectral gaps of $\sigma(H_{r,V})$, indicated by red points in Figure~\ref{fig:Defects-Kohmoto-butterfly}.

Already at the numerical level, one observes that these limiting spectra differ: the eigenvalues in $\sigma(H_{r_+,V})$ and $\sigma(H_{r_-,V})$ do not coincide. We refer to these additional eigenvalues -- present in the limit spectra but absent in $\sigma(H_{r,V})$ itself -- as \emph{spectral defects}. 
This numerical observation -- already noted in~\cite[Rem.~3]{BIT91} -- is rigorously proven in~\cite[Prop.~5.5]{BecBelTho25} by building on the hierarchical structure of the spectral bands established in~\cite{Raym95,BaBeBiRaTh24,BaBeLo24}; see also the discussion in Section~\ref{sec:Types_Spectral_Bands}.

\begin{figure}[hptb]
    \centering
    \includegraphics[scale=1.2, angle=90]{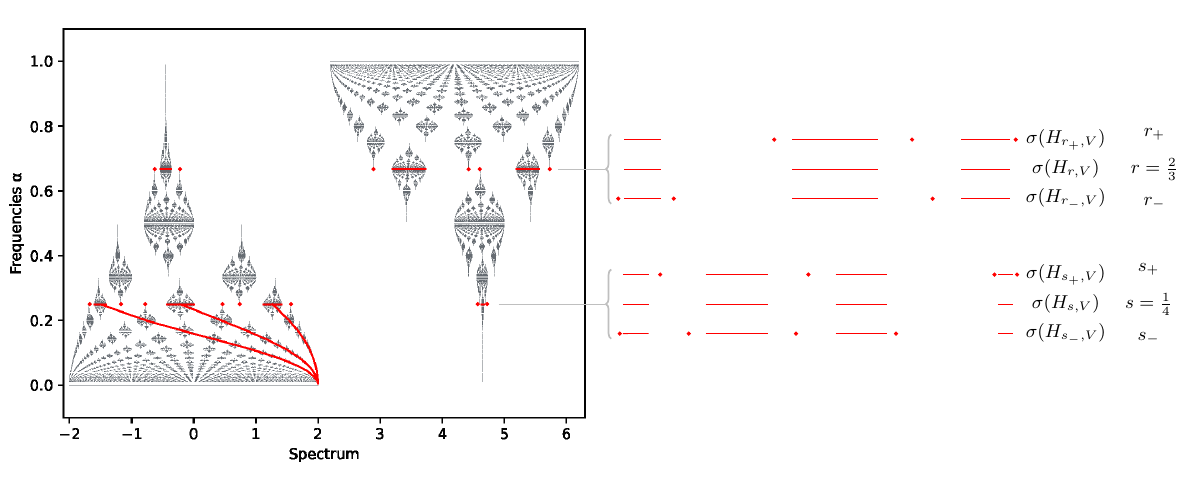}
    \captionsetup{width=0.95\linewidth}
    \caption{The Kohmoto butterfly: For rational $\alpha \in [0,1]$, the spectrum $\sigma(H_{\alpha,V})$ is shown for coupling $V = 4$. At selected rational values $r, s \in [0,1]$, spectral defects (marked as red points) arising from the left-sided limits ($r_-$ and $s_-$) and right-sided limits ($r_+$ and $s_+$) are highlighted. The number of such defects depends on the denominators of $r$ and $s$. For details, see \cite{BecBelTho25} and the discussion below.}
    \label{fig:Defects-Kohmoto-butterfly}
\end{figure}

To rigorously study both the spectral defects and the regularity of the spectral map, consider the closed subspace
\[
\overline{\mathfrak{S}} := \overline{\set{\Omega_\alpha}{\alpha \in [0,1]}}
\]
in the metric space $(\inv, \delta_\Aa)$ as introduced in Section~\ref{subsec:symb_syst_as_weighted_Delone}. In \cite{BecBelTho25}, a \emph{Farey topology} on $[0,1]$ was introduced, which characterizes the pullback topology induced by the parameter map $[0,1]\ni \alpha \mapsto \Omega_\alpha \in \inv$. This topology is strictly finer than the standard Euclidean topology and -- when combined with \cite{BBdN18} -- ensures the continuity of the spectral map $\Sigma_V$.

Each rational number $r \in [0,1]$ is identified with its unique irreducible representation $r = \tfrac{p}{q}$, where $p$ and $q$ are coprime integers. In particular, we adopt the convention that $0 = \tfrac{0}{1}$ and $1 = \tfrac{1}{1}$. A reduced rational number $r = \tfrac{p}{q}$ is called an \emph{$m$-Farey number} for $m \in \N$ if its denominator satisfies $q \leq m$. The set of all such numbers is denoted by $F_m$, which is clearly finite. For instance,
\[
F_1 = \{0,1\}, \quad 
F_2 = \left\{ 0, \tfrac{1}{2}, 1 \right\}, \quad 
F_3 = \left\{ 0, \tfrac{1}{3}, \tfrac{1}{2}, \tfrac{2}{3}, 1 \right\}.
\]

The $m$-Farey numbers form a totally ordered set. Given $r, s \in F_m$ with $r < s$, we say that $r$ and $s$ are \emph{$m$-Farey neighbors} if the open interval $(r, s)$ contains no other elements of $F_m$. For each $r \in F_m$, we denote by $r_\ast$ and $r^\ast$ the unique $m$-Farey neighbors satisfying $r_\ast < r < r^\ast$, except in the cases $r = 0$ or $r = 1$, where only one such neighbor exists.

Based on the Farey numbers, we define the \emph{Farey metric} $d_F: [0,1] \times [0,1] \to [0,\infty)$ by
\[
d_F(\alpha, \beta) :=
\begin{cases}
	0, & \alpha = \beta, \\
	1, & \alpha \in \{0,1\} \text{ or } \beta \in \{0,1\}, \\
	\min \left\{ \tfrac{1}{m+1} \;\middle|\; \exists\, r \in F_m \text{ with } \alpha, \beta \in (r, r^\ast) \right\}, & \text{otherwise}.
\end{cases}
\]

Following \cite{BecBelTho25}, the metric space $([0,1], d_F)$ is not complete. Its completion, denoted by $\overline{[0,1]}_F$, defines a compact metric space which we refer to as the \emph{Farey space}. The topology induced by $d_F$ is strictly finer than the Euclidean topology -- that is, convergence in the Farey metric implies convergence in the usual Euclidean sense.

In \cite[Cor.~3.4]{BecBelTho25}, we showed that the Farey space $\overline{[0,1]}_F$ can be realized by adjoining to each rational number $r \in [0,1]$ a left- and right-sided limit point $r_-$ and $r_+$, respectively:
\[
\overline{[0,1]}_F = [0,1] \cup \{ r_- \mid r \in \Q \cap (0,1] \} \cup \{ r_+ \mid r \in \Q \cap [0,1) \}.
\]

Moreover, by \cite[Thm.~3.8]{BecBelTho25}, there exists a unique surjective isometry
\[
\Phi: \overline{[0,1]}_F \to \overline{\mathfrak{S}}
\]
satisfying $\Phi(\alpha) = \Omega_\alpha$ for all $\alpha \in [0,1]$.

\begin{proposition}[\cite{BecBelTho25}]
\label{thm:Sturm_Defects}
For each $r \in [0,1]$, there exist configurations $\omega_{r_+}, \omega_{r_-} \in \Aa^\Z$ such that
\[
\Phi(r_+) = \overline{\Orb(\omega_{r_+})}
\qquad\text{and}\qquad
\Phi(r_-) = \overline{\Orb(\omega_{r_-})}.
\]
\end{proposition}

\begin{proof}
This is proven in \cite[Prop.~4.6]{BecBelTho25}.
\end{proof}

The proposition \cite[Prop.~4.6]{BecBelTho25} asserts even more. In particular, the configurations $\omega_{r_\pm}$ can be regarded as perturbations of the original periodic sequence $\omega_r\in\Aa^\Z$ for $r \in [0,1] \cap \Q$. The dynamical impurity is explicitly determined by the Farey neighbors and the continued fraction expansion of $r$; see again \cite[Prop.~4.6]{BecBelTho25}. The phenomenon is illustrated by the following example.

\begin{example}
Let $r = \tfrac{1}{2}$, whose associated configuration is
\[
\omega_r = \ldots\, 01\ 01\ 01 \cdot \underline{01}\ 01\ 01\, \ldots\,.
\]
Here, the symbol $\cdot$ marks the origin, i.e., the first letter to the right is $\omega_r(0)$. The sequence $\omega_r$ is periodic with period $2$, and the underlined word $01$ repeats indefinitely. 
Observe that $r=\tfrac{1}{2}$ appears in $F_2$ the first time.
In $F_2$, the number $r$ has the Farey neighbors $r_\ast = 0$ and $r^\ast = 1$. Their associated configurations are given by
\[
\omega_0 = \ldots\, 0\ 0\ 0 \cdot \underline{0}\ 0\ 0\, \ldots
\qquad\text{and}\qquad
\omega_1 = \ldots\, 1\ 1\ 1 \cdot \underline{1}\ 1\ 1\, \ldots\,.
\]
Both are periodic with period one, and the underlined words represent the building blocks for the impurities in $\omega_{r_-}$ and $\omega_{r_+}$ for $r=\tfrac{1}{2}$. These underlined words appear as local impurities (highlighted by a box) of the periodic structure of $\omega_r$, namely
\[
\omega_{r_-} = \ldots\, 01\ 01\ 01\ \boxed{0} \cdot 01\ 01\ 01\, \ldots
\qquad\text{and}\qquad
\omega_{r_+} = \ldots\, 10\ 10\ 10\ \boxed{1} \cdot 10\ 10\ 10\, \ldots\,.
\]
\end{example}

Using the general version \cite[Prop.~4.6]{BecBelTho25}, each $x \in \ol{[0,1]}_F$ can be associated with a configuration $\omega_x \in \Aa^\Z$ such that
\[
\Phi(x) = \overline{\Orb(\omega_x)}.
\]
To be consistent with the above definition on $\omega_\alpha$, we require for $x \in [0,1]$,
\[
\omega_x(n) = \chi_{[1 - x, 1)}(\{n x\}), \qquad n\in\Z.
\]

For each such configuration $\omega_x$ and every coupling constant $V \in \R$, we consider the corresponding Schrödinger operator
\[
H_{\omega_x,V} : \ell^2(\Z) \to \ell^2(\Z),
\]
as defined in~\eqref{eq:Schroedinger_Z}. Recall that $H_{\omega_x,V} = H_{x,V}$ whenever $x \in [0,1]$.

\begin{theorem}[\cite{BecBelTho25}]
\label{thm:SpectralEstim_Sturm}
For all $V \in \R$, the spectral map
\[
\Sigma_V : \ol{[0,1]}_F \to \Kk(\R), \qquad x \mapsto \sigma(H_{\omega_x,V}),
\]
is Lipschitz continuous. More precisely, there exists a constant $C > 0$ such that
\[
d_H\big( \sigma(H_{\omega_x,V}), \sigma(H_{\omega_y,V}) \big) \leq C \cdot d_F(x,y), \qquad x,y \in \ol{[0,1]}_F.
\]
\end{theorem}

\begin{proof}
This is proven in \cite[Thm.~1.2]{BecBelTho25}.
\end{proof}

\begin{corollary}[\cite{BecBelTho25}]
\label{cor:SpectralEstim_Sturm}
For each $V \in \R$, there exists a constant $C > 0$ such that for all $\alpha, \tfrac{p}{q} \in [0,1]$ with $0 < \big|\alpha - \tfrac{p}{q}\big| < \tfrac{1}{q^2}$, the estimate
\[
d_H \left( \sigma(H_{\alpha,V}), \sigma\big(H_{\tfrac{p}{q},V}\big) \right) \leq \frac{C}{q}
\]
holds.
\end{corollary}

\begin{proof}
This is proven in \cite[Cor.~1.3]{BecBelTho25}.
\end{proof}

This recovers the continuity result obtained in \cite{BIT91} and provides a quantitative refinement. In contrast to earlier approaches, the proof is based on different techniques, relying on the results from \cite{Mig91,BeBeCo19,BecTak25} partially discussed in Section~\ref{sec:SpectralEstimates}. Moreover, it is proven in \cite[Thm.~6.2]{BecBelTho25} that the obtained spectral estimates are optimal at all $r_{\pm}$ provided that $V > 4$. More precisely, one constructs a sequence $(x_n)_{n\in\N} \subseteq [0,1]\cap\Q$ converging to $r_{\pm}$ in the Farey metric such that there exist constants $0 < C_1 \leq C_2$ satisfying
\begin{equation}
\label{eq:Optimality_Kohmoto}
C_1\, d_F(x_n, r_{\pm}) 
\;\leq\; 
d_H\big( \sigma(H_{\omega_{x_n},V}),\, \sigma(H_{\omega_{r_{\pm}},V}) \big) 
\;\leq\;
C_2\, d_F(x_n, r_{\pm}).
\end{equation}
As an example, for $r_+ = 0_+$ one can trace spectral gaps that close as $x_n \to r_+$, illustrated by the red lines in Figure~\ref{fig:Defects-Kohmoto-butterfly}. A related result for the almost Mathieu operator is proven in \cite{BeRa90}; see also the discussion in \cite{BecTak25}.

\section{The Lamplighter group}
\label{sec:Lamplighter}

Let us consider the Lamplighter group~\cite{KaiVer83,Lin01,Butk01}, which is the countable discrete group given by 
\[
G := \left\{ (m,\gamma) \;\middle|\; m \in \Z,\; \gamma \in \bigoplus_{j\in\Z} \Z_2 \right\},
\]
where $\Z_2$ denotes the cyclic group of order two. The second component $\gamma$ is a finitely supported map $\gamma:\Z\to\Z_2$ (since we consider the direct sum).

The group multiplication is defined by
\[
(m,\gamma)(n,\eta) := \left(m+n,\, S^n(\gamma) + \eta\right),
\]
where $S : \bigoplus_{j\in\Z} \Z_2 \to \bigoplus_{j\in\Z} \Z_2$ is the shift given by $S(\gamma)(j) := \gamma(j+1)$ for all $j \in \Z$.

The inverse of an element is given by
\[
(m,\gamma)^{-1} = \left(-m,\, S^{-m}(\gamma)\right).
\]
The lamplighter group is discrete and so the Haar measure (being the counting measure) is left and right invariant, namely $G$ is unimodular.

This group is for instance generated by the set $S = \{t, at\}$ where $t = (1,0)$ and $a = (0,\delta_0)$, and $\delta_0 : \Z \to \Z_2$ is the sequence that vanishes everywhere except at zero. We use a word metric\footnote{This is the shortest path metric on the Cayley graph of $G$ with generating set $S = \{t, at\}$.} $d_G$ by the generator set $S = \{t, at\}$ following \cite[Sec.~2]{JoKe18}. 

The Lamplighter group is exponentially growing (so it is not strictly polynomially growing), but $G$ is amenable; see e.g.\ \cite[Example~1.50]{Butk01} and \cite[Ex.~3.7]{Pog10_Dipl}.

\begin{example}\label{ex:LL_Schroedinger}
Let $G$ be the Lamplighter group, and consider the symbolic dynamical system $(\Aa^G,G,\tau)$ introduced in Example~\ref{ex:SymbDynSyst} with finite alphabet $\Aa\subseteq \R$.
Fix a function $v : \Aa^G \to \R$ which is locally constant, e.g.
\[
v(\gamma) := \gamma(e),
\]
where $e = (0,0)$ is the neutral element in $G$. Define the kernel $k : G \times \Aa^G \to \R$ by
\[
k(g, \gamma) := \begin{cases}
		1 & \text{if } g \in S \cup S^{-1},\\
		v(\gamma) & \text{if } g = e,\\
		0 & \text{otherwise},
	\end{cases}
\]
where $S = \{t, at\}$ is a generating set of $G$.

It is straightforward to check that this is a bounded and locally constant kernel with finite support $\{e\}\cup S \cup S^{-1}$ in $G$. The associated dynamically-defined operator $A_\gamma : \ell^2(G) \to \ell^2(G)$ with $k$ is self-adjoint and
\[
(A_\gamma \psi)(g) 	= \left(\sum_{s \in S \cup S^{-1}} \psi(gs) \right) + \gamma(g) \psi(g)
	, \qquad \psi \in \ell^2(G),\ g \in G.
\]
This defines a Schr\"odinger operator on the Cayley graph of the lamplighter group with generating set $S$. We now address the question if we can prove also spectral estimates for such operators.
\end{example}

Using the method developed in \cite{BecTak25} (see Theorem~\ref{thm:SpecEst_DynDefOp} as well as the detailed presentation in Appendix~\ref{App:SpectralEstimates}), we establish spectral estimates for dynamically-defined operators on the Lamplighter group. In this section, we focus on locally constant potentials, motivated by the symbolic dynamical system $(\Aa^G, G, \tau)$ associated with the Lamplighter group $G$, see Example~\ref{ex:LL_Schroedinger}. Similar results can be obtained for Lipschitz continuous kernels.

\begin{proposition}
\label{prop:Lamplighter}
Let \( (\Aa^G,G,\tau) \) be a symbolic dynamical system for a finite alphabet $\Aa$ where $G$ is the lamplighter group equipped with the word metric $d_G$ defined by the generating set $S=\{t,at\}$.
For $N\in\N$, let \( k: G \times Z \to \C^{N\times N} \) be a bounded kernel with dynamically-defined operator $(A_x)_{x\in Z}$ such that $A_x$ is normal for every $x \in Z$. 
If furthermore
\begin{enumerate}[label=(\alph*)]
\item 
$k$ is decaying: there exists $c_s > 0$ such that 
\[
\sup_{x \in Z} \sum_{(j,\eta)\in G}  \sup_{g\in G} \big\| k\big((j,\eta),g^{-1}x\big)\|_M \, |j| \leq c_s.
\]
\item 
$k$ is locally constant: there exists \( 0< c_k \in \ell^\infty(G) \) such that for all \( h \in G \),
\[
d(x,y) < \frac{1}{c_k(h)} \quad \Rightarrow \quad k(h,x) = k(h,y).
\]
\end{enumerate}
Then there exists a constant $C>0$ such that
\[
d_H\big(\sigma(A_X), \sigma(A_Y)\big) \le C\, \delta_H(X, Y), \qquad X,Y\in\inv.
\]
\end{proposition}

\begin{proof}
We postpone the proof to the appendix since it is slightly more involved, see Proposition~\ref{prop:Lamplighter_app}. 
The argument essentially follows the lines of the proof of Theorem~\ref{thm:LocalConst-Spectrum_app} (see also \cite[Thm.~4.7]{BecTak25}). We use a F\o lner sequence defined in \cite{Butk01,Pog10_Dipl} to define a suitable cut-off function.
\end{proof}

\chapter[The DTMP for Sturmian Hamiltonians]{The dry ten Martini problem for Sturmian Hamiltonians}
\label{chap:DTMP}

During the AMS Annual Meeting in 1981, Mark Kac asked whether the almost Mathieu operator ``has all its gaps there''~\cite{Sim82-review} for irrational \( \alpha \), offering ten Martinis to whoever could solve the problem. Shortly thereafter, Barry Simon~\cite{Sim82-review} coined the name of the \emph{(dry) ten Martini problem}.

The question is closely tied to the study of spectral gaps in quasiperiodic operators and the so-called Gap Labeling Theorem, which predicts the set of possible gap labels for the spectrum based on the underlying dynamical system. 
In the case of the almost Mathieu operator, it was conjectured that the spectrum is a Cantor set and that all allowed gaps are there -- that is, the spectrum realizes the maximal number of spectral bands separated by actual gaps.

A closely related question also remained a long-standing open problem for the Sturmian Hamiltonians -- a prototypical toy model for one-dimensional quasicrystals. It was recently fully resolved in the joint work \cite{BaBeLo24}.

We discuss these problems in parallel for both the almost Mathieu operator and the Sturmian Hamiltonians, highlighting their structural similarities.

For \( \alpha \in [0,1] \setminus \Q \) and \( V \in \R \), recall the definitions of the almost Mathieu operator \( \AMO : \ell^2(\Z) \to \ell^2(\Z) \) from Section~\ref{sec:Hofstadter} and the Kohmoto model \( \Koh : \ell^2(\Z) \to \ell^2(\Z) \) from Section~\ref{sec:Kohmoto}.\footnote{We omit the dependence on \( \theta \) in both the almost Mathieu operator and the Kohmoto model, since for all irrational \( \alpha \) the spectrum is independent of \( \theta \); see eq.~\eqref{eq:SpectrConst_almMath} and eq.~\eqref{eq:SpectrConst_Kohm}.} Recall that $\Koh$ is also called a Sturmian Hamiltonian if $\alpha$ is irrational.
In the following, we write \( H_{\alpha,V} \) to refer to either \( \AMO \) or \( \Koh \).

\medskip

\underline{\textit{The ten Martini problem (TMP):}} One asks whether the spectrum \( \sigma(H_{\alpha,V}) \) is a Cantor set for \( V \neq 0 \) and irrational \( \alpha \in [0,1] \setminus \Q \).

For the almost Mathieu operator, the final solution of the TMP was given by the breakthrough work of Avila and Jitomirskaya~\cite{AviJit09-TMP}, building on a long line of partial results and techniques developed in~\cite{BeSi82,Sin87,HelSjo89,ChElYu90,Last94,Puig04,AvKri06}.

For the Sturmian Hamiltonians, it is known that the spectrum is a Cantor set of zero Lebesgue measure~\cite{Sut89,BIST89}. Substantial generalizations were established in~\cite{Len02,DaLe06_Boshernitzan,DaLe06_ZeroMeasure} based on the so-called \emph{Boshernitzan condition}~\cite{Bos84}, which was further extended to Jacobi operators in the joint work~\cite{BecPog13}. A key tool in this analysis is Kotani theory~\cite{Sim83,Kot84,Kotani89}.

\medskip

\underline{\textit{The dry ten Martini problem (DTMP):}}\\
Let \( (H_{\alpha,V})|_{[0,n-1]} \) denote the Hermitian \( n \times n \) matrix given by the restriction of \( H_{\alpha,V} \) to the finite-dimensional subspace \( \ell^2(\{0,\ldots,n-1\}) \). This self-adjoint matrix has exactly \( n \) real eigenvalues (counted with multiplicity), denoted by \( \sigma\left((H_{\alpha,V})|_{[0,n-1]}\right) \).

The \emph{integrated density of states} (IDS) \( N_{\alpha,V} : \R \to [0,1] \) is defined by
\[
N_{\alpha,V}(E) := \lim_{n\to\infty} \frac{\sharp \set{\lambda\in \sigma\left( (H_{\alpha,V})|_{[0,n-1]}\right) }{\lambda\leq E} }{n}.
\]
This limit exists for irrational \( \alpha \in [0,1] \setminus \Q \) and arbitrary \( V \in \R \), since the associated dynamical systems are uniquely ergodic; see e.g.~\cite{Kirsch08,DaFi22-book_1,DaFi24-book_2}. The IDS satisfies the following properties:

\begin{enumerate}[label=\textbf{(IDS--\arabic*)},ref=IDS--\arabic*,labelwidth=!,align=left]
\item \label{item:IDS-1} The function \( N_{\alpha,V} : \R \to [0,1] \) is monotone increasing and continuous.

\item \label{item:IDS-2} We have \( E \in \R \setminus \sigma(H_{\alpha,V}) \) if and only if there exists \( \varepsilon > 0 \) such that the restriction \( N_{\alpha,V}|_{(E - \varepsilon,\, E + \varepsilon)} \) is constant.
\end{enumerate}

An open interval \( \gap := (a,b) \subseteq \R \) is called a \emph{spectral gap} if
\[
 a, b \in \sigma(H_{\alpha,V}) 
 \qquad \text{and} \qquad 
 \gap \cap \sigma(H_{\alpha,V}) = \emptyset. 
\] 
Fix such a spectral gap \( \gap \). Using~\eqref{item:IDS-1} and~\eqref{item:IDS-2}, we conclude that the IDS is constant on \( \gap \), i.e.,
\[
N_{\alpha,V}(E) = N_{\alpha,V}(E') \quad \text{for all } E, E' \in \gap.
\]
The common value of the IDS on \( \gap \) is called the \emph{gap label} of \( \gap \), see Figure~\ref{fig:IDS_gap_label}.

\begin{figure}[htb]
    \centering
    \includegraphics[scale=1.4]{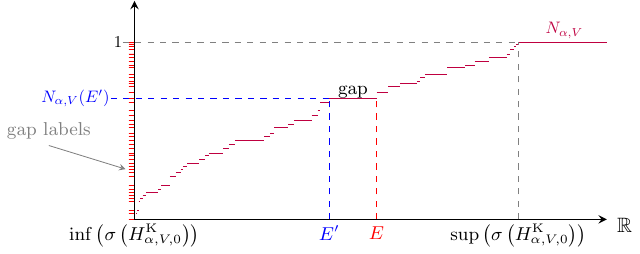}
    \captionsetup{width=0.95\linewidth}
    \caption{An illustration of the plateaus of an integrated density of states and a gap $(E',E)$ with gap label $N_{\alpha,V}(E')=N_{\alpha,V}(E)$.
    }
    \label{fig:IDS_gap_label}
\end{figure}

The possible gap labels are described by the \emph{Gap Labeling Theorem}~\cite{Bellis82_Overview,JohMos82,Bell92-Gap,BelBovGhe92,DamFill23-GapLabel,DamFilZha23}, which can be approached via different frameworks. One approach, originally developed by Jean Bellissard, is based on \( K \)-theory. For one-dimensional models such as \( H_{\alpha,V} \), an alternative approach is due to Johnson and Moser.

For both the almost Mathieu operator and the Sturmian Hamiltonians, one obtains the same set of possible gap labels, namely
\[
\set{N_{\alpha,V}(E)}{E \in \R \setminus \sigma(H_{\alpha,V})} \subseteq \set{ \{n\alpha\} }{n \in \Z} \cup \{1\}, \qquad \alpha \in [0,1] \setminus \Q,
\]
where \( \{n\alpha\} := n\alpha - \lfloor n\alpha \rfloor \) denotes the fractional part of \( n\alpha \). With this in hand, we can formally state the question originally raised by Kac: When do we have equality in the above inclusion for irrational $\alpha$? That is, does there exist, for each \( n \in \Z \), a spectral gap \( \gap \subseteq \R \setminus \sigma(H_{\alpha,V}) \) such that \( N_{\alpha,V}(E) = \{n\alpha\} \) for all \( E \in \gap \)? Note that the possible gap labels are dense in $[0,1]$, if $\alpha\in[0,1]$ is irrational. Also the gap label \( 1 \) is always attained for any energy above the supremum of the spectrum.

Substantial progress toward resolving this question for the almost Mathieu operator has been made over the past decades~\cite{ChElYu90,Puig04,AvJit10-JEMS,LiYu15}, and most recently in~\cite{AviYouZho23}. 

While the question was originally formulated for the almost Mathieu operator, it has since been studied for a wide range of other models; see, for instance,~\cite{Han18,DamLi24,DamEmiFil24,CedLi25,GeWanXu25}. 
We emphasize, however, that the set of possible gap labels can vary depending on the underlying model.

For the Sturmian Hamiltonians, the set of possible gap labels coincides with that of the almost Mathieu operator. The solution of the DTMP for the Sturmian Hamiltonians was achieved in the joint work~\cite{BaBeLo24}; see also the corresponding Oberwolfach preprint~\cite{BaBeLo23-MFO}. A first significant step in this direction was taken by Raymond in his PhD thesis~\cite{Raym95-thesis}, see also~\cite{Raym95}. This preprint had a notable influence on the field (see e.g.~\cite{Dam07_survey,DaEmGo15-survey,Dam17-Survey,DaFi24-book_2} and references therein), and a modernized presentation of these ideas is provided in the joint work~\cite{BaBeBiRaTh24}, adapted to the perspective of~\cite{BaBeLo24}. Raymond's results will also appear in revised form in~\cite{Raym-AperiodicOrder}.

In his work, Raymond classified the spectral bands of periodic approximations into three types under the assumption \( V > 4 \). The resulting coding allowed him to express the IDS explicitly (see Proposition~\ref{prop:IDS_Kohmoto}) and thereby to resolve the DTMP for the Sturmian Hamiltonians if \( V > 4 \).

The case \( V > 4 \) corresponds to the regime in which the potential dominates the kinetic term of the Schrödinger operator. Similar assumptions -- often even stronger -- appear in the literature when analyzing the fractal nature of the spectrum. The main technical difficulty beyond this regime is the overlap of spectral band in different approximations and the lack of techniques to control such overlaps.

A particular focus in the literature has been placed on the case where \( \alpha \) equals the inverse of the golden mean
\[
\phi := \frac{\sqrt{5} - 1}{2} = \left(\frac{\sqrt{5} + 1}{2}\right)^{-1},
\]
whose continued fraction expansion consists entirely of ones. The associated operator \( \Fib \) is commonly referred to as the \emph{Fibonacci Hamiltonian}, since its rational approximations are determined by the Fibonacci numbers.

Damanik and Gorodetski~\cite{DamGor11} showed that, for the Fibonacci Hamiltonian, all predicted gap labels appear if the coupling constant \( V \) is sufficiently small. This technique was extended by Mei~\cite{Mei14} to irrational \( \alpha \) with eventually periodic continued fraction expansions. In this setting, the threshold for how small the coupling must be depends on the continued fraction expansion of the given irrational number.
Later, Damanik, Gorodetski, and Yessen~\cite{DaGoYe16} proved -- among many other remarkable results -- that all spectral gaps are there for the Fibonacci Hamiltonian solving the DTMP for the Fibonacci Hamiltonian.
Note that the golden mean, as well as irrational numbers with eventually periodic continued fraction expansions, satisfy a Diophantine condition excluding many other irrational numbers.

A unifying theme in the aforementioned results is the study of traces of the transfer matrices and their recursive relations. While these algebraic relations provide deep insight, they are not sufficient on their own to resolve the DTMP for the Sturmian Hamiltonians.

In our joint work~\cite{BaBeLo24}, two new ingredients --  allowing us to control also overlaps of spectral bands -- were introduced:
\begin{itemize}
\item We studied the space of all continued fraction expansions simultaneously, rather than fixing a single irrational \( \alpha \).
\item We employed interlacing arguments for eigenvalues to control spectral bands in the small coupling regime \( 0 < V \leq 4 \).
\end{itemize}
Using these tools, we established the solution of the DTMP for Sturmian Hamiltonians: all gaps are there for all Sturmian Hamiltonians \( \Koh \) with non-zero coupling $V$.

\begin{theorem}[\cite{BaBeLo24}]
\label{thm:DTMP}
For all \( \alpha \in [0,1] \setminus \Q \) and all \( V \neq 0 \), the equality
\[
\set{N_{\alpha,V}(E)}{E \in \R \setminus \sigma(\Koh)} = \set{ \{n\alpha\} }{n \in \Z} \cup \{1\}
\]
holds.
\end{theorem}

\begin{remark}
By standard arguments, Theorem~\ref{thm:DTMP} also implies that \(\sigma(\Koh)\) is a Cantor set for all \(\alpha \in [0,1] \setminus \Q\); see the brief discussion in Remark~\ref{rem:DTMP->TMP}. That the spectrum is a Cantor set -- indeed, even of Lebesgue measure zero -- was previously established in \cite{BIST89}. Theorem~\ref{thm:DTMP} thus provides an alternative proof of the Cantor structure of the spectrum.
\end{remark}

In the following, we provide an overview of the strategy used to prove this result, as well as additional results obtained along the way.

\section[Classifying the spectral bands]{Classifying the spectral bands of periodic approximations}
\label{sec:Types_Spectral_Bands}

The first step in the proof is to classify the spectral bands of periodic approximations. This classification is of independent interest and has various other applications. Inspired by the work of Casdagli~\cite{Cas86}, Raymond~\cite{Raym95} encoded the spectral bands of \( \Koh \) for rational \( \alpha \in [0,1] \cap \Q \) and \( V > 4 \) into three types: a spectral gap (type~I), and two types of spectral bands (type~II and type~III).
Inspired by this work,  a tree representation of the periodic approximations was developed in \cite{BaBeLo24}. 
In~\cite{BaBeBiRaTh24}, we revisit Raymond's classification including the new perspectives from the joint work \cite{BaBeLo24}. Only two types of spectral bands are used: type $A$ (corresponding to type~III) and type $B$ (corresponding to type~II) -- a notion also used before for the Fibonacci Hamiltonian \cite{KiKiLa03}.

This classification depends on the finite continued fraction expansion of the rational number \( \alpha \). Recall that every \( \alpha \in [0,1] \) defines a symbolic subshift \( \Omega_\alpha \subseteq \Aa^\Z \) with \( \Aa := \{0,1\} \), see eq.~\eqref{eq:Omega_alpha} in Section~\ref{sec:Kohmoto}. Moreover, as discussed in that section, the spectrum of \( \Koh \) for irrational \( \alpha \in [0,1] \setminus \Q \) can be approximated by the spectra of \( \Kohbet \) with rational \( \beta \in [0,1] \cap \Q \); see Corollary~\ref{cor:SpectralEstim_Sturm}.

Every irrational number \( \alpha \in [0,1] \) admits a unique continued fraction expansion
\[
\alpha = c_0 + \frac{1}{c_1 + \frac{1}{c_2+ \frac{1}{\ddots}}},
\]
where \( c_k \in \N \) for \( k \geq 0 \). The corresponding finite truncations define rational approximants
\[
\alpha_k 
	:= c_0 + \frac{1}{c_1 + \frac{1}{c_2+ \frac{1}{\ddots+\frac{1}{c_k}}}} \in\Q.
\]
Since \( \alpha \in [0,1] \), we have \( c_0 = 0 \) in our setting. 
For conceptual reasons that will become clear later, we prepend an additional digit \( c_{-1} := 0 \) to the expansion. This digit later serves as the root in an associated tree structure; see Section~\ref{sec:SpectralTree}.

Summing up, we identify \( \alpha \) with an infinite tuple and \( \alpha_k \) with its finite prefix:
\[
\alpha \mapsto [0,0,c_1,c_2,c_3,\ldots], \qquad \alpha_k \mapsto [0,0,c_1,c_2,\ldots,c_k].
\]

The rational approximants \( \alpha_k \) are known as the \emph{best approximations of the second kind} to \( \alpha \); see~\cite[Thm.~17]{Kh64}. In particular, if \( \alpha_k = \tfrac{p_k}{q_k} \) with $p_k$ and $q_k$ coprime, then
\[
\left| \alpha - \frac{p_k}{q_k} \right| < \frac{1}{q_k^2}.
\]
Each such \( \alpha_k \) defines a periodic subshift $\Omega_{\alpha_k}$ approximating the Sturmian dynamical system \( \Omega_\alpha \). Thus, the operators \( \big( \Kohk \big)_{k\in\N} \) define periodic approximations of the Sturmian Hamiltonian \( \Koh \) with converging spectra; see Theorem~\ref{thm:SpectralEstim_Sturm} and Corollary~\ref{cor:SpectralEstim_Sturm}.

Following the previous considerations, we introduce the following \emph{set of (finite) continued fraction expansions}
\[
\Co := \left\{ [0],~[0,0] \right\} \cup \bigcup_{k \in \N} \set{ [0,0,c_1,\ldots,c_k] }{ c_1,\ldots,c_{k-1} \in \N,~ c_k \in \N_{-1} },
\]
where \( \N_{-1} := \N \cup \{-1,0\} \). Note that we allow the last digit \( c_k \) of a tuple \( \co = [0,0,c_1,\ldots,c_k] \) with \( k \geq 1 \) to also take the values \( 0 \) and \( -1 \). This extension enables us to describe the different types of spectral bands. For \( \co = [0,0,c_1,\ldots,c_k] \) with \( k \geq 0 \), we use the shorthand notation
\[
[\co,m] := [0,0,c_1,\ldots,c_k,m] \in \Co,
\]
whenever it is defined. We define the \emph{evaluation map} \( \varphi : \Co \to \R \cup \{\infty\} \) by setting \( \varphi([0]) := \infty \), \( \varphi([0,0,-1]) := -1 \), and
\[
\varphi([0,c_0,c_1,\dots,c_k]) :=
\begin{cases}
\varphi([0,c_0,c_1,\dots,c_{k-2},c_{k-1}-1]), & \text{if } k \geq 2 \text{ and } c_k = -1, \\[0.2em]
\varphi([0,c_0,c_1,\dots,c_{k-2}]), & \text{if } k \in \N \text{ and } c_k = 0, \\[0.2em]
c_0 + \dfrac{1}{c_1 + \dfrac{1}{\ddots + \dfrac{1}{c_k}}}, & \text{otherwise}.
\end{cases}
\]

This definition connects \( \Co \) with the rational numbers. For a more detailed explanation and motivation, we refer the reader to~\cite{BaBeBiRaTh24,BaBeLo24}.

We only note that if \( [0,0,c_1,\ldots,c_k] \in \Co \) with \( c_k > 1 \), then
\[
\varphi\big( [0,0,c_1,\ldots,c_k] \big) = \varphi\big( [0,0,c_1,\ldots,c_k-1,1] \big).
\]
This reflects the fact that every rational number admits exactly two representations as a finite continued fraction~\cite[Chap.~I.4]{Kh64}. On the other hand, for irrational \( \alpha \in [0,1] \setminus \Q \), the truncation at level \( k \) defines a unique finite continued fraction expansion \( [0,0,c_1,\ldots,c_k] \), called the \( k \)-th convergent of \( \alpha \).

\begin{remark}
Let \( \Aa := \N_{-1} \cup \{\infty\} \) be the one-point compactification of \( \N_{-1} \). This compact alphabet can be used to describe the Farey space (excluding the rational number itself) as discussed in Section~\ref{sec:Kohmoto}. Every irrational \( \alpha \in [0,1] \setminus \Q \) can be identified with its unique infinite continued fraction \( [0,0,c_1,c_2,c_3,\ldots] \), and thus embedded in the product space \( X := \{ x : \N_{-1} \to \Aa \} \). Each \( \co \in \Co \) is embedded into \( X \) by continuing the finite tuple with the constant letter \( \infty \). Let \( Y \subseteq X \) be the subset consisting of all such embedded \( \co \in \Co \) and all irrational numbers, equipped with the induced product topology.

By the previous considerations, every rational \( r \in [0,1] \cap \Q \) appears with two continued fraction expansions, which can be interpreted as left and right limit points $r_-,r_+$ of \( r \). Indeed, the sequence \( x_m := [\co,m,\infty,\infty,\ldots] \) converges to \( [\co,\infty,\infty,\ldots] \) in the product topology. The same convergents were used in~\cite{BecBelTho25} to describe spectral defects.
\end{remark}

Let \( \beta = \varphi(\co) = \tfrac{p}{q} \in [0,1] \) for some \( \co \in \Co \), where \( p \) and \( q \) are coprime. Then the spectrum
\[
\sigma_\co(V) := \sigma(\Kohbet)
\]
is the disjoint union of exactly \( q \) intervals; see, e.g.,~\cite[Prop.~3.1]{Raym95} or \cite[Prop.~4.1]{BaBeBiRaTh24}. Each such interval is referred to as a \emph{spectral band of} \( \Kohbet \). In the following, we restrict the coupling constant \( V \) to be nonzero. In various considerations throughout~\cite{BaBeLo24,BaBeBiRaTh24}, the case \( V > 0 \) is also treated separately for simplicity. Negative coupling constants are handled via the observation (see~\cite[Lem.~7.5]{BaBeLo24})
\[
\sigma_\co(V) = -\sigma_\co(-V), \qquad V \in \R.
\]

Since we are dealing with one-dimensional operators acting on \( \ell^2(\Z) \), the spectrum \( \sigma_\co(V) \) can be described via trace properties of the associated transfer matrices; see e.g.~\cite{Simon11,DaFi22-book_1,DaFi24-book_2,BaBeBiRaTh24}. In the setting of the Kohmoto model, these traces satisfy recursive relations governed by the continued fraction expansion. These trace identities constitute a key tool in the spectral analysis of these operators. For instance, the following proposition is well known; see~\cite{BIST89,BIT91,Raym95}.

\begin{proposition}
\label{prop:MonotonConv_Spectrum_Kohmoto}
Let \( V \neq 0 \), \( \alpha \in [0,1] \setminus \Q \) with continued fraction expansion \( [0,0,c_1,c_2,\ldots] \) and the $k$th convergent \( \co_k = [0,0,c_1,\ldots,c_k] \in \Co \) for \( k \in \N_0 \). Then the compact sets
\[
\spec_{k,V} := \sigma_{\co_k}(V) \cup \sigma_{\co_{k+1}}(V) \subseteq \R
\]
satisfy \( \spec_{k+1,V} \subseteq \spec_{k,V} \), and
\[
\lim_{k \to \infty} \spec_{k,V} = \bigcap_{k \in \N} \spec_{k,V} = \sigma(\Koh),
\]
where the limit is taken in the Hausdorff metric on compact subsets of \( \R \).
\end{proposition}

\begin{proof}
This is proven in~\cite{BIST89,BIT91,Raym95}, see also~\cite[Cor.~4.5, Prop.~4.6]{BaBeBiRaTh24}. Once the monotonicity \( \spec_{k+1,V} \subseteq \spec_{k,V} \) is established, the remaining assertions follow directly from Corollary~\ref{cor:SpectralEstim_Sturm}.
\end{proof}

A more refined statement was shown in~\cite{Raym95}; see also~\cite[Lem.~4.3, Prop.~4.7]{BaBeBiRaTh24}. By convention, we set 
\[
\sigma_{[0]}(V) := \R 
\qquad \text{ for all } \qquad
V \in \R.
\]

\begin{proposition}
\label{prop:Dichotomy_Spectrum_Kohmoto_V>4}
Let \( V \neq 0 \), \( \co = [0,0,c_1,\ldots,c_k] \in \Co \) with \( \varphi(\co) \geq 0 \), and \( k \in \N_0 \). Then
\[
\sigma_\co(V) \subseteq \sigma_{[\co,0]}(V) \cup \sigma_{[\co,-1]}(V),
\]
and, if in addition \( V > 4 \), then
\[
\sigma_\co(V) \cap \sigma_{[\co,0]}(V) \cap \sigma_{[\co,-1]}(V) = \emptyset.
\]
\end{proposition}

We emphasize that the intersection of the three spectra -- the \emph{three-intersection property} -- in Proposition~\ref{prop:Dichotomy_Spectrum_Kohmoto_V>4} is only empty under the additional assumption \( V > 4 \). The following example demonstrates that this no longer holds for \( V \leq 4 \).

\begin{example}
\label{ex:ThreeIntersection}
Let \( \co = [0,0,1] \). Then
\[
\varphi(\co) = 1, \quad
\varphi([\co,0]) = \varphi([0,0]) = 0, \quad
\varphi([\co,-1]) = \varphi([0]) = \infty.
\]
The associated spectra can be computed explicitly (see, e.g., \cite[Ex.~2.9]{BaBeLo24} and \cite[Ex.~4.12]{BaBeBiRaTh24}) for \( V \in \R \):
\[
\sigma_\co(V) = [-2+V, 2+V], \quad
\sigma_{[\co,0]}(V) = [-2,2], \quad
\sigma_{[\co,-1]}(V) = \R.
\]
Thus,
\[
\sigma_\co(V) \cap \sigma_{[\co,0]}(V) \cap \sigma_{[\co,-1]}(V) = \{2\} \quad \text{if } V = 4,
\]
and
\[
\sigma_\co(V) \cap \sigma_{[\co,0]}(V) \cap \sigma_{[\co,-1]}(V) = [0,2] \quad \text{if } V = 2.
\]

Let \( \widetilde{\co} = [0,0,1,1] \). Then
\[
\varphi(\widetilde{\co}) = \tfrac{1}{2}, \quad
\varphi([\widetilde{\co},0]) = \varphi([0,0,1]) = 1, \quad
\varphi([\widetilde{\co},-1]) = \varphi([0,0]) = 0.
\]
The associated spectra (see, e.g.,~\cite[Ex.~2.3, 2.9]{BaBeLo24} and~\cite[Ex.~4.12]{BaBeBiRaTh24}) are
\[
\sigma_{\widetilde{\co}}(V) = [E_0(V), 0] \cup [V, E_2(V)], \quad
\sigma_{[\widetilde{\co},0]}(V) = [-2+V, 2+V], \quad
\sigma_{[\widetilde{\co},-1]}(V) = [-2, 2],
\]
with \( E_0(V) < 0 \) and \( E_2(V) > V \). Thus,
\[
\sigma_{\widetilde{\co}}(V) \cap \sigma_{[\widetilde{\co},0]}(V) \cap \sigma_{[\widetilde{\co},-1]}(V) = \{0,2\} \quad \text{if } V = 2.
\]
\end{example}

Returning to Proposition~\ref{prop:Dichotomy_Spectrum_Kohmoto_V>4}, we conclude that for \( V > 4 \), each spectral band \( I(V) \subseteq \sigma_\co(V) \) is strictly contained either in a spectral band of \( \sigma_{[\co,0]}(V) \) or of \( \sigma_{[\co,-1]}(V) \). This characterizes the spectral bands as one of two types. Here, an interval \( [a,b] \) is said to be \emph{strictly contained} in \( [c,d] \) if 
\[ 
[a,b] \strict [c,d] 
\quad:\Leftrightarrow\quad 
c < a \leq b < d. 
\] 
In~\cite{Raym95}, only non-strict inclusions of spectral bands were studied. The refinement to strict inclusions is essential when extending these ideas to the small coupling regime $0<V\leq 4$.

Recall from the definition of the evaluation map $\varphi$ that appending a ``\( 0 \)'' to a finite continued fraction \( \co \in \Co \) effectively deletes its last digit -- this operation characterizes type \( A \). In contrast, appending a ``\( -1 \)'' reduces the last digit by one; repeating this procedure until the last digit becomes zero corresponds to deleting the last two digits under the evaluation map -- this defines type \( B \).

This inclusion-based behavior provides a natural classification of spectral bands into two types, as illustrated in panel~(a) of Figure~\ref{fig:Type_Bands}. Each spectral band from a given approximation \( \alpha_k = \varphi(\co_k) \) of an irrational \( \alpha \) can thus be assigned a type based on its containment in the bands of previous approximations\footnote{In \cite{BaBeLo24,BaBeBiRaTh24}, we refer to this condition as backward type.}.

\begin{definition}[Type \( A \) and \( B \)]
\label{def:SpectralType}
Let \( V \neq 0 \), \( \alpha \in [0,1] \setminus \Q \) with continued fraction expansion \( [0,0,c_1,c_2,\ldots] \), the $k$th convergent \( \co_k = [0,0,c_1,\ldots,c_k] \in \Co \) and \( \alpha_k = \varphi(\co_k) \) for \( k \in \N \). A spectral band \( I_{\co_k}(V) \subseteq \sigma_{\co_k}(V) \) is called
\begin{itemize}
\item of \emph{type \( A \)} if \( I_{\co_k}(V) \strict J \) for some spectral band \( J \subseteq \sigma_{\co_{k-1}}(V) \),
\item of \emph{type \( B \)} if \( I_{\co_k}(V) \strict J \) for some spectral band \( J \subseteq \sigma_{\co_{k-2}}(V) \) and \( I_{\co_k}(V) \) is not contained in any spectral band of \( \sigma_{\co_{k-1}}(V) \).
\end{itemize}
\end{definition}

\begin{figure}[htb]
    \centering
    \includegraphics[scale=1.2]{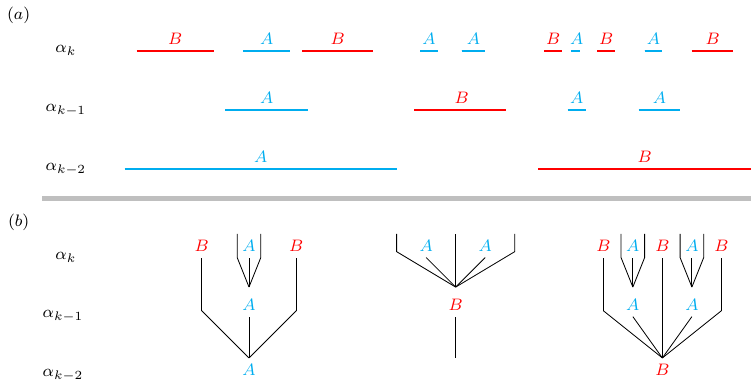}
    \captionsetup{width=0.95\linewidth}
    \caption{(a) Illustration of the classification of spectral bands into types~$A$ and~$B$ for $\alpha_k=\varphi(\co_k)$: For a given band in $\sigma_{\co_k}(V)$, inclusion into a band of $\sigma_{\co_{k-1}}(V)$ (type~$A$) or of $\sigma_{\co_{k-2}}(V)$ but not $\sigma_{\co_{k-1}}(V)$ (type~$B$) is visualized.\\
    (b) First idea towards constructing a spectral tree: Each spectral band from panel~(a) is represented as a vertex, and directed edges encode strict containment in bands from previous levels. The relative horizontal position of the bands is preserved to reflect their location on the real line. This graphical representation suggests how a tree structure may be built by organizing spectral bands according to their inclusion relations, see Section~\ref{sec:SpectralTree}.}
    \label{fig:Type_Bands}
\end{figure}

Note that the definition of spectral band types provided here differs from the formulation in~\cite{BaBeLo24,BaBeBiRaTh24}, but is equivalent; see~\cite[Prop.~7.19]{BaBeLo24}. Establishing this equivalence for arbitrary \( V \neq 0 \) is non-trivial and requires proving first that every spectral band possesses a well-defined type; see Theorem~\ref{thm:Types_A-B} below. For the purposes of presenting the spectral tree in the next section, the above definition suffices and is simpler to describe. A more detailed discussion is given in~\cite{BaBeLo24,BaBeBiRaTh24}.

Let \( \co \in \Co \) be such that \( \varphi(\co) = \tfrac{p}{q} \in [0,1] \), where \( p \) and \( q \) are coprime integers. Recall that for \( V \neq 0 \), the spectrum \( \sigma_\co(V) \) consists of exactly \( q \) disjoint intervals, each referred to as a \emph{spectral band} and denoted \( I_\co(V) \).

Following~\cite[Def.~1.5]{BaBeLo24}, we extend the notion of a spectral band by exploiting the continuity of the map \( V \mapsto \sigma_\co(V) \); see, e.g.,~\cite[Cor.~3.2]{BaBeLo24}. More precisely, for a fixed index \( 0 \leq j < q \), let \( I_\co(V) \) denote the \( j \)-th spectral band in \( \sigma_\co(V) \), counted from left to right. Then the map
\[
(0,\infty) \ni V \mapsto I_\co(V)
\]
is continuous in the Hausdorff metric, meaning the band edges vary continuously with \( V \). With a slight abuse of terminology, we will also refer to the function \( (0,\infty)\ni V \mapsto I_\co(V) \) as a \emph{spectral band}. An analogous statement holds for negative coupling constants \( V < 0 \).

Since the spectral bands vary continuously in \( V \), the strict inclusion relations that define their type cannot change as long as \( V > 4 \) using the three-intersection property in Proposition~\ref{prop:Dichotomy_Spectrum_Kohmoto_V>4}. Thus, the type of each spectral band remains invariant in this regime. However, for \( V \leq 4 \), overlaps between spectral bands at different approximation levels may occur, as illustrated in Example~\ref{ex:ThreeIntersection}. These overlaps can lead to changes or even loss of type, presenting a major conceptual difficulty.

A central achievement of~\cite{BaBeLo24} is to overcome this issue by proving that every spectral band has a well-defined and consistent type for all \( V \neq 0 \), as formulated in the following theorem.

\begin{theorem}
\label{thm:Types_A-B}
For all \( V \neq 0 \) and all \( \co \in \Co \) with \( \varphi(\co) \not\in \{-1,\infty\} \), each spectral band \( I_\co(V) \subseteq \sigma_\co(V) \) is of either type~\( A \) or type~\( B \), and its type remains invariant for all \( V > 0 \), respectively all \( V < 0 \).
\end{theorem}

\begin{proof}
For the case \( V > 4 \), the result was shown by Raymond~\cite{Raym95}; see also~\cite{BaBeBiRaTh24}. The general case \( V \neq 0 \) was proven in~\cite[Thm.~2.15, Cor.~7.7]{BaBeLo24}. Conceptually, the latter proof relies on the following key ingredients:
\begin{itemize}
\item the recursive structure of trace polynomials of the transfer matrices via Chebyshev polynomials~\cite[Sec.~4.6]{BaBeLo24},
\item interlacing properties of spectral band edges for admissible pairs of spectral band edges~\cite[Thm.~3.4]{BaBeLo24},
\item uniform (in \( \co \in \Co \)) Lipschitz continuity of band edges~\cite[Cor.~3.2]{BaBeLo24},
\item and a two-level induction over the space $\Co$ of finite continued fraction expansions~\cite[Sec.~5 and 6]{BaBeLo24}.\qedhere
\end{itemize}
\end{proof}

This classification has also been employed to estimate the fractal dimension of the spectrum \cite{KiKiLa03,LiuWen04,DaEmGoTc08,DamGor11,LiQuWe14,DamaGor15,CaoQu23,Lun25}. Therefore it is interesting in itself as mentioned in Section~\ref{sec:Kohmoto} before.

At this point, we emphasize again that the type of a spectral band \( I_\co(V) \subseteq \sigma_\co(V) \) depends on the specific continued fraction expansion \( \co \), and not merely on its evaluation \( \varphi(\co) \); see~\cite[Lem.~2.10]{BaBeLo24}. However, the rational approximations $\alpha_k$ of an irrational $\alpha\in[0,1]$ have a unique finite continued fraction expansion given by the $k$th convergent $\co_k$.

Following Definition~\ref{def:SpectralType}, the type of a spectral band \( I_{\co_k}(V) \) is determined by whether it is strictly contained in a band at level \( k-1 \) or level \( k-2 \). Moreover, knowing the next digit \( c_{k+1} \) of the continued fraction expansion allows us to determine how many spectral bands of type~\( A \) and type~\( B \) appear inside \( I_{\co_k}(V) \) at the following two levels. This condition is referred to as forward type in \cite{BaBeLo24,BaBeBiRaTh24}. These bands are interlaced in a precise way, for which we define the order relation\footnote{We use the convention $a\leq b$ and $c\leq d$.}
\[
[a,b] \prec [c,d] \qquad \Longleftrightarrow \qquad a < c \quad \text{and} \quad b < d.
\]

\begin{figure}[htb]
    \centering
    \includegraphics[scale=0.82]{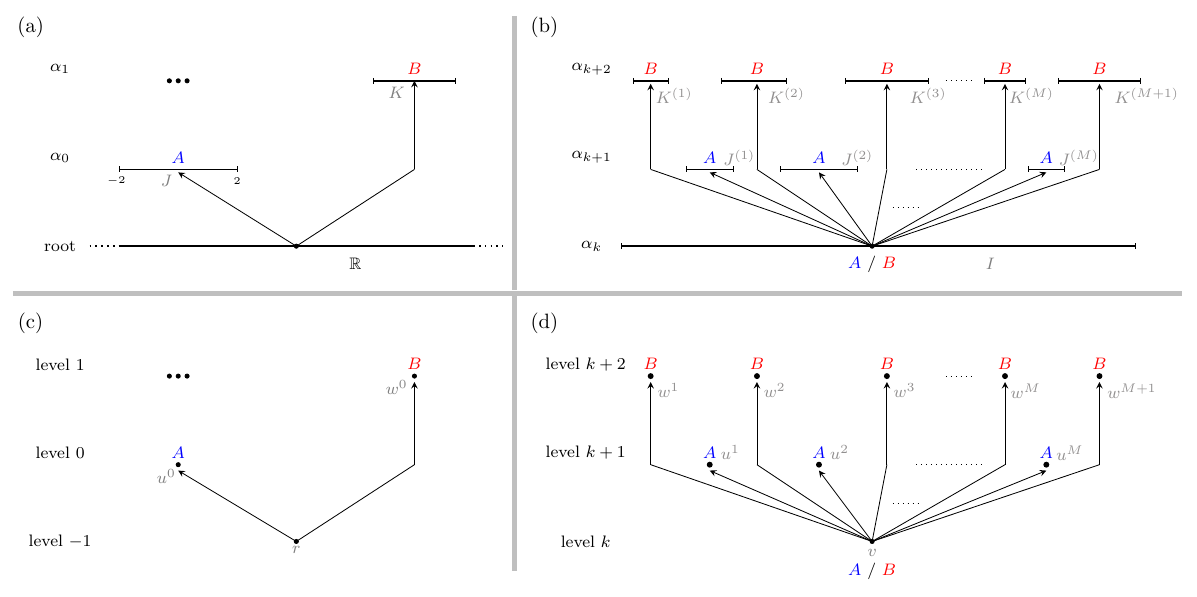}
    \captionsetup{width=0.95\linewidth}
    \caption{Panel~(a) illustrates the local structure of the first spectral bands for $\alpha_0$ and $\alpha_1$ where $J \prec K$ for $V>0$. The dots serve as placeholders for potential additional spectral bands strictly contained in $J$ for $\alpha_1$, as predicted in Proposition~\ref{prop:Forwardtype}, depending on the value of $c_1 \in \N$.
Panel~(b) visualizes the forward type structure described in Proposition~\ref{prop:Forwardtype}.
Panels~(c) and (d) show how the structure depicted in panels~(a) and (b) is encoded in a tree, which is formally defined in Definition~\ref{def:SpectralTree}. In this representation, $J$ corresponds to $u^0$, $K$ to $w^0$, $I$ to $v$, $J^{(i)}$ to $u^i$, and $K^{(j)}$ to $w^j$.
    }
    \label{fig:TreeRoot}
\end{figure}

\begin{proposition}[Forward type structure]
\label{prop:Forwardtype}
Let \( V \neq 0 \), and let \( \alpha \in [0,1] \setminus \Q \) with continued fraction expansion \( [0,0,c_1,c_2,\ldots] \). Consider the $k$th convergent \( \co_k := [0,0,c_1,\ldots,c_k] \in \Co \) with evaluation let \( \alpha_k := \varphi(\co_k) \in[0,1]\cap\Q \). 

Let \( I_{\co_k}(V) \subseteq \sigma_{\co_{k}}(V) \) be a spectral band, and define
\[
M :=
	\begin{cases}
		c_{k+1}-1, & \text{if } I_{\co_k}(V) \text{ is of type } A,\\
		c_{k+1}, & \text{if } I_{\co_k}(V) \text{ is of type } B.\\
	\end{cases}
\]
Then there exist exactly
\begin{itemize}
	\item \( M \) spectral bands \( J^{(1)}(V),\ldots, J^{(M)}(V) \subseteq \sigma_{\co_{k+1}}(V) \) of type~\( A \), and
	\item \( M+1 \) spectral bands \( K^{(1)}(V),\ldots, K^{(M+1)}(V) \subseteq \sigma_{\co_{k+2}}(V) \) of type~\( B \),
\end{itemize}
such that all these bands are strictly contained in \( I_{\co_k}(V) \) and interlace as
\[
K^{(j)}(V) \prec J^{(j)}(V) \prec K^{(j+1)}(V), \qquad 1 \leq j \leq M.
\]
\end{proposition}

\begin{proof}
For \( V > 4 \), the result was shown in~\cite[Lem.~3.3]{Raym95}; see also~\cite[Lem.~4.14]{BaBeBiRaTh24}. The general case \( V \neq 0 \) was established in~\cite[Thm.~2.15, Cor.~7.7]{BaBeLo24} using the equivalent notion of types defined in~\cite[Prop.~7.19]{BaBeLo24}.
Note that the formal definition of forward type in \cite{BaBeLo24,BaBeBiRaTh24} does not require the uniqueness of the \(M\) spectral bands of type~\(A\) and the \(M+1\) spectral bands of type~\(B\). The uniqueness stated in this proposition follows instead from a counting argument based on the total number of spectral bands.
\end{proof}

An illustration of this forward type structure is provided in panel~(b) of Figure~\ref{fig:TreeRoot}. As sketched there, neighboring bands \( K^{(j)}(V) \) and \( J^{(j)}(V) \) may overlap for small coupling $0<V\leq 4$. If \( V > 4 \), such overlaps do not occur by the three-intersection property (Proposition~\ref{prop:Dichotomy_Spectrum_Kohmoto_V>4}), and a stricter interlacing property holds.

Proposition~\ref{prop:Forwardtype} shows that the continued fraction expansion of \( \alpha \) uniquely determines how many spectral bands of type~\( A \) and type~\( B \) are found in each spectral band at level \( k \), together with their relative position. This provides a combinatorial encoding of the spectral structure, which will serve as the foundation for the tree construction discussed below; see also Figure~\ref{fig:Type_Bands}.

\section[The spectral tree]{The spectrum as the boundary of the spectral tree}
\label{sec:SpectralTree}

The forward classification of spectral bands described in Proposition~\ref{prop:Forwardtype} naturally leads to a hierarchical and recursive organization of spectral data. This structure can be encoded in a rooted tree, referred to as the \emph{spectral $\alpha$-tree} associated with a fixed irrational number \( \alpha \in [0,1] \setminus \Q \). Each level \( k \) in this tree corresponds to the rational approximation \( \alpha_k = \varphi(\co_k) \), and its vertices represent the spectral bands in \( \sigma_{\co_k}(V) \), as determined by Theorem~\ref{thm:Types_A-B} and Proposition~\ref{prop:Forwardtype}.

Figure~\ref{fig:Type_Bands} and Figure~\ref{fig:TreeRoot} illustrate the core idea behind the construction: each spectral band at level \( k \) gives rise to a fixed number of type \( A \) and type \( B \) spectral bands at the next levels, encoded via the branching of the tree. This spectral $\alpha$-tree retains all combinatorial information relevant to the associated integrated density of state; see Proposition~\ref{prop:IDS_Kohmoto}. At the same time, it omits non-essential details such as the precise positions of band edges.

Importantly, the spectral $\alpha$-tree is entirely independent of the coupling constant \( V > 0 \); its structure is governed solely by the continued fraction expansion of \( \alpha \).

We begin by recalling the formal notion of a directed graph, as previously used in Section~\ref{sec:Subst-Approximations}.

A tuple \( G = (\Vv, \Ee) \) is called a \emph{directed graph} if \( \Vv \) is a countable set (the \emph{vertex set}) and \( \Ee \subseteq \Vv \times \Vv \) is a set of directed edges. An element \( (u, w) \in \Ee \) is a directed edge from \( u \) to \( w \), with \( (u, w) \neq (w, u) \) in general.

A \emph{(directed) path} in \( G \) is a finite or infinite sequence of vertices \( (u_0, u_1, u_2, \ldots) \) such that \( (u_i, u_{i+1}) \in \Ee \) for all \( i \in \N_0 \).  
For two vertices \( u, w \in \Vv \), we write \( u \to w \) if there exists a finite path \( (u, u_1, \ldots, u_m, w) \)  from \( u \) to \( w \).  
A path \( (u_0, u_1, \ldots, u_m) \) is called a \emph{cycle} if \( u_0 = u_m \) and \( m \in \N \).  
A directed graph is called a \emph{tree} if it contains no cycles.  
It is called a \emph{rooted tree} if there exists a unique vertex \( r \in \Vv \) (the \emph{root}) such that for all \( v \in \Vv \), there exists a path \( r \to v \).  
A rooted tree is called an \emph{ordered tree} if the vertex set \( \Vv \) carries a strict (i.e., irreflexive) partial order relation \( \prec \).

Motivated by Proposition~\ref{prop:Forwardtype}, we now define the \emph{spectral $\alpha$-tree}, which organizes the spectral data in a recursive and hierarchical manner; see the illustration in Figure~\ref{fig:TreeRoot}.

\begin{definition}[Spectral \( \alpha \)-tree]
\label{def:SpectralTree}
Let \( \alpha \in [0,1] \setminus \Q \) with continued fraction expansion \( [0,0,c_1,c_2,\ldots] \).  
The \emph{spectral \( \alpha \)-tree} \( \Tt_\alpha = (\Vv_\alpha, \Ee_\alpha) \) is defined recursively as follows:

\begin{itemize}
\item Let \( r \in \Vv_\alpha \) be the root at level \( k = -1 \).
\item Add a vertex \( u^0 \in \Vv_\alpha \) at level \( k = 0 \) with label \( A \), and an edge \( (r, u^0) \in \Ee_\alpha \).
\item Add a vertex \( w^0 \in \Vv_\alpha \) at level \( k = 1 \) with label \( B \), and an edge \( (r, w^0) \in \Ee_\alpha \).
\item Declare the initial order relation \( u^0 \prec w^0 \).
\end{itemize}

Now, recursively extend the tree:  
For each vertex \( v \in \Vv_\alpha \) at level \( k \geq 0 \), set
\[
M := 
\begin{cases}
c_{k+1} - 1, & \text{if } v \text{ has label } A, \\
c_{k+1},     & \text{if } v \text{ has label } B.
\end{cases}
\]
Then:
\begin{itemize}
\item Add $M$ vertices \( u^1, \ldots, u^M \in \Vv_\alpha \) at level \( k+1 \) with label \( A \), and connect each via \( (v, u^i) \in \Ee_\alpha \) for \( 1 \leq i \leq M \).
\item Add $M+1$ vertices \( w^1, \ldots, w^{M+1} \in \Vv_\alpha \) at level \( k+2 \) with label \( B \), and connect each via \( (v, w^j) \in \Ee_\alpha \) for \( 1 \leq j \leq M+1 \).
\item Define the strict order relation on these new vertices as\footnote{In particular, we have $u^i\prec u^{i+1}$ and $w^j\prec w^{j+1}$ by transitivity.}
\[
w^1 \prec u^1 \prec w^2 \prec \ldots \prec u^M \prec w^{M+1}.
\]
\end{itemize}

Extend the strict partial order \( \prec \) on \( \Vv_\alpha \) by declaring \( w_1 \prec w_2 \) for \( w_1, w_2 \in \Vv_\alpha \) if there exist \( u_1, u_2 \in \Vv_\alpha \) such that
\[
u_1 \prec u_2,
\quad \text{and} \quad 
(u_1 \to w_1 \text{ or } u_1 = w_1),
\quad \text{and} \quad 
(u_2 \to w_2 \text{ or } u_2 = w_2).
\]
The \emph{level function} $\ell_\alpha:\Vv_\alpha\to\N_{-1}$ assigns to each vertex $v\in \Vv_\alpha$ its level used in the construction.
\end{definition}

We emphasize here that the level of a vertex in the spectral \( \alpha \)-tree should not be confused with its combinatorial distance from the root \( r \). Rather, the level refers to the index \( k \) of the corresponding rational approximation \( \alpha_k \). For convenience, an example of a spectral \( \alpha \)-tree \( \Tt_\alpha \) is plotted in Figure~\ref{fig:SpectralTree}. Note also that the spectral \( \alpha \)-tree up to level \( k \) is uniquely determined by the initial segment \( [0,0,c_1,\ldots,c_k] \) of the continued fraction expansion.

\begin{figure}[htb]
    \centering
    \includegraphics[scale=1.05]{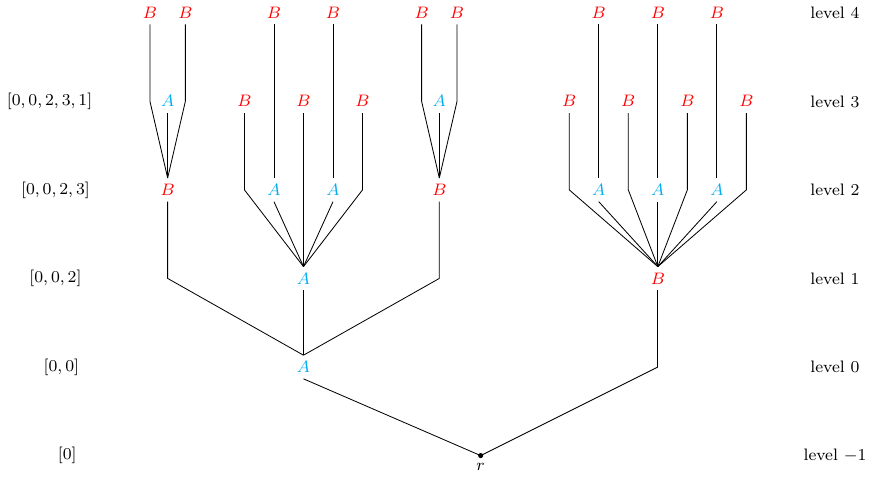}
    \captionsetup{width=0.95\linewidth}
    \caption{Initial part of the spectral \( \alpha \)-tree for a continued fraction expansion beginning with $[0,0,2,3,1]$.
    }
    \label{fig:SpectralTree}
\end{figure}

Let \( \alpha \in [0,1] \setminus \Q \) have continued fraction expansion \( [0,0,c_1,c_2,\ldots] \). Recall that we define the $k$th convergent \( \co_k := [0,0,c_1,\ldots,c_k] \) and denote by \( \alpha_k := \varphi(\co_k) \in[0,1]\cap\Q \) its evaluation. In order to consistently identify vertices of the spectral \( \alpha \)-tree with spectral bands, we extend the indexing to \( k = -1 \) by setting 
\[ 
	\alpha_{-1} := \varphi([0]) 
	\qquad \text{and}\qquad 
	\sigma_{[0]}(V) := \R  \quad \text{for all } V \in \R.
\] 
We also call the connected component of \( \sigma_{[0]}(V) \) a spectral band. This spectral set corresponds to the root $r$ of the tree, see Figure~\ref{fig:TreeRoot}~(a) and (c).

As discussed in Section~\ref{sec:Types_Spectral_Bands}, spectral bands can either be viewed as fixed intervals \( I_\co(V) \subseteq \sigma_\co(V) \), or, more generally, as continuous maps \( (0,\infty) \ni V \mapsto I_\co(V) \), where each \( I_\co(V) \) remains a connected component of \( \sigma_\co(V) \) and retains a fixed type \( A \) or \( B \) for all \( V > 0 \); see Theorem~\ref{thm:Types_A-B}.

We now formalize the correspondence between the spectral \( \alpha \)-tree and the set of spectral bands of rational approximations. 

\begin{proposition}[\cite{BaBeLo24}, the spectral $\alpha$-tree]
\label{prop:BijectionSpectralTreeBands}
Let \( \alpha \in [0,1] \setminus \Q \) with continued fraction expansion \( [0,0,c_1,c_2,\ldots] \). Then there exists a unique bijection
\[
\Psi : \Vv_\alpha \to \left\{ \text{spectral bands of } \sigma_{\co_k}(V),~ k \in \N_{-1} \right\}
\]
(viewed as continuous maps on \( (0,\infty) \)) satisfying the following properties:
\begin{enumerate}[label=(\alph*)]
\item For each \( k \in \N_{-1} \), the restriction of \( \Psi \) to level \( k \) gives a bijection between the vertices in level \( k \) and the spectral bands in \( \sigma_{\co_k} \).
\item If \( u, w \in \Vv_\alpha \) satisfy \( u \to w \), then \( \Psi(w)(V) \strict \Psi(u)(V) \) for all \( V > 0 \).
\item If \( u_1, u_2 \in \Vv_\alpha \) are vertices in levels \( k_1, k_2 \in \N_{-1} \) with \( |k_1 - k_2| \leq 1 \), then
\[
u_1 \prec u_2 \quad \Longleftrightarrow \quad \Psi(u_1)(V) \prec \Psi(u_2)(V) \quad \text{for all } V > 0.
\]
\item A vertex \( u \in \Vv_\alpha \setminus \{r\} \) is labeled \( A \) (respectively \( B \)) if and only if \( \Psi(u)(V) \) is of type \( A \) (respectively \( B \)) for all \( V > 0 \).
\end{enumerate}
\end{proposition}

\begin{proof}
This is proven in \cite[Prop.~7.1]{BaBeLo24}.
\end{proof}

Combining Proposition~\ref{prop:BijectionSpectralTreeBands} with the monotone convergence of the spectrum from Proposition~\ref{prop:MonotonConv_Spectrum_Kohmoto}, we obtain that the boundary of the spectral $\alpha$-tree encodes the spectrum $\sigma(\Koh)$ for $V>0$:

\begin{figure}[htb]
    \centering
    \includegraphics[scale=1.1]{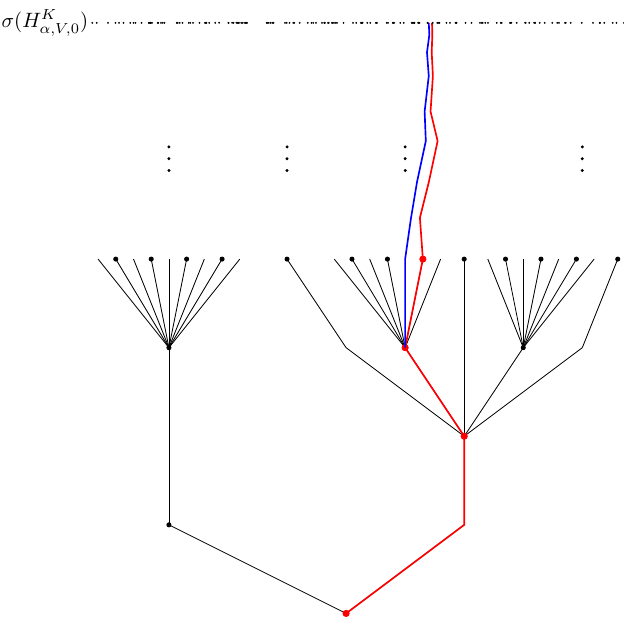}
    \captionsetup{width=0.95\linewidth}
    \caption{
    	A sketch of a spectral $\alpha$-tree. The two infinite paths $\gamma\in\partial\Tt_\alpha$ (blue) and $\eta\in\partial\Tt_\alpha$ (red) satisfy $\gamma\prec \eta$ and \( N_\alpha(\gamma)=N_\alpha(\eta) = \{n\alpha\} \).
    }
    \label{fig:TreeSpectrum_OhneLabels}
\end{figure}

The boundary of the tree $\Tt_\alpha$ is defined as
\[
\partial\Tt_\alpha := 
	\set{\gamma=(u_0,u_1,u_2,\ldots)}{u_0=r, \ (u_m,u_{m+1})\in\Ee_\alpha \text{ for all } m\in\N_0},
\]
that is, the set of all infinite directed paths starting at the root $r$. It inherits a natural total order from the partial order $\prec$ on $\Vv_\alpha$: given two elements $\gamma=(u_0,u_1,u_2,\ldots)$ and $\eta=(w_0,w_1,w_2,\ldots)$ in $\partial\Tt_\alpha$, we define (see Figure~\ref{fig:TreeSpectrum_OhneLabels} for two paths $\gamma\preceq \eta$)
\[
\gamma \preceq \eta 
	\quad:\Longleftrightarrow\quad
	\gamma=\eta \quad\text{ or }\quad \exists k_0\in\N_0 \;\;\text{ such that }\;\; u_{k_0}=w_{k_0} \text{ and } u_{k_0+1}\prec w_{k_0+1}.
\]
By construction, such an index $k_0$ is uniquely determined, and the ordering persists for all subsequent vertices, namely $u_k\prec w_k$ for all $k>k_0$.

Let $V>0$. Following \cite{BaBeLo24}, we associate each path $\gamma=(u_0,u_1,u_2,\ldots)$ with a point in the spectrum via the nested intervals $\Psi(u_m)(V)$, where $\Psi$ is the bijection from Proposition~\ref{prop:BijectionSpectralTreeBands}. Since $\Psi(u_{m+1})(V) \strict \Psi(u_m)(V)$ for all $m\in\N_0$, the intersection
\[
I_\gamma(V) := \bigcap_{m\in\N} \Psi(u_m)(V)
\]
is a non-empty compact interval. Proposition~\ref{prop:MonotonConv_Spectrum_Kohmoto} ensures that $I_\gamma(V) \subseteq \sigma(\Koh)$. As $\sigma(\Koh)$ is a Cantor set of Lebesgue measure zero~\cite{BIST89}, $I_\gamma(V)$ must be a singleton. We denote its unique element by $E_\alpha(\gamma;V)$, which defines a map
\begin{equation}
\label{eq:SpectralEncodingMap}
E_\alpha:\partial\Tt_\alpha \times(0,\infty) \to \sigma(\Koh), \qquad (\gamma,V) \mapsto E_\alpha(\gamma;V).
\end{equation}
This map encodes our spectrum as the boundary of a tree satisfying the following properties:

\begin{theorem}[\cite{BaBeLo24}, Encoding of the spectrum]
\label{thm:EncodingSpectrum}
Let $\alpha \in [0,1] \setminus \Q$. Then the map
\[
E_\alpha : \partial \Tt_\alpha \times(0,\infty)  \to \sigma(\Koh), \qquad (\gamma,V) \mapsto E_\alpha(\gamma;V),
\]
is
\begin{enumerate}[label=(\alph*)]
\item a bijection for each fixed $V > 0$,
\item order-preserving, i.e., $\gamma \preceq \eta$ implies $E_\alpha(\gamma;V) \leq E_\alpha(\eta;V)$ for all $V > 0$,
\item Lipschitz continuous in $V$.
\end{enumerate}
\end{theorem}

\begin{proof}
This result is proven in \cite[Thm.~1.9]{BaBeLo24}. Properties (b) and (c) follow directly from the nested structure of spectral bands established in Proposition~\ref{prop:BijectionSpectralTreeBands} and the monotone convergence from Proposition~\ref{prop:MonotonConv_Spectrum_Kohmoto}. The surjectivity in (a) is a direct consequence of the convergence of the spectrum.

The central difficulty lies in proving injectivity of the map. This is resolved by a refined analysis of how spectral bands intersect in the critical regime $V \leq 4$, building on the classification of band types from Theorem~\ref{thm:Types_A-B} and an extended version of the three-intersection property in Proposition~\ref{prop:Dichotomy_Spectrum_Kohmoto_V>4}. For details, we refer to the comprehensive argumentation in \cite[Thm.~1.9]{BaBeLo24}.
\end{proof}

\begin{remark}
\label{rem:DTMP->TMP}
In \cite{BIST89}, the authors prove that \(\sigma(\Koh)\) is a Cantor set of Lebesgue measure zero for irrational \(\alpha\), thereby solving the ten Martini problem (TMP) for Sturmian Hamiltonians. We use this result only once in \cite{BaBeLo24}, as mentioned above -- specifically, to show that \(I_\gamma(V)\) is a singleton. However, the application of the TMP for Sturmian Hamiltonians can be avoided at this point, as follows:

By construction, \(\Psi(u_m)(V)\) is a spectral band of a rational approximation \(\tfrac{p_m}{q_m}\) of \(\alpha\), where \(p_m\) and \(q_m\) are coprime. The length of such a spectral band can be estimated by \(\tfrac{2\pi}{q_m}\); see, e.g., \cite[Thm~7.5.1]{DaFi24-book_2}. Since \(\alpha\) is irrational, the denominators \(q_m\) tend to infinity as \(m \to \infty\). Consequently, the length of the band \(\Psi(u_m)(V)\) tends to zero, which implies that \(I_\gamma(V)\) consists of a single point.

This argument shows that one can establish the solution of the DTMP for Sturmian Hamiltonians without relying on the fact that the spectrum is a Cantor set of Lebesgue measure zero. 
Since the gap labels (for irrational \(\alpha\)) are dense in \([0,1]\), the DTMP implies the TMP by standard arguments, see, e.g., \cite[Rem~5.2]{ChElYu90}. Specifically, Theorem~\ref{thm:DTMP} implies that the spectrum \(\sigma(\Koh)\) is a Cantor set for all $\alpha\in[0,1]\setminus\Q$. Thus, \cite{BaBeLo24} also provides an alternative proof of this result.
\end{remark}

\medskip

\noindent\textbf{Comparison to existing descriptions of the spectrum.} Theorem~\ref{thm:EncodingSpectrum} shows that the spectral tree $\Tt_\alpha$ fully encodes the spectrum of the operator $\Koh$ in terms of its boundary. Various strategies have been employed in the literature (see \cite[Thm.~10.5.2]{DaFi24-book_2}) to characterize the spectrum $\sigma(\Koh)$ of the Sturmian Hamiltonians:
\begin{itemize}
\item The set $\mathcal{Z}$ of energies where the Lyapunov exponent vanishes plays a key role in proving that $\sigma(\Koh)$ is a Cantor set of zero Lebesgue measure. This is based on Kotani theory~\cite{Sim83,Kot84,Kotani89} and was developed in works such as \cite{Sut89,BIST89,Len02,DaLe06_Boshernitzan,DaLe06_ZeroMeasure}, see also a discussion in \cite[Sec.~4.8]{DaFi22-book_1}.

\item The set $\mathcal{B}$ of energies for which the positive orbit under the associated trace map remains bounded has been used to analyze the fractal geometry of the spectrum, including its Hausdorff dimension; see \cite{Cas86,Raym95,DaEmGoTc08,DaGoYe16}.

\item A symbolic \emph{coding scheme} $\Pi$, originally inspired by the trace map analysis of Casdagli~\cite{Cas86} and further developed by Raymond~\cite{Raym95} for $V>4$ was used to prove that all spectral gaps are there if $V>4$. This coding also provides insight into the fractal geometry of the spectrum and related quantities such as transport exponents \cite{KiKiLa03,LiuWen04,DaEmGoTc08,DamGor11,LiQuWe14,DamaGor15,CaoQu23,Lun25}. Coding methods have also been applied to other substitution models, e.g., the period-doubling sequence in \cite{LiQuYa22}.
\end{itemize}

In \cite{BaBeLo24}, the perspective was changed replacing the symbolic coding $\Pi$ with the boundary of a spectral tree $\partial\Tt_\alpha$; see also a discussion in \cite{BaBeBiRaTh24}.

\section{Solving the dry ten Martini problem for Sturmian Hamiltonians}
\label{sec:DTMP_Sturmian}

Building on the hierarchical structure of spectral data encoded in the spectral $\alpha$-tree and the bijective correspondence with the spectrum $\sigma(\Koh)$, we now turn to the resolution of the dry ten Martini problem for the Sturmian Hamiltonians. In particular, Theorem~\ref{thm:DTMP} asserts for all \( \alpha \in [0,1] \setminus \Q \) and all \( V \neq 0 \), that
\[
\set{N_{\alpha,V}(E)}{E \in \R \setminus \sigma(\Koh)} = \set{ \{n\alpha\} }{n \in \Z} \cup \{1\}.
\]
Recall that \( \{n\alpha\}=n\alpha-\lfloor n\alpha\rfloor \) is the fractional part of $n\alpha$.
Thanks to the Gap Labeling Theorem, it suffices to show that for each $n\in\Z$, there is an $E \in \R \setminus \sigma(\Koh)$ such that $N_{\alpha,V}(E)=\{n\alpha\}$.

Representing \( \sigma(\Koh) \) as the boundary of a tree provides a natural and combinatorial framework for describing the integrated density of states \( N_{\alpha,V}:\R \to [0,1] \) on the spectrum. This approach makes the structure explicit while suppressing the dependence on the coupling parameter \( V \).

Let \( \alpha \in [0,1] \setminus \Q \) be fixed, with rational approximants \( \alpha_k = \frac{p_k}{q_k} \) obtained by truncating the continued fraction expansion of \( \alpha \) after the \( k \)-th digit, where \( p_k \) and \( q_k \) are coprime integers. Recall from Definition~\ref{def:SpectralTree} that the level function \( \ell_\alpha: \Vv_\alpha \to \N_{-1} \)  (Definition~\ref{def:SpectralTree}) assigns to each vertex its level in the spectral \( \alpha \)-tree.

To quantify the relative position of a given path \( \gamma \in \partial\Tt_\alpha \), we introduce the \emph{relative \( A \)-index}
\[
\ind(\gamma,j) := \sharp \set{u\in\Vv_\alpha}{\ell_\alpha(u)=\ell_\alpha\big(\gamma(j)\big) +1, \, \big(\gamma(j),u\big)\in\Ee_\alpha,\, u\prec \gamma(j+1)}.
\]
This counts how many vertices in the next level -- each\footnote{This follows from $\ell_\alpha(u)=\ell_\alpha\big(\gamma(j)\big) +1$ and $\big(\gamma(j),u\big)\in\Ee_\alpha$.} with label \( A \) -- are connected to \( \gamma(j) \) and lie to the left of \( \gamma(j+1) \) in the partial order \( \prec \).
Due to the interlacing in the forward property, this also approximates (up to an error of one) the number of vertices \( w \) with label \( B \) satisfying \( ( \gamma(j), w ) \in \Ee_\alpha \) and \( w \prec \gamma(j+1) \).

Additionally, define the indicator function
\[
\delta_A(\gamma,j) :=
\begin{cases}
1, & \text{if } \gamma(j) \text{ is of type } A, \\
0, & \text{if } \gamma(j) \text{ is of type } B.
\end{cases}
\]

With these definitions, we define a map \( N_\alpha: \partial\Tt_\alpha \to [0,1] \) by\footnote{Note that in \cite[Prop.~5.21]{BaBeBiRaTh24}, the levels do not explicitly appear in the summation. However, all coefficients $\mu_j$ vanish whenever the path $\gamma$ contains no vertex in level $j$.
}
\begin{equation}
\label{eq:IDS-formula}
N_\alpha(\gamma) := -\alpha + \sum_{j \in \N_-1} 
(-1)^{\ell_\alpha(\gamma(j))} 
\big( \ind(\gamma,j) + \delta_A(\gamma,j) \big)
\big( q_{\ell_\alpha(\gamma(j))} \alpha - p_{\ell_\alpha(\gamma(j))} \big).
\end{equation}
Intuitively, each summand in the series captures the normalized number of vertices at the corresponding level that lie to the left of the path \( \gamma \); a detailed explanation is provided in \cite[Sec.~5.3]{BaBeBiRaTh24}.
Via the bijection \( \Psi \) from Proposition~\ref{prop:BijectionSpectralTreeBands}, this corresponds to counting the number of spectral bands to the left of \( \Psi(\gamma(j)) \). In turn, this is closely related to counting the eigenvalues of the finite matrix \( \Koh|_{[0,q_{\ell_\alpha(\gamma(j))}-1]} \) that lie below \( E_\alpha(\gamma;V) \), thus connecting the construction directly to the integrated density of states. A rigorous justification of this correspondence is provided in \cite[Lem.~7.5]{BaBeLo24}.

Summing up, the function \( N_\alpha \) gives an explicit and purely combinatorial (independent of $V>0$) description of the integrated density of states on \( \sigma(\Koh) \), parametrized by the boundary \( \partial\Tt_\alpha \) of the spectral tree:

\begin{proposition}
\label{prop:IDS_Kohmoto}
Let \( \alpha \in [0,1] \setminus \Q \) and let \( E_\alpha : \partial \Tt_\alpha\times(0,\infty) \to \sigma(\Koh) \) be the map defined in \eqref{eq:SpectralEncodingMap} encoding the spectrum \( \sigma(\Koh) \). Then
\[
N_{\alpha,V}\big(E_\alpha(\gamma;V)\big) = N_\alpha(\gamma)
\]
holds for all \( V > 0 \).
\end{proposition}

\begin{proof}
For \( V > 4 \), this formula was first described by Raymond~\cite{Raym95} using the coding scheme. An adapted version based on the spectral \( \alpha \)-tree is provided in~\cite[Sec.~5.3]{BaBeBiRaTh24}. We show in~\cite[Thm.~1.9~(d)]{BaBeLo24} that the identity remains valid for all \( 0 < V \leq 4 \).
\end{proof}

Note that the integrated density of states \( N_{\alpha,V}:\R \to [0,1] \) for \( V < 0 \) can be expressed in terms of the integrated density of states for positive coupling $V>0$; see~\cite[Thm.~1.9~(e)]{BaBeLo24}. This allows the extension of the previous results to negative coupling constants as well. Alternatively, one may define a spectral $\alpha$-tree directly for \( V \in (-\infty,0) \), requiring only a minor adjustment in Definition~\ref{def:SpectralTree}, namely replacing the ordering \( u^0 \prec w^0 \) with \( w^0 \prec u^0 \).

With all preparations in place, Theorem~\ref{thm:DTMP} now follows from the preceding considerations; see~\cite[Thm.~1.1]{BaBeLo24}. We sketch the main steps of the argument for fixed \( V > 0 \) and the reader is referred also to Figure~\ref{fig:IDS_gap_label} and Figure~\ref{fig:TreeSpectrum_OhneLabels}

\begin{enumerate}
\item Fix an integer \( n \in \Z \). Using Proposition~\ref{prop:IDS_Kohmoto}, one explicitly constructs a path \( \eta = (u_0, u_1, u_2, \ldots) \in \partial\Tt_\alpha \) such that
\[
N_\alpha(\eta) = \{n\alpha\} = n\alpha - \lfloor n\alpha \rfloor.
\]
This path is constructed so that the associated summands (see eq.~\eqref{eq:IDS-formula}) in \( N_\alpha \) vanish beyond some index \( k_0 \in \N_0 \).

\item The path \( \eta \) is then modified to a second path \( \gamma \in \partial\Tt_\alpha \) such that \( N_\alpha(\gamma) = N_\alpha(\eta) \), \( \gamma \preceq \eta \) and  \( \gamma \neq \eta \). A detailed construction of \( \eta \) is provided in~\cite{BaBeLo24,BaBeBiRaTh24}. 

\item The encoding map \( E_\alpha \) is order preserving and injective by Theorem~\ref{thm:EncodingSpectrum}, and in combination with Proposition~\ref{prop:IDS_Kohmoto}, we obtain
\[
E_\alpha(\gamma;V) < E_\alpha(\eta;V)
\quad\text{and}\quad
N_{\alpha,V}(E_\alpha(\gamma;V)) = \{n\alpha\} = N_{\alpha,V}(E_\alpha(\eta;V)).
\]
Hence, property~\eqref{item:IDS-1} implies that the integrated density of states remains constant on the interval \( (E_\alpha(\gamma;V), E_\alpha(\eta;V)) \). By property~\eqref{item:IDS-2}, this interval lies entirely in the resolvent set and thus forms a spectral gap. The value \( \{n\alpha\} \) is then the corresponding gap label.
\end{enumerate}

Let us conclude this chapter with a brief discussion on how to label or color the Kohmoto butterfly, now that Theorem~\ref{thm:DTMP} ensures that all spectral gaps are open.

Let $\alpha \in [0,1]$ be irrational. Then, by Theorem~\ref{thm:DTMP}, for each fractional part $\{n\alpha\}$, there exists a spectral gap $\gap$ of $\Koh$ such that
\[
N_{\alpha,V}(E) = \{n\alpha\} = k\alpha + m, \qquad E \in \gap,
\]
where $k, m \in \Z$ are uniquely determined due to the irrationality of $\alpha$. In the following, we refer to the integer $k$ as the \emph{gap index} of $\gap$.

In the context of the almost Mathieu operator, this integer $k$ corresponds to the Hall conductance~\cite{OsaAvr01}. In that work, all spectral gaps are labeled by identifying each gap index $k$ with a distinct color. We refer the reader to~\cite{OsaAvr01} for further background and references.

Let us now return to the Kohmoto butterfly. At first glance, the butterfly seems to consist of connected spectral gaps (white regions in Figure~\ref{fig:Kohmoto-butterfly}). However, as emphasized in~\cite{BecBelTho25} and discussed in Section~\ref{sec:Kohmoto}, the Euclidean topology on $[0,1]$ is not adequate for analyzing the Kohmoto model. Instead, one should consider the finer Farey topology in which $[0,1]$ is cut at every rational value disconnecting the spectral gaps.

As previously mentioned, the butterfly is generated numerically by computing the spectra for rational values of $\alpha$. Consequently, only these spectral gaps can be colored. Let $\alpha = \tfrac{p}{q} \in [0,1]$ with $p$ and $q$ coprime. Then
\[
\set{N_{\alpha,V}(E)}{E \in \R \setminus \sigma(\Koh)} = \left\{ \frac{j}{q} \ \middle| \ 0 \leq j \leq q-1 \right\}.
\]
Thus, to determine the gap index, we solve
\[
\frac{j}{q} = k\frac{p}{q} + m
\qquad \Longleftrightarrow \qquad
j = kp + mq.
\]
However, the integer \(k \in \Z\) is not uniquely determined in this case. In fact it is defined modulo~$q$.

In the joint project~\cite{BanBec25}, we propose for assigning the gap index \( k \) within the interval \( \big[ -\tfrac{q}{2}, \tfrac{q}{2} \big) \cap \Z \) (left-closed, right-open). For a large class of spectral gaps, this choice can be justified by structural arguments showing that the corresponding indices stabilize in the limit as the system approaches an irrational parameter \( \alpha \). For the remaining cases, this convention provides a consistent extension as well. The resulting coloring is displayed in Figure~\ref{fig:ColoredKohmoto}.

Inspired by~\cite{KelPro17,KelSca25}, one may further seek to associate this gap index with a topological winding number. 
It would be interesting to explore whether such a winding number is compatible with the labeling scheme proposed in~\cite{BanBec25}.

\begin{figure}[htb]
    \centering
    \includegraphics[scale=0.37]{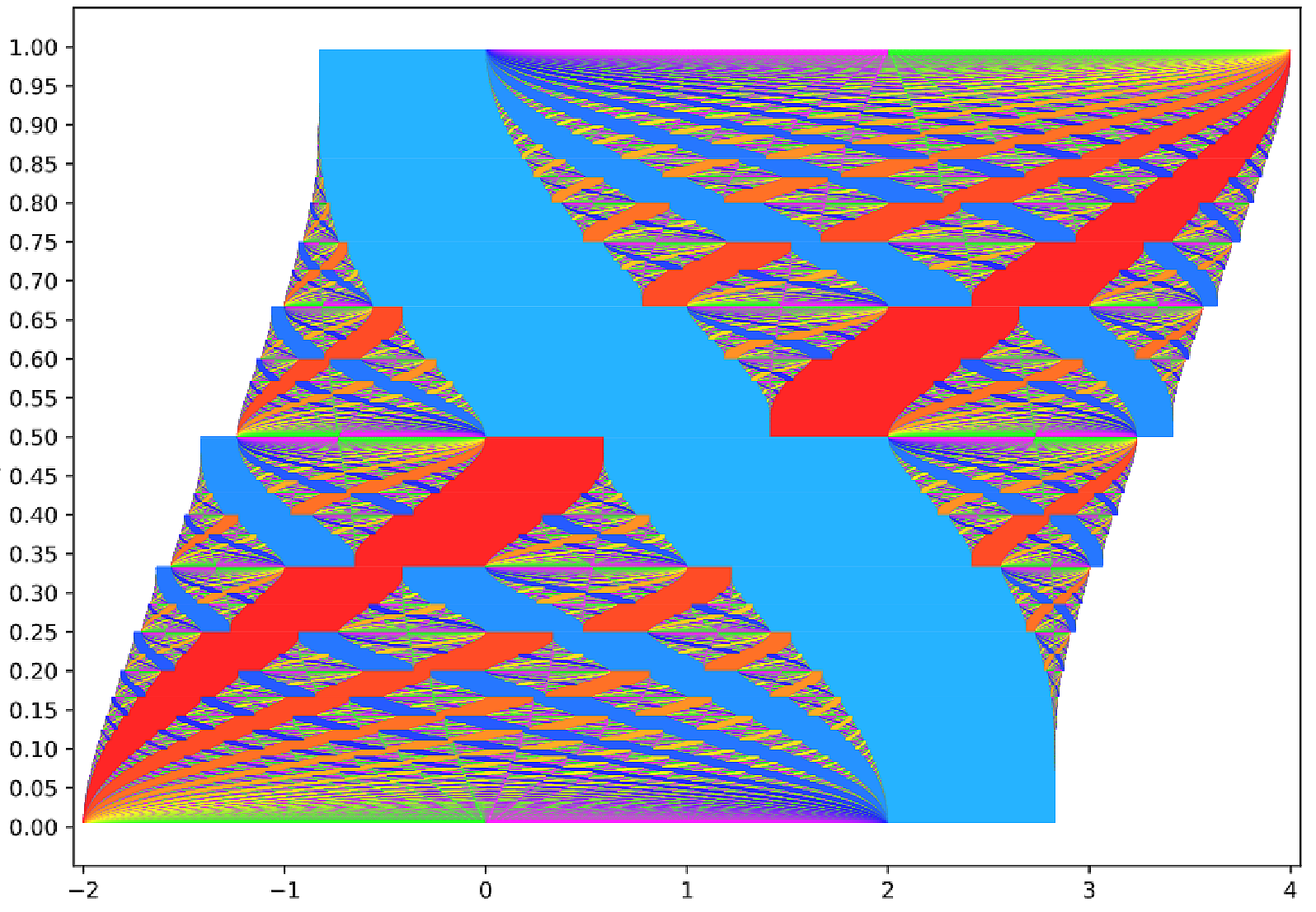}
    \captionsetup{width=0.95\linewidth}
	\caption{This figure was created in collaboration with Ram Band as part of the joint project~\cite{BanBec25}. For each rational value \( \alpha = \tfrac{p}{q} \), spectral gaps are indexed by their gap index, and each index is assigned a distinct color, see a more detailed discussion in Section~\ref{sec:DTMP_Sturmian}. We use the same color scheme as in~\cite{OsaAvr01}. 
	}
    \label{fig:ColoredKohmoto}
\end{figure}

\chapter[Generalized eigenfunctions]{Generalized eigenfunctions of Schr\"odinger operators}
\label{chap:GeneralEigenfct}

In Chapter~\ref{chap:WorldButterflies}, we presented several results to approximate the spectrum by the spectra of other operators, which was employed in Chapter~\ref{chap:DTMP}. We now return to the more general setting of Schrödinger-type operators in both the continuous and discrete framework. In this chapter, we present results from the joint works \cite{BecPinc20,BecDev22}, where we analyze when a generalized eigenfunction gives rise to an element of the spectrum. In the subsequent section, we characterize the essential spectrum of Schrödinger operators on graphs, following the joint work \cite{BecEli21}. Note that discrete Schrödinger operators associated with Delone sets fall into this framework of Schrödinger operators on graphs; see \cite[Sec.~5]{BecPog20} and Example~\ref{ex:Graph_Delone}.

Under suitable assumptions on the underlying space, the results of this chapter can be summarized as follows:
\begin{itemize}
\item The spectrum $\sigma(H)$ of a Schrödinger-type operator can be characterized by generalized eigenfunctions whose growth is either controlled by the geometry (e.g., subexponential growth) or by an intrinsic object such as the Agmon ground state; see \cite{BecPinc20,BecDev22} and Section~\ref{sec:EllOp_RiemManif}.
\item The essential spectrum $\sigma_{\mathrm{ess}}(H)$ is characterized by bounded generalized eigenfunctions under suitable growth conditions of the underlying graph; see \cite{BecEli21} and Section~\ref{sec:DiscretOp_generEigenfct}.
\end{itemize}

\section[Shnol type theorem]{The spectrum of Schrödinger-type operators: Shnol type theorem}
\label{sec:EllOp_RiemManif}

This section follows the joint works \cite{BecPinc20,BecDev22}, which address the question under which conditions a generalized eigenvalue \( \lambda \), associated with a generalized eigenfunction of a self-adjoint operator \( H \), belongs to the spectrum of \( H \).

Before introducing the formal framework, we briefly outline the background of this question. Shnol \cite{Shnol57} showed that for a Schrödinger operator \( H \) on \( \R^d \) with a potential bounded from below, any generalized eigenfunction with at most polynomial growth implies that the corresponding energy lies in the spectrum of \( H \); see also \cite{Sim81}. Furthermore, it was shown in \cite{Shu79,Sim81,CyconFroeseKirschSimon87} that there exists a dense set of spectral values admitting such generalized eigenfunctions.

Extensions to the framework of Dirichlet forms were provided in \cite{BouSto03,BoLeSt09,HaeKel11,FraLenWin14}, where the polynomial growth condition is replaced by a subexponential growth requirement for the generalized eigenfunction.

The underlying proofs use:
\begin{itemize}
\item a Weyl sequence argument (approximate eigenfunctions),\footnote{These methods were also applied in \cite{BecTak25}, see Section~\ref{sec:SpectralEstimates} and Appendix~\ref{App:SpectralEstimates}.}
\item suitable cut-off functions, called admissible cut-off sequences in \cite{BecDev22},
\item local gradient estimates -- so-called Caccioppoli type estimates \cite{BirMos95,HeiKilMar06,BoLeSt09}.
\end{itemize}

The work \cite{BecPinc20} was motivated by the conjecture stated in \cite[Conj.~9.9]{DevFraPin14}, which asks whether the polynomial growth condition can be replaced by a bound involving the Agmon ground state of the Schrödinger operator (provided it exists). In \cite{BecDev22}, a more general framework was developed that unifies various growth bounds under a single concept. It is also shown that the resulting condition is close to optimal (see \cite[Rem.~1.5]{BecDev22}).

Additionally, \cite[Exam.~7.3]{BecDev22} provides examples of generalized eigenfunctions exhibiting exponential growth but where one can still prove with similar techniques that the generalized eigenvalue lies in the spectrum of the operator.

In the following, let $\Omega$ denote a domain in $\R^d$ or, more generally, a domain in a non-compact, connected, $d$-dimensional Riemannian manifold. Let $m$ be a strictly positive measurable function on $\Omega$ such that both $m$ and $m^{-1}$ are bounded on every compact subset of $\Omega$. We define the weighted measure by $dm := m(x)\,dx$, where $dx$ denotes the Riemannian volume form on $\Omega$. In the case $\Omega \subseteq \R^d$, this corresponds to the Lebesgue measure. Then \( (\Omega,dm) \) is a \emph{weighted manifold}.

Let $T_x\Omega$ be the tangent space to $\Omega$ at $x\in \Omega$ and $T\Omega$ be the tangent bundle.
The bundle of endomorphisms of the tangent bundle $T\Omega$ is denoted by $\text{End}(T\Omega)$, equipped with the induced inner product $\langle\cdot,\cdot\rangle$ and corresponding norm $|\cdot|$.

A symmetric measurable section $A$ of $\text{End}(T\Omega)$ is called \emph{locally uniformly elliptic} if for every compact set $K \subseteq \Omega$, there exists a constant $\lambda_K > 0$ such that
\[
\frac{1}{\lambda_K} |\xi|^2 \leq \langle A(x)\xi, \xi \rangle \leq \lambda_K |\xi|^2, \qquad \text{for all } x \in K \text{ and } (x,\xi) \in T\Omega.
\]

Fix such an $A$ and a potential $V \in L^p_{\mathrm{loc}}(\Omega,\R)$ with $p > \tfrac{d}{2}$. Here $L^p_{\mathrm{loc}}(\Omega,\R)$ denotes the equivalence classes of all measurable functions $f:\Omega\to\R$ that have finite $L^p$-norm if restricted to any compact subset in $\Omega$. We define the symmetric bilinear form
\[
Q : C_c^\infty(\Omega) \times C_c^\infty(\Omega) \to \C, \qquad
Q(u, w) := \int_\Omega \big(\langle A \nabla u, \nabla w \rangle + V u \overline{w}\big) \,dm,
\]
where $C_c^\infty(\Omega)$ denotes the space of smooth, compactly supported functions on $\Omega$.
We assume that $Q$ is semibounded from below, i.e., there exists a constant $c \in \R$ such that
\[
Q(u, u) \geq c\, \|u\|_{2,m}^2, \qquad u \in C_c^\infty(\Omega),
\]
where $\|u\|_{2,m}$ denotes the $L^2$-norm of $u$ in $L^2(\Omega, dm)$. Under this assumption, the form $Q$ is closable (see e.g.~\cite{Stoll01}), and we denote its closure by the same symbol. The domain of the closed form is given by $D(Q) := \overline{C_c^\infty(\Omega)}^{\|\cdot\|_Q}$, where the $Q$-norm is defined by
\[
\|u\|_Q := \sqrt{Q(u, u) + (1 - c)\, \|u\|_{2,m}^2}.
\]
By construction, $C_c^\infty(\Omega)$ is a core for $D(Q)$, and by general theory, there exists a unique self-adjoint operator $H_V$ with domain $D(H_V) \subseteq D(Q)$ such that
\[
Q(u, w) = \langle H_V u, w \rangle \qquad \text{for all } u\in D(H_V),\,  w \in D(Q).
\]
This Schrödinger-type operator has the formal expression
\begin{equation}
\label{eq:SchroedType_Continuous}
H_V = -\mathrm{div}(A \nabla \cdot) + V
\end{equation}
and we are interested in the spectrum $\sigma(H_V)$ of $H_V$. The operator is called \emph{nonnegative} (i.e., $H_V \geq 0$) if $Q(u,u) \geq 0$ for all $u \in C_c^\infty(\Omega)$.

The framework of criticality theory \cite{Pinsky95,PinTin06,Pin07,PinPsa16,KelPinPog20} is fundamental.

\begin{definition}
Let $H_V$ be as in~\eqref{eq:SchroedType_Continuous}. The operator is called
\begin{itemize}
\item \emph{supercritical} if $H_V$ is not nonnegative,
\item \emph{critical} if $H_V \geq 0$ and for every nonzero, nonnegative $W \in L^p_{\mathrm{loc}}(\Omega, \R)$ with $p > \tfrac{d}{2}$, the perturbed operator $H_{V - W}$ is supercritical,
\item \emph{subcritical} if $H_V \geq 0$ and not critical.
\end{itemize}
\end{definition}

If $H_V$ is critical, then there exists a unique (up to a multiplicative constant) positive function $u \in W^{1,2}_{\mathrm{loc}}(\Omega)$ solving (in the weak sense)
\[
\langle H_V u, w \rangle = 0, \qquad w \in C_c^\infty(\Omega).
\]
This function is called the \emph{Agmon ground state} or a minimal positive \emph{$H_V$-harmonic function}. As shown in the above references (see also \cite[Thm.~3.5]{BecPinc20}), the operator $H_V$ is critical if and only if there exists a \emph{null sequence} $(\varphi_n)_{n \in \N}$, i.e., a sequence in $C_c^\infty(\Omega)$ such that for some bounded ball $B \subseteq \Omega$,
\[
\int_B \varphi_n^2\,dm = 1
\qquad \text{and} \qquad
\lim_{n \to \infty} \langle H_V \varphi_n, \varphi_n \rangle = 0.
\]

We now return to the notion of generalized eigenfunctions.

\begin{definition}
\label{def:generEigenfct}
Let $(\Omega, dm)$ be a weighted manifold and let $H_V$ be a Schrödinger-type operator of the form~\eqref{eq:SchroedType_Continuous}. A function $u \in W^{1,2}_{\mathrm{loc}}(\Omega)$ is called a \emph{generalized eigenfunction} of $H_V$ with \emph{(generalized) eigenvalue} $\lambda \in \R$ if
\[
\langle H_V u, v \rangle = \lambda \int_\Omega u\, \overline{v} \, dm, \qquad \text{for all } v \in C_c^\infty(\Omega).
\]
\end{definition}

To verify whether a generalized eigenvalue $\lambda$ belongs to the spectrum $\sigma(H_V)$, the Weyl criterion is employed. This involves the construction of approximate eigenfunctions via suitable cut-off functions. In \cite{BecPinc20}, null sequences are used as cut-off functions to resolve the conjecture \cite[Conj.~9.9]{DevFraPin14} formulated in the following statement.

\begin{theorem}[\cite{BecPinc20}]
\label{thm:ShnolType_AgmonGroundState}
Let $(\Omega, dm)$ be a weighted manifold and let $H_V$ be a Schrödinger-type operator on $(\Omega, dm)$ of the form~\eqref{eq:SchroedType_Continuous}. Suppose there exists a function $W \in L^\infty(\Omega, dm)$ such that $H_{V + W}$ is critical with Agmon ground state $\varphi$.

If $u$ is a generalized eigenfunction of $H_V$ with generalized eigenvalue $\lambda \in \R$ satisfying
\[
|u| \leq C \varphi \quad \text{in } \Omega
\]
for some constant $C > 0$, then $\lambda \in \sigma(H_V)$.
\end{theorem}

\begin{proof}
This is proven in \cite[Thm.~1.2]{BecPinc20}.
\end{proof}

The Agmon ground state is intrinsically defined by a minimal positive $H_V$-harmonic function. Theorem~\ref{thm:ShnolType_AgmonGroundState} states that whenever a generalized eigenfunction is not growing faster than an $H_V$-harmonic function, then the corresponding generalized eigenvalue is in the spectrum of $H_V$. Examples illustrating this result can be found in \cite[Sec.~1]{BecPinc20}.

\medskip

As mentioned before, discrete structures are also treated in \cite[Sec.~6]{BecPinc20}. We summarize the relevant framework and results here. Let $X$ be a countably infinite set equipped with the discrete topology and let $m:X\to(0,\infty)$ be a measure with full support. We consider the Hilbert space $\ell^2(X,m)$ of real-valued functions $u:X\to\R$ with finite $\ell^2$-norm
\[
\|u\|_{2,m}^2 := \sum_{x\in X} |u(x)|^2\, m(x) < \infty.
\]

\begin{definition}
\label{def:weightGraph}
A (weighted) graph $(X,b)$ is a countable set $X$ and a function $b:X\times X\to [0,\infty)$ satisfying:
\begin{itemize}
\item $b(x,x) = 0$ for all $x \in X$ (no loops),
\item $b(x,y) = b(y,x)$ for all $x,y \in X$ (undirected graph),
\item $\deg(x):=\sum_{y\in X} b(x,y) < \infty$ for all $x \in X$ (bounded weighted vertex degree).
\end{itemize}
\end{definition}

In the notation used in Chapter~\ref{chap:WorldButterflies} and Chapter~\ref{chap:DTMP}, the edge set of this graph is given by $\set{(x,y) \in X \times X}{b(x,y) > 0}$. If $b(x,y)>0$, we say that $x$ and $y$ are connected by an edge with edge weight $b(x,y)$. The graph $(X,b)$ is called \emph{connected} if every pair of vertices can be joined by a finite path.

Analogous to the continuous case, one defines a symmetric bilinear form on $C_c(X)\times C_c(X)$ associated with the graph $(X,b)$ and a potential $c:X\to[0,\infty]$ (also referred to as a killing term) by
\[
Q_{b,c}(u,w) := \sum_{x,y\in X} b(x,y)\, (u(x)-u(y))(w(x)-w(y)) + \sum_{x\in X} c(x)\, u(x) w(x).
\]
Note that in the discrete setting, $C_c(X)$ simply denotes the space of finitely supported functions on $X$. The form $Q_{b,c}$ is semibounded from below and hence closable. Its closure, still denoted by $Q_{b,c}$, has domain given by the completion $D(Q_{b,c}) := \overline{C_c(X)}^{\|\cdot\|_{Q_{b,c}}}$ with form norm
\[
\|u\|_{Q_{b,c}} := \sqrt{Q_{b,c}(u,u) + \|u\|_{2,m}^2}.
\]
By general theory (see e.g.~\cite{Stoll01}), there exists a unique self-adjoint operator $H_c$ on $\ell^2(X,m)$ associated with $Q_{b,c}$. This operator acts as
\begin{equation}
\label{eq:GraphSchrOp}
(H_c u)(x) = \sum_{y\in X} b(x,y)\, (u(x) - u(y)) + c(x)\, u(x), \qquad u \in C_c(X).
\end{equation}

According to \cite[Thm.~7]{KelLen12}, every regular Dirichlet form on $\ell^2(X,m)$ arises from such a graph representation. That is, for each regular Dirichlet form $Q$, there exists a pair $(b,c)$ such that $Q = Q_{b,c}$. For definitions and further details regarding regular Dirichlet forms in the discrete setting, we refer the reader to \cite[Sec.~6]{BecPinc20} and the references therein.

The notions of (super-/sub-)criticality in the discrete setting are analogous to those in the continuous case. A comprehensive criticality theory for discrete structures was developed in \cite{KelPinPog20}, where the existence of null sequences and minimal positive $H_c$-harmonic functions was also established.

To define a notion of generalized eigenfunction that may lie outside the domain of $H_c$ but still makes sense pointwise, we introduce the function space
\[
F(X) := \set{u : X \to \R}{\sum_{y \in X} b(x,y) |u(y)| < \infty \text{ for each } x \in X}.
\]
A function $u \in F(X)$ is called \emph{$H_c$-harmonic} if $H_c u = 0$ in $X$.

Let $C_b(X)$ denote the space of bounded, real-valued functions on $X$. With these definitions in place, we now state the discrete analogue of Theorem~\ref{thm:ShnolType_AgmonGroundState}.

\begin{theorem}[\cite{BecPinc20}]
\label{thm:ShnolType_AgmonGroundState_Discrete}
Let $b$ be a connected graph on $X$, let $c : X \to [0,\infty)$, and let $Q_{b,c}$ be the associated closed, semibounded, symmetric form on $\ell^2(X,m)$ with associated self-adjoint operator $H_c$. Suppose $W \in C_b(X)$ is such that $H_{c+W}$ is critical in $X$ with Agmon ground state $\varphi$. If $u \in F(X)$ is a generalized eigenfunction of $H_c$ with generalized eigenvalue $\lambda \in \R$ satisfying
\[
|u(x)| \leq C \varphi(x) \quad \text{for all } x \in X
\]
for some constant $C > 0$, then $\lambda \in \sigma(H_c)$.
\end{theorem}

\begin{proof}
This is proven in \cite[Thm.~6.3]{BecPinc20}.
\end{proof}

Before turning to the generalization of Theorem~\ref{thm:ShnolType_AgmonGroundState} established in \cite{BecDev22}, we discuss the connection to Delone sets. In particular, we explain how a graph structure can be naturally associated with a Delone set.

\begin{example}
\label{ex:Graph_Delone}
Recall the basic notions introduced in Definition~\ref{def:Delone} for Delone sets in a locally compact, second countable Hausdorff group $G$. Every Delone set $D$ is countable and will serve as the vertex set $X$ in the previous setting.

Let $D \in \Del$ with fixed parameters $r_0, R_0 > 0$. We define a graph structure on $D$ via a function
\[
b = b_f : D \times D \to \R, \qquad b(x,y) := f(x^{-1}y),
\]
where $f : G \to [0,\infty)$ is a continuous function with compact support satisfying $f(e) = 0$ and $f(g) = f(g^{-1})$ for all $g \in G$. This symmetry condition can be ensured by symmetrizing an arbitrary continuous function $\tilde{f} : G \to [0,\infty)$ via
\[
f(g) := \tilde{f}(g) + \tilde{f}(g^{-1}), \qquad g\in G.
\]

The assumptions imply:
\begin{itemize}
\item $b(x,x) = 0$ for all $x \in D$ since $f(e)=0$,
\item $b(x,y) = b(y,x)$ for all $x, y \in D$ since $f(g) = f(g^{-1})$ for all $g \in G$,
\item since $x^{-1}D$ is uniformly discrete for all $x \in D$ and $f$ has compact support, we have
\[
\deg(x) = \sum_{y\in D} b(x,y)
	= \sum_{y\in D} f(x^{-1}y)
	= \sum_{z\in x^{-1}D\cap\supp(f)} f(z)
	<\infty,
	\qquad x\in D.
\] 
\end{itemize}
Thus, $(D, b_f)$ defines a weighted, undirected graph, which we refer to as the \emph{Delone graph} associated with $D$ and $f$.
The condition on $f$ can be relaxed by requiring sufficient decay to ensure summability. 

Figure~\ref{fig:DeloneGraph} illustrates the principle: edges are determined by evaluating $f$ on the relative positions of elements in $D$. Note that the Delone graph $(D,b_f)$ has bounded vertex degree $\sup_{x\in D} \deg(x)<\infty$ if $f$ is continuous with compact support.

\begin{figure}[htb]
    \centering
    \includegraphics[scale=1.5]{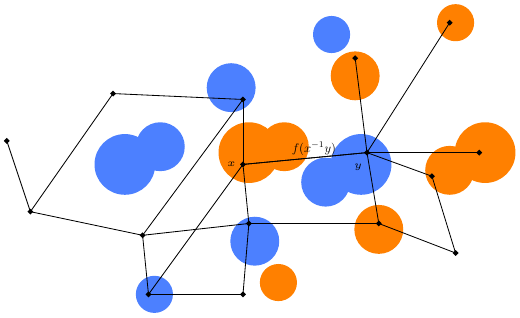}
    \captionsetup{width=0.95\linewidth}
    \caption{Sketch of the Delone graph. Black points represent elements of $D$, with particular points $x,y \in D$ indicated. The blue and orange regions illustrate the support of the shifted functions $f(x^{-1} \cdot)$ and $f(y^{-1} \cdot)$, respectively.}
    \label{fig:DeloneGraph}
\end{figure}

To define a potential term, let $V : \Del \to [0,\infty)$ be a continuous function. We set $c : D \to [0,\infty)$ by
\[
c(x) := V(x^{-1}D).
\]
This yields a self-adjoint operator $H_{D,c}$ on $C_c(D)$ acting as in equation~\eqref{eq:GraphSchrOp}.

Although this construction does not produce a dynamically-defined operator family in the strict sense of Definition~\ref{def:DynDefOp} (since $D$ is not a group, in general), the operator $H_{D,c}$ formally resembles such an operator. To explain this, we recall that $\Del$ is a compact metric space with metric $\dw$, see Proposition~\ref{prop:metric_weightDelone}. The topology coincides with the Chabauty-Fell topology. Define
\[
\varphi_f : \Del \to \R, \qquad \varphi_f(D) := \sum_{z \in D} f(z),
\]
which is continuous by \cite[Lem.~3.5]{BecPog20}. Observe that 
\[
\varphi_f(x^{-1}D) = \sum_{y \in D} f(x^{-1}y) = \sum_{y \in D} b(x,y), \qquad x\in D,
\]
which coincides with the weighted vertex degree $\deg(x)$ at $x \in D$ in the associated Delone graph $(D,b_f)$.

Let $\rho : G \to [0,1]$ be a continuous function with $\rho(e) = 1$ and $\rho(g) = 0$ for $g \in G \setminus B_{r_0}(e)$. Define a kernel
\[
k : G \times \Del \to \C, \qquad k(g,D) := -f(g) + \rho(g) \cdot \big(V(D) + \varphi_f(D)\big),
\]
which is continuous and has compact support in $G \times \Del$.

Since $D$ is $r_0$-uniformly discrete and $\rho$ vanishes outside $B_{r_0}(e)$, for $x,y \in D$ the condition $\rho(x^{-1}y) \neq 0$ implies $x = y$. Inserting this into the operator expression yields:
\begin{align*}
(H_{D,c}u)(x) = &\sum_{y\in D} -f(x^{-1}y)\cdot u(y) + \left(V(x^{-1}D) + \sum_{z\in x^{-1}D}f(z) \right)\cdot u(x)\\
	= &\sum_{y\in D} k\big(x^{-1}y,x^{-1}D\big) \cdot u(y)
\end{align*}
for $x \in D$ and $u \in C_c(D)$.

This representation suggests that $H_{D,c}$ resembles a dynamically-defined operator restricted to a transversal of $D$, see the definition of the transveral on page~\pageref{page:transversal}. However, since $D$ is generally not a group, it does not directly fit into the standard framework. The correct setting involves groupoids over the transversal of $D$, as used in \cite{BBdN18}. 
In particular, combining the above considerations with the results from \cite[Thm.~3]{BBdN18} shows that the spectral map
\[
\Sigma_c : \inv \to \Kk(\R), \qquad \hull(D) \mapsto \sigma(H_{D,c}),
\]
is continuous if $G$ is amenable, where $(\inv, \delta_H)$ is the compact metric space introduced in Section~\ref{sec:Weight_Delone}.
Proposition~\ref{prop:Homeo_hulls_transversals} ensures that convergence in $\delta_H$ is equivalent to convergence of the corresponding transversals.

While this yields qualitative spectral continuity, quantitative spectral estimates analogous to Theorem~\ref{thm:SpecEst_DynDefOp} are currently unavailable in the literature and are the subject of ongoing research; see also the discussion in Chapter~\ref{chap:discussion}.
\end{example}

Checking the Shnol-type theorems, one observes that approximate eigenfunctions can also be constructed using distance functions defined by so-called intrinsic metrics. In \cite[Sec.~5 and 6]{BecDev22}, the notion of an \emph{admissible cut-off sequence} is introduced. It is shown that this notion encompasses both null sequences and distance functions associated with intrinsic metrics, and that it suffices to construct a Weyl sequence yielding a result analogous to Theorem~\ref{thm:ShnolType_AgmonGroundState}.
We use the notation $|\nabla v|_A^2:=\langle A\nabla v, \nabla v\rangle$.

\begin{theorem}[\cite{BecDev22}]
\label{thm:ShnolType}
Let $(\Omega,dm)$ be a weighted manifold and $H_V$ be a Schrödinger-type operator on $(\Omega,dm)$ of the form \eqref{eq:SchroedType_Continuous}. Let $\varphi \in C(\Omega)\cap W^{1,2}_{\mathrm{loc}}(\Omega)$ and $\{\varphi_n\}_{n\in\N}\subseteq C_c(\Omega)\cap D(Q_V)$ be an admissible cut-off sequence for $(H,\varphi)$. Let $u\in W^{1,2}_{\mathrm{loc}}(\Omega)$ be a generalized eigenfunction of $H$ with associated eigenvalue $\lambda\in\R$. Let $A_n$ be the support of the function $\varphi_{n+1}(\varphi - \varphi_{n-1})$. Suppose that one of the following growth conditions on $u$ holds:

\begin{itemize}
\item[(i)] There exists a bounded potential $W:\Omega\to\R$ such that $(H_{V+W})\varphi=0$ in the weak sense, and
\[
\liminf_{n\to\infty} 
	\frac{\max_{A_n}\left|\frac{u}{\varphi}\right|}{\|\varphi_n\,\frac{u}{\varphi}\|_{2,m}}
	\left( \int_\Omega \left|\nabla \left(\frac{\varphi_n}{\varphi}\right)\right|^2_A\, \varphi^2 \,dm \right)^{1/2}
	= 0,
\]

\item[(ii)] $\varphi$ is constant, and
\[
\liminf_{n\to\infty}
\frac{
\|u\|_{L^2(A_n,dm)}
+
\left( \int_\Omega |u|^2\left(|\nabla \varphi_{n-1}|^2_A + |\nabla \varphi_n|^2_A + |\nabla \varphi_{n+1}|^2_A \right)\,dm \right)^{1/2}
}{\|\varphi_n u\|_{2,m}} = 0.
\]
\end{itemize}
Then $\lambda \in \sigma(H)$.
\end{theorem}

\begin{proof}
This is proven in \cite[Thm.~1.3]{BecDev22}.
\end{proof}

Statement (i) generalizes Theorem~\ref{thm:ShnolType_AgmonGroundState} from \cite{BecPinc20}, see \cite[Sec.~5]{BecDev22}. Therein, a particular null sequence is constructed using the so-called Evans potential. The condition in (i) is strictly weaker than the bound by the ground state and allows applications to a broader class of examples; see \cite[Exam.~7.4]{BecDev22}.

Statement (ii) follows by similar arguments as in \cite[Thm.~4.4]{BoLeSt09}. However, (ii) is more general than their theorem, as illustrated in \cite[Exam.~7.1]{BecDev22}. Moreover, \cite[Exam.~7.3]{BecDev22} discusses the Laplace-Beltrami operator on the hyperbolic plane, where the generalized eigenfunctions exhibit exponential growth but Theorem~\ref{thm:ShnolType} is still applicable. 

Finally, condition (i) does not follow from (ii), as shown in \cite[Exam.~7.2]{BecDev22}.

\section[The essential spectrum]{The essential spectrum and bounded generalized eigenfunctions}
\label{sec:DiscretOp_generEigenfct}

This section follows \cite{BecEli21}, where the essential spectrum of discrete Schrödinger operators on graphs is analyzed in terms of generalized eienfunctions.

Recall the notion of a weighted graph $(X,b)$ introduced in Definition~\ref{def:weightGraph}. Such a graph is called \emph{combinatorial} if the edge weights $b$ take only values in $\{0,1\}$. In this case, the relation $y \sim x$ defined by $b(x,y) = 1$ induces a symmetric relation on $X$. A graph $(X,b)$ is said to be \emph{$d$-bounded} if the vertex degree is uniformly bounded by $d \in\N$, that is,
\[
\sup_{x\in X} \deg(x) \leq d.
\]
Thus, a $d$-bounded combinatorial graph has uniformly bounded vertex degrees. Throughout this section, we fix the weight $m : X \to (0,\infty)$ to be identically one.

Let $(X,b)$ be a combinatorial $d$-bounded graph and let $c : X \to [0,\infty)$ be a bounded potential. We consider the associated self-adjoint operator $H_c$ defined by
\[
(H_c u)(x) := \sum_{y \sim x} \big(u(x) - u(y)\big) + c(x) u(x), \qquad u \in \ell^2(X).
\]
The \emph{discrete spectrum} $\sigma_{\mathrm{disc}}(H_c)$ is defined as the set of isolated eigenvalues of $H_c$ with finite multiplicity. The \emph{essential spectrum} is then given by
\[
\sigma_{\mathrm{ess}}(H_c) := \sigma(H_c) \setminus \sigma_{\mathrm{disc}}(H_c).
\]
By Weyl's theorem, the essential spectrum is stable under compact perturbations.

Recall the notion of generalized eigenfunctions introduced in the previous section. There, we established that if such functions exhibit sufficiently slow growth, the corresponding generalized eigenvalue lies in the spectrum $\sigma(H_c)$. In the characterization of the essential spectrum, only bounded generalized eigenfunctions are relevant under suitable growth assumptions of the graph. We define the set
\[
\sigma_\infty(H_c) 
	:= \set{ \lambda \in \C }{ 
		\begin{array}{c}
			\lambda \text{ is a generalized eigenvalue of }\\
			\text{a bounded generalized eigenfunction }
		\end{array}
		}.
\]

We have previously encountered the essential spectrum in the context of systems with purely essential spectrum, a concept that is intimately related to minimality. As shown in \cite{LenSto03Alg,Bec16}, minimality is equivalent to the constancy of the spectrum across the entire dynamical system, and this common spectrum coincides with the essential spectrum. For weighted Delone sets in $G$, minimality is further characterized by repetitivity; see \cite[Prop.~2.4]{BecHarPog25} and \cite{FrRi14} for unweighted Delone sets. Roughly speaking, repetitivity means that every finite pattern (or patch) appears relatively densely in the Delone set, i.e., with uniformly bounded gaps. Hence, repetitive weighted Delone sets exhibit no dynamical impurities, meaning that no patch is confined to a region of the space without recurring throughout with bounded gaps.

Minimality implies, even for non-normal operators, that the spectrum is constant; see \cite{BLLS17,BLLS18} for further discussion. This phenomenon underlines the intrinsic relation between the spectral type and the global topological dynamical structure of the system.

The underlying philosophy of these results is that the essential spectrum captures the behavior "at infinity" of the space or structure in question. In \cite{BecEli21}, we established such a result for $d$-bounded combinatorial graphs. Earlier results of this type were obtained for operators on $\N$ or $\Z^d$; see \cite{ChaLin11,Simon11}. There, the concept of \emph{right limits} plays a central role. Informally, a right limit is a limit point in the shift space of the system that corresponds to the strong limit of a sequence of shifted operators. This notion was extended to higher dimensions in \cite{LastSimon06}, and to general $d$-bounded graphs in \cite{BreDenEli18}, where it was adapted to the concept of \emph{$\Rr$-limits}.

We omit the technical definition of $\Rr$-limits here and refer the reader to \cite[Def.~4.2]{BecEli21}. It suffices to note that the set $\Rr(H_c)$ of all $\Rr$-limits consists of bounded self-adjoint operators acting on possibly different graphs $(X', b')$ that arise as a limits of $(X,b)$ shifting a root.

In \cite{BreDenEli18}, it was proven that the union of the spectra of $\Rr$-limits contains the essential spectrum of $H_c$. Moreover, they showed that the reverse inclusion holds for regular trees. On the other hand, they also provided an example of an infinite, connected, $d$-bounded graph for which the equality fails. This observation served as a motivation for the joint work \cite{BecEli21}. 

A \emph{rooted graph} $(X,b,v)$ is a graph $(X,b)$ together with a distinguished vertex $v\in X$. The ball of radius $r\in\N$ centered at $v$ is defined by
\[
B_r(v) 
	:= \set{u\in X}{ 
		\begin{array}{c}
			\exists\, n\leq r,\ \exists\, x_0,\ldots,x_n\in X \text{ with }\\ 
			x_0=v,\ x_n=u,\ x_{i-1}\sim x_i\ \text{ for all }\, 1\leq i\leq n
		\end{array}
	},
\]
which is the ball with respect to the combinatorial graph distance.
The rooted graph $(X,b,v)$ is said to have \emph{subexponential growth} if for every $\gamma > 0$, there exists a constant $C = C(\gamma, v) > 0$ such that
\[
\sharp B_r(v) < C \gamma^r, \qquad r\in\N,
\]
where $\sharp B_r(v)$ denotes the number of elements in $B_r(v)$.
If the constant $C(\gamma,v)$ can be chosen uniformly in $v\in X$, the graph $(X,b)$ is said to have \emph{uniform subexponential growth}, i.e. $\sup_{v\in X} C(\gamma,v)<\infty$ for all $\gamma>0$.

\begin{theorem}[\cite{BecEli21}]
\label{thm:EssSpectr_subset_siginf}
Let $(X,b,v)$ be an infinite, connected, $d$-bounded rooted graph of subexponential growth. For every bounded Schrödinger operator $H_c$ on $\ell^2(X)$ with bounded potential $c:X\to[0,\infty)$, we have
\[
\sigma_{\mathrm{ess}}(H_c) \subseteq \bigcup_{H'\in \Rr(H_c)} \sigma_\infty(H').
\]
\end{theorem}

\begin{proof}
This is proven in \cite[Thm.~2.2]{BecEli21}. 
\end{proof}

In general, the inclusion in Theorem~\ref{thm:EssSpectr_subset_siginf} is strict. For instance, such a strict inclusion occurs for regular trees, which exhibit exponential growth; see \cite{FigPic83,Broo91}. Moreover, in \cite[Prop.~2.3]{BecEli21}, we construct an explicit example of a rooted graph with subexponential growth where the inclusion remains strict.

However, under the stronger assumption of uniform subexponential growth, one obtains equality:

\begin{theorem}[\cite{BecEli21}]
\label{thm:EssSpectr_=_siginf_unif-SubExp}
Let $(X,b)$ be an infinite, connected, $d$-bounded graph of uniform subexponential growth. For every bounded Schrödinger operator $H_c$ on $\ell^2(X)$ with bounded potential $c : X \to [0,\infty)$, we have
\[
\sigma_{\mathrm{ess}}(H_c) = \bigcup_{H'\in \Rr(H_c)} \sigma(H') = \bigcup_{H'\in \Rr(H_c)} \sigma_\infty(H').
\]
\end{theorem}

\begin{proof}
This is proven in \cite[Thm.~2.4]{BecEli21}. We briefly outline the main steps: As shown in \cite[Thm.~2]{BreDenEli18}, the inclusion
\[
\bigcup_{H'\in \Rr(H_c)} \sigma(H') \subseteq \sigma_{\mathrm{ess}}(H_c)
\]
holds for general $d$-bounded graphs. Theorem~\ref{thm:EssSpectr_subset_siginf} complements this result by establishing the reverse inclusion
\[
\sigma_{\mathrm{ess}}(H_c) \subseteq \bigcup_{H'\in \Rr(H_c)} \sigma_\infty(H').
\]
Under the additional assumption of uniform subexponential growth, a Shnol-type theorem \cite[Thm.~4.8]{HaeKel11} implies that $\sigma_\infty(H') \subseteq \sigma(H')$ for every $H' \in \Rr(H_c)$. Combining these inclusions yields
\[
\bigcup_{H'\in \Rr(H_c)} \sigma(H') = \sigma_{\mathrm{ess}}(H_c) = \bigcup_{H'\in \Rr(H_c)} \sigma_\infty(H'),
\]
which completes the proof.
\end{proof}

We refer to \cite[Sect.~1--3]{BecEli21} for additional background and further examples. Let us also point out that the identity
\[
\sigma_{\mathrm{ess}}(H_c) = \bigcup_{H'\in \Rr(H_c)} \sigma(H')
\]
was independently obtained in \cite{HagSei20} for certain classes of metric spaces. In contrast, the main focus of \cite{BecEli21} lies in establishing a direct connection to bounded generalized eigenfunctions.

\markboth{Outlook}{Outlook}
  \include{discussion}

\bibliographystyle{plain}
\bibliography{referencesHabil}

\cleardoublepage

\markboth{Appendix}{Appendix}
  \include{appendix}

\end{document}

%% file: acknowledgements.tex
\addcontentsline{toc}{chapter}{\protect Acknowledgments}

\chapter*{Acknowledgments}

The following acknowledgements are directed to colleagues and friends who have supported me in various ways. As is customary in mathematics, the names are listed alphabetically and not according to the extent of their contributions, which I would not be in a position to rank appropriately.

I am particularly grateful to the following colleagues and friends for reading parts of this work, for discussing its structure or overall approach, or for providing valuable feedback: 
Ram Band, Juliane Beckus, Paul Beckus, Peter Beckus, Jean Bellissard, Christopher Cedzich, Baptiste Devyver, Tobias Hartnick, 
Johannes Kellendonk, Matthias Keller, Marie Landsem, Daniel Lenz, Raphael Loewy, 
Yehuda Pinchover, Felix Pogorzelski, Sven Raum, Sylvie Roelly, Annette Sviland, Alberto Takase, Lior Tenenbaum and Yannik Thomas.

In particular, I am also deeply thankful to my coauthors for many stimulating discussions that challenged me at times, but ultimately pushed me further: 
Ram Band, Jean Bellissard, Barak Biber, Horia Cornean, Giuseppe De Nittis, Baptiste Devyver, Latif Eliaz, Ulrik Enstad, Tobias Hartnick, Daniel Lenz, Marko Lindner, Raphael Loewy, 
Yehuda Pinchover, Felix Pogorzelski, Laurent Raymond, Christian Seifert, Alberto Takase, Lior Tenenbaum, Yannik Thomas and Jordy van Velthoven.

In addition to my coauthors, I am grateful to many others for inspiring and insightful conversations that have played a key role in shaping my perspective, and I am thankful for each of them. I am particularly grateful to the following individuals: 
Lior Alon, Eric Akkerman, Kyle Austin, Yossi Avron, Philipp Bartmann, Michael Baake, Simon Becker, Gregory Berkolaiko, Jonathan Breuer, Christopher Cedzich, Aviv Censor, David Damanik, Adam Dor-on, 
Mark Embree, Jake Fillman, Florian Fischer, Anton Gorodetski, Franz G\"ahler, Giles Gardam, Uwe Grimm, Johannes Happich, Alan Haynes, Paul Hege, Bernard Helffer, 
Svetlana Jitomirskaya, Johannes Kellendonk, Victor Kleptsyn, Daniel Lenz, Renaud Leplaideur, Wencai Liu, Franz Luef, 
Neil Manibo, May Mei, Noema Nicolussi, Tobias Oertel-J\"ager, Sanaz Pooya, Emil Prodan, Sven Raum, Christoph Richard, Matti Richter,  
Alexander Schmeding, Marcel Schmidt, Claude Schochet, Christian Scholz, Hermann Schulz-Baldes, Daniel Sell, Mehran Seyed Hosseini, Orr Shalit, Uzy Smilansky, Yotam Smilansky, Peter Stollmann, Boris Solomyak, Gilad Sofer, 
Pascal Vanier, Jamie Walton, Melchior Wirth, William Wood, Qi Zhou and Elias Zimmermann.


My sincere thanks goes to Ram Band and Yehuda Pinchover for their warm hospitality, both at the Technion -- Israel Institute of Technology -- and in Israel, as well as for their constant support and friendship. I am glad to have found in them two generous hosts who not only support young mathematicians with great commitment.

I am particularly thankful to Ram Band for the time we spent together, first in Haifa and later in Potsdam, which I greatly enjoyed. I am also especially grateful for all the stimulating and enlightening discussions -- both personal and professional.

Special thanks goes to Jean Bellissard for many insightful discussions in mathematics, physics, and beyond over the past years, and for his constant support.

I am thankful to Horia Cornean for the many inspiring discussions. I still like to recall how we first got in touch: after uploading an arXiv preprint in the evening, I received an email the next morning at 9 a.m., which ultimately led to our collaboration. I very much enjoyed the stays in Aalborg.

The many inspiring discussions with Raphael Loewy, Orr Shalit and Claude Schochet have greatly influenced my perspective and introduced me to new mathematical techniques, for which I am sincerely grateful.

In particular, I am deeply thankful to Felix Pogorzelski for encouraging me to come to the Technion. I am very glad that we were able to continue our fruitful collaboration,  and I have always valued his constant support and the openness with which he shared ideas and thoughtful advice.

I am grateful to my PhD students for bringing fresh perspectives and challenges into the projects. I am glad I had the chance to work with each of them: Alberto Takase, Lior Tenenbaum and Yannik Thomas.


Beyond my coauthors, I am thankful to Michael Baake for many enlightening discussions and organizing a meeting, where we got in touch with Laurent Raymond.

I would like to express my special thanks to David Damanik, Anton Gorodetski and Jake Fillman for numerous stimulating discussions at conferences, workshops, and other occasions, as well as for the opportunity to present my research to a wider audience.

Moreover, I am deeply grateful to Matthias Keller for the constant support I have received. He opened the door for me in Potsdam and has consistently supported me in all matters ever since.

I am very thankful to Marcel Schmidt for his openness and for the many enriching conversations, both personal and scientific.


My special thanks also go to my colleagues in Potsdam for encouraging me to complete this habilitation, the many valuable discussions, and the morning smiles that made even difficult days easier. In particular, I am grateful to: 
Christian B\"ar, Andreas Braunss, Alexandra Carpentier, J\"org Enders, Heiko Etzold, Myfanwy Evans, Melina Freitag, Moritz Gerlach, Claudia Grabs, 
Florian Hanisch, Niklas Hartung, Lukas Hellwig, Wilhelm Huisinga, 
Markus Klein, J\"org Koppitz, Ulrich Kortenkamp, Tania Kosenkova, Max Lein, Hannelore Liero, 
Thomas Mach, Jan Metzger, Sylvie Paycha, Sanaz Pooya, Sven Raum, Sylvie Roelly, Elke Rosenberger, Bernhard Stankewitz, Christoph Stephan and Gert Z\"oller.

Finally, I am also grateful to Winnie Enders, Peter Grabs, Volker Gustavs, Katrin Kania, Shiri Kaplan-Shabat, Galya Khanin, Christian Molle, Sylke Pfeiffer, Manuela Scheffel, Yael Stern and Frauke Stobbe for their valueable support behind the scenes.


This work was supported by the Deutsche Forschungsgemeinschaft [BE 6789/1-1] and the Maria-Weber Grant 2022 offered by the Hans B\"ockler Stiftung.
The research was also supported through the program ``Research in Pairs'' by the Mathematisches Forschungsinstitut Oberwolfach in 2018. 
I acknowledge the support of the Israel Science Foundation (grants No. 970/15) founded by the Israel Academy of Sciences and Humanities.

%% file: discussion.tex
\chapter{Outlook}
\label{chap:discussion}

Throughout this work, we have highlighted several open questions, which we briefly revisit here. The primary focus lies on potential spectral applications in higher-dimensional settings. We also discuss possible future uses of the spectral tree construction developed for the Kohmoto model. In certain contexts, it may be of independent interest to investigate extensions of the presented results to broader classes of operators, more general dynamical systems, or frameworks beyond the specific setting studied in this work. The topics are organized according to the chapter structure of this thesis.

\section*{Linearly repetitive Delone sets}

In \cite{BecHarPog21}, we constructed a class of symbolic dynamical systems over non-Abelian groups that exhibit linear repetitivity. To the best of our knowledge, these represent the first non-trivial examples in this setting. From these symbolic systems, one can also derive linearly repetitive Delone sets.

An interesting question is whether there exist geometric constructions of linearly repetitive Delone sets in homogeneous Lie groups that do not arise from symbolic substitutions. One possible approach would be to formulate a suitable notion of geometric substitutions for these groups. 

Another promising direction lies in the theory of cut-and-project sets, where linear repetitivity is well understood in the Euclidean case \cite{HaKoWa18,KoWa22}. It remains to be seen whether similar arguments can be extended to non-Abelian settings.

Compared to substitution systems, the greater structural flexibility of cut-and-project sets offers additional possibilities for constructing periodic approximations. Following \cite{DMO89}, periodic approximations were provided for the octagonal tiling by exploiting its substitutive structure; see the discussion in Section~\ref{sec:Octagonal_Penrose}. The initial periodic configuration was constructed using a symmetry of the underlying cut-and-project set. It remains an open question whether a general approximation scheme can be developed for a broader class of cut-and-project sets. In particular, it is unknown whether this approach can yield periodic approximations for the Penrose tiling.

If not, one should study the nature of possible dynamical impurities and spectral defects in analogy to our analysis of the Kohmoto model; see \cite{BecBelTho25} and the discussion in Section~\ref{sec:Kohmoto}. Recent developments \cite{HaKoWa18,KoWa22} providing characterizations of linear repetitivity for regular model sets may serve as a promising starting point for further investigations.

\section*{Higher-dimensional models}

A central motivation for the approximation method developed in this work is the hope that it provides access to the spectral theory of higher-dimensional aperiodic Schr\"odinger operators, where currently only limited results are available. A natural next step is to investigate which additional spectral properties can be recovered or passed to the limit. 

More precisely, let \( A_n \) be a sequence of dynamically defined operators. We have studied whether the spectra \( \sigma(A_n) \) converge to the spectrum \( \sigma(A) \) of a limiting operator \( A \), as discussed in Chapter~\ref{chap:WorldButterflies}. In this context, each operator \( A_n \) (respectively \( A \)) is defined via a kernel on a dynamical system \( X_n \) (respectively \( X \)), and we assume that the sequence of dynamical systems \( (X_n)_{n\in\N} \) converges to \( X \).

However, spectral convergence alone does not generally imply convergence of other spectral properties. For instance, convergence of the spectra does not guarantee convergence of the Lebesgue measure. In fact, only the inequality
\[
\limsup_{n \to \infty} \lambda(\sigma(A_n)) \leq \lambda(\sigma(A))
\]
holds in general where $\lambda$ denotes the Lebesgue measure. Additional structure, such as monotonic convergence (discussed below), can imply convergence of the Lebesgue measure. Alternatively, quantitative control on spectral convergence combined with information on the number of spectral gaps may also lead to convergence of the Lebesgue measure. A particular application can be found in \cite[Thm.~2]{Last93} and \cite[Thm.~1]{Last94} for the almost Mathieu operator, see also a discussion in \cite{Ten25-Phd}.

To explore these questions, we initiated a joint project \cite{BaBePoTe25-Spec} using numerical computations. These computations indicate new phenomena that have not been observed in one-dimensional models:

\begin{itemize}
\item Consider the table tiling and its periodic approximations as constructed in Proposition~\ref{prop:PeriodAppr_Table}~(a), where the associated dynamical systems converge. The spectral bands rapidly become extremely thin. Numerically, the spectrum appears to vanish after only a few iterations. If one could confirm that the Lebesgue measures of the spectra converge and that the limit has zero measure, this would imply that the spectrum of the associated Schr\"odinger operator has zero Lebesgue measure as well.

\item In contrast, the single-letter approximations discussed in Proposition~\ref{prop:PeriodAppr_Table}~(b) yield dynamical systems that do not converge. In the numerics, these approximations lead to substantially thicker spectra, raising the question of how large the deviation in the limit can be when using poorly chosen approximations.
\end{itemize}

In the forthcoming work \cite{BaBePoTe25-Spec} (see also \cite{Ten25-Phd}), we demonstrate that for specific Schr\"odinger operators, the limiting essential spectra may differ if one takes poor approximations. In one-dimensional systems, only compact perturbations are possible, so differences in the limit can affect at most the discrete spectrum. In higher dimensions, however, we show that the limiting Lebesgue measure itself may even change.

Our methods build on the computation of spectral limits and identification of dynamical impurities, in the spirit of \cite{BecBelTho25} for the Kohmoto model; see also Section~\ref{sec:Kohmoto}. In higher dimensions, these impurities may take the form of infinite strips, which can perturb the essential spectrum and even alter its Lebesgue measure.

These observations mark only the beginning of a range of interesting open problems.

\section*{Monotonicity of the Spectrum}

In Chapter~\ref{chap:DTMP}, we studied the classification of spectral bands arising from approximations of a specific aperiodic Schr\"odinger operator. This classification was based on tracking whether spectral bands from a given approximation were contained in those of earlier ones; see Proposition~\ref{prop:MonotonConv_Spectrum_Kohmoto}. While this provides a first structural insight, a more detailed understanding of the spectral bands requires additional information; see \cite{Raym95,Raym95-thesis,BaBeBiRaTh24,BaBeLo24}. Nevertheless, the observed monotonicity of the spectrum, i.e.,
\[
\sigma(A_n) \subseteq \sigma(A_{n-1}) \cup \sigma(A_{n-2}),
\]
serves as a key guiding principle at an initial stage.

Analogous hierarchical structures and codings have been investigated for other models, such as the period-doubling chain~\cite{LiQuYa22}.

In Chapter~\ref{chap:WorldButterflies}, we observed that spectral convergence is tightly linked to the convergence of the underlying dynamical systems. This naturally leads to the question of whether one can establish monotonicity of the spectrum in a more abstract dynamical framework. More precisely, let \( (Z, G, \tau) \) be a dynamical system, and let \( \inv \) denote the space of invariant dynamical subsystems of \( Z \). Can one find a partial order relation \( \prec \) on \( \inv \) such that
\[
X \prec Y \quad \Longrightarrow \quad \sigma(A_X) \subseteq \sigma(A_Y),
\]
where \( (A_x)_{x \in Z} \) is a dynamically-defined operator family?

A quick but trivial answer would be to define \( X \prec Y \) by \( X \subseteq Y \), in which case the above implication holds by definition of the spectrum. However, the question should be understood as seeking nontrivial order structures that induce monotonicity of the associated spectra beyond such elementary inclusions. 

A natural starting point is to analyze substitution systems, where the recursive definition of the approximations via the substitution rule may lead to spectral inclusions.

\section*{Spectral estimates beyond operators on lattices}

In Example~\ref{ex:Graph_Delone}, we described how to construct a Delone graph associated with a given Delone set. We observed that the corresponding graph Schr\"odinger operator almost fits into the framework of dynamically-defined operators, except that it acts on the Hilbert space $\ell^2(D)$ over the Delone set $D$, which typically does not possess a group structure. 

To study such operators, a natural framework is provided by groupoids, as used in \cite{BBdN18}. A central open question is whether the spectral estimates obtained in \cite{BeBeCo19,BecTak25} can be extended to this more general setting. This requires techniques for comparing operators acting on different Hilbert spaces, as a change of the Delone set alters the underlying $\ell^2$-space. Preliminary investigations with Horia Cornean suggest that such a comparison might be feasible.

With such tools at hand, one could revisit the question of periodic approximations for well-known aperiodic tilings such as the octagonal tiling or the Ammann-Beenker tiling (see Section~\ref{sec:Octagonal_Penrose}). Moreover, there exist many additional substitution-based models to which the approximation results developed in \cite{BaBePoTe25}, discussed in Section~\ref{sec:Subst-Approximations}, could be extended.

\section*{Improving spectral estimates when gaps do not close}

In Chapter~\ref{chap:WorldButterflies}, we discussed that the spectral map associated with a family of dynamically-defined operators typically exhibits a square root behavior; see also Remark~\ref{rem:Choice-cut_Interplay-amenable}. More precisely, if \( (A_z)_{z \in Z} \) is a family of self-adjoint operators with Lipschitz continuous kernel, then Theorem~\ref{thm:SpecEst_DynDefOp}~(a) implies
\[
d_H(\sigma(A_X), \sigma(A_Y)) \leq C \, \delta_H(X,Y)^{\tfrac{1}{2}},
\]
i.e., although the kernel is Lipschitz continuous, the spectral map is only \( \tfrac{1}{2} \)-H\"older continuous. This reduced regularity originates from the possible closing of spectral gaps~\cite{BB16}.

More precisely, let \( (X_n)_{n \in \N} \subseteq \inv \) be a sequence converging to \( X \in \inv \) in the Hausdorff topology. A spectral gap of \( A_X \) is an open interval \( g = (a,b) \subseteq \R \) such that \( a, b \in \sigma(A_X) \) and \( (a,b) \cap \sigma(A_X) = \emptyset \). By continuity of the spectrum with respect to the dynamical system, there exists a sequence of spectral gaps \( g_n = (a_n, b_n) \subseteq \R\) of \( A_{X_n} \) such that
\[
|a_n - a| \leq C \, \delta_H(X_n,X)^{\tfrac{1}{2}}, \qquad |b_n - b| \leq C \, \delta_H(X_n,X)^{\tfrac{1}{2}}.
\]
In particular, \( \lim_{n \to \infty} |a_n - b_n| > 0 \), so the gaps do not close in the limit.

This raises the natural question whether the above estimates can be improved under the additional assumption that the limiting gap \( g = (a,b) \) remains open. More precisely, can one prove that
\[
|a_n - a| \leq \tilde{C} \, \delta_H(X_n,X), \qquad |b_n - b| \leq \tilde{C} \, \delta_H(X_n,X),
\]
for some constant \( \tilde{C} > 0 \)?

This question is motivated by several observations. First, when spectral gaps close (i.e., \( |a_n - b_n| \to 0 \)), the convergence of the spectrum slows down -- a phenomenon exploited in the analysis of the almost Mathieu operator~\cite{BeRa90} and the Kohmoto model~\cite{BecBelTho25} to show optimality of the estimate in Theorem~\ref{thm:SpecEst_DynDefOp}; see also the discussions in Sections~\ref{sec:Hofstadter} and Section~\ref{sec:Kohmoto}. This reduction in spectral regularity is also discussed in~\cite{BB16}.

On the other hand, using \( C^* \)-algebraic methods, it was shown in~\cite{Bell94} that this square root behavior disappears if spectral gaps do not close, at least for the rotation algebra including the almost Mathieu operator. This raises the question of whether such an improvement is a general phenomenon. In particular, it would be interesting to explore whether the techniques developed in~\cite{BeBeCo19} or~\cite{BecTak25} can be refined to yield Lipschitz spectral estimates under the assumption that gaps remain open.

\section*{Examples beyond the discrete case}

In the discussion following Theorem~\ref{thm:SpecEst_DynDefOp}, we exclusively considered dynamically-defined operator families with a discrete underlying group. It is a natural and intriguing question whether the general framework also extends to continuous settings, since we allow the underlying group to be non-discrete in general. 

To be more specific, let us consider the continuous group $G=\R^n$ and the Laplace operator $-\Delta$. Fix parameters $r_0, R_0 > 0$ and consider the compact metric space $\Del$ of Delone sets as introduced in Section~\ref{sec:Weight_Delone}. Given a continuous potential function $V : \Del \to \R$ (which is automatically bounded due to compactness), we define for each $D \in \Del$ the Schr\"odinger operator
\[
H_D := -\Delta + V_D,
\]
where the multiplication operator \( V_D \) is given by
\[
(V_D \psi)(x) := V(-x + D)\, \psi(x).
\]
We omit the technical details regarding domains, which can be defined via suitable sesquilinear forms as in Chapter~\ref{chap:GeneralEigenfct}, using a completion of smooth, compactly supported functions on \( \R^n \). Since \( V_D \) is uniformly bounded in \( D \in \Del \) and \( -\Delta \) is nonnegative, the operator \( \exp(-H_D) \) is well defined and bounded on \( L^2(\R^n) \).

By the Dunford-Pettis criterion for kernels, \( \exp(-H_D) \) is an integral operator on \( \R^n \), and hence possesses an integral kernel \( k_D \). A direct computation shows that the family of integral kernels \( (k_D)_{D \in \Del} \) arises from a bounded kernel \( k:\R^n\times\Del \to\R \) defined on the dynamical system \( (\Del, \R^n, \tau) \). Moreover, classical heat kernel estimates provide quantitative decay bounds on \( k \) in \( \R^n \).

The open question is whether Lipschitz continuity of the function \( V:\Del \to \R \) implies that the associated kernel \( k \) is also Lipschitz continuous in the sense of Definition~\ref{def:Regularity_kernel}. If this is the case -- or even under a weaker assumption -- then Theorem~\ref{thm:SpecEst_DynDefOp} would apply to the operator family \( (\exp(-H_D))_{D \in \Del} \), enabling spectral estimates for the corresponding Schr\"odinger operators.

This approach would also yield local spectral estimates for the original operators \( H_D \), see the related discussion in \cite[Sec.~2.5]{Bec16}.

\section*{The Kohmoto butterfly: Coloring and self-similarity}

Following Chapter~\ref{chap:DTMP}, the resolution of the dry ten Martini Problem (DTMP) for Sturmian Hamiltonians~\cite{BaBeLo24} opens up various further questions. As mentioned earlier, the Kohmoto butterfly (and similar spectral figures) is numerically generated from periodic approximations, i.e., for rational values of the frequency parameter. In contrast, the DTMP addresses the spectral properties of irrational rotations. 

In~\cite{BanBec25}, we propose a coloring scheme for the Kohmoto butterfly (see Figure~\ref{fig:ColoredKohmoto}), as briefly discussed at the end of Section~\ref{sec:DTMP_Sturmian}. The labeling of spectral gaps for rational parameters is chosen to be consistent with any continuation of their continued fraction expansion, yielding a diagram structurally analogous to that of the Hofstadter butterfly~\cite{OsaAvr01}. This construction raises the question of whether a winding number can be associated with these gap labels. Initial steps in this direction have been made in~\cite{KelPro17,KelSca25}.

The symbolic coding developed by Raymond~\cite{Raym95} has been applied in the literature (see e.g. \cite{KiKiLa03,LiuWen04,DaEmGoTc08,DamGor11,LiQuWe14,DamaGor15,CaoQu23,Lun25}) to estimate the fractal dimensions of spectra. Since the coding structure --- represented in the spectral tree --- is now available for all coupling constants, a natural question is whether this framework can be used to study fractal dimension properties more systematically, particularly for small coupling constants. It is worth noting, however, that existing estimates often rely on further estimations that also require the coupling constant to be sufficiently large.

As already hinted at in the introduction, these butterfly plots suggest an intrinsic self-similarity of the underlying spectra. Proving such self-similarity rigorously remains a longstanding open problem. In a joint project with Ram Band and Yannik Thomas, we address this question using the techniques developed in~\cite{BaBeLo24,BaBeBiRaTh24}. It would also be of interest to investigate whether these methods can be extended to other models, most notably the Hofstadter butterfly. A potential starting point is the nested covering constructed in~\cite{HelSjo88,HelLiuQuZho25}.

%% file: appendix.tex
\chapter{Appendix}

\section{Weighted Delone sets}
\label{App:Weight_Delone}

In this section, we provide the proofs of the statements presented in Section~\ref{sec:Weight_Delone}.  
We recall the relevant notions introduced therein, including weighted Delone sets $\Del$ and the space $\wPat$ of $r_0$-uniformly discrete sets (see page~\pageref{page:wPat}) defined on a locally compact second countable (lcsc) group $G$ equipped with an adapted metric $d_G$.

We denote by $B(x, r)$ the open ball of radius $r > 0$ centered at $x \in G$, and write $B(r) := B(e, r)$ for the ball centered at the identity element $e \in G$. Furthermore, $C_c(G)$ denotes the space of continuous functions with compact support in $G$.

For $\Pi:=(P,\alpha)\in\wPat$, the weighted dirac comb on the Borel $\sigma$-algebra $\bs(G)$ of $G$ is defined in eq.~\eqref{eq:DiracComp_measure} by
\begin{equation}
\label{eq:DiracComp_measure-appendix}
\mu_{\Pi}:\bs(G)\to[0,\infty], \qquad
\mu_{\Pi} := \sum_{x\in P} \alpha(x) \delta_x.
\end{equation}
Clearly, $\mu_\Pi$ is a Radon measure (regular Borel measure on $\bs(G)$). The support of the measure $\mu_{\Pi}$ satisfies $\supp(\Pi):=\supp(\mu_\Pi)=P$.
For $\Pi,\Xi\in\wPat$, define (see  eq.~\eqref{eq:delta_weightedPatches})
\begin{equation}
\label{eq:wdelt-appendix}
\wdelt(\Pi,\Xi) 
	:= \inf\set{\varepsilon>0}{ 
		\begin{array}{c}
		\big|\mu_\Pi(B(x,\varepsilon)) - \mu_\Xi(B(x,\varepsilon))\big|<\varepsilon\\[0.2cm]
		\textrm{ for all } x\in B(1/\varepsilon)\cap \big( \supp(\Pi)\cup \supp(\Xi) \big)
		\end{array} },
\end{equation}
where the infimum is infinite, if there is no such $\varepsilon$. If $\varepsilon<r_0$, then $\mu_\Pi(B(x,\varepsilon))>0$ then $B(x,\varepsilon)\cap \supp(\Pi)$ contains exactly one point.
We restate and prove Proposition~\ref{prop:metric_weightDelone} where the following proof is inspired by \cite{Sol98,Sol06}.

\begin{proposition}
\label{prop:metric_weightDelone-App}
Let $G$ be a lcsc group, $r_0,R_0>0$ and $\sigma\geq 1$. Then $(\wPat,\dw)$ is a compact metric space where
\[
\dw:\wPat\times\wPat\to[0,\infty),\quad
	\dw(\Pi,\Xi):= \min\big\{\tfrac{r_0}{2},\tfrac{\sigma^{-1}}{2},\wdelt(\Pi,\Xi) \big\}
\]
induces the weak* topology on $\wPat$ viewing every element $\Pi\in\wPat$ as the Radon measure $\mu_\Pi$ defined in \eqref{eq:DiracComp_measure-appendix}.
Furthermore $\wDel\subseteq\wPat$ is a closed subset and so $(\wDel,\dw)$ is a compact metric space.
\end{proposition}

\begin{proof}
\underline{$\dw$ is a metric:} Clearly, $\dw$ is symmetric and $\dw(\Pi,\Xi)=0$ holds if and only if $\Pi=\Xi$. It is left to show the triangle inequality. 

Let $\Pi,\Xi,\Theta\in\wPat$ and we prove
$$
\dw(\Pi,\Xi) \leq \underbrace{\dw(\Pi,\Theta)}_{=:a'} + \underbrace{\dw(\Theta,\Xi)}_{=:b'}.
$$
Set $c_0:=\min\big\{\tfrac{r_0}{2},\tfrac{\sigma^{-1}}{2}\big\}$. Since the range of $\dw$ is bounded by $c_0$, there is no loss of generality in assuming $a'+b'<c_0$. Choose $0<\tepsilon<c_0-(a'+b')<c_0$ and set
$$
	a:=a'+\tfrac{\tepsilon}{2} 
\qquad \textrm{and}\qquad 
	b:=b'+\tfrac{\tepsilon}{2}.
$$
Observe that $a<c_0<\frac{1}{\sqrt{2}}$ and $b<c_0<\frac{1}{\sqrt{2}}$ since $\sigma^{-1}<1$. Thus,
\begin{equation}
\label{eq:proof_metric_weightDelone}
\frac{1}{a+b}\leq \frac{1}{a}-b
	\qquad\textrm{and}\qquad
\frac{1}{a+b}\leq \frac{1}{b}-a
\end{equation}
follow implying
$$
B\big(\tfrac{1}{a+b}\big)\subseteq B\big(\tfrac{1}{a}-b\big)
	\qquad\textrm{and}\qquad
B\big(\tfrac{1}{a+b}\big)\subseteq B\big(\tfrac{1}{b}-a\big).
$$
Since $\dw(\Pi,\Theta)=a'< a$, we conclude
$$
\big|\mu_\Pi(B(x,a)) - \mu_\Theta(B(x,a))\big|<a, \qquad
	\textrm{for all } x\in B\big(\tfrac{1}{a}\big)\cap \big( \supp(\Pi)\cup \supp(\Theta) \big).
$$
Let $x\in B\big(\tfrac{1}{a+b}\big)\cap \supp(\Pi) \subseteq B\big(\frac{1}{a}-b\big) \cap \supp(\Pi)$. Then the previous considerations and $a<c_0\leq \frac{\sigma^{-1}}{2}$ lead to
$$
\mu_\Theta\big(B(x,a)\big) 
	\geq \underbrace{\mu_\Pi\big(B(x,a)\big)}_{=\mu_\Pi(\{x\})\geq \sigma^{-1}} - a 
	>\frac{\sigma^{-1}}{2} 
	>0.
$$
Hence, there is a unique $y=y(x)\in B(x,a)\cap \supp(\Theta)$ where uniqueness follows from $a\leq \tfrac{r_0}{2}$ and that $\supp(\Theta)$ is $r_0$-uniformly discrete. Using \eqref{eq:proof_metric_weightDelone}, we conclude
$$
d_G(e,y)
	\leq d_G(e,x) + d_G(x,y)
	< \frac{1}{a+b} + a
	\leq \frac{1}{b}.
$$
Thus, $y=y(x)\in B\big(\tfrac{1}{b}\big)\cap \supp(\Theta)$.
Since $\dw(\Theta,\Xi)=b'< b$, 
$$
\big|\mu_\Theta\big(B(y(x),b)\big) - \mu_\Xi\big(B(y(x),b)\big)\big|< b.
$$
Using again $b<c_0\leq \frac{\sigma^{-1}}{2}$, we conclude $\mu_\Xi\big(B(y(x),b)\big)>\tfrac{\sigma^{-1}}{2}>0$. Since $b\leq \tfrac{r_0}{2}$ and $\supp(\Xi)$ is $r_0$-uniformly discrete, there is a unique $z=z(y,x)\in B(y(x),b)\cap \supp(\Xi)$. 
Then
$$
d_G(x,z)
	\leq d_G(x,y) + d_G(y,z) 
	= a+b
$$
follows. Since the support of $\mu_\Pi,\mu_\Theta$ and $\mu_\Xi$ are all $r_0$-uniformly discrete and $a+b < r_0$, we have
\begin{align*}
&\mu_\Pi\big(B(x,a+b)\big) = \mu_\Pi(\{x\}),\\ 
&\mu_\Theta\big(B(x,a+b)\big) = \mu_\Theta(\{y(x)\}),\\
&\mu_\Xi\big(B(x,a+b)\big) = \mu_\Xi(\{z(y,x)\}).
\end{align*} 
Summing up, we get for each $x\in B\big(\tfrac{1}{a+b}\big)\cap \supp(\Pi)$,
\begin{align*}
&\left| \mu_\Pi\big(B(x,a+b)\big) - \mu_\Xi\big(B(x,a+b)\big) \right|\\
	\leq &\left| \mu_\Pi(\{x\}) - \mu_\Theta(\{y(x)\}) \right| + \left| \mu_\Theta(\{y(x)\}) - \mu_\Xi(\{z(y,x)\}) \right|\\
	< &a+b.
\end{align*}
Switching the role of $\Pi$ and $\Xi$, we yield for all $x\in B\big(\tfrac{1}{a+b}\big)\cap \supp(\Xi)$,
$$
\left| \mu_\Pi\big(B(x,a+b)\big) - \mu_\Xi\big(B(x,a+b)\big) \right|
	< a+b.
$$
This shows $\wdelt(\Pi,\Xi) <a+b <c_0$ implying $\wdelt(\Pi,\Xi) = \dw(\Pi,\Xi)$ and
$$
\dw(\Pi,\Xi)
	<a+b
	= \dw(\Pi,\Theta) + \dw(\Theta,\Xi) + \tepsilon.
$$
Sending $\tepsilon$ to zero implies $\dw(\Pi,\Xi) \leq \dw(\Pi,\Theta) + \dw(\Theta,\Xi)$ as claimed.

\medskip

\underline{$\dw$ induces the weak* topology:} Recall that the weak*-topology on the Radon measures $\mu$ on $G$ is the coarsest topology so that $\mu\mapsto \mu(\varphi)$ is continuous for all $\varphi\in C_c(G)$. By definition of $\dw$ it is elementary to see that $\dw$ induces the weak*-topology see also a detailed elaboration in \cite[Prop.~B.1]{BecHarPog25}.
We refer the reder also to \cite{BjoHarPog22} where this topology is compared to the Chabauty-Fell topology if $\sigma=1$.

\medskip

\underline{$(\wPat,\dw)$ is compact:} Following \cite{BaLe04,BjoHarPog22}, a Radon measure $\mu$ on $G$ is called $(C,V)$-translation bounded if $C>0$, $V\subseteq G$ is a relatively compact open subset and 
$$
|\mu|(gV)\leq C,\qquad \textrm{for all } g\in G,
$$
where $|\mu|$ denotes the total variation of the measure $\mu$. Then \cite[Thm.~2]{BaLe04} respectively \cite[Lem.~3.5]{BjoHarPog22} imply that the space $\mathcal{M}_{C,V}(G)$ of $(C,V)$-translation bounded measures (for fixed $C$ and $V$) is compact in the weak*-topology. Note that in \cite{BaLe04} the group $G$ is assumed to be a $\sigma$-compact locally compact Abelian group. The presented result is only based on fundamental functional analysis tools and is not using that the group is Abelian. Using \cite[Prop.~2.A.10]{CorHar16}, we conclude that the group $G$ that we consider is also $\sigma$-compact and so $\mathcal{M}_{C,V}(G)$ is compact in our situation. 

Set $C:= \sigma$ and $V:=B(r_0)$ be the open ball of radius $r_0$ around $e\in G$. 
Since $d_G$ is a proper metric $V$ is relatively compact and open. 
For $\Pi\in\wPat$ and $g\in G$, we have $|\mu_\Pi|(gV)\leq C$ by construction of the measure $\mu_\Pi$ and using that the support $\supp(\mu_\Pi)$ is $r_0$-uniformly discrete. 
Hence, via the identification $\wPat \ni\Pi \mapsto \mu_\Pi$ we embed $\wPat$ into $\mathcal{M}_{C,V}(G)$ and this embedding is a homeomorphism (onto its image) by the previous considerations stating that $\wPat$ equipped with the topology inherited from $\dw$ coincides with the weak*-topology on this embedding. To simplify notation we view $\wPat$ as a subset of $\mathcal{M}_{C,V}(G)$ via this embedding. Since $\mathcal{M}_{C,V}(G)$ is compact, the metric space $(\wPat,\dw)$ is compact if $\wPat$ is a closed subset of $\mathcal{M}_{C,V}(G)$.
For colored point sets this is proven in \cite[Sec.~2.1, Prop.~3.5]{MueRic13}. For completeness, we provide the details here.

Let $\Pi_n\in\wPat, n\in\N,$ be such that $\mu_n:=\mu_{\Pi_n}$ converges in the weak*-topology to $\mu\in \mathcal{M}_{C,V}(G)$ and we show $\mu=\mu_\Pi$ for a suitable $\Pi\in\wPat$.  

We first prove that $\supp(\mu)$ is $r_0$-uniformly discrete. Assume towards contradiction that this is not the case. Thus, there is a $g\in G$ and $x,y\in B(g,r_0)$ with $x,y\in\supp(\mu)$ and $x\neq y$. By Urysohn's lemma there are two continuous, non-negative functions $f_x,f_y$ with $\supp(f_x)\subseteq B(g,r_0)$, $\supp(f_y)\subseteq B(g,r_0)$, $\supp(f_x)\cap \supp(f_x)=\emptyset$ and $f_x(x)=1=f_y(y)$. Since $B(g,r_0)\cap\supp(\mu_n)$ is a singleton for each $n\in\N$, it cannot intersect both supports of $f_x$ and $f_y$ simultaneously. Hence, there is no loss of generality in assuming that there is a subsequence $\mu_{n_k}$ such that $\mu_{n_k}(f_x)=0$ for all $k\in\N$. Consequently, 
$$
0<\mu(f_x)=\lim_{n\to\infty} \mu_{n_k}(f_x) = 0,
$$
a contradiction. 
Thus, $\supp(\mu_n)\subseteq G$ is $r_0$-uniformly discrete implying that $\mu$ is a dirac measure. 

Define $P:=\supp(\mu)$ and $\alpha:\supp(\mu)\to\R$ by $\alpha(x):=\mu(\{x\})$. It is left to show that the range of $\alpha$ is contained in $[\sigma^{-1},\sigma]$.
Therefore, let $x\in \supp(\mu)$ and choose $f\in C_c(G)$ such that $f(x)=1$, $0\leq f\leq 1$ and $\supp(f)\subseteq B(x,r_0)$. 
Since $\mu_n$ converges in the weak*-topology to $\mu$, there exists an $n_0\in\N$ such that $\mu_n(f)>0$ for all $n\geq n_0$.
Then, $\mu_n(f)\in[\sigma^{-1},\sigma]$ for all $n\geq n_0$, lead to
$$
\mu(\{x\}) = \mu(f)-\mu_n(f) + \mu_n(f) 
	\leq \mu(f)-\mu_n(f) + \sigma 
	\overset{n\to\infty}{\longrightarrow} \sigma
$$
and
$$
\sigma^{-1} \leq \mu_n(f)-\mu(f) + \mu(f) 
	\overset{n\to\infty}{\longrightarrow} \mu(f) 
	= \mu(\{x\}).
$$
Thus, $\alpha(x)=\mu(\{x\})\in[\sigma^{-1},\sigma]$ follows proving that $(P,\alpha)\in\wPat$ and $\mu=\mu_{(P,\alpha)}$.

\medskip

\underline{$(\wPat,\dw)$ is a complete metric space:} Since $(\wPat,\dw)$ is a compact metric space, a generalization of the Heine-Borel theorem yields that $(\wPat,\dw)$ is also a complete metric space.

\underline{$(\wDel,\dw)$ is compact metric space:} Since $(\wPat,\dw)$ is a compact metric space, it suffices to prove that $\wDel\subseteq\wPat$ is a closed subset. Let $\Pi_n\in\wDel, n\in\N,$ and $\Pi\in\wPat$ be such that $\Pi_n\to\Pi$. We prove that $\Pi\in\wDel$, namely $\supp(\Pi)\subseteq G$ is $R_0$-relatively dense.

Assume towards contradiction that $\supp(\Pi)$ is not $R_0$-relatively dense, namely $\supp(\Pi)\cap\oB(g,R_0)=\emptyset$ for some $g\in G$. Then there is an open set $U\subseteq G$ such that $\oB(g,R_0)\subseteq U$ and $\mu(U)=0$. By Urysohn's lemma, there is a continuous function $f$ with $\supp(f)\subseteq U$, $0\leq f\leq 1$ and $f|_{\oB(g,R_0)}\equiv 1$. For $n\in\N$, $\supp(\Pi_n)\cap \oB(g,R_0)\neq \emptyset$ since $\Pi_n\in\wDel$ implying $\mu_{\Pi_n}(f)\geq \sigma^{-1}$. Thus,
$$
0=\mu_\Pi(f) = \lim_{n\to\infty} \mu_{\Pi_n}(f) \geq \sigma^{-1}>0
$$
follows, a contradiction.
\end{proof}

Next, we prove that this group actions is actually Lipschitz continuous (see Definition~\ref{def:LipschitzAction}) stated in Proposition~\ref{prop:wDel_groupAction_Lipschitz}. 

\begin{proposition}
\label{prop:wDel_groupAction_Lipschitz-App}
Let $G$ be a lcsc group, $r_0,R_0>0$ and $\sigma\geq 1$.
Then $(\wPat,G,\tau)$ and $(\wDel,G,\tau)$ are a topological dynamical system. Moreover, the group action is Lipschitz continuous, namely for all $g\in G$ and $\Pi,\Xi\in\wPat$, we have
$$
\dw(g.\Pi,g.\Xi) \leq (d_G(e,g)+1) \dw(\Pi,\Xi).
$$
\end{proposition}

\begin{proof}
By Proposition~\ref{prop:metric_weightDelone}, $(\wPat,\dw)$ and  $(\wDel,\dw)$ are compact metric spaces. Thus, if we prove the claimed Lipschitz continuity, then the statement follows using $\wDel\subseteq \wPat$.

Let $g\in G$ and $\Pi,\Xi\in\wPat$.
To shorten the notation set $|g|:=d_G(e,g)\geq 0$ and $a:= \dw(\Pi,\Xi)$. Since $\dw$ is a metric, there is no loss of generality in assuming $\dw(\Pi,\Xi)>0$. Since $\dw(g.\Pi,g.\Xi)\leq c_0$, we can assume without loss of generality that $0<(|g|+1) a<c_0$. Choose $\varepsilon>0$ such that $0<b:=(|g|+1) (a+\varepsilon)<c_0$. 
We first prove that for all $x\in B\big(\tfrac{1}{b}\big) \cap \big( \supp(g.\Pi)\cup \supp(g.\Xi)\big)$, we have
\begin{equation}
\label{eq:proof_WeightDelone_Action_Lipschitz}
\left| \mu_{g.\Pi}\big(B(x,b)\big) - \mu_{g.\Xi}\big(B(x,b)\big) \right| < b.
\end{equation}
If $\Pi=(D,\alpha)$, then $g.\Pi=(gD,\alpha\circ g^{-1})$ and $\supp(g.\Pi)=gD = g\supp(\Pi)$. Furthermore, a short computation yields
$$
\mu_{g.\Pi} 
	= \sum_{x\in gD} \alpha(g^{-1}x) \ \delta_x
	= \sum_{g^{-1}x\in D} \alpha(g^{-1}x) \ \delta_{gg^{-1}x}\\
	= \sum_{y\in D} \alpha(y) \ \delta_{y}\circ g^{-1}
	= \mu_\Pi \circ g^{-1}.
$$
Let $x\in B\big(\tfrac{1}{b}\big) \cap \supp(g.\Pi)$. By the previous considerations, we have $g^{-1}x\in\supp(\Pi)$. Then the left-invariance of the metric $d_G$ and $b<c_0<1$ lead to
$$
d_G(g^{-1}x,e) 
	= d_G(x,g) 
	\leq d_G(x,e) + d_G(e,g)
	\leq \frac{1}{b} + \frac{|g|}{b}
	= \frac{1}{a+\varepsilon}.
$$
Thus, $g^{-1}x\in B\big(\tfrac{1}{a+\varepsilon}\big) \cap \supp(\Pi)$ follows. Since $a+\varepsilon<c_0$, we have $\wdelt(\Pi,\Xi)=\dw(\Pi,\Xi)<a+\varepsilon$ where $\wdelt(\Pi,\Xi)$ is defined in \eqref{eq:wdelt-appendix}. Thus, 
$$
\left| \mu_\Pi\big(B(g^{-1}x,a+\varepsilon)\big) - \mu_\Xi\big(B(g^{-1}x,a+\varepsilon)\big) \right| < a+\varepsilon <c_0 \leq \frac{\sigma^{-1}}{2}
$$
and so $g^{-1}x\in \supp(\Pi)$ implies
$$
\mu_\Xi\big(B(g^{-1}x,a+\varepsilon)\big) 
	\geq \underbrace{\mu_\Pi\big(B(g^{-1}x,a+\varepsilon)\big)}_{\geq \sigma^{-1}} - (a+\varepsilon)
	\geq \frac{\sigma^{-1}}{2}
	>0.
$$
Since $a+\varepsilon<r_0$ and $\supp(\Xi)$ is $r_0$-uniformly discrete, there is a unique element $y\in B(g^{-1}x,a+\varepsilon)\cap \supp(\Xi)$. Since $(|g|+1)(a+\varepsilon)<c_0<r_0$, we conclude
$\mu_\Xi\big(B(g^{-1}x,b)\big) = \mu_\Xi(\{y\})$. Combining the previous considerations, we deduce
\begin{align*}
\left| \mu_{g.\Pi}\big(B(x,b)\big) - \mu_{g.\Xi}\big(B(x,b)\big) \right| 
	= &\left| \mu_{\Pi}\big(B(g^{-1}x,b)\big) - \mu_{\Xi}\big(B(g^{-1}x,b)\big) \right| \\
	= &\left| \mu_{\Pi}(\{g^{-1}x\}) - \mu_{\Xi}(\{y\}) \right| \\
	= &\left| \mu_\Pi\big(B(g^{-1}x,a+\varepsilon)\big) - \mu_\Xi\big(B(g^{-1}x,a+\varepsilon)\big) \right|\\
	< &a+\varepsilon
	< b.
\end{align*}
The case $x\in B\big(\tfrac{1}{b}\big) \cap \supp(g.\Xi)$ is treated similarly.

We have proven \eqref{eq:proof_WeightDelone_Action_Lipschitz} for all $x\in B\big(\tfrac{1}{b}\big) \cap \big( \supp(g.\Pi)\cup \supp(g.\Xi)\big)$ where $0<b:=(|g|+1) (a+\varepsilon)<c_0$. 
Hence, 
$$
d(g.\Pi,g.\Xi) = \wdelt(g.\Pi,g.\Xi)
	\leq b 
	= (|g|+1) (a+\varepsilon).
$$
Thus the desired claim follows by sending $\varepsilon$ to zero and using $a:= \dw(\Pi,\Xi)$.
\end{proof}

We now present an alternative proof of Proposition~\ref{prop:Homeo_hulls_transversals}, which characterizes the convergence of weighted Delone dynamical systems in terms of the convergence of the associated transversals; see also \cite[Thm.~2.A]{BecPog20} for the unweighted case. To this end, recall that \(\wInvT\) denotes the space of all transversals of elements in \(\wDel\), and \(\wInv\) the space of all (weighted) Delone dynamical systems.

\begin{proposition}
\label{prop:Homeo_hulls_transversals_Appendix}
Let $R_0>r_0>0$ and $\sigma\geq 1$.
The map 
\[
\Phi:\wInvT\to\wInv, \Tt\mapsto G.\Tt=\oB(R_0).\Tt
\] 
is a homeomorphism. Moreover, $\Phi$ is bi-Lipschitz, namely
\[
\frac{1}{2}\delta_H(\Tt_1,\Tt_2) 
	\leq \delta_H\big( \Phi(\Tt_1),\Phi(\Tt_2) \big)
	\leq (R_0+1)\delta_H(\Tt_1,\Tt_2),
	\qquad \Tt_1,\Tt_2\in\wInvT.
\]
\end{proposition}

\begin{proof}
The map $\Phi$ is clearly surjective using that $\tran(\Omega)\in\wInvT$ for $\Omega\in\wInv$ satisfying $\Phi(\tran(\Omega))=\Omega$ by Lemma~\ref{lem:Connec_Transv-Hull}. Now the claimed statement follows if we prove $\Phi$ is bi-Lipschitz.

Let $\Tt_1,\Tt_2\in\wInvT$ be fixed and for the sake of simplicity set 
\[
\Omega_1:=\Phi(\Tt_1),\qquad
\Omega_2:=\Phi(\Tt_2),\qquad
\delta_{\text{hull}} := \delta_H\big( \Omega_1,\Omega_2 \big)
\qquad\textrm{and}\qquad
\delta_{\text{tran}} := \delta_H(\Tt_1,\Tt_2).
\]
We first prove $\delta_{\text{hull}}\leq (R_0+1)\delta_{\text{tran}}$. 
Let $\Pi\in\Omega_1$.
Then there is a $g\in\supp(\Pi)\cap \oB(R_0)$ and so $g^{-1}\Pi\in\Tt_1$. By definition of the Hausdorff metric, there is a $\Xi\in\Tt_2$ satisfying $\dw(g^{-1}\Pi,\Xi)<\delta_{\text{tran}}+\varepsilon$. Then Proposition~\ref{prop:wDel_groupAction_Lipschitz-App} and $g^{-1}\in \oB(R_0)$ (here we use that the metric $d_G$ is left-invariant) imply
\[
\dw(\Pi,g\Xi) \leq \big( d_G(e,g^{-1})+1\big) \dw(g^{-1}\Pi,\Xi) 
	\leq \big( R_0+1\big) \big( \delta_{\text{tran}}+\varepsilon \big).
\]
A similar estimate follows if $\Pi\in\Omega_2$ proving 
\[
\delta_{\text{hull}}\leq (R_0+1) (\delta_{\text{tran}}+\varepsilon)
\]
for all $\varepsilon>0$. Thus, $\Phi$ is Lipschitz continuous by sending $\varepsilon$ to zero.

\begin{figure}[htb]
    \centering
    \includegraphics[scale=1]{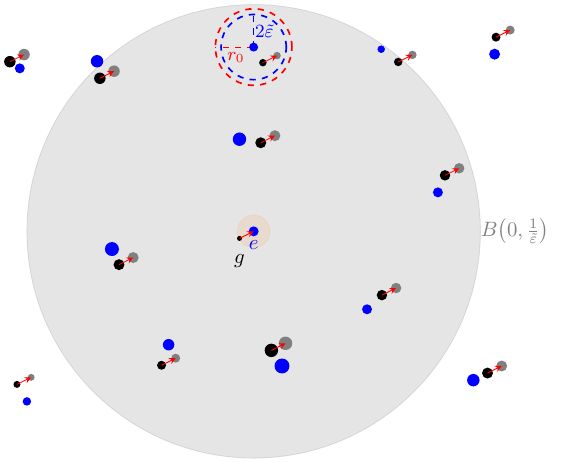}
	\captionsetup{width=0.95\linewidth}
	\caption{The weighted Delone set \(\Pi\) is shown in blue, \(\Xi\) in black, and the shifted set \(g^{-1} \Xi\) in gray.}
    \label{fig:Homeo-Hull-Transversal}
\end{figure}

Next we prove $\delta_{\text{tran}} \leq 2\delta_{\text{hull}}$. Since $\delta_{\text{tran}}\leq c_0:=\min\big\{\tfrac{r_0}{2},\tfrac{\sigma^{-1}}{2}\big\}$, there is no loss in generality in assuming $\delta_{\text{hull}}<c_0$. 
Let $\varepsilon>0$ be such that  $\tepsilon:= \delta_{\text{hull}} +\varepsilon<c_0$. 
For $\Pi\in\Tt_1 \subseteq \Omega_1$, there exists a $\Xi\in\Omega_2$ such that $\dw(\Pi,\Xi)<\tepsilon$. 
Using eq.~\eqref{eq:wdelt-appendix}, we conclude
\begin{equation}
\label{eq:proof_Homeo_hulls_transversals}
\big|\mu_\Pi(B(x,\tepsilon)) - \mu_\Xi(B(x,\tepsilon))\big|<\tepsilon,\quad
		\textrm{ for all } x\in B(\tfrac{1}{\tepsilon})\cap \big( \supp(\Pi)\cup \supp(\Xi) \big).
\end{equation}
Since $e\in\supp(\Pi)$ and $\tepsilon<\tfrac{\sigma^{-1}}{2}$, there is a $g\in B(\tepsilon)\cap \supp(\Xi)$, see Figure~\ref{fig:Homeo-Hull-Transversal}. Hence, $g^{-1}\Xi\in\Tt_2$. 
Since $\tepsilon< c_0$, we conclude $2\tepsilon<r_0$. Thus, $\mu_\Lambda(B(x,\tepsilon)) = \mu_\Lambda(B(x,2\tepsilon))$ holds for all $\Lambda\in\wDel$ and $x\in G$. 
In particular, $\mu_{g^{-1}\Xi}(B(x,2\tepsilon))=\mu_\Xi(B(x,\tepsilon))$ follows whenever $\mu_\Xi(B(x,\tepsilon))>0$.
Combining the previous considerations with eq.~\eqref{eq:proof_Homeo_hulls_transversals}, we derive
\[
\big|\mu_\Pi(B(x,2\tepsilon)) - \mu_{g^{-1}\Xi}(B(x,2\tepsilon))\big|<\tepsilon < 2\tepsilon,\quad
		\textrm{ for all } x\in B\big(\tfrac{1}{\tepsilon}-\tepsilon\big)\cap \big( \supp(\Pi)\cup \supp(g^{-1}\Xi) \big).
\]
Since $\sigma\geq 1$, we have $\tepsilon< c_0\leq \tfrac{1}{2} \leq \tfrac{1}{\sqrt{2}}$. Thus, $\tfrac{1}{\tepsilon}-\tepsilon\geq \tfrac{1}{2\tepsilon}$ follows like in eq.~\eqref{eq:proof_metric_weightDelone}. Hence,
$\dw(\Pi,g^{-1}\Xi)<2\tepsilon = 2(\delta_{\text{hull}} +\varepsilon)$ follows where $g^{-1}\Xi\in\Tt_2$. A similar argument for an arbitrary element $\Pi\in\Tt_2$ leads then to
\[
\delta_{\text{tran}} \leq 2(\delta_{\text{hull}} +\varepsilon).
\]
This finishes the proof by sending $\varepsilon$ to zero.
\end{proof}

\section{Substitution systems}
\label{App:Substituion_Systems}

Following the discussion in Section~\ref{sec:Subst_BeyondAbel} (see the remark on page~\pageref{rem:Histor_suff-large}), a modified notion of ``sufficiently large'' is used here and in \cite{BaBePoTe25}, compared to the version employed in the original ArXiv manuscript \cite{BecHarPog21} from September 2021.\footnote{Note that all explicit references to \cite{BecHarPog21} refer to the version posted in September 2021.} 

The main focus of \cite{BecHarPog21} was to prove the existence of linearly repetitive systems when the underlying group is no longer abelian. The modified notion introduced here allows one to incorporate a broader class of examples.
For the sake of transparency, we include the main adjustments required in the proofs from \cite{BecHarPog21} when working with this new notion of ``sufficiently large''. These modifications are also reflected in an updated version of \cite{BecHarPog21}.

The notion of ``sufficiently large'' is introduced to ensure that the supports of iterated substitution patches eventually cover the entire space in the limit. More precisely, it guarantees that the support of an $n$-fold substitution contains any given compact subset of the group (up to translation), which in turn ensures that the associated subshift is non-empty.

To establish linear repetitivity, one further requires that these supports can be bounded from both inside and outside by balls whose radii grow proportionally to $\lambda_0^n$, where $n$ denotes the number of iterations and $\lambda_0 > 1$ is the stretch factor. 

As discussed in Section~\ref{sec:Subst_BeyondAbel}, the sets $V(n)$ defined in equation~\eqref{eq:V(n,M)-def} (see page~\pageref{eq:V(n,M)-def}) represent the support of the $n$-fold substitution. As illustrated in Example~\ref{ex:HeisenbergExample-Subst} (see also \cite{BecHarPog21}), these supports may exhibit substantially different behavior in the non-Abelian case compared to the classical Abelian setting.

Recall the setting introduced in Section~\ref{chap:Dynam_LR}, $G$ is a locally compact, second countable Hausdorff group (lcsc) with an adapted metric $d_G$.
The open ball around $x\in G$ with radius $r>0$ is denoted by $B(x,r):=\set{y\in G}{d_G(x,y)<r}$. For a subset $F\subseteq G$ and $r>0$, define
\[
F_{+r} := \bigcup_{x\in F} B(x,r)
\qquad\text{and}\qquad
F_{-r} := \set{x\in F}{B(x,r)\subseteq F}.
\]
By construction we have for $r',r>0$,
\begin{equation}
\label{eq:Inner-outer_approx-balls}
B(x,r'-r)\subseteq B(x,r')_{-r}
\qquad\text{and}\qquad
B(x,r')_{+r}\subseteq B(x,r'+r).
\end{equation}
Let $\Gamma\subseteq G$ be a uniform lattice with a (left) fundamental domain $V$ such that the identity $e\in G$ lies in the interior of $V$. Then the \emph{inner $\Gamma$-approximation} of a bounded set $F\subseteq G$ is defined by
\[
F_\Gamma:= (F\cap\Gamma)V
\]

\begin{figure}[htb]
    \centering
    \includegraphics[scale=1]{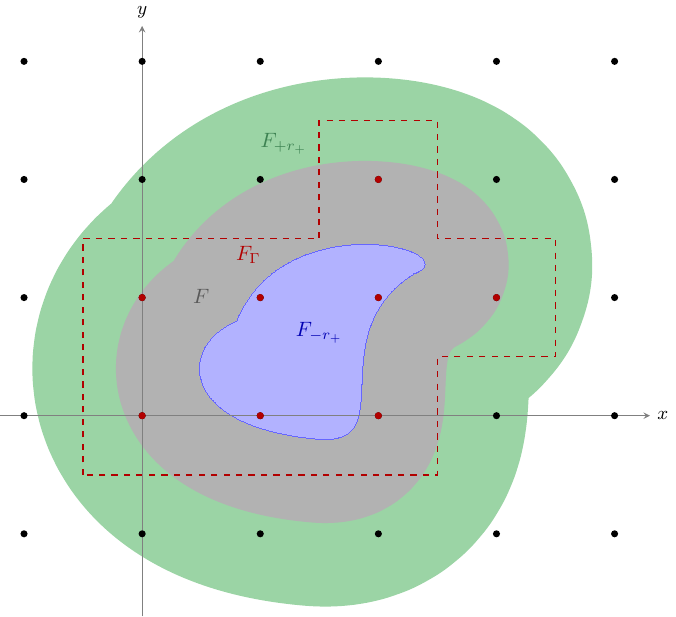}
	\captionsetup{width=0.95\linewidth}
	\caption{The set \(F \subseteq\R^2\) (gray), its inner approximation \(F_{-r_+}\) (blue), and outer approximation \(F_{+r_+}\) (green) are shown. Lattice points in \(\Gamma = \mathbb{Z}^2\) that lie in \(F\) are marked in red. The dashed red line marks the boundary of the set \(F_\Gamma\), where \(V = [-\tfrac{1}{2}, \tfrac{1}{2})^2\).}
    \label{fig:F_Gamma-set}
\end{figure}

Note that neither $F_\Gamma\subseteq F$ nor $F\subseteq F_\Gamma$ needs to hold in general. However, the following inclusion are exploited (proven in \cite[Prop.~5.11]{BecHarPog21})
\begin{equation}
\label{eq:prop5.11-suff-large}
F_{-r_+} \subseteq F_\Gamma \subseteq F_{+r_+},
\end{equation}
where $r_+>0$ is chosen such that $\overline{V} \subseteq B(e, r_+)$. These properties are illustrated by the example in Figure~\ref{fig:F_Gamma-set}.

Let \( \Dd = (G, d_G, (D_\lambda)_{\lambda > 0}, \Gamma, V) \) be a dilation datum, see Section~\ref{sec:Subst_BeyondAbel} for the definition.  
Recall that \( \lambda_0 > 1 \) is sufficiently large relative to \( V \) if there exist a constant \( C_- > 0 \), an integer \( s \in \N_0 \), and a point \( z \in \Gamma \) such that for all \( n \in \N_0 \),
\[
D_{\lambda_0}^n\big( B(z, C_-) \big) \subseteq V(s+n).
\]
Here $V(k)=V_{\lambda_0}(k)$ are recursively defined in eq.~\eqref{eq:V(n,M)-def} (for $M=\{e\}$) via
\[
V(0) :=  V
\quad \text{and} \quad
V(n) := D_{\lambda_0} \big( (V(n-1) \cap \Gamma) \cdot V \big) = D_{\lambda_0}\big( V(n-1)_\Gamma \big), \quad n\in\N.
\]
We begin by providing two sufficient criteria ensuring that \(\lambda_0\) is sufficiently large relative to \(V\). In Lemma~\ref{lem:suff-large}~(a), we show that the modified notion of ``sufficiently large'' generalizes the original definition used in \cite{BecHarPog21}. Lemma~\ref{lem:suff-large}~(b) presents an alternative sufficient criterion, which is applied in Examples~\ref{ex:Geom_Comb_data-TableTiling}, \ref{ex:NonSquare-Substitution-Z2}, and \ref{ex:NonSquare-Substitution-Z2_lambda=3}.

\begin{lemma}
\label{lem:suff-large}
Let \( \Dd = (G, d_G, (D_\lambda)_{\lambda > 0}, \Gamma, V) \) be a dilation datum, and let \( r_-, r_+ > 0 \) be such that
\[
B(e, r_-) \subseteq V \subseteq \overline{V} \subseteq B(e, r_+).
\]
\begin{enumerate}[label=(\alph*)]
\item If \( \lambda_0 > 1 + \tfrac{r_+}{r_-} \), then \( \lambda_0 \) is sufficiently large relative to \( V \), with parameters
\[
C_- = \tfrac{r_-}{\lambda_0} \left( \lambda_0 - \left(1 + \tfrac{r_+}{r_-} \right) \right), \qquad s = 0, \qquad z = e.
\]
\item If there exist \( r_0 > \tfrac{r_+ \lambda_0}{\lambda_0 - 1} \), \( s_0 \in \N_0 \), and \( z_0 \in \Gamma \) such that \( B(z_0, r_0) \subseteq V(s_0) \), then \( \lambda_0 \) is sufficiently large relative to \( V \), with parameters
\[
C_- = (r_0 - r_+) - \tfrac{r_0}{\lambda_0}, \qquad s=s_0, \qquad z=z_0.
\]
\end{enumerate}
\end{lemma}

\begin{proof}
(a) This is proven in \cite[Thm.~5.13]{BecHarPog21} using that \( \lambda_0 > 1 + \tfrac{r_+}{r_-} \) was the original definition of $\lambda_0$ being sufficiently large.

\medskip

(b) Let \( r_0 > \tfrac{r_+ \lambda_0}{\lambda_0 - 1} \), \( s_0 \in \N_0 \), and \( z_0 \in \Gamma \) be such that \( B(z_0, r_0) \subseteq V(s_0) \). First note that
\begin{equation}
\label{eq:proof_suff-large}
r_0 > \tfrac{r_+ \lambda_0}{\lambda_0 - 1} 
	\quad\Longleftrightarrow\quad
(r_0-r_+)\lambda_0>r_0
	\quad\Longleftrightarrow\quad	
\frac{r_0-r_+}{r_0} > \frac{1}{\lambda_0}.
\end{equation}
In particular, $C_->0$ holds. Define $\varepsilon:= \tfrac{C_-}{r_0}>0$ and observe $\varepsilon+\tfrac{1}{\lambda_0} = \tfrac{r_0-r_+}{r_0}$. 
Next, we prove by induction that
\[
\text{for all } k \in \N: \quad
	B\left(D_{\lambda_0}^k(z_0),\, \lambda_0^k \varepsilon r_0 + r_0 \right) \subseteq V(s_0 + k).
\]
Before proceeding with the induction, we show that this implication yields the desired statement.  
Since \( D_{\lambda_0} \) is a dilation, we obtain for all \( k \in \N \) that
\[
D_{\lambda_0}^k\big(B(z_0, C_-)\big) = B\big(D_{\lambda_0}^k(z_0),\, C_- \lambda_0^k \big).
\]
Thus, the above inclusion and \( \varepsilon r_0 = C_- \) imply
\[
D_{\lambda_0}^k\left(B(z_0, C_-)\right) \subseteq B\left(D_{\lambda_0}^k(z_0),\, \lambda_0^k \varepsilon r_0 + r_0 \right) \subseteq V(s_0 + k).
\]
This completes the proof of the desired statement.
It remains to prove the claimed inclusion by induction.

For the induction base $k=1$, the assumed inclusion \( B(z_0, r_0) \subseteq V(s_0) \) together with eq.~\eqref{eq:Inner-outer_approx-balls} and eq.~\eqref{eq:prop5.11-suff-large} imply
\[
B(z_0,r_0-r_+) \subseteq B(z_0,r_0)_{-r_+} \subseteq V(s_0)_\Gamma.
\]
Applying the dilation and the choice of $C_-$ yield
\begin{align*}
B\left(D_{\lambda_0}(z_0), \lambda_0\varepsilon r_0 +r_0 \right) 
	= B\left(D_{\lambda_0}(z_0), \lambda_0(r_0-r_+)\right)
	= &D_{\lambda_0} \left(B(z_0, r_0-r_+)\right)\\
	\subseteq &D_{\lambda_0}\big( V(s_0)_\Gamma\big)\\
	= &V(s_0+1),
\end{align*}
where we used the recursive definition of $V(n)$ in the last step.
For the induction step suppose 
\[
B\left(D_{\lambda_0}^k(z_0),\, \lambda_0^k \varepsilon r_0 + r_0 \right) \subseteq V(s_0 + k)
\]
holds for a given $k\in\N$. Applying eq.~\eqref{eq:proof_suff-large}, eq.~\eqref{eq:Inner-outer_approx-balls} and eq.~\eqref{eq:prop5.11-suff-large} lead to
\begin{align*}
B\left(D_{\lambda_0}^{k+1}(z_0),\, \lambda_0^{k+1} \varepsilon r_0 + r_0 \right)
	\subseteq &B\left(D_{\lambda_0}^{k+1}(z_0),\, \lambda_0^{k+1} \varepsilon r_0 + (r_0-r_+)\lambda_0 \right)\\
	= &D_{\lambda_0}\left( B\left(D_{\lambda_0}^{k}(z_0),\, \lambda_0^{k} \varepsilon r_0 + r_0-r_+ \right) \right)\\
	\subseteq &D_{\lambda_0}\left( V(s_0 + k)_\Gamma \right) \\
	= &V(s_0+k+1)
\end{align*}
finishing the proof.
\end{proof}

We now provide an adjusted version of \cite[Thm.~5.13, Lem.~5.15]{BecHarPog21} stating that the sets $V(n)$ can be bounded from the inside and outside by balls whose radii are proportional to $\lambda_0^n$.

\begin{lemma}\label{lem:GrowthLemma1} 
Let \( \Dd = (G, d_G, (D_\lambda)_{\lambda > 0}, \Gamma, V) \) be a dilation datum and $r_+>0$ be such that $\overline{V}\subseteq B(e,r_+)$. Let $s\in\N_0$, $z\in\Gamma$ and $R_+>0$ be such that $V(s)\subseteq B(z,R_+)$. Then
$$
V(s+k) \subseteq B\big(D_{\lambda_0}^k(z),C_+ \lambda_0^{k}\big)
$$
holds for all $k\in\N_0$, where $C_+:=R_+ +\tfrac{r_+\lambda_0}{\lambda_0-1}>0$. 
\end{lemma}

\begin{proof} 
First note that the inclusion holds trivially for $k=0$ since $C_+>R_+$. 
Define $S_k:=\sum_{n=0}^k \frac{1}{\lambda_0^n}$ for $k\in\N$. We first inductively prove
\[
\text{for all } k\in\N:
	\quad
	V(s+k) \subseteq B\big(D_{\lambda_0}^k(z),\lambda_0^{k} (R_+ + r_+S_{k-1})\big).
\]
For the induction base $k=1$, the inclusion \(V(s)\subseteq B(z,R_+)\) and eq.~\eqref{eq:prop5.11-suff-large} yield
\[
V(s)_\Gamma \subseteq B(z,R_+)_\Gamma \subseteq B(z,R_+ + r_+).
\]
Together with $S_{k-1}=S_0=1$, we conclude
\[
V(s+1) 
	= D_{\lambda_0}\big(V(s)_\Gamma\big) 
	\subseteq  B\big(D_{\lambda_0}(z),\lambda_0 (R_+ + r_+S_{k-1})\big)
\]
follows. For the induction step, the induction hypothesis combined with  eq.~\eqref{eq:prop5.11-suff-large} lead to
\[
V(s+k)_\Gamma 
	\subseteq B\big(D_{\lambda_0}^k(z),\lambda_0^{k} (R_+ + r_+S_{k-1})\big)_\Gamma
	\subseteq B\big(D_{\lambda_0}^k(z),\lambda_0^{k} (R_+ + r_+S_{k-1}) + r_+\big).
\]
Applying the dilation $D_{\lambda_0}$ yields
\begin{align*}
V(s+k+1)
	= D_{\lambda_0}\big(V(s+k)_\Gamma \big)
	\subseteq &B\left(D_{\lambda_0}^{k+1}(z),\lambda_0^{k+1} \big(R_+ + r_+S_{k-1} + r_+\tfrac{1}{\lambda_0^k}\big)\right)\\
	= &B\big(D_{\lambda_0}^{k+1}(z),\lambda_0^{k+1} (R_+ + r_+S_{k})\big)
\end{align*}
finishing the induction.

Observe that $(S_k)_{k\in\N}$ is a monotonically increasing sequence converging to $\tfrac{\lambda_0}{\lambda_0-1}$ (geometric series) using $\lambda_0>1$. Thus, the previous considerations yield for all $k\in\N$,
\[
	V(s+k) \subseteq B\big(D_{\lambda_0}^k(z),\lambda_0^{k} (R_+ + r_+S_{k-1})\big)
		\subseteq B\big(D_{\lambda_0}^k(z),\lambda_0^{k} (R_+ + \tfrac{r_+\lambda_0}{\lambda_0-1})\big).
		\qedhere
\]
\end{proof}

\begin{theorem}
\label{thm:SupportGrowth}
Let \( \Dd = (G, d_G, (D_\lambda)_{\lambda > 0}, \Gamma, V) \) be a dilation datum, and let \( \lambda_0 > 1 \) be sufficiently large relative to \( V \), with parameters \( C_- > 0 \), \( s \in \N_0 \), and \( z \in \Gamma \). Let $R_+>0$ be such that $V(s)\subseteq B(z,R_+)$. Then for all \( n \in \N \),
\[
D_{\lambda_0}^n(z) B\big(e, C_-\lambda_0^n\big) \big) \subseteq V(s+n) \subseteq D_{\lambda_0}^n(z) B\big(e, C_+\lambda_0^n\big).
\]
where \( C_+:=R_+ +\tfrac{r_+\lambda_0}{\lambda_0-1}>0 \).
Moreover, for all $n\in\N_0$,
\[
V(n) \subseteq B\big(e, \tilde{C}_+\lambda_0^n\big)
\]
holds with $\tilde{C}_+=r_+\big( 1+ \tfrac{\lambda_0}{\lambda_0-1}\big)>0$.
\end{theorem}

\begin{proof}
Since $D_{\lambda_0}:G\to G$ is a dilation and the metric $d_G$ is left-invariant, we conclude 
\[
D_{\lambda_0}(z) B\big(e, r\lambda_0\big) 
	= B\big(D_{\lambda_0}(z) , r\lambda_0\big)
	= D_{\lambda_0} \big(B(z, r)\big).
\]
The first inclusion follows since \( \lambda_0 > 1 \) is sufficiently large relative to \( V \), with parameters \( C_- > 0 \), \( s \in \N_0 \), and \( z \in \Gamma \). The second inclusion follows directly from Lemma~\ref{lem:GrowthLemma1}. Also the inclusion
\[
V(n) \subseteq B\big(e, \tilde{C}_+\lambda_0^n\big)
\]
is a consequence of Lemma~\ref{lem:GrowthLemma1} using $V(0)=V\subseteq B(e,r_+)$.
\end{proof}

Comparing Theorem~\ref{thm:SupportGrowth} with the original version \cite[Thm.~5.13]{BecHarPog21}, we observe the following differences:
\begin{enumerate}[label=(\alph*)]
\item The constants $C_-, C_+$ (respectively $\tilde{C}_+$) are modified;
\item The sets $V(s+n)$ contain only a ball of radius comparable to $\lambda_0^n$ centered at $D_{\lambda_0}^n(z)$, where $s \in \N$ may be larger than $1$.
\end{enumerate}
The difference in (a) affects only the involved constants and does not change the qualitative nature of the results in \cite{BecHarPog21}. 

Regarding (b), we note that up to a shift by $D_{\lambda_0}^n(z)$, the set $V(s+n)$ still contains arbitrarily large finite subsets of $\Gamma$. This likewise leaves the qualitative content of the results in \cite{BecHarPog21} unaffected. In particular, due to the equivariance condition from Proposition~\ref{prop:SubstRule-Properties}, all relevant arguments remain valid. The additional parameter $s$ can be controlled using $V(s+n) = V(n,V(s)\cap\Gamma)$ proven in \cite[Lem.~5.16]{BecHarPog21}.

More precisely, the inclusion of such balls is used only in \cite[Lem.~7.3 and Lem.~7.8]{BecHarPog21}, where the arguments can be straightforwardly adapted to incorporate the corresponding shift.

\section{Proving spectral estimates of dynamically-defined operators}
\label{App:SpectralEstimates}

In this section, we provide the modified proofs of \cite{BecTak25} stated in Section~\ref{sec:SpectralEstimates} to in cooperate that the kernel not only takes values in $\C$ but in $\C^{N\times N}$ for some $N\in\N$. Here $\C^{N\times N}$ denotes the set of $N\times N$-matrices $A = (a_{m,n})_{m,n=1}^N$ with values in $\C$ (i.e. $a_{m,n}\in\C$) equipped with the norm
$$
\|A\|_M := \max_{1\leq m,n\leq N} |a_{m,n}|.
$$

Let $G$ be a second countable, locally compact, Hausdorff group $G$. Then $G$ admits a unique left-invariant Haar measure \( m_G \) on the Borel \( \sigma \)-algebra of \( G \); see, e.g., \cite{Fol16}. Let $L^2(G):=L^2(G,m_G)$ be the Hilbert space of square integrable (w.r.t. $m_G$) functions with complex values. Denote by $\|\psi\|_2^2 := \int_G |\psi(g)|^2 \ dg$ the $L^2$-norm of $\psi\in L^2(G)$. Here $dg$ denotes the integration with respect to the Haar measure $m_G$. 

Throughout this section we assume the group to be unimodular, i.e. $m_G$ is left and right-invariant with respect to the group action, i.e. $m_G(gA) = m_G(A) = m_G(Ag)$ for all measurable $A$ and $g\in G$.

We are interested in linear bounded operators on the Hilbert space $H_N:=L^2(G)\otimes \C^N$, which is also identified with (see \cite[Rem.~2.6.8]{KadRin97})
$$
\bigoplus_{n=1}^N L^2(G)
	:= \set{ \psi:= (\psi_1,\ldots,\psi_N) \in\prod_{n=1}^N L^2(G) }{ \|\psi\|_{2,N}^2:=\sum_{n=1}^N \|\psi_n\|_2^2<\infty}
$$
equipped with norm $\| \cdot \|_{2,N}$.

Before discussing the dynamically-defined operators, let us observe the following.
\begin{lemma}
\label{lem:Matrix}
Let $A(h)= (a_{m,n}(h))_{m,n=1}^N \in\C^{N\times N}, h\in G,$ be a family of matrices with measurable coefficients $G\ni h\mapsto a_{m,n}(h)$ defining a linear and bounded operator on the Hilbert space $H_N$ via 
$$
(A\psi)_n(h) := (A(h)\psi(h))_n = \sum_{m=1}^N a_{m,n}(h)\psi_m(h), 
	\qquad 1\leq n\leq N.
$$
Then the following estimate holds
$$
\sum_{n=1}^N \left( \int_G \big| \big(A(h)\psi(h)\big)_n \big| dh \right)^2 
	\leq N^2 \cdot \sum_{m=1}^N \left(\int_G \; \|A(h)\|_M |\psi_m(h)| \ dh \right)^2.
$$
\end{lemma}

\begin{proof}
By the Cauchy-Schwartz inequality on $\C^N$, observe for a vector $(v_1,\ldots, v_N)\in\C^N$ that
$$
\left( \sum_{m=1}^N 1\cdot |v_m| \right)^2 \leq N \cdot \left( \sum_{m=1}^N |v_m|^2 \right).
$$
Thus, a short computation yields
\begin{align*}
\sum_{n=1}^N \left( \int_G \big| (A\psi)(h)_n \big| dh \right)^2 
	\leq &\sum_{n=1}^N \left( \sum_{m=1}^N  \int_G |a_{m,n}(h)\psi(h)_m| dh \right)^2\\
	\leq &N\cdot \sum_{n=1}^N \sum_{m=1}^N \left(\int_G |a_{m,n}(h)\psi(h)_m| dh \right)^2
\end{align*}
proving the claimed estimate.
\end{proof}

Following Section~\ref{sec:SpectralEstimates}, we are interested in dynamically-defined operators. We use the notation that \( k: G \times Z \to \C^{N\times N} \) is a map into the $N\times N$-matrices and its $m,n$th coefficients are given by $k(\cdot,\cdot)_{m,n}$.

\begin{definition}[Dynamically-defined operators]	
Let \( (Z, G, \tau) \) be a dynamical system. A measurable function \( k: G \times Z \to \C^{N\times N} \) is called a {\em (bounded) kernel} if for each \( g \in G \), the map \( k(g,\cdot): Z \to \C^{N\times N} \) is continuous and
\[
\sup_{x \in Z} \int_G \sup_{g\in G}\left\| k\big( h, g^{-1}x \big) \right\|_M \, dh < \infty.
\]
Then the operator family \( A_x \in \Ll\big(H_N\big) \) for \( x \in Z \) defined by
\[
(A_x \psi)_n(g) := \int_G  \left(k\big( g^{-1}h, g^{-1}x \big) \psi(h)\right)_n\, dh, 
		= \int_G  \left(k\big( h, g^{-1}x \big) \psi(gh)\right)_n\, dh,\qquad 1\leq n \leq N,
\]
is called a \emph{dynamically-defined operator} associated with the kernel $k$.
\end{definition}

Note that $k\big( g^{-1}h, g^{-1}x \big)$ is an $N\times N$ matrix and $\psi(h) = (\psi_1(h),\ldots,\psi_N(h))\in\C^N$ is a vector for each $h\in G$ (after choosing a suitable representative of each coefficient of $\psi$). Then the matrix acts via matrix $k\big( g^{-1}h, g^{-1}x \big)$ multiplication on the vector $\psi(h)$ and in the integral, we take the $n$th coefficient that we integrate with respect to $m_G$. Specifically, the operator acts on the $n$th coefficient like an integral operator. Writing the matrix multiplication out, gives
$$
\left(k\big( g^{-1}h, g^{-1}x \big) \psi(h)\right)_n
	= \sum_{m=1}^N k\big( g^{-1}h, g^{-1}x \big)_{m,n} \psi_m(h),
	\qquad 1\leq n\leq N.
$$
All the statements in \cite{BecTak25} are proven for normal operators, i.e. $A_x^\ast A_x = A_x A_x^\ast$ for $x\in Z$. Here $A_x^\ast$ is the adjoint operator of $A_x$. Since $A_x$ is dynamically-defined a short computation gives
\begin{equation}
\label{eq:Adjoint_DynDefOp}
(A_x^\ast \psi)_n(g) := \int_G  \left( k\big( h^{-1}, (gh)^{-1}x \big)^\ast \psi(gh)\right)_n\, dh, \qquad 1\leq n \leq N, g\in G,
\end{equation}
where $k(h,y)^\ast$ denotes the adjoint matrix of $k(h,y)$.

We now provide the analogs of various statements in \cite{BecTak25}. We stay relatively close to the original formulation to make it easier to spot the modifications. We start with \cite[Prop.~5.1]{BecTak25}.

In the following, we consider the function 
$$
\mathbf{1}_{F}:G\to\{0,1\},
	\qquad \mathbf{1}_F(g) :=
		\begin{cases}
			1, \qquad &g\in F,\\
			0, \qquad &g\not\in F,
		\end{cases}
$$
for a subset $F\subseteq G$,

Let \( C_c(G,\C^N) \) be the set of all \( \psi = (\psi_1,\ldots,\psi_N) \in H_N \) such that each component \( \psi_n : G \to \C \) is continuous and has compact support in \( G \). The support of this vector-valued function \( \psi \) is defined by
\[
\supp(\psi) := \bigcup_{n=1}^N \supp(\psi_n).
\]
Then \( C_c(G,\C^N) \) embeds into \( H_N \), and is dense with respect to the \( L^2 \)-norm \( \|\cdot\|_{2,N} \). We formulate the following statement for $\psi\in C_c(G,\C^N)$, which is sufficient for our further applications. However, the statement extends to all $\psi:G\to\C^N$ being $L^2$-integrable componentwise. 

\begin{proposition}
\label{prop:Basic_DynDefOp_app}
Let \( (Z, G, \tau) \) be a dynamical system where $G$ is an unimodular lcsc group. 
For $N\in\N$, consider a bounded kernel \( k: G \times Z \to \C^{N\times N} \) with dynamically-defined operator $(A_x)_{x\in Z}$.
Let $x,y\in Z$, $\psi\in C_c(G,\C^N)$ for $N\in\N$ and $j\in G$.
Then 
\begin{align*}
&(a) \ \|A_x\psi\|_{2,N} \le
N\cdot \left(\int_G \; \sup_{g\in \supp(\psi) h^{-1}} \|k(h,g^{-1}x)\|_M \ dh \right) \|\psi\|_{2,N}, \\
&(b) \ \|A_x^*\psi\|_{2,N} \le
N\cdot \left(\int_G \; \sup_{g\in \supp(\psi)} \|k(h,g^{-1}x)\|_M\cdot \ dh \right) \|\psi\|_{2,N}, \\
&(c) \ \|(A_y - A_x)\psi\|_{2,N} \le
N\cdot \left( \int_G \sup_{g \in \supp(\psi) h^{-1}}  \|k(h,g^{-1}y) - k(h,g^{-1}x)\|_M \ dh \right) \|\psi\|_{2,N}, \\
&(d) \ \|(A_y - A_x)^*\psi\|_{2,N} \le
N\cdot \left( \int_G \sup_{g \in \supp(\psi)} \|k(h,g^{-1}y) - k(h,g^{-1}x)\|_M \ dh \right) \|\psi\|_{2,N}, \\
&(e) \ U_{j^{-1}} A_x U_j = A_{j^{-1}x}.
\end{align*}
\end{proposition}

\begin{proof}
Statements (c) and (d) follow directly from (a) and (b) respectively using that linear combinations of dynamically-defined operators are again dynamically-defined operators. 
We only provide the proof for (b), since we mainly employ these estimates for the adjoint operator. Set 
\[
C(k,\psi) := \int_G \; \sup_{\tilde{g}\in \supp(\psi)} \|k(h^{-1},\tilde{g}^{-1}x)\|_M  dh.
\]
Observe
\begin{align*}
\|A_x^\ast\psi\|_{2,N}^2
	&= \int_G  	\sum_{n=1}^N \left| \big(A_x^\ast\psi\big)_n(g)\right|^2 dg\\
	&= \int_G  	\sum_{n=1}^N 	\left| \int_G \left( k(h^{-1}, (gh)^{-1}x)^\ast \psi(gh) \right)_n dh \right|^2 dg \\
\text{(Lemma~\ref{lem:Matrix})}\qquad	&\leq N^2  \int_G  \sum_{m=1}^N 	\left( \int_G \left\| k(h^{-1}, (gh)^{-1}x)\right\|_M \cdot \left| \psi_m(gh) \right| dh \right)^2 dg.
\end{align*}
Here we used that $\|A^\ast\|_M=\|A\|_M$ for any matrix $A\in\C^{N\times N}$. Then H\"older's inequality and Fubini-Tonelli lead to
\begin{align*}
&\|A_x^\ast\psi\|_{2,N}^2 \\
	&\le N^2  \int_G \sum_{m=1}^N  \underbrace{ \Big(\scalebox{0.9}{$ \int_G  \left\|k(h^{-1},(gh)^{-1}x)\right\|_M\cdot \mathbf{1}_{\supp(\psi)}(gh) dh $}\Big) }_{\le C(k,\psi)}
        	\cdot\Big( \scalebox{0.9}{$\int_G \left\|k(h^{-1},(gh)^{-1}x)\right\|_M \cdot |\psi_m(gh)|^2 dh$} \Big) dg \\
	&\le N^2 C(k,\psi) \int_G \left( \sup_{\tilde{g}\in \supp(\psi)} \|k(h^{-1},\tilde{g}^{-1}x)\|_M \right) 
			\underbrace{\int_G \sum_{m=1}^N |\psi(gh)_m|^2 dg}_{=\|\psi\|_{2,N}^2} \ dh\\
	&\le  N^2 C(k,\psi)^2 \|\psi\|_{2,N}^2. 
\end{align*}
The covariance property $U_{j^{-1}} A_x U_j = A_{j^{-1}x}$ follows by a straightforward computation.
\end{proof}

Let us highlight the difference that arises from taking the adjoint operator in the previous estimates. When the adjoint operator is considered, the support of $\psi$ remains unchanged, whereas without taking the adjoint, the support is shifted. 
Not having this additional shift is central to the proof, and we therefore primarily work with the adjoint operator.

This can be employed using that the operator $A_x$ is normal if and only if $\|A_x \psi\| = \|A_x^\ast \psi\|$ for all $\psi \in H_N$.

Another central tool is a Weyl sequence argument, stating that a normal operator admits approximate eigenfunctions. More precisely, if \( A:H_N \to H_N \) is a linear bounded operator that is normal, then
\begin{equation}
\label{eq:ApproxEigenfct}
E \in \sigma(A)
	\quad\Longleftrightarrow\quad
\begin{array}{c}
\forall \varepsilon>0, \ \exists\, \psi_\varepsilon \in C_c(G,\C^N) \text{ such that} \\ 
\|\psi_\varepsilon\|_{2,N} = 1 \quad \text{and} \quad \|(A-E)\psi_\varepsilon\|_{2,N} < \varepsilon.
\end{array}
\end{equation}

We continue with the proof of semi-continuity of the spectrum as stated in \cite[Prop.~5.2]{BecTak25}. Note that the semi-continuity of spectra, especially for strongly convergent self-adjoint operators, is a classical result. We provide a proof here as it follows naturally from our techniques.

For a compact set \( K \subseteq \C \) and a point \( E \in \C \), we define
\[
\dist(E,K) := \inf_{z \in K} |E - z|.
\]

\begin{proposition}
\label{prop:Semi-cont-Spectrum_app}
Let \( (Z, G, \tau) \) be a dynamical system where $G$ is an unimodular lcsc group. 
For $N\in\N$, let \( k: G \times Z \to \C^{N\times N} \) be a bounded kernel with dynamically-defined operator $(A_x)_{x\in Z}$ such that $A_x$ is normal for every $x \in Z$ and
\begin{enumerate}[label=(\alph*)]
\item \label{item:Semi-cont:a} $k$ satisfies a continuity condition: for each $x \in Z$, each $\varepsilon > 0$, and each nonempty compact subset $F \subseteq G$, there exists an open neighborhood $U_c\subseteq Z$ of $x$ such that
\[
\int_G \|k(h,y) - k(h,x)\|_M \cdot \mathbf{1}_F(h)\, dh < \varepsilon, \qquad y\in U_c,
\]
\item \label{item:Semi-cont:b} $k$ satisfies a decay condition: for each $x \in Z$ and each $\varepsilon > 0$, there exists a nonempty compact subset $F \subseteq G$ and an open neighborhood $U_d\subseteq Z$ of $x$ such that
\[
\int_G \; \sup_{g\in G} \|k(h,g^{-1}y)\|_M \cdot \mathbf{1}_{G\setminus F}(h)\, dh < \varepsilon, \qquad y\in U_d.
\]
\end{enumerate}
Then for every $x \in Z$ and each $\varepsilon > 0$, there exists an open neighborhood $U\subseteq Z$ of $x$ such that
\[
\sup_{E \in \sigma(A_x)} \dist(E, \sigma(A_y)) < \varepsilon, \qquad y\in U.
\]
In particular, we have
\[
\lim_{y \to x} \sup_{E \in \sigma(A_x)} \dist(E, \sigma(A_y)) = 0.
\]
\end{proposition}

\begin{remark}
\label{rem:UpperSemiCont-Spectrum}
Proposition~\ref{prop:Semi-cont-Spectrum_app} yields the following upper semi-continuity of the spectra: If \( y_n \to x \) in \( Z \), then
\[
\limsup_{n \to \infty} \sigma(A_{y_n}) 
:= \bigcap_{n \in \N} \overline{\bigcup_{k \geq n} \sigma(A_{y_k})}
\supseteq \sigma(A_x).
\]
We also refer the reader to Section~\ref{sec:SemiContMeasure} and Theorem~\ref{thm:SemiContinuity_measures} for analogous results concerning the space of invariant measures.
\end{remark}

\begin{proof}
By compactness of $\sigma(A_x)$, it suffices to show that for each $E \in \sigma(A_x)$ and each $\varepsilon > 0$, we have $\dist(E, \sigma(A_y)) < \varepsilon$ for every $y$ in some neighborhood of $x$.

Let $E \in \sigma(A_x)$ and $\varepsilon > 0$.
Without loss of generality, assume $E \notin \sigma(A_y)$.
By the decay condition, there exists a nonempty compact subset $F \subseteq G$ and an open neighborhood $U_d\subseteq Z$ of $x$ such that
\[
\int  \sup_{g\in G} \|k(h,g^{-1}y)\|_M \cdot \mathbf{1}_{G\setminus F}(h)\, dh < \frac{1}{4N}\varepsilon, \qquad y\in U_d.
\]
Since $A_x$ is normal, Weyl's criterion (see eq.~\eqref{eq:ApproxEigenfct}) implies that there is a $0\neq \psi\in C_c(G,\C^N) \subseteq H_N$ such that
\[
\|(A_x - E)\psi\|_{2,N} < \frac{1}{4}\varepsilon \cdot \|\psi\|_{2,N}.
\]
By continuity of the group action, compactness of $\supp(\psi)$ and the continuity condition on $k$, there exists an open neighborhood $U_c\subseteq Z$ of $x$ such that
\[
\int_G  \left(\sup_{g\in \supp(\psi)}\|(k(h, g^{-1}y) - k(h, g^{-1}x))\|_M \right) \cdot \mathbf{1}_F(h)\, dh < \frac{1}{4N}\varepsilon, \qquad y\in U_c.
\]
Since $A_z$ is normal for all $z \in Z$, we conclude for every $y \in U:=U_c \cap U_d$:
\begin{align*}
\|(A_y - E)^{-1}\|^{-1} \cdot \|\psi\|_{2,N}
\le \|(A_y - E)\psi\|_{2,N}
= &\|(A_y - E)^*\psi\|_{2,N}\\
< &\left( \left\|(A_y - A_x)^* \tfrac{\psi}{\|\psi\|}_{2,N} \right\|_{2,N} + \tfrac{1}{4}\varepsilon \right) \cdot \|\psi\|_{2,N}.
\end{align*}
Hence, normality of $A_y$ for $y\in U$ leads to
\begin{align*}
\dist(E, \sigma(A_y))
	&= \|(A_y - E)^{-1}\|^{-1} \\
	&< \left\|(A_y - A_x)^* \tfrac{\psi}{\|\psi\|_{2,N}} \right\|_{2,N} + \tfrac{1}{4}\varepsilon \\
(\text{\small Proposition~\ref{prop:Basic_DynDefOp_app}~(d)} )\quad	&\le N\cdot \int_G \sup_{g\in \supp(\psi)}\| (k(h, g^{-1}y) - k(h, g^{-1}x))\|_M \  dh + \tfrac{1}{4}\varepsilon \\
( \small{y\in U_d} )\quad	&  < N\cdot \int_G \left(\sup_{g\in \supp(\psi)} \| (k(h, g^{-1}y) - k(h, g^{-1}x))\|_M \right) \cdot \mathbf{1}_F(h) dh + \tfrac{3}{4}\varepsilon \\
( \small{y\in U_c} )\quad		&< \varepsilon. \qedhere
\end{align*}
\end{proof}

Recall that $\inv$ denotes the set of all invariant, closed and non-empty subsets of $Z$, if \( (Z, G, \tau) \) be a dynamical system.
For $Y\in \inv$, the spectrum of a dynamically-defined operator $(A_x)_{x\in Z}$ associated with $Y$ is defined by
$$
\sigma(A_Y) := \ol{\bigcup_{y\in Y} \sigma(A_y)}
$$
Whenever $Y$ is topological transitive (i.e. there exists a $y\in Y$ such that $Y=\ol{\Orb(y)}$), then the spectrum $\sigma(A_Y)$ equals to the spectrum $\sigma(A_y)$. This follows directly from the semi-continuity of the spectra proven in Proposition~\ref{prop:Semi-cont-Spectrum_app} and the covariance property of dynamically-defined operators.

\begin{corollary}
\label{cor:SemiCont_Spectrum}
Let \( (Z, G, \tau) \) be a dynamical system where $G$ is an unimodular lcsc group. 
For $N\in\N$, let \( k: G \times Z \to \C^{N\times N} \) be a bounded kernel with dynamically-defined operator $(A_x)_{x\in Z}$ such that $A_x$ is normal for every $x \in Z$. Suppose further that the kernel $k$ satisfies \ref{item:Semi-cont:a} and \ref{item:Semi-cont:b} of Proposition~\ref{prop:Semi-cont-Spectrum_app}. Then the implication holds 
$$
Y=\ol{\Orb(y)} \in\inv \text{ for some } y\in Y
	\qquad\Longrightarrow\qquad \sigma(A_Y)=\sigma(A_y).
$$
\end{corollary}

\begin{proof}
By definition of $\sigma(A_Y)$, we have $\sigma(A_y)\subseteq \sigma(A_Y)$. For the converse inclusion we use $Y=\ol{\Orb(y)} \in\inv$. 

Let $E\in\sigma(A_Y)$ and $\varepsilon>0$, then there is an $x\in Y$ and an $E_1\in \sigma(A_x)$ such that $|E-E_1|<\tfrac{\varepsilon}{2}$. Since $Y=\ol{\Orb(y)}$, Proposition~\ref{prop:Semi-cont-Spectrum_app} implies that there is a $g\in G$ and an $E_2\in \sigma(A_{gy})$ such that $|E_1-E_2|<\tfrac{\varepsilon}{2}$. 
Using the covariance condition in Proposition~\ref{prop:Basic_DynDefOp_app}~(e), we conclude $E_2\in \sigma(A_{gy})=\sigma(A_{y})$ since unitary transformation do not change the spectrum.
Thus, $|E-E_2|<\varepsilon$ where $E_2\in\sigma(A_y)$. Since $\varepsilon>0$ was arbitrary and $\sigma(A_y)$ is closed, we conclude $E\in\sigma(A_y)$ finishing the proof.
\end{proof}

To prove the spectral estimates, we use cut-off functions \( \chi \). In the following, we first estimate the corresponding error term arising when commuting a dynamically-defined operator with such a cut-off function. 

Let $L^\infty(G)$ denote the space of (equivalence classes of) all essentially bounded functions. In the following, we always identify an element of $L^\infty(G)$ with a bounded representative $f:G\to \C$ , i.e. \( \sup_{g \in G} f(g) < \infty \).

For \( 0 \neq \chi \in L^\infty(G) \), define the linear bounded operator (componentwise multiplication operator)
\[
[\chi] : H_N \to H_N, \qquad ([\chi]\psi)_n(g) := \chi(g)\psi_n(g).
\]
Furthermore, for \( j, h \in G \), define the corresponding left and right shifts of \( \chi \) by
\begin{align*}
\chi_j : G \to [0,\infty), \qquad &\chi_j(g) := \chi(j^{-1}g), \\
\chi^h : G \to [0,\infty), \qquad &\chi^h(g) := \chi(gh).
\end{align*}
The commutator of two linear bounded operators $A$ and $B$, is defined by $[A,B]:= AB-BA$.
We continue with the analog of \cite[Lem.~4.2]{BecTak25}. 

\begin{lemma}\label{lem:Technical_SpectralEst_app}
Let \( (Z, G, \tau) \) be a dynamical system where $G$ is an unimodular lcsc group. 
For $N\in\N$, consider a bounded kernel \( k: G \times Z \to \C^{N\times N} \) with dynamically-defined operator $(A_x)_{x\in Z}$.
Let $0\neq \chi \in L^\infty(G)\cap L^2(G)$, $x \in Z$ and $\psi \in C_c(G,\C^N)$. Then
\begin{align*}
&(a) \ \int_G \big\|[\chi_j]\psi\big\|_{2,N}^2 \, dj
	= \|\chi\|_2^2 \cdot \|\psi\|_{2,N}^2, \\
&(b) \ \int_G \big\|\big[A_x^*, [\chi_j]\big]\psi\big\|_{2,N}^2 \, dj
	\le N\cdot \left( \int_G \sup_{g\in \supp(\psi)} \|k(h^{-1}, g^{-1}x)\|_M \|\chi^h - \chi\|_2\, dh \right)^2 \|\psi\|_{2,N}^2, \\
&(c) \ \int_G \big\|\big[A_x, [\chi_j]\big]\psi\big\|_{2,N}^2 \, dj
\le N\cdot \left( \int_G \sup_{g\in \supp(\psi) h^{-1}} \|k(h, g^{-1}x)\|_M \|\chi^h - \chi\|_2\, dh \right)^2 \|\psi\|_{2,N}^2.
\end{align*}
\end{lemma}

\begin{proof}
Using Fubini-Tonelli, a short computation gives
\begin{align*}
\int \|[\chi_j]\psi\|_{2,N}^2 dj
= \iint \sum_{m=1}^N|\chi_j(g)|^2\, |\psi_m(g)|^2\, dg\, dj
= &\int \sum_{m=1}^N \left(\int |\chi(j^{-1}g)|^2\, dj \right) |\psi_m(g)|^2\, dg\\
= &\|\chi\|_2^2 \cdot \|\psi\|_{2,N}^2. 
\end{align*}
This proves (a). 

Using eq.~\eqref{eq:Adjoint_DynDefOp} (representation of the adjoint) and Lemma~\ref{lem:Matrix}, we obtain
\begin{align*}
&\int \|[A_x^*,[\chi_j]]\psi\|^2 dj\\
&= \iint \sum_{n=1}^N \left| \int \sum_{m=1}^N \overline{k(h^{-1}, (gh)^{-1}x)}_{n,m} (\chi_j(gh) - \chi_j(g)) \psi_m(gh)\, dh \right|^2 dg\, dj \\
&\le N^2\iint \sum_{m=1}^N \left(\int \left\|k(h^{-1}, (gh)^{-1}x)\right\|_M \cdot \frac{\|\chi^h - \chi\|^{1/2}}{\|\chi^h - \chi\|^{1/2}} \cdot |\chi_j(gh) - \chi_j(g)|  \cdot |\psi_m(gh)|\,  dh\right)^2 dg\, dj .
\end{align*}
Using H\"older's inequality and Fubini-Tonelli, we further estimate the integral
\begin{align*}
&\int \|[A_x^*,[\chi_j]]\psi\|_{2,N}^2 dj \\
&\le N^2\iint \sum_{m=1}^N
	\left( \int \left\|k(h^{-1}, (gh)^{-1}x)\right\|_M\cdot \mathbf{1}_{\supp(\psi)}(gh)\cdot \|\chi^h - \chi\|_2\, dh \right) \\
&\quad \cdot
	\left( \int \left\|k(h^{-1}, (gh)^{-1}x)\right\|_M\cdot \frac{|\chi_j(gh) - \chi_j(g)|^2}{\|\chi^h - \chi\|_2}\cdot |\psi_m(gh)|^2 \, dh \right) dg\, dj \\
&\le N^2 \sum_{m=1}^N
	\left( \int \left(\sup_{\check{g}\in \supp(\psi)}\|k(h^{-1}, \check{g}^{-1}x)\|_M \right)\cdot \|\chi^h - \chi\|_2\, dh \right) \\
&\quad \cdot \int
	\left( \left(\sup_{\check{g}\in \supp(\psi)} \|k(h^{-1}, \check{g}^{-1}x)\|_M \right)\cdot 
		\int \underbrace{\int \frac{|\chi_j(gh) - \chi_j(g)|^2}{\|\chi^h - \chi\|_2}\, dj}_{=\|\chi^h - \chi\|_2}
		 \cdot |\psi_m(gh)|^2 \, dg \right) dh \\
&= N^2\left( \int \left(\sup_{\check{g}\in \supp(\psi)} \|k(h^{-1}, \check{g}^{-1}x)\|_M \right)\cdot \|\chi^h - \chi\|_2\, dh \right)^2 \|\psi\|_{2,N}^2. 
\end{align*}
This completes the proof of the estimate in~(b). The corresponding estimate in~(c) can be shown similarly.
\end{proof}

Next we formulate and prove the analog of \cite[Lem.~4.3]{BecTak25}, which is the heart of the proof. Recall the notion of the Hausdorff metric discussed in Chapter~\ref{chap:Dynam_LR} (see page~\pageref{eq:Hausdorff_Metric}). The Hausdorff metric on the compact subsets of $\C$ is denoted by $d_H$ and is used to measure the distance of spectra of dynamically-defined operators. The induced Hausdorff metric on $\inv$ was denoted by $\delta_H$. 

With these notions at hand, we provide the key estimate for the Hausdorff distance between the spectra of two operators using a cut-off function \(0 \neq \chi \in L^\infty(G)\cap L^2(G)\). In concrete situations, this function will either be the characteristic function of a compact set (see Section~\ref{sec:Lamplighter} and Proposition~\ref{prop:Lamplighter_app} below) or a continuous function with compact support, as used in the proofs of Theorem~\ref{thm:Lipschitz-Spectrum_app} and Theorem~\ref{thm:LocalConst-Spectrum_app}.

\begin{lemma}\label{lem:po(3)}
Let \( (Z, G, \tau) \) be a dynamical system where $G$ is an unimodular lcsc group. 
For $N\in\N$, let \( k: G \times Z \to \C^{N\times N} \) be a bounded kernel with dynamically-defined operator $(A_x)_{x\in Z}$ such that $A_x$ is normal for every $x \in Z$. 
Let $0 \neq \chi \in L^\infty(G)\cap L^2(G)$ be bounded. 
Fix $X, Y \in \inv$.
Then for each $\rho > 0$,
\begin{align*}
d_H(\sigma(A_X), \sigma(A_Y))
\le \sup_{\substack{x \in X,\ y \in Y,\ z \in Z \\ d(x,y) < \delta_H(X, Y) + \rho}} 
	&\int_G \left( N\cdot\sup_{g\in \supp(\chi)}\|k(h, g^{-1} y) - k(h, g^{-1} x)\|_M \right)\, dh\\
	&\qquad + \int_G \left( \sup_{g\in G} \|k(h, g^{-1} z)\|_M \right) \cdot \frac{\|\chi^h - \chi\|_2}{\|\chi\|_2} \, dh.
\end{align*}
\end{lemma}

\begin{proof}
Fix $\rho > 0$ and let $E \in \sigma(A_X)$. 
By the definition of the Hausdorff metric, we need to prove
\begin{align*}
\dist(E,\sigma(A_Y))
\le \sup_{\substack{x \in X,\ y \in Y,\ z \in Z \\ d(x,y) < \delta_H(X, Y) + \rho}} 
	&\int_G \left( N\cdot\sup_{g\in \supp(\chi)}\|k(h, g^{-1} y) - k(h, g^{-1} x)\|_M \right)\\
	&\qquad + \left( \sup_{g\in G} \|k(h, g^{-1} z)\|_M \right) \cdot \frac{\|\chi^h - \chi\|_2}{\|\chi\|_2} \, dh.
\end{align*}
For any $\varepsilon > 0$, there exist $x \in X$ and $E' \in \sigma(A_x)$ such that $|E - E'| < \varepsilon$.
Without loss of generality, assume $E' \notin \sigma(A_Y)$.

Since $A_x$ is normal, there exists a $0\neq \psi \in C_c(G,\C^N)$ such that 
\[
\|(A_x - E')^\ast\psi\|_{2,N} = \|(A_x - E')\psi\|_{2,N} < \varepsilon \|\psi\|_{2,N},
\]
using Weyl's criterion, see eq.~\eqref{eq:ApproxEigenfct}.
For $j \in G$, define $\chi_j(g) := \chi(j^{-1}g)$, and set
\[
C := \frac{N}{\varepsilon} \int \sup_{g\in G}\|k(h, g^{-1} x)\|_M \cdot \frac{\|\chi^h - \chi\|_2}{\|\chi\|_2} \, dh.
\]
Note that $C$ is a finite value since $k$ is a bounded kernel.
We now estimate:
\begin{align*}
\int \|(A_x - E')^\ast [\chi_j] \psi\|_{2,N}^2 dj
		&= \int \left\|\big[(A_x - E')^\ast, [\chi_j]\big] \psi + [\chi_j] (A_x - E') ^\ast\psi \right\|_{2,N}^2 dj \\
	(\text{\small Peter-Paul's inequality})\quad 
		&\le  \int\scalebox{0.9}{$ \frac{1+C}{C} \big\|[(A_x - E')^\ast, \chi_j] \psi\big\|_{2,N}^2 + (1 + C) \big\|[\chi_j] (A_x - E')^\ast\psi\big\|_{2,N}^2 $} \, dj \\
	(\text{\small Lemma~\ref{lem:Technical_SpectralEst_app}~(a), (b)})\quad 
		&\le \frac{1+C}{C} \left( 
				\underbrace{N\int \sup_{g\in G}\|k(h, g^{-1} x)\|_M \cdot \|\chi^h - \chi\|_2 \, dh}_{=C\varepsilon \|\chi\|_2}
			\right)^2 \|\psi\|_{2,N}^2 \\
			&\quad + (1 + C) \ \|\chi\|_2^2 \ \|(A_x - E') \ \psi\|_{2,N}^2 \\
		&< (1+C) \ C \ \varepsilon^2 \ \|\chi\|_2^2 \ \|\psi\|_{2,N}^2
			+ (1 + C) \ \varepsilon^2 \ \|\chi\|_2^2 \ \|\psi\|_{2,N}^2.
\end{align*}
Define $\psi_j := [\chi_j] \psi$. Applying Lemma~\ref{lem:Technical_SpectralEst_app}~(a) to the previous estimate leads to
\begin{align*}
\int \|(A_x - E')^\ast \psi_j\|_{2,N}^2 dj
	< &\big((1+C) \ C \ \varepsilon^2 + (1 + C) \ \varepsilon^2\big) \ \int\|\psi_j\|_{2,N}^2 \, dj \\
	= &\int (C\ \varepsilon + \varepsilon)^2 \|\psi_j\|_{2,N}^2 \, dj.
\end{align*}
The previous estimate can only hold if there exists a $j\in G$ such that
\[
\|(A_x - E')^\ast \psi_j\|_{2,N}
	< (C\ \varepsilon + \varepsilon) \ \|\psi_j\|_{2,N}
	= \left( \int \sup_{g\in G}\|k(h, g^{-1} x)\|_M \cdot \frac{\|\chi^h - \chi\|_2}{\|\chi\|_2} \, dh + \varepsilon \right) \|\psi_j\|_{2,N}.
\]
Next, we use that $X$ and $Y$ are invariant subsets.
For this particular $x\in X$ and $j\in G$, there exists $y \in Y$ such that
\[
d(j^{-1}x, j^{-1}y) <  \delta_H(X, Y)  + \rho.
\]
Moreover, $A_y$ is normal and the previous considerations yield
\begin{align*}
\dist(E, \sigma(A_Y))
	&< \dist(E', \sigma(A_y)) + \varepsilon \\
	&< \|(A_y-E')^{-1}\|^{-1} + \varepsilon \\
	&\le \|(A_y - E')^\ast \tfrac{\psi_j}{\|\psi_j\|_{2,N}}\|_{2,N} + \varepsilon \\
	&<  \|(A_y - A_x)^* \tfrac{\psi_j}{\|\psi_j\|_{2,N}} \|_{2,N}
		+ \int \sup_{g\in G}\|k(h, g^{-1} x)\|_M \cdot \frac{\|\chi^h - \chi\|_2}{\|\chi\|_2}  \, dh + 2\varepsilon .
\end{align*}
We continue estimating the first term. By construction, we have $\supp(\psi_j) \subseteq \supp(\chi_j) = j \supp(\chi)$. Thus, Proposition~\ref{prop:Basic_DynDefOp_app}~(d) implies for 
\begin{align*}
\|(A_y - A_x)^* \tfrac{\psi_j}{\|\psi_j\|_{2,N}} \|_{2,N}
	\leq &N\cdot  \int_G \sup_{g \in \supp(\psi_j)} \|k(h,g^{-1}y) - k(h,g^{-1}x)\|_M \ dh \\
	= & \int_G N\cdot \sup_{g \in \supp(\chi)} \left\|k\big(h,g^{-1}(j^{-1}y)\big) - k\big(h,g^{-1}(j^{-1}x)\big) \right\|_M \ dh 
\end{align*}
Combining the last two estimates and the choice of $y\in Y$ before, the desired estimate is concluded by sending $\varepsilon$ to zero.
\end{proof}

\begin{remark}
\label{rem:Optimization-Regularity-Amenable}
Lemma~\ref{lem:po(3)} provides an estimate on the Hausdorff distance between dynamically-defined operators restricted to different dynamical subsystems, with two main contributions:
\begin{itemize}
\item The first term depends on the regularity of the kernel (i.e., the operator coefficients) with respect to the second variable \(Z\). Its size increases with the support of \(\chi\), as the group action explores larger regions in \(Z\).
\item The second term depends on the choice of the cut-off function \(\chi\). It becomes small if \(\chi\) approximates an invariant mean, which exists when the group \(G\) is amenable. Such approximations typically require \(\chi\) to have large support.
\end{itemize}
Hence, there is a trade-off: small support of \(\chi\) improves the first term, while large support improves the second. Balancing these effects leads to the characteristic square-root behavior in the spectral map; see Theorem~\ref{thm:Lipschitz-Spectrum_app}. We emphasize that this square-root behavior is optimal in general and reflects a conceptual insight rather than a technical artifact of the proof.
\end{remark}

Following the previous remark, this is the point where additional assumptions on the group become necessary -- in particular, we require the group to be amenable. 
In order to make the upper bound in Lemma~\ref{lem:po(3)} small, we require the term
\[
\frac{\|\chi^h - \chi\|_2}{\|\chi\|_2}
\]
to be small for a suitable choice of $\chi$. Suppose for the moment that $\chi$ is compactly supported, i.e., there exists a compact set \( K \subseteq G \) such that \( \chi(g) = 0 \) for all \( g \notin K \). Since $\chi^h$ denotes the right shift of $\chi$, this condition expresses that the $L^2$-mass on the ``boundary'', as captured by $\|\chi^h - \chi\|_2$, must be small relative to the total mass $\|\chi\|_2$.

More precisely, a unimodular locally compact second countable (lcsc) group \( G \) with Haar measure \( m_G \) is called \emph{amenable} if for every \( \varepsilon > 0 \) and every compact set \( K \subseteq G \), there exists a bounded function \( 0 \neq \chi:G \to [0,\infty) \) such that
\[
\sup_{h \in K} \|\chi^h - \chi\|_2 < \varepsilon \|\chi\|_2.
\]
In this case, one also says that \( G \) satisfies the Reiter condition on \( L^2 \); see \cite[Thm.~10.15]{EinWar17}.

One natural choice is to define \( \chi \) as the characteristic function supported on the inverse of a left-F\o lner sequence. A sequence \( (F_n)_{n \in \N} \) of compact subsets \( F_n \subseteq G \) is called a \emph{(left-)F\o lner sequence} if for every compact subset \( K \subseteq G \), we have
\[
\lim_{n \to \infty} \sup_{h \in K} \frac{m_G(hF_n \Delta F_n)}{m_G(F_n)} = 0,
\]
where \( hF_n \Delta F_n := (hF_n \setminus F_n) \cup (F_n \setminus hF_n) \) denotes the symmetric difference of these sets. 

For the Lamplighter group (discussed in Section~\ref{sec:Lamplighter}), this is done; see the proof of Proposition~\ref{prop:Lamplighter_app}. The Lamplighter group has exponential growth, and thus balls do not yield a F\o lner sequence. However, alternative sets can be employed to construct appropriate cut-off functions~$\chi$.

In \cite{BecTak25} and in the following sections, we primarily focus on groups of strict polynomial growth. These groups are amenable and admit suitable cut-off functions~$\chi$ supported on balls; see, for instance, \cite[Lem.~4.5]{BecTak25} and Equation~\eqref{eq:cut-off_ball-estim_app} below. Balls in such groups form a F\o lner sequence. Note, however, that we do not use characteristic functions on balls, but rather tent-shaped functions; see Remark~\ref{rem:Choice-cut_Interplay-amenable} and Figure~\ref{fig:CutOff} for an illustration. In fact, choosing tent-shaped functions optimizes the spectral estimates.

With the previous estimates at hand, one could prove an analog of \cite[Thm.~4.4]{BecTak25} -- stating the continuity of the spectrum -- for $N\in\N$. We omit this here and only provide the quantitative estimates. We also note that the continuity of the spectrum follows also by different methods from \cite{BBdN18}.

We now strengthen the assumptions of Proposition~\ref{prop:Semi-cont-Spectrum_app} by requiring explicit decay estimates of the off-diagonal entries of the kernel (linear decay) as well as regularity of the kernel in \( Z \) (either Lipschitz continuity or local constancy). These assumptions may be relaxed depending on the specific model, which then leads to modified spectral estimates.

Recall the notion of Lipschitz continuous dynamical system (Definition~\ref{def:LipschitzAction}) and strictly polynomially growing group (Definition~\ref{def:strict-polyn-growth}).

\begin{theorem}\label{thm:Lipschitz-Spectrum_app}
Let \( (Z, G, \tau) \) be a dynamical system where $G$ is an unimodular lcsc group with adapted metric $d_G$. 
For $N\in\N$, let \( k: G \times Z \to \C^{N\times N} \) be a bounded kernel with dynamically-defined operator $(A_x)_{x\in Z}$ such that $A_x$ is normal for every $x \in Z$. 
If furthermore
\begin{enumerate}[label=(\alph*)]
\item \label{item:Lip-cont:a}
\( (Z, G, \tau) \) is Lipschitz continuous: there exists $c_\tau > 0$ such that
\[
d(gx, gy) \le (c_\tau d_G(e,g) + 1)\, d(x, y), \qquad g\in G, \, x,y\in Z.
\]
\item \label{item:Lip-cont:b}
$G$ is strictly polynomially growing: there exist $\kappa > 0$ and $c_1 \ge c_0 > 0$ such that for each $r \ge 1$,
\[
c_0 r^\kappa \le m_G\big(B(r)\big) \le c_1 r^\kappa.
\]
\item \label{item:Lip-cont:c}
$k$ is linearly decaying: there exists $c_s > 0$ such that 
\[
\sup_{x \in Z} \int_G  \sup_{g\in G} \big\| k\big(h,g^{-1}x\big) \big\|_M \, d_G(e,h) \, dh \leq c_s.
\]
\item \label{item:Lip-cont:d}
$k$ is Lipschitz continuous:  there exists \( c_k \in L^1_+(G) \) such that for \( m_G \)-a.e.\ \( h \in G \),
\[
\|k(h,x) - k(h,y)\|_M \leq c_k(h) \cdot d(x,y), \qquad x,y \in Z.
\]

\end{enumerate}
Fix $X, Y \in \inv$. Then for each $r \ge 1$,
\[
d_H(\sigma(A_X), \sigma(A_Y))
\le N\|c_k\|_1 (c_\tau r + 1)\, \delta_H(X, Y)
+ \frac{c_s}{r} \left( \frac{(2 + \kappa)^{2 + \kappa} c_1}{2 \kappa^\kappa c_0} \right)^{1/2},
\]
and 
\[
\delta_H(X, Y) \le 1 \qquad \Longrightarrow\qquad
	d_H(\sigma(A_X), \sigma(A_Y)) \le C\, \delta_H(X, Y)^{1/2},
\]
where
\[
C :=  N \|c_k\|_1 (c_\tau+1) + c_s \left( \frac{(2 + \kappa)^{2 + \kappa} c_1}{2 \kappa^\kappa c_0} \right)^{1/2}.
\]
\end{theorem}

\begin{proof}
For $r \ge 1$, define the cut-off function 
\[
\chi_r\in C_c(G), \qquad \chi_r(g) := \left( \frac{r - d_G(e,g)}{r} \right) \cdot \mathbf{1}_{B(r)}(g).
\] 
Observe that $\chi_r$ is supported on the ball $B(r):=B(e,r)$, see a sketch in Figure~\ref{fig:CutOff}~(b).
Recall for $h\in G$ that $\chi_r^h\in C_c(G)$ denotes the function $\chi_r^h(g):= \chi_r(gh),\, g\in G$.
Since $G$ has strict polynomial growth, \cite[Lem.~4.5]{BecTak25} leads to the estimate for $h\in G$,
\begin{equation}
\label{eq:cut-off_ball-estim_app}
\frac{\|\chi_r^h - \chi_r\|_2}{\|\chi_r\|_2} \leq
	\eta_r(h) 
	:= \min\left\{ \frac{d_G(e,h)}{r} \left( \frac{(2 + \kappa)^{2 + \kappa} c_1}{2 \kappa^\kappa c_0} \right)^{1/2}, \sqrt{2} \right\}
\end{equation}

Let $\rho > 0$.
By Lemma~\ref{lem:po(3)}, we estimate:
\begin{align*}
d_H(\sigma(A_X), \sigma(A_Y))
	&\le \sup_{\substack{x \in X,\ y \in Y,\ z \in Z \\ d(x,y) < \delta_H(X, Y) + \rho}} 
		\int_G \left( N\cdot\sup_{g\in \supp(\chi_r)}\|k(h, g^{-1} y) - k(h, g^{-1} x)\|_M \right) \, dh\\
		&\quad\qquad + \int_G \left( \sup_{g\in G} \|k(h, g^{-1} z)\|_M \right) \cdot \frac{\|\chi_r^h - \chi_r\|_2}{\|\chi_r\|_2} \, dh\\
\left(\substack{\small \text{$k$ Lipschitz \ref{item:Lip-cont:d},} \\ \text{eq.~\eqref{eq:cut-off_ball-estim_app}}}\right)\quad
	&\le \sup_{\substack{x, y, z\in Z \\ d(x, y) < \delta_H(X,Y) + \rho}}
		N \cdot \int_G c_k(h) \cdot \sup_{g\in \supp(\chi_r)} d(g^{-1} y, g^{-1} x) \, dh \\
		&\quad\qquad + \sup_{z \in Z} \int_G \sup_{g\in G}\|k(h, g^{-1} z)\|_M \cdot \eta_r(h)\, dh\\
\left(\substack{\small \text{action Lipschitz \ref{item:Lip-cont:a},} \\ \supp(\chi_r) = B(r), \\ \text{$k$ linearly decaying}}\right) \quad
	&\le N \ \|c_k\|_1 \ (c_\tau r + 1)(\delta_H(X,Y) + \rho) + \frac{c_s}{r} \left( \frac{(2 + \kappa)^{2 + \kappa} c_1}{2 \kappa^\kappa c_0} \right)^{1/2}
\end{align*}
This proves the first claimed estimate by sending $\rho$ to zero. 
If $\delta_H(X,Y) \le 1$ and we choose
\[
r := \left( \frac{1}{\delta_H(X,Y)} \right)^{1/2} \ge 1,
\]
then
\begin{align*}
d_H(\sigma(A_X), \sigma(A_Y))
&\le N \, \|c_k\|_1 \, (c_\tau r + 1) \,  \delta_H(X,Y)
+ \frac{c_s}{r} \left( \frac{(2 + \kappa)^{2 + \kappa} c_1}{2 \kappa^\kappa c_0} \right)^{1/2} \\
&\le C\, \delta_H(X,Y)^{1/2}
\end{align*}
follows.
\end{proof}

\begin{figure}[htb]
    \centering
    \includegraphics[scale=1]{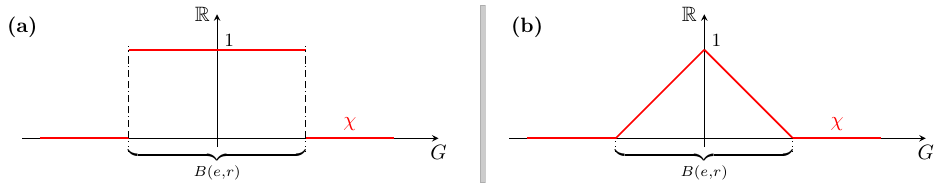}
    \captionsetup{width=0.95\linewidth}
    \caption{Different choices for the cut-off function $\chi_r$ supported on the ball $B(e,r)$ are illustrated. In panel~(a), $\chi_r$ is chosen as the characteristic function of the ball, while in panel~(b), it is given by a tent-shaped function.}
    \label{fig:CutOff}
\end{figure}

\begin{remark}
\label{rem:Choice-cut_Interplay-amenable}
In the final part of the proof (after the first estimate was established), we observe the interplay between the amenability of the group and the regularity of the kernel. The first summand becomes small when the radius is small, while the second term decreases as the radius grows. The smallness of the second term originates from the amenability property. This trade-off ultimately results in the characteristic square-root behavior in the spectral estimates.

Furthermore, we emphasize the importance of the specific choice of the cut-off function $\chi$, as illustrated by the two alternatives shown in Figure~\ref{fig:CutOff}. For the almost Mathieu operator, the characteristic function of a ball $B(e,r)$ was initially employed in~\cite{ChElYu90}, which led to a spectral estimate with a $\tfrac{1}{3}$-H\"older behavior. In contrast, the use of a tent function in~\cite{AvrMouSim90} resulted in a $\tfrac{1}{2}$-H\"older behavior, consistent with Theorem~\ref{thm:Lipschitz-Spectrum_app}. This exponent is in fact optimal for the almost Mathieu operator, as established in~\cite{BeRa90}; see also the discussion in~\cite{BecTak25}. In light of this, a tent-shaped cut off function is preferable than a characteristic function.
\end{remark}

Next we prove the analog of \cite[Thm.~4.7]{BecTak25} for general $N\in\N$.

\begin{theorem}\label{thm:LocalConst-Spectrum_app}
Let \( (Z, G, \tau) \) be a dynamical system where $G$ is an unimodular lcsc group with adapted metric $d_G$. 
For $N\in\N$, let \( k: G \times Z \to \C^{N\times N} \) be a bounded kernel with dynamically-defined operator $(A_x)_{x\in Z}$ such that $A_x$ is normal for every $x \in Z$. 
If furthermore
\begin{enumerate}[label=(\alph*)]
\item \label{item:Loc-const:a}
\( (Z, G, \tau) \) is Lipschitz continuous: there exists $c_\tau > 0$ such that
\[
d(gx, gy) \le (c_\tau d_G(e,g) + 1)\, d(x, y), \qquad g\in G,\, x,y\in Z.
\]
\item \label{item:Loc-const:b}
$G$ is strictly polynomially growing: there exist $\kappa \ge 0$ and $c_1 \ge c_0 > 0$ such that for each $r \ge 1$,
\[
c_0 r^\kappa \le m_G\big(B(r)\big) \le c_1 r^\kappa.
\]
\item \label{item:Loc-const:c}
$k$ is linearly decaying: there exists $c_s > 0$ such that 
\[
\sup_{x \in Z} \int_G  \sup_{g\in G} \big\| k\big(h,g^{-1}x\big) \big\|_M \, d_G(e,h)  \, dh \leq c_s.
\]
\item \label{item:Loc-const:d}
$k$ is locally constant: there exists \( c_k \in L^\infty_+(G) \) such that for \( m_G \)-a.e.\ \( h \in G \),
\[
d(x,y) < \frac{1}{c_k(h)} \quad \Rightarrow \quad k(h,x) = k(h,y).
\]
\end{enumerate}
Fix $X, Y \in \inv$. Then
\[
\delta_H(X, Y) \le \delta
\;\Longrightarrow\;
d_H\big(\sigma(A_X), \sigma(A_Y)\big) \le C\, \delta_H(X, Y),
\]
where
\[
\delta := \frac{1}{(\|c_k\|_\infty + 1)(c_\tau + 1)}
\quad \text{and} \quad
C := (\|c_k\|_\infty + 1)(c_\tau + 1) c_s \left( \frac{(2 + \kappa)^{2 + \kappa} c_1}{2 \kappa^\kappa c_0} \right)^{1/2}.
\]
\end{theorem}

\begin{proof}
Let $X,Y\in\inv$ be such that $\delta_H(X, Y) \le \delta$. Then the choice of $\delta$ yields
\[
\frac{1}{(\|c_k\|_\infty + 1)\, \delta\, c_\tau} - \frac{1}{c_\tau} = 1 
\qquad\text{and}\qquad
\frac{c_\tau}{1 - (\|c_k\|_\infty + 1)\, \delta_H(X, Y)} \le c_\tau + 1.
\]
With this at hand, define the radius
\[
r := \frac{1}{(\|c_k\|_\infty + 1)\, \delta_H(X, Y)\, c_\tau} - \frac{1}{c_\tau}
\]
satisfying
\begin{equation}
\label{eq:est-r-loc-const_app}
r\ge 1
\qquad\text{and}\qquad
\frac{1}{r} \le (\|c_k\|_\infty + 1)(c_\tau + 1)\, \delta_H(X, Y).
\end{equation}
Define the cut-off function 
\[
\chi_r\in C_c(G), \qquad \chi_r(g) := \left( \frac{r - d_G(e,g)}{r} \right) \cdot \mathbf{1}_{B(r)}(g).
\] 
Observe that $\chi_r$ is supported on the ball $B(r):=B(e,r)$.
Recall for $h\in G$ that $\chi_r^h\in C_c(G)$ denotes the functiong $\chi_r^h(g):= \chi_r(gh),\, g\in G$.
Since $G$ has strict polynomial growth, \cite[Lem.~4.5]{BecTak25} leads to the estimate for $h\in G$,
\begin{equation}
\label{eq:cut-off_ball-estim2_app}
\frac{\|\chi_r^h - \chi_r\|_2}{\|\chi_r\|_2} \leq
	\eta_r(h) 
	:= \min\left\{ \frac{d_G(e,h)}{r} \left( \frac{(2 + \kappa)^{2 + \kappa} c_1}{2 \kappa^\kappa c_0} \right)^{1/2}, \sqrt{2} \right\}
\end{equation}
Let \(\rho := \tfrac{\delta_H(X, Y)}{\|c_k\|_\infty}>0.\) Thus, if $d(x, y) < \delta_H(X, Y) + \rho$, then 
\[
\sup_{g\in \supp(\chi_r)} d(g^{-1} y, g^{-1} x)
	\le (c_\tau r + 1)\, d(x, y)
	< \frac{1}{\|c_k\|_\infty}
\]
follows since the group action is Lipschitz continuous \ref{item:Loc-const:a} and by the choice of $r$. Since $k$ is locally constant \ref{item:Loc-const:d}, we conclude
\[
d(x, y) < \delta_H(X, Y) + \rho
	\quad\Longrightarrow\quad
	\sup_{g\in \supp(\chi_r)}\|k(h, g^{-1} y) - k(h, g^{-1} x)\|_M = 0.
\]
Plugging this in Lemma~\ref{lem:po(3)}, we derive
\begin{align*}
d_H(\sigma(A_X), \sigma(A_Y))
	&\le \sup_{\substack{x \in X,\ y \in Y,\ z \in Z \\ d(x,y) < \delta_H(X, Y) + \rho}} 
		\int_G \left( N\cdot\sup_{g\in \supp(\chi_r)}\|k(h, g^{-1} y) - k(h, g^{-1} x)\|_M \right) \, dh\\
		&\quad\qquad + \int_G \left( \sup_{g\in G} \|k(h, g^{-1} z)\|_M \right) \cdot \frac{\|\chi_r^h - \chi_r\|_2}{\|\chi_r\|_2} \, dh\\
\left(\substack{\small \text{previous considerations,} \\ \text{eq.~\eqref{eq:cut-off_ball-estim2_app}}}\right)\quad
	&\le 0 + \sup_{z \in Z} \int \sup_{g\in G}\|k(h, g^{-1} z)\|_M \cdot \eta_r(h)\, dh\\
\left(\small \text{$k$ linearly decaying}\right) \quad
	&\le \frac{c_s}{r} \left( \frac{(2 + \kappa)^{2 + \kappa} c_1}{2 \kappa^\kappa c_0} \right)^{1/2}\\
\left(\small \text{eq.~\eqref{eq:est-r-loc-const_app}}\right) \quad
	&\le C \delta_H(X,Y). \qedhere
\end{align*}
\end{proof}

\section{Spectral estimates for the Lamplighter group}
\label{App:SpectralEstimates_Lampl}

As discussed in Section~\ref{sec:Lamplighter}, the proof strategy for strictly polynomially growing groups applies to general amenable groups. For illustration, we prove here spectral estimates for symbolic dynamical systems over the lamplighter group as stated in Proposition~\ref{prop:Lamplighter}.

First, recall the definition of the lamplighter group
\[
G := \left\{ (m,\gamma) \;\middle|\; m \in \Z,\; \gamma \in \bigoplus_{j\in\Z} \Z_2 \right\},
\]
where \( \Z_2 \) denotes the cyclic group of order two. The second component \( \gamma \) is a finitely supported map \( \gamma:\Z\to\Z_2 \), since we consider the direct sum.

The group multiplication is defined by
\[
(m,\gamma)(n,\eta) := \left(m+n,\, S^n(\gamma) + \eta\right),
\]
where \( S : \bigoplus_{j\in\Z} \Z_2 \to \bigoplus_{j\in\Z} \Z_2 \) is the shift operator given by \( S(\gamma)(j) := \gamma(j+1) \) for all \( j \in \Z \).

The inverse of an element in \( G \) is given by
\[
(m,\gamma)^{-1} = \left(-m,\, S^{-m}(\gamma)\right).
\]
The set
\[
S = \{t, at\} = \{(1,0),\, (1, \delta_1)\}
\]
generates the lamplighter group, where \( t = (1,0) \) and \( a = (0, \delta_0) \).
Here, \( \delta_j : \Z \to \Z_2 \) denotes the sequence that vanishes everywhere except at the \( j \)th position, where it is one.

A generating set \( S \) can be used to define the word metric on \( G \), namely, \( d_G(g,h) \) is given by the minimal number \( m \in \N_0 \) of elements \( s_j \in S \cup S^{-1} \) such that
\[
g^{-1}h = s_1 \cdots s_m,
\]
with the convention that \( m = 0 \) if \( g^{-1}h = e \). Note that the word metric can also be interpreted as the shortest-path metric on the Cayley graph of \( G \) with respect to the generating set \( S \). It is straightforward to verify that every word metric is adapted on a discrete group \( G \).

In the following, we work with the word metric \( d_G \) defined by the generating set \( S = \{t, at\} \). This word metric is studied and computed in \cite[Sec.~2]{JoKe18}.

We now consider the symbolic dynamical system \( (\Aa^G, G, \tau) \) (see Example~\ref{ex:SymbDynSyst}) for the lamplighter group and a finite alphabet \( \Aa \). Recall from Section~\ref{subsec:symb_syst_as_weighted_Delone} that the metric \( d_G \) induces a metric \( \dAG \) on \( \Aa^G \). Moreover, \( \dAG \) induces the Hausdorff metric \( \delta_\Aa \) on the set \( \inv \) of invariant, closed, and non-empty subsets.

We now restate Proposition~\ref{prop:Lamplighter} and proceed with its proof.

\begin{proposition}
\label{prop:Lamplighter_app}
Let \( (\Aa^G,G,\tau) \) be a symbolic dynamical system for a finite alphabet $\Aa$ where $G$ is the lamplighter group equipped with the word metric $d_G$ defined by the generating set $S=\{t,at\}$.
For $N\in\N$, let \( k: G \times Z \to \C^{N\times N} \) be a bounded kernel with dynamically-defined operator $(A_x)_{x\in Z}$ such that $A_x$ is normal for every $x \in Z$. 
If furthermore
\begin{enumerate}[label=(\alph*)]
\item 
$k$ is decaying: there exists $c_s > 0$ such that 
\[
\sup_{x \in Z} \sum_{(j,\eta)\in G}  \sup_{g\in G} \big\| k\big((j,\eta),g^{-1}x\big)\|_M \, |j| \leq c_s.
\]
\item 
$k$ is locally constant: there exists \( 0< c_k \in \ell^\infty(G) \) such that for all \( h \in G \),
\[
d(x,y) < \frac{1}{c_k(h)} \quad \Rightarrow \quad k(h,x) = k(h,y).
\]
\end{enumerate}
Then there exists a constant $C>0$ such that
\[
d_H\big(\sigma(A_X), \sigma(A_Y)\big) \le C\, \delta_H(X, Y), \qquad X,Y\in\inv.
\]
\end{proposition}

\begin{proof}
Since $G$ is discrete, $G$ itself is a Delone set of $G$ (for suitable parameters) and so the symbolic dynamical system \( (\Aa^G,G,\tau) \) is Lipschitz continuous by Proposition~\ref{prop:Conv_Symb-Dyn-Syst}~(b). Thus, there exists \( c_\tau > 0 \) such that
\[
\dAG(gx, gy) \le (c_\tau d_G(e,g) + 1)\, \dAG(x, y),
\qquad g \in G,\ x, y \in Z.
\]
We note at this point that the metric \( \dAG \) on \( \Aa^G \) depends on the choice of the metric \( d_G \), since the latter determines the shape of balls in \( G \); see Section~\ref{subsec:symb_syst_as_weighted_Delone}.

Following~\cite[Example~1.50]{Butk01} (see also \cite[Ex.~3.7]{Pog10_Dipl}), define for each $\ell \in \N$ the finite sets
\[
F_\ell := \left\{ (m,\gamma) \;\middle|\; |m| \leq \ell,\; \gamma(j) = 0 \text{ for all } |j| > 2 \ell \right\}.
\]
For two sets $A,B \subseteq G$, denote their symmetric difference by $A \,\Delta\, B := (A \setminus B) \cup (B \setminus A)$.

According to~\cite[Example~1.50]{Butk01}, for each $h = (j,\eta) \in G$, the following estimate holds:
\[
\frac{\sharp(F_\ell \,\Delta\, h F_\ell)}{\sharp F_\ell} \leq \frac{|j|}{2\ell+1},
\]
which tends to zero as $\ell \to \infty$. Hence, the sequence $(F_\ell)_{\ell \in \N}$ forms a (left-)F\o lner sequence.

Define the indicator function $\chi_\ell : G \to \{0,1\}$ by\footnote{Note that $(F_\ell)_{\ell \in \N}$ defines a left-F\o lner sequence (action from the left). Then its inverse $(F_\ell^{-1})_{\ell \in \N}$ defines a right-F\o lner sequence. Since we need a cut-off function where we act from the right, $\chi_\ell$ is the characteristic function of $F_\ell^{-1}$.}
\[
\chi_\ell(g) :=
\begin{cases}
1, & \text{if } g \in F_\ell^{-1}, \\
0, & \text{otherwise}.
\end{cases}
\]
Define the shifted function
\[
\chi_\ell^h : G \to [0,\infty), \qquad \chi_\ell^h(g) := \chi_\ell(gh).
\]
Using the previous estimate, we conclude that for any $h = (j,\eta) \in G$,
\begin{equation}
\label{eq:Decay_CutOff_Lampl}
\frac{\| \chi_\ell^h - \chi_\ell \|_2}{\|\chi_\ell \|_2} 
	\leq \frac{\sharp(F_\ell \,\Delta\, h^{-1} F_\ell)^{-1}}{\sharp F_\ell^{-1}} 
	= \frac{\sharp(F_\ell \,\Delta\, h^{-1} F_\ell)}{\sharp F_\ell} 
	\leq \frac{|j|}{2\ell+1},
\end{equation}
where we used that the cardinality of an inverse of a set equals the cardinality of the set.

Since \( d_G \) is the word metric defined by the generating set \( S = \{t, at\} \) of \( G \), \cite[Lem.~2.2]{JoKe18} allows us to estimate \( d_G(e,h) \) for \( h := (j,\eta) \in F_\ell \). Specifically, there exists a constant \( C_F > 1 \) such that
\[
\sup_{h\in F_\ell} d_G(e,h)
	\leq C_F \ell
\]
Choose 
\[
\delta := \frac{1}{(\|c_k\|_\infty + 1)(c_\tau \, C_F + 1)}>0.
\]
Let $X,Y\in\inv$ be such that $\delta_H(X, Y) \le \delta$. Then the choice of $\delta$ yields
\[
\frac{1}{(\|c_k\|_\infty + 1)\, \delta\, c_\tau \, C_F} - \frac{1}{c_\tau \, C_F} = 1 
\qquad\text{and}\qquad
\frac{c_\tau \, C_F}{1 - (\|c_k\|_\infty + 1)\, \delta_H(X, Y)} \le c_\tau \, C_F + 1.
\]
Define the $\ell_0\in\N$ as follows
\[
\ell_0 := \left\lfloor\frac{1}{(\|c_k\|_\infty + 1)\, \delta_H(X, Y)\, c_\tau \, C_F} - \frac{1}{c_\tau \, C_F}\right\rfloor
\]
where $\lfloor z \rfloor :=\max \set{n\in\N}{n\leq z}$. Then the previous considerations imply
\begin{equation}
\label{eq:est-r-loc-const_lampl}
\ell_0\ge 1
\qquad\text{and}\qquad
\frac{1}{2\ell_0} \le \frac{1}{\ell_0+1} \le (\|c_k\|_\infty + 1)(c_\tau \, C_F + 1)\, \delta_H(X, Y).
\end{equation}
Let \(\rho := \tfrac{\delta_H(X, Y)}{\|c_k\|_\infty}>0.\) Thus, if $d(x, y) < \delta_H(X, Y) + \rho$, then 
\begin{align*}
\sup_{h\in \supp(\chi_{\ell_0})} d(h^{-1} y, h^{-1} x)
	\le &\sup_{h\in F_{\ell_0}} (c_\tau \ d_G(e,h) + 1)\, d(x, y)\\
	\le &\ (c_\tau \ C_F \ \ell_0 + 1)\, d(x, y)\\
(\text{\small choice of $\ell_0$})\qquad	< &\frac{1}{(\|c_k\|_\infty + 1)\, \delta_H(X, Y)} \big(\delta_H(X, Y)+\rho\big)\\
	= &\frac{1}{\|c_k\|_\infty}
\end{align*}
follows since the group action is Lipschitz continuous \ref{item:Loc-const:a}. Since $k$ is locally constant (b), we conclude
\[
d(x, y) < \delta_H(X, Y) + \rho
	\quad\Longrightarrow\quad
	\sup_{g\in \supp(\chi_{\ell_0})}\|k(h, g^{-1} y) - k(h, g^{-1} x)\|_M = 0.
\]
Plugging this in Lemma~\ref{lem:po(3)}, we derive
\begin{align*}
d_H(\sigma(A_X), \sigma(A_Y))
	&\le \sup_{\substack{x \in X,\ y \in Y,\ z \in Z \\ d(x,y) < \delta_H(X, Y) + \rho}} 
		\sum_{h\in G} \left( N\cdot\sup_{g\in \supp(\chi_{\ell_0})}\|k(h, g^{-1} y) - k(h, g^{-1} x)\|_M \right)\\
		&\quad\qquad + \sum_{h\in G} \left( \sup_{g\in G} \|k(h, g^{-1} z)\|_M \right) \cdot \frac{\|\chi_{\ell_0}^h - \chi_{\ell_0}\|_2}{\|\chi_{\ell_0}\|_2}\\
\left(\substack{\small \text{previous considerations,} \\ \text{eq.~\eqref{eq:Decay_CutOff_Lampl} }}\right)\quad
	&\le 0 + \sup_{z \in Z} \sum_{(j,\eta)\in G} \sup_{g\in G}\|k((j,\eta), g^{-1} z)\|_M \, \frac{|j|}{2\ell_0+1} \\
\left(\small \text{$k$ decaying}\right) \quad
	&\le \frac{1}{2\ell_0} c_s\\
\left(\small \text{eq.~\eqref{eq:est-r-loc-const_lampl} }\right) \quad
	&\le (\|c_k\|_\infty + 1)(c_\tau + 1)\, c_s\, \delta_H(X, Y).
\end{align*}
The previous estimate holds under the assumption that $\delta_H(X, Y) \le \delta$. If $\delta_H(X, Y) \ge \delta$, one gets
\[
d_H(\sigma(A_X), \sigma(A_Y))
	\le 2\sup_{x\in Z} \|A_x\|
	\le 2\sup_{x\in Z} \|A_x\|\, \tfrac{1}{\delta}\, \delta_H(X,Y),
\]
finishing the proof.
\end{proof}

We note that, in the latter proof, the cut-off function \( \chi_\ell \) was chosen to be the characteristic function on the F\o lner set \( F_\ell^{-1} \). The estimates might improve by choosing a tent function instead, particularly in the case of Lipschitz continuous kernels; see the discussion in Remark~\ref{rem:Choice-cut_Interplay-amenable}.